%% file: HOESMMFD.tex
\documentclass[11pt]{elsarticle}
\usepackage{amsmath,amssymb,amsthm}
\usepackage{mathrsfs}
\usepackage[bookmarks=true,
  bookmarksnumbered=true,
  bookmarksopen=true,
  colorlinks,
  pdfborder=001,
linkcolor=black]{hyperref}
\usepackage{bm}
\usepackage{array}
\usepackage{float}
\usepackage{color}
\usepackage{lineno}
\usepackage[labelformat=simple]{subcaption}

\usepackage{stmaryrd}
\usepackage{extarrows}
\usepackage{multirow}
\newtheorem{thm}{Theorem}[section]
\newtheorem{lem}[thm]{Lemma}
\newtheorem{rmk}{Remark}[section]
\newtheorem{example}{Example}[section]
\newtheorem{proposition}{Proposition}[section]
\newtheorem{definition}{Definition}[section]
\newproof{pf}{Proof}
\numberwithin{equation}{section}
\numberwithin{figure}{section}
\numberwithin{table}{section}
\allowdisplaybreaks

\newcommand\diag{\mathrm{diag}}

\newcommand\dd{\mathrm{d}}

\newcommand\bF{\bm{F}}

\newcommand\bv{\bm{v}}

\newcommand\bx{\bm{x}}

\newcommand\bV{\bm{V}}
\newcommand\bU{\bm{U}}

\newcommand\bT{\bm{T}}

\newcommand{\Ucurv}{\bm{\mathcal{U}}}
\newcommand{\Fcurv}{\bm{\mathcal{F}}}
\newcommand{\Bcurv}{\mathcal{B}}
\newcommand{\qcurv}{\mathfrak{q}}
\newcommand{\UCurvFlux}{\bm{\mathring{{U}}}}
\newcommand{\FCurvFlux}{\bm{\mathring{{F}}}}
\newcommand{\twop}{{\scriptsize2p\mbox{\text{th}}}}
\newcommand{\wth}{{\scriptsize w\mbox{\text{th}}}}
\newcommand{\WENO}{{\scriptsize\mbox{\text{WENO}}}}

\newcommand\vx{v_1}
\newcommand\vy{v_2}
\newcommand\vz{v_3}
\newcommand\ux{u^x}
\newcommand\uy{u^y}
\newcommand\uz{u^z}

\newcommand\Bx{B_1}
\newcommand\By{B_2}
\newcommand\Bz{B_3}
\newcommand\pt{p_\text{tot}}
\newcommand\bb{\abs{b}^2}
\newcommand\BB{\abs{\bm{B}}^2}

\newcommand\vB{\bv\cdot\bm{B}}

\newcommand\pd[2]{\dfrac{\partial {#1}}{\partial {#2}}}
\newcommand\abs[1]{\lvert #1 \rvert}
\newcommand\jump[1]{\llbracket #1 \rrbracket}
\newcommand\mean[1]{\{\!\!\{ #1 \}\!\!\}}
\newcommand\meanln[1]{\{\!\!\{ #1 \}\!\!\}^{\text{ln}}}
\newcommand\jumpangle[1]{\langle\!\langle #1 \rangle\!\rangle}

\begin{document}

\begin{frontmatter}

  \title{High-order accurate entropy stable adaptive moving mesh finite difference schemes for special relativistic (magneto)hydrodynamics}

  \author{Junming Duan}
  \ead{duanjm@pku.edu.cn}
  \author{Huazhong Tang\corref{cor1}}
  \ead{hztang@math.pku.edu.cn}
  \address{Center for Applied Physics and Technology, HEDPS and LMAM,
  School of Mathematical Sciences, Peking University, Beijing 100871, P.R. China}
  \cortext[cor1]{Corresponding author. Fax:~+86-10-62751801.}

  \begin{abstract}
    This paper develops high-order accurate entropy stable (ES) adaptive moving mesh finite difference schemes for the two- and three-dimensional special relativistic hydrodynamic (RHD) and magnetohydrodynamic (RMHD) equations,
    which is the high-order accurate extension of [J.M. Duan and H.Z. Tang, Entropy stable adaptive moving mesh schemes for 2D and 3D special relativistic hydrodynamics, J. Comput. Phys., 426(2021), 109949].
    The key point is the derivation of the higher-order accurate entropy conservative (EC) and ES finite difference schemes in the curvilinear coordinates
    by carefully dealing with
the discretization of the temporal and spatial metrics and the Jacobian of the coordinate
transformation and constructing the high-order EC and ES fluxes with the discrete metrics. The spatial derivatives in the source
terms of the symmetrizable RMHD equations and the geometric conservation laws are discretized by
using the linear combinations of the corresponding second-order case to obtain high-order accuracy.
Based on the proposed high-order accurate EC schemes and the high-order accurate dissipation
terms built on the WENO reconstruction, the high-order accurate ES schemes are obtained for
the RHD and RMHD equations in the curvilinear coordinates.
The mesh iteration redistribution or adaptive moving mesh strategy is built on the minimization of the
mesh adaption functional.
Several
numerical tests are conducted to validate the shock-capturing ability and high efficiency of our
high-order accurate ES adaptive moving mesh methods on the parallel computer system with the
MPI communication. The numerical results show that the high-order accurate ES adaptive moving
mesh schemes outperform both their counterparts on the uniform mesh and the second-order ES
adaptive moving mesh schemes.
  \end{abstract}


\end{frontmatter}

\input{Intro}

\input{EntropyCondition}

\input{ECScheme}

\input{ESScheme}

\input{MovingMesh}

\input{NumTests}

\input{Conclusion}

\section*{Acknowledgments}
The authors were partially supported by
the National Key R\&D Program of China, Project Number 2020YFA0712000,
Science Challenge Project (No. TZ2016002),
and High-performance Computing Platform of Peking University.

\input{App}


\input{HOESMMFD.bbl}
\end{document}

%% file: Intro.tex
\section{Introduction}
This paper is concerned with the high-order accurate numerical schemes for the
special relativistic hydrodynamic (RHD) and magnetohydrodynamic (RMHD) equations,
which consider the relativistic description for the dynamics of the fluid (gas) at nearly the speed of light when the astrophysical phenomena are investigated from stellar to galactic scales, e.g. the core collapse supernovae, the coalescing neutron stars, the active galactic nuclei, the formation of black holes,
the gamma-ray bursts, and the superluminal jets etc.
In the covariant form, the four-dimensional space-time RMHD equations can be written as
follows \cite{Anile1987On}
\begin{equation}\label{eq:RMHD}
  \partial_{\alpha}(\rho u^\alpha)=0,~
  \partial_{\alpha} (\mathrm{T}^{\alpha\beta})=0,~
  \partial_{\alpha} (\Psi^{\alpha\beta})=0,
\end{equation}
where the Einstein summation convention has been used, $\rho$ and $u^\alpha$ denote the rest-mass density and the
four-velocity vector, respectively,
$\partial_{\alpha}$ denotes the covariant derivative operator with respect to
the four-dimensional space-time coordinates $(t,x^1,x^2,x^3)$,
the Greek indices $\alpha,\beta$ run from $0$ to $3$.
In \eqref{eq:RMHD}, the tensor $\Psi^{\alpha\beta}$ can be expressed by $u^\alpha$ and four-dimensional magnetic field $b^\alpha$ as
\begin{equation}
  \Psi^{\alpha\beta}=u^\alpha b^\beta-u^\beta b^\alpha,
\end{equation}
and  the energy-momentum tensor $\mathrm{T}^{\alpha\beta}$  can be decomposed into
the fluid part $\mathrm{T}^{\alpha\beta}_{f}$ and the electromagnetic part
$\mathrm{T}^{\alpha\beta}_{m}$, defined by
\begin{align}
  \mathrm{T}^{\alpha\beta}_{f}&=\rho hu^\alpha u^\beta +pg^{\alpha\beta}, \\
  \mathrm{T}^{\alpha\beta}_{m}&=\bb(u^\alpha u^\beta+g^{\alpha\beta}/2)-b^\alpha b^\beta, \label{eq:RMHD2}
\end{align}
where $p$ and $h=1+e+p/\rho$ are respectively the pressure and specific enthalpy,
with $e$ the specific internal energy.
Throughout this paper, the metric tensor $g^{\alpha\beta}$ is taken as the
Minkowski tensor, i.e. $g^{\alpha\beta}=\pm\diag\{-1,1,1,1\}$, and units in which the speed of light is equal to one will be used.
The relations between the four-vectors $u^\alpha$ and $b^\alpha$ and the spatial components
of the velocity $\bv=(\vx,\vy,\vz)$ and the laboratory magnetic field
$\bm{B}=(\Bx,\By,\Bz)$ are
\begin{align}
  &u^\alpha=W(1,\bv),\\
  &b^\alpha=W\left(\vB,\frac{\bm{B}}{W^2}+\bv(\vB)\right),
\end{align}
where $W=1/\sqrt{1-\abs{\bv}^2}$ is the Lorentz factor. It is easy to
verify the following relations
\begin{align*}
  u^\alpha u_\alpha=-1,\quad u^\alpha b_\alpha=0,\quad \bb\equiv b^\alpha b_\alpha=\frac{\BB}{W^2}+(\vB)^2.
\end{align*}
To close the system \eqref{eq:RMHD}-\eqref{eq:RMHD2}, this paper
considers the equation of state (EOS) for the perfect gas
\begin{equation}\label{eq:EOS}
  p=(\Gamma-1)\rho e,
\end{equation}
with the adiabatic index $\Gamma\in(1,2]$.
The RHD case can be obtained by setting $\bm{B}\equiv \bm{0}$.

Numerical simulation is a powerful way to help us better understand the physical mechanisms in the RHDs and RMHDs.
For the computational purpose, the system \eqref{eq:RMHD}-\eqref{eq:EOS} is rewritten in a
lab frame as follows
\begin{equation}\label{eq:RMHDdiv1}
  \pd{\bU}{t}+\sum_{k=1}^d \pd{\bF_k(\bU)}{x_k}=0,
\end{equation}
with the divergence-free constraint on the magnetic field
\begin{equation}\label{eq:divB}
  \sum_{k=1}^d \pd{B_k}{x_k}=0,
\end{equation}
where $\bU$ and $\bF_k$ are respectively
the conservative variable vector and the flux vector in the $x_k$-direction, and defined by
\begin{equation}\label{eq:RMHDdiv2}
  \begin{aligned}
    &\bU=(D,\bm{m},E,\bm{B})^\mathrm{T},\\
    &\bF_k=(Dv_k,\bm{m}v_k-B_k(\bm{B}/W^2+(\vB)\bv)+\pt\bm{e}_k,m_k,v_k\bm{B}-B_k\bv)^\mathrm{T},
  \end{aligned}
\end{equation}
with the mass density $D=\rho W$, the momentum density
$\bm{m}=(\rho hW^2+\BB)\bv-(\bv\cdot\bm{B})\bm{B}$,
and the energy density $E=DhW-\pt+\BB$.
 Here $\bm{e}_k$ denotes the $k$-th row of the
$d\times d$ unit matrix, and
 $\pt$ denotes the total pressure containing the gas pressure $p$ and
the magnetic pressure $p_m=\frac12\bb$.
Due to no explicit expression for the primitive variables
$(\rho,\bv,p,\bm{B})^\mathrm{T}$ and the flux $\bF_k$ in terms of $\bU$,
a nonlinear algebraic equation, see e.g. \cite{Koldoba2002An}, has to be solved in order to recover the values of the primitive variables and the flux from the given $\bU$.
It is obvious that the nonlinearity of \eqref{eq:RMHDdiv1}-\eqref{eq:RMHDdiv2} becomes much
stronger than the non-relativistic case due to the relativistic effect,
thus its analytical treatment is very challenging.
The first numerical work may date back to the artificial viscosity method
for the RHD equations in the Lagrangian coordinates \cite{May1966Hydrodynamic,May1967Stellar}
and the Eulerian coordinates \cite{Wilson1972Numerical}.
Since the early 1990s, the modern shock-capturing methods were extended to the
RHD and  RMHD equations, such as
the Roe-type scheme \cite{Anton2010Relativistic,Eulderink1994General},
the Harten-Lax-van Leer (HLL) method \cite{DelZanna2002An,DelZanna2003An,Schneider1993New},
the Harten-Lax-van Leer-Contact (HLLC) method \cite{Ling2019Physical,Mignone2005An,Mignone2006An},
the Harten-Lax-van Leer-Discontinuities (HLLD) method \cite{Mignone2009A},
the essentially non-oscillatory (ENO) and  weighted ENO (WENO) methods
\cite{Dolezal1995Relativistic,DelZanna2002An,DelZanna2003An},
the piecewise parabolic methods \cite{Marti1996Extension,Mignone2005The},
the Runge-Kutta discontinuous Galerkin (DG) methods with WENO limiter
\cite{Zhao2013Discontinuous,Zhao2017Runge},
the direct Eulerian generalized Riemann problem schemes
\cite{Yang2011A1D,Yang2012A2D,Wu2014A,Wu2016A},
the gas kinetics schemes \cite{Chen2017Anderson, Chen2021Second},
the two-stage fourth-order time discretization \cite{Yuan2020Two},
the adaptive moving mesh methods \cite{He2012RHD,He2012RMHD}, and so on.
Recently, the properties of the admissible state set and the
physical-constraints-preserving (both the rest-mass density and the kinetic
pressure of the numerical solutions are positive and the magnitude of the fluid velocity is less than the speed of light) numerical schemes were well studied for
the RHD and RMHD equations, see \cite{Ling2019Physical, Ling2020A,Wu2015Finite,Wu2016Physical,Wu2017Admissible,Wu2018Admissible,Wu2017Design}.
The readers are also referred to the early review articles
\cite{Font2008Numerical,Marti2003Numerical,Marti2015Grid} for more references.

For the RHD and RMHD equations, the entropy condition is an important property which should be respected according to the second law of thermodynamics.
On the other hand, it is well known that the weak solution of the
quasi-linear hyperbolic conservation laws nay not be unique so that the
entropy condition is needed to single out the unique physical relevant solution among
all the weak solutions.
Thus it is of great significance to seek the entropy stable (ES) schemes (satisfying some discrete or semi-discrete entropy conditions) for the quasi-linear system of hyperbolic conservation laws.
For the scalar conservation laws, the fully-discrete conservative monotone schemes
were nonlinearly stable and satisfied the entropy conditions,
thus they could converge to the entropy solution \cite{Harten1976,Crandall1980Monotone}.
A class of the so-called E-schemes satisfying the  semi-discrete entropy conditions for any convex
entropy was studied in \cite{Osher1984Riemann,Osher1988On}, but they were restricted to the first-order accuracy.
Generally, it is difficult to show that the high-order schemes of
the scalar conservation laws and the schemes for the system of hyperbolic conservation laws
satisfy the entropy inequality for any convex entropy function.
In \cite{Bouchut1996A}, a second-order accurate scheme is shown to satisfy all the entropy conditions, which evolves not only the cell averages but also the solution values at half nodes.
Many researchers are trying to study the high-order accurate
ES schemes, which satisfy the entropy inequality for a given entropy pair.
The two-point entropy conservative (EC) flux and corresponding second-order EC schemes (satisfying the   semi-discrete entropy
identity) were proposed in \cite{Tadmor1987The,Tadmor2003Entropy},
and their higher-order extension was studied in \cite{Lefloch2002Fully}.
It is known that the EC schemes may become oscillatory near the
shock waves so that some additional dissipation terms have to be added
to obtain the ES schemes.
Combining the EC flux with the ``sign'' property of the ENO reconstruction, the
arbitrary high-order ES schemes were constructed by using high-order dissipation terms \cite{Fjordholm2012Arbitrarily}.
The ES schemes were then extended to the finite difference schemes based on summation-by-parts (SBP) operators \cite{Fisher2013High}.
Some ES schemes were also studied in the DG framework, such as the space-time DG formulation \cite{Hughes1986A,Hiltebrand2014Entropy},
the DG spectral element methods \cite{Gassner2013A, Carpenter2014Entropy},
and the nodal DG schemes on the simplex meshes \cite{Chen2017Entropy}.
More ES DG methods can be found in the review articles \cite{Chen2020Review,Gassner2021A}.

Recently, the high-order accurate ES finite difference schemes for the RHD equations were firstly studied in \cite{Duan2020RHD}, in which the dissipation terms built on  the fifth-order WENO reconstruction and the switch function in \cite{Biswas2018Low} was of the fifth-order accuracy and the ``sign" property simultaneously.
Later, the TeCNO scheme  \cite{Fjordholm2012Arbitrarily} was extended to the RHD equations \cite{Bhoriya2020Entropy}, where the dissipation terms were based on the ENO reconstruction.
For the ideal RMHDs,  the high-order accurate ES finite difference schemes  were proposed in \cite{Wu2020Entropy} and the ES DG schemes  were studied in \cite{Duan2020RMHD} by using the symmetrizable RMHD equations and the suitable discretization of the source terms.

In view of the fact that the solutions of the RHD equations often exhibit localized structures,
e.g. containing sharp transitions or discontinuities in relatively localized regions, the second-order accurate ES adaptive moving mesh schemes for the RHD equations are proposed in \cite{Duan2021RHDMM} to improve the efficiency and quality of numerical simulation.
This paper is devoted to extend such ES adaptive moving mesh  schemes
 as the high-order (greater than second-order) accurate  schemes for the RHD and RMHD equations.
The key point is the derivation of the higher-order accurate EC and ES finite difference schemes in the curvilinear coordinates.
For such purpose, one should carefully deal with the discretization of the temporal and spatial metrics and the Jacobian introduced by the coordinate transformation and construct
the high-order EC and ES fluxes with the discrete metrics.
We prove that the suitable linear combinations of the two-point EC flux in the curvilinear coordinates give the high-order EC fluxes, which can be regarded as a refinement of the arbitrarily high-order accurate EC fluxes in the Cartesian coordinates in \cite{Lefloch2002Fully}.
The spatial derivatives in the source terms of the symmetrizable RMHD equations and the geometric conservation laws are discretized by using the linear combinations of the corresponding second-order case to obtain high-order accuracy.
%
%
Based on the proposed high-order accurate EC schemes and the high-order accurate dissipation terms built on the WENO reconstruction, the high-order accurate ES schemes are obtained for the
RHD and RMHD equations in the curvilinear coordinates.
  Several two- and three-dimensional  numerical tests are conducted to validate the shock-capturing ability and high efficiency of our high-order accurate ES adaptive moving mesh methods on the parallel computer system with the MPI communication.
The  numerical results show that the high-order accurate ES adaptive moving mesh schemes outperform both their counterparts on the uniform mesh and the second-order ES adaptive moving mesh schemes \cite{Duan2021RHDMM}.

The paper is organized as follows.
Section \ref{section:EntropyCondition} gives the symmetrizable RMHD equations  in the curvilinear coordinates and  corresponding entropy conditions. It involves the special case of the RHD equations, i.e. \eqref{eq:RMHDdiv1}-\eqref{eq:RMHDdiv2} with $\bm{B}\equiv \bm{0}$.
Section \ref{section:ECScheme} presents the high-order accurate EC finite difference schemes in the curvilinear coordinates,
while Section \ref{section:ESScheme} gives the high-order accurate ES finite difference schemes by adding suitable dissipation terms based on the WENO reconstruction.
The adaptive moving mesh strategy is introduced in Section \ref{section:MM}.
Several numerical tests are conducted in Section \ref{section:NumTests} to validate the high-order accuracy, the shock-capturing ability and the efficiency of the proposed schemes.
Section \ref{section:Conclusion} concludes the work with further remarks.

%% file: EntropyCondition.tex
\section{Entropy conditions for  symmetrizable RMHD equations}\label{section:EntropyCondition}
This section introduces
some basic notations and  the entropy conditions for the symmetrizable RMHD equations.
\begin{definition}\label{def:entropy_pair}\rm
  A strictly convex scalar function $\eta(\bU)$ is called an \emph{entropy function} of the
  system \eqref{eq:RMHDdiv1} if there exists associated entropy
  fluxes $q_k(\bU)$ such that
  \begin{equation}\label{eq:EntropyConsistentCond}
    q_k'(\bU)=\bV^\mathrm{T}\bF_k'(\bU),\ \
    k=1,2,\cdots,d,
  \end{equation}
  where $\bV=\eta'(\bU)^\mathrm{T}$ is called the \emph{entropy variables}, and
  $(\eta,q_k)$ is an {\em entropy pair}.
\end{definition}
For the smooth solutions of \eqref{eq:RMHDdiv1}-\eqref{eq:RMHDdiv2},
multiplying \eqref{eq:RMHDdiv1} by $\bV^\mathrm{T}$ gives the entropy identity
\begin{equation*}
  \pd{\eta(\bU)}{t}+\sum_{k=1}^d\pd{q_k(\bU)}{x_k} = 0.
\end{equation*}
However, if the solutions contain discontinuities, then the above identity does not hold and the weak solutions should be considered.
\begin{definition}\rm
  A weak solution $\bU$ of \eqref{eq:RMHDdiv1}-\eqref{eq:RMHDdiv2} is called an {\em entropy solution} if for
  all entropy functions $\eta$, the inequality
  \begin{equation}\label{eq:entropyineq}
    \pd{\eta(\bU)}{t}+\sum_{k=1}^d\pd{q_k(\bU)}{x_k} \leqslant 0,
  \end{equation}
  holds in the sense of distributions.
\end{definition}

For the system \eqref{eq:RMHDdiv1}-\eqref{eq:RMHDdiv2} with zero magnetic field ($\bm{B}\equiv\bm{0}$), the entropy pair can be defined by the thermodynamic entropy \cite{Duan2020RHD, Rezzolla2013Relativistic} as follows
\begin{equation}\label{eq:EntropyPair}
  \eta(\bU)=-\dfrac{\rho Ws}{\Gamma-1}, \quad q_k(\bU)=\eta v_k,
\end{equation}
where $s=\ln(p/\rho^\Gamma)$ is the thermodynamic entropy,  $\eta$ is a convex function of $\bU$ and $(\eta,q_k)$ satisfies the consistent condition \eqref{eq:EntropyConsistentCond}.
However, when $\bm{B}\not\equiv\bm{0}$,
the function pair in \eqref{eq:EntropyPair} does not satisfy  \eqref{eq:EntropyConsistentCond}, and it can be verified that in general the system \eqref{eq:RMHDdiv1}-\eqref{eq:RMHDdiv2} cannot be symmetrized \cite{Duan2020RMHD, Wu2020Entropy}.
Motivated by the symmetrization of the non-relativistic magnetohydrodynamics \cite{Godunov1972, Powell1994},
some source terms can be added to get a symmetrizable RMHD system as follows \cite{Wu2020Entropy}
\begin{equation}\label{eq:RMHDSymm}
  \pd{\bU}{t}+\sum_{k=1}^d\pd{\bF_k}{x_k}=-\Phi'(\bV)^\mathrm{T}\sum_{k=1}^d \pd{B_k}{x_k},\ \bV:=\eta'(\bU)^\mathrm{T},
\end{equation}
where $\Phi(\bV)$ is a homogeneous function of degree one, i.e. $\Phi=\Phi'(\bV)\bV$, with
\begin{equation}\label{eq:Phi}
  \Phi=\dfrac{\rho W(\vB)}{p},\quad
  \Phi'(\bV)=\left(0,\bm{B}/W^2+\bv(\vB),~\vB,~\bv\right),
\end{equation}
that is to say, the function pair in \eqref{eq:EntropyPair}
can symmetrize the modified RMHD system \eqref{eq:RMHDSymm}
so that it is an entropy pair of \eqref{eq:RMHDSymm}.
The entropy variable $\bV$ can be explicit expressed as
\begin{equation*}
  \bV=\eta'(\bU)^\mathrm{T}=\left(\dfrac{\Gamma-s}{\Gamma-1}+\frac{\rho}{p},
  \dfrac{\rho W\bv^\mathrm{T}}{p}, -\dfrac{\rho W}{p}, \dfrac{\rho (\bm{B}+W^2\bv(\vB))}{pW}\right)^\mathrm{T}.
\end{equation*}
For the smooth solutions, taking the dot product of $\bV$ with \eqref{eq:RMHDSymm} yields the entropy identity
\begin{equation*}
  \pd{\eta(\bU)}{t}+\sum_{k=1}^d\pd{q_k(\bU)}{x_k} = 0,
\end{equation*}
while for the discontinuous solutions,  it is replaced with the entropy inequality
\begin{equation*}
  \pd{\eta(\bU)}{t}+\sum\limits_{k=1}^d\pd{q_k(\bU)}{x_k} \leqslant 0,
\end{equation*}
which holds in the sense of distributions.
One can further define the \emph{entropy potential} $\phi$ and \emph{entropy  flux potential} $\psi_k$ from the given  $(\eta(\bU), q_k(\bU))$ and $\Phi(\bV)$ by
\begin{subequations}\label{eq:EntropyPotential}
  \begin{align}
    \phi:&=\bV^\mathrm{T}\bU-\eta(\bU)=\rho W+\dfrac{\rho W\bb}{2p}, \\
    \psi_k:&=\bV^\mathrm{T} \bF_k(\bU)+\Phi(\bV) B_k-q_k(\bU)=\rho v_kW+\dfrac{\rho v_kW\bb}{2p},
  \end{align}
\end{subequations}
which are important in obtaining the sufficient condition for the two-point EC flux.

Similar to \cite{Duan2021RHDMM}, let us derive the curvilinear coordinate form of the symmetrizable RMHD equations \eqref{eq:RMHDSymm}  and  corresponding entropy conditions.
Let $\Omega_p$ be the physical domain with coordinates $\bm{x}=(x_1,\cdots,x_d)$, in which
\eqref{eq:RMHDSymm} is specified,
and $\Omega_c$ be the computational domain with coordinates $\bm{\xi}=(\xi_1,\cdots,\xi_d)$ that
is artificially chosen for the sake of the
mesh redistribution or movement.
Our adaptive moving meshes for $\Omega_p$ can be generated as the images of a reference mesh in $\Omega_c$ by a time dependent,
differentiable, one-to-one coordinate mapping {$\bx = \bx(\tau, \bm{\xi})$},
which can be written as
\begin{align}\label{eq:transf}
  t=\tau,\ \ \bx=\bx(\tau, \bm{\xi}),\ \
  \bm{\xi}=(\xi_1,\cdots,\xi_d)\in\Omega_c.
\end{align}
Under this transformation, the system \eqref{eq:RMHDSymm} in the coordinates {$(\tau, \bm{\xi})$} reads
\begin{equation}\label{eq:RMHDCurv}
  \pd{\Ucurv}{\tau}+\sum_{k=1}^d\pd{\Fcurv_k}{\xi_k}=-\Phi'(\bV)^\mathrm{T}\sum_{k=1}^d\pd{\Bcurv_k}{\xi_k},
\end{equation}
with
\begin{equation*}
  \Ucurv=J\bU,~
  \Fcurv_k=\left(J\pd{\xi_k}{t}\bU\right)+\sum_{j=1}^d{\left(J\pd{\xi_k}{x_j}\bF_j\right)},~
  \Bcurv_k=\sum_{j=1}^d\left(J\pd{\xi_k}{x_j}B_j\right),
\end{equation*}
where  $J$ denotes the determinant of the Jacobian matrix and
its 3D version is explicitly given by
\begin{equation*}
  J=\det\left(\pd{(t,\bx)}{(\tau,\bm{\xi})}\right)=
  \begin{vmatrix}
    1 & 0 & 0 & 0\\
    \pd{x_1}{\tau} & \pd{x_1}{\xi_1} & \pd{x_1}{\xi_2} & \pd{x_1}{\xi_3} \\
    \pd{x_2}{\tau} & \pd{x_2}{\xi_1} & \pd{x_2}{\xi_2} & \pd{x_2}{\xi_3} \\
    \pd{x_3}{\tau} & \pd{x_3}{\xi_1} & \pd{x_3}{\xi_2} & \pd{x_3}{\xi_3}
  \end{vmatrix}.
\end{equation*}
The metrics should satisfy the following geometric conservation laws (GCLs) consisting of
the volume conservation law (VCL) and the surface conservation laws (SCLs)
\begin{subequations}\label{eq:GCL}
  \begin{align}
    \label{eq:VCL}
    &\text{VCL:}\quad \pd{J}{\tau}+\sum_{k=1}^d\dfrac{\partial}{\partial\xi_k}{\left(J\pd{\xi_k}{t}\right)}=0,\\
    \label{eq:SCL}
    &\text{SCLs:}\quad \sum_{k=1}^d\dfrac{\partial}{\partial\xi_k}{\left(J\pd{\xi_k}{x_j}\right)}=0,~ j=1,\cdots,d.
  \end{align}
\end{subequations}
The former indicates that the volumetric increment of a moving cell must be equal to the sum of the changes along the surfaces that enclose the volume,
while the latter indicates that the cell volumes must be closed by its surfaces \cite{Zhang1993Discrete}.
Those GCLs imply that free-stream solution is preserved by \eqref{eq:RMHDCurv},
in other words, a physical constant state is an exact solution of \eqref{eq:RMHDCurv}.
Finally, by using the GCLs \eqref{eq:GCL}, see \cite{Duan2021RHDMM},  the entropy identity for \eqref{eq:RMHDCurv} in the coordinates $(\tau,\bm{\xi})$ is
\begin{align}\label{eq:RMHDEntropyIdCurv}
  \pd{\left(J\eta\right)}{\tau}+\sum_{k=1}^d\pd{\mathfrak{q}_k}{\xi_k}=0,
\end{align}
with
\begin{equation*}
  \mathfrak{q}_k=\left(J\pd{\xi_k}{t}\eta\right)
  +\sum_{j=1}^d\left(J\pd{\xi_k}{x_j}q_j\right),
\end{equation*}
but when the solutions are not smooth, it is replaced with
the entropy inequality
\begin{align}\label{eq:RMHDEntropyInqCurv}
  \pd{\left(J\eta\right)}{\tau}+\sum_{k=1}^d\pd{\mathfrak{q}_k}{\xi_k}\leqslant0,
\end{align}
which holds in the sense of distribution.

\begin{rmk}\rm
  The continuous GCLs \eqref{eq:GCL} are crucial in deriving the governing equations  \eqref{eq:RMHDCurv},
  the entropy identity \eqref{eq:RMHDEntropyIdCurv} and the entropy inequality \eqref{eq:RMHDEntropyInqCurv}.
Corresponding discrete GCLs will be  important in proving the EC property of our schemes, see Section \ref{section:ECScheme}.
\end{rmk}

\begin{rmk}\rm
  For the RHD case, the entropy variable $\bV$, the entropy potential $\phi$, and the entropy flux potential  $\psi_k$ can be obtained by setting $\bm{B}\equiv\bm{0}$.
\end{rmk}

%% file: ECScheme.tex
\section{High-order accurate EC schemes}\label{section:ECScheme}
This section presents the 3D high-order accurate EC finite difference schemes for the RMHD system \eqref{eq:RMHDCurv},
 which gives  corresponding schemes for the RHD equations by setting $\bm{B}\equiv\bm{0}$.
The 1D and 2D schemes in the curvilinear coordinates are given in \ref{app:1DEC} and \ref{app:2DEC}, respectively.
For simplicity, only 3D adaptive moving mesh schemes for the RMHD equations in curvilinear coordinates \eqref{eq:RMHDCurv} on structured meshes are presented hereafter.

\subsection{Two-point EC flux}
To develop the high-order accurate EC schemes, one of the main ingredient is the so-called \emph{two-point EC flux}.

\begin{definition}\rm
  For the RMHD system \eqref{eq:RMHDCurv},
  a numerical flux
  $\widetilde{\Fcurv_k}\Big(\bU_l, \bU_r, \big(J\pd{\xi_k}{\zeta}\big)_l,
  \big(J\pd{\xi_k}{\zeta}\big)_r \Big)$ is called two-point EC flux, ${\zeta}=t,x_1,x_2,x_3$,
  if it is consistent with $\Fcurv_k$
  and satisfies
  \begin{align}\label{eq:ECConditionCurv}
    \left(\bV(\bU_r)-\bV(\bU_l)\right)^\mathrm{T}&\widetilde{\Fcurv_k}=
    ~\dfrac12\left(\left(J\pd{\xi_k}{t}\right)_l+\left(J\pd{\xi_k}{t}\right)_r\right)\left(\phi(\bU_r)-\phi(\bU_l)\right) \nonumber\\
    &+\sum_{j=1}^3\dfrac12\left(\left(J\pd{\xi_k}{x_j}\right)_l+\left(J\pd{\xi_k}{x_j}\right)_r\right)\left(\psi_{j}(\bU_r)-\psi_{j}(\bU_l)\right) \nonumber\\
    &-\sum_{j=1}^3\dfrac14\left(\left(J\pd{\xi_k}{x_j}\right)_l+\left(J\pd{\xi_k}{x_j}\right)_r\right)\left(B_{j}(\bU_l)+B_{j}(\bU_r)\right)\left(\Phi(\bU_r)-\Phi(\bU_l)\right),
  \end{align}
  where $\Phi$ and $\phi,\psi_j$ are defined in \eqref{eq:Phi} and \eqref{eq:EntropyPotential}, respectively.
\end{definition}

\begin{rmk}\rm
  If $\bm{B}\equiv\bm{0}$, then \eqref{eq:ECConditionCurv} reduces to the RHD case  \cite{Duan2021RHDMM},
  while, if $(t,\bm{x})\equiv (\tau,\bm{\xi})$, then \eqref{eq:ECConditionCurv} reduces to  the Cartesian coordinate case \cite{Wu2020Entropy}.
\end{rmk}

What follows is to find such a two-point EC flux satisfying \eqref{eq:ECConditionCurv}.
Similar to \cite{Duan2021RHDMM}, the EC flux can be chosen as follows
\begin{align}\label{eq:ECFluxCurv}
  \widetilde{\Fcurv_k}\Big(\bU_l,\bU_r,
  \big(J\pd{\xi_k}{\zeta}\big)_l,
  \big(J\pd{\xi_k}{\zeta}\big)_r\Big)
  =&~\UCurvFlux(l;r) + \sum_{j=1}^3\FCurvFlux_j(l;r),
\end{align}
where   ${\zeta}=t,x_1,x_2,x_3$, and
\begin{equation}\label{eq:UFcircle}
  \UCurvFlux(l;r)=\dfrac12\left(\left(J\pd{\xi_k}{t}\right)_l + \left(J\pd{\xi_k}{t}\right)_r\right)\widetilde{\bU},~
  \FCurvFlux_j(l;r)=\dfrac12\left(\left(J\pd{\xi_k}{x_j}\right)_l + \left(J\pd{\xi_k}{x_j}\right)_r\right)\widetilde{\bF}_j,
\end{equation}
with
$\widetilde{\bU}$ and $\widetilde{\bF}_j$ satisfying
\begin{align*}
  &\left(\bV_r-\bV_l\right)^\mathrm{T}\widetilde{\bU}=\phi_r - \phi_l,\\
  &\left(\bV_r-\bV_l\right)^\mathrm{T}\widetilde{\bF}_j=
  \left[(\psi_j)_r - (\psi_j)_l\right]-\dfrac12\left[(B_j)_l+(B_j)_r\right]\left(\Phi_r-\Phi_l\right).
\end{align*}
For  the RMHD system \eqref{eq:RMHDCurv},
$\widetilde{\bF}_j^{\scriptsize\text{RMHD}}$  in \cite{Duan2020RMHD} is used, while
$\widetilde{\bU}^{\scriptsize\text{RMHD}}$
can be given by  following the  derivation of
$\widetilde{\bF}_j^{\scriptsize\text{RMHD}}$ and reads
\begin{align*}
  &\widetilde{\bU}^{\scriptsize\text{RMHD}}=
  \begin{pmatrix}
    \meanln{\rho}\mean{W}\\
    \mean{\ux}\widetilde{U}_5^{\scriptsize\text{RMHD}}/\mean{W}+R_2/\mean{\beta}\\
    \mean{\uy}\widetilde{U}_5^{\scriptsize\text{RMHD}}/\mean{W}+R_3/\mean{\beta}\\
    \mean{\uz}\widetilde{U}_5^{\scriptsize\text{RMHD}}/\mean{W}+R_4/\mean{\beta}\\
    \mathcal{D}^{-1}\left[\mean{\ux}R_2+\mean{\uy}R_3+\mean{\uz}R_4-\mean{\beta}R_1\right]\\
    \left(\mean{W}^2\mean{b^x}-\mean{Wb^0}\mean{u^x}\right)/\mean{W}\\
    \left(\mean{W}^2\mean{b^y}-\mean{Wb^0}\mean{u^y}\right)/\mean{W}\\
    \left(\mean{W}^2\mean{b^z}-\mean{Wb^0}\mean{u^z}\right)/\mean{W}\\
  \end{pmatrix},
\end{align*}
where
\begin{align*}
  \mathcal{D}&={\mean{\beta}(\mean{W}^2-\sum_{k=1}^3\mean{u^k}^2)}/{\mean{W}},~\beta=\rho/p,\\
  R_1&=-\alpha_0\widetilde{U}_1^{\scriptsize\text{RMHD}}-\frac12\mean{W(b^0)^2}
  +\sum_{k=1}^2\left[\frac12\mean{W}\mean{(b^k)^2}
  -\mean{b^k}\widetilde{U}_{k+5}^{\scriptsize\text{RMHD}}\right],\\
  R_2&=\left(\alpha_1\mean{u^x}
  -\mean{\beta}\mean{Wb^0}\mean{b^x}\right)/\mean{W},\\
  R_3&=\left(\alpha_1\mean{u^y}
  -\mean{\beta}\mean{Wb^0}\mean{b^y}\right)/\mean{W},\\
  R_4&=\left(\alpha_1\mean{u^z}
  -\mean{\beta}\mean{Wb^0}\mean{b^z}\right)/\mean{W},\\
  \alpha_0&=1+1/(\Gamma-1)/\meanln{\beta},\\
  \alpha_1&=\mean{\rho}+\frac12\mean{\beta}\sum_{k=1}^3\mean{(b^k)^2}+\frac12\mean{\beta}\mean{W(b^0)^2}/\mean{W},
\end{align*}
here $\meanln{a}=\jump{a}/\jump{\ln{a}}$ is the logarithmic mean, see \cite{Ismail2009Affordable},
and $\widetilde{\bU}_n^{\rm RMHD}$ denotes the $n$th component of $\widetilde{\bU}^{\rm RMHD}$.
For the RHD equations, a two-point EC flux in curvilinear coordinates can be found in \cite{Duan2021RHDMM}.

\subsection{Discretization of  RMHD system and VCL}
Assume that the 3D computational domain $\Omega_c$
is chosen as a cuboid for convenience, e.g.
$[a_1,b_1]\times[a_2,b_2]\times[a_3,b_3]$,
and divided into a fixed orthogonal mesh
$\{(\xi_{1,i_1},\xi_{2,i_2},\xi_{3,i_3})$:\\
  $a_k=\xi_{k,0}<\xi_{k,1}<\cdots<\xi_{k,i_k}
<\cdots<\xi_{k,N_k-1}=b_k$, $k=1,2,3\}$
with the constant step-size $\Delta \xi_k=\xi_{k,i_k+1}-\xi_{k,i_k}$.
For the sake of brevity, the index $\bm{i}=(i_1,i_2,i_3)$ is used to denote the point
$(\xi_{1,i_1},\xi_{2,i_2},\xi_{3,i_3})$,
and the subscript $\{\bm{i},k,n\}$ denotes the index $\bm{i}$ increases $n$ in the $i_k$-direction,
e.g., $\{\bm{i},3,\frac12\}$ denotes $(i_1, i_2, i_3+\frac12)$.

Based on the above notations, consider the following $2p$th-order ($p\geq1$) semi-discrete conservative finite difference schemes for the RMHD system \eqref{eq:RMHDCurv} and the VCL \eqref{eq:GCL}
\begin{align}
  \label{eq:RMHDSemiU_O2p}
  &\dfrac{\dd}{\dd t}\bm{\mathcal{U}}_{\bm{i}}=
  -\sum_{k=1}^3\dfrac{1}{\Delta \xi_k}\left((\widetilde{\Fcurv_k})_{\bm{i},k,+\frac12}^{\twop}-(\widetilde{\Fcurv_k})_{\bm{i},k,-\frac12}^{\twop}\right)
  -\Phi'(\bV_{\bm{i}})^\mathrm{T}\sum_{k=1}^3\dfrac{1}{\Delta \xi_k}\left((\widetilde{\Bcurv_k})_{\bm{i},k,+\frac12}^{\twop}-(\widetilde{\Bcurv_k})_{\bm{i},k,-\frac12}^{\twop}\right),
  \\
  \label{eq:RMHDSemiJ_O2p}
  &\dfrac{\dd}{\dd t}J_{\bm{i}}=
  -\sum_{k=1}^3\dfrac{1}{\Delta \xi_k}\Bigg(\left(\widetilde{J\pd{\xi_k}{t}}\right)_{\bm{i},k,+\frac12}^{\twop}-\left(\widetilde{J\pd{\xi_k}{t}}\right)_{\bm{i},k,-\frac12}^{\twop}\Bigg),
\end{align}
where $J_{\bm{i}}(t)$ and $\Ucurv_{\bm{i}}(t)$ approximate the point values of
$J\left(t,\bm{\xi}\right)$ and $\Ucurv(t,\bm{\xi})$ at $\bm{i}$, respectively,
and $(\widetilde{\Fcurv_k})_{\bm{i},k,\pm\frac12}^{\twop}$,
$(\widetilde{\Bcurv_k})_{\bm{i},k,\pm\frac12}^{\twop}$,
$\left(\widetilde{J\pd{\xi_k}{t}}\right)_{\bm{i},k,\pm\frac12}^{\twop}$
are the numerical fluxes used in the approximations of the flux derivative, source terms, and spatial derivatives in the VCL, respectively.

The high-order ($p>1$) accurate EC schemes \eqref{eq:RMHDSemiU_O2p} for the system \eqref{eq:RMHDCurv} are mainly built on the following parts.
\begin{enumerate}
  \item  For the given entropy pair, the two-point EC flux $\widetilde{\Fcurv_k}$  is first derived from \eqref{eq:ECConditionCurv}, and then the high-order EC flux $(\widetilde{\Fcurv_k})_{\bm{i},k,\pm\frac12}^{\twop}$
is gotten by some linear combination of the two-point EC flux $\widetilde{\Fcurv_k}$ in \eqref{eq:ECFluxCurv},
such that the approximation of the flux derivative $\pd{\Fcurv_k}{\xi_k}$ is $2p$th-order accurate. It is considered as
an extension of the high-order accurate EC schemes in the Cartesian coordinates \cite{Lefloch2002Fully} to the curvilinear coordinates.

  \item Compute $(\widetilde{\Bcurv_k})_{\bm{i},k,\pm\frac12}^{\twop}$ and
  $\Big(\widetilde{J\pd{\xi_k}{t}}\Big)_{\bm{i},k,\pm\frac12}^{\twop}$ by  the same linear combinations of corresponding 2nd-order case  as that of the $2p$th-order EC flux, so that the approximations of the spatial derivatives in source terms and the VCL are also $2p$th-order accurate. The discretization of the latter degenerates to the $2p$th-order accurate central difference.

  \item The metrics $\Big(\widetilde{J\pd{\xi_k}{x_j}}\Big)_{\bm{i}}$ used in the above two parts are discretized by the $2p$th-order central difference based on the conservative metrics method (CMM) \cite{Thomas1979} such that the SCLs hold in the discrete level.
  \item The schemes \eqref{eq:RMHDSemiU_O2p}-\eqref{eq:RMHDSemiJ_O2p} can be proved to be $2p$th-order accurate and EC by combing the above three parts, which mimics the derivation of the continuous entropy identity \eqref{eq:RMHDEntropyIdCurv} in the curvilinear coordinates.
\end{enumerate}
The first two parts are given in Proposition \ref{prop:ECScheme_HOacc},
the third is addressed in Section \ref{subsec:GCLs},
and the last one is summarized in Theorem \ref{thm:ECScheme_HOEC}.

\begin{proposition}\rm\label{prop:ECScheme_HOacc}
  If the $2p$th-order fluxes $(\widetilde{\Fcurv_k})_{\bm{i},k,\pm\frac12}^{\twop}$, $(\widetilde{\Bcurv_k})_{\bm{i},k,\pm\frac12}^{\twop}$ and
  $\Big(\widetilde{J\pd{\xi_k}{t}}\Big)_{\bm{i},k,\pm\frac12}^{\twop}$ are chosen as follows
  \begin{align}
    \label{eq:ECFlux_O2p}
    &(\widetilde{\Fcurv_k})_{\bm{i},k,+\frac12}^{\twop}=
    \sum_{n=1}^p\alpha_{p,n}\sum_{s=0}^{n-1}
    \widetilde{\Fcurv_k}\left(\bU_{\bm{i},k,-s},\bU_{\bm{i},k,-s+n},
    \left(J\pd{\xi_k}{\zeta}\right)_{\bm{i},k,-s},
    \left(J\pd{\xi_k}{\zeta}\right)_{\bm{i},k,-s+n}\right), \\
    \label{eq:BFlux_O2p}
    &(\widetilde{\Bcurv_k})_{\bm{i},k,+\frac12}^{\twop}=
    \sum_{n=1}^p\alpha_{p,n}\sum_{s=0}^{n-1}
    \widetilde{\Bcurv_k}\left(B_{\bm{i},k,-s},B_{\bm{i},k,-s+n},\left(J\pd{\xi_k}{x_j}\right)_{\bm{i},k,-s},\left(J\pd{\xi_k}{x_j}\right)_{\bm{i},k,-s+n}\right), \\
    \label{eq:MetricsFlux_O2p}
    &\left(\widetilde{J\pd{\xi_k}{\zeta}}\right)_{\bm{i},k,+\frac12}^{\twop}=
    \sum_{n=1}^p\alpha_{p,n}\sum_{s=0}^{n-1}
    \left(\widetilde{J\pd{\xi_k}{\zeta}}\right)
    \left(\left(J\pd{\xi_k}{\zeta}\right)_{\bm{i},k,-s},
    \left(J\pd{\xi_k}{\zeta}\right)_{\bm{i},k,-s+n}\right),
  \end{align}
  where $\zeta=t,x_1,x_2,x_3$, $\widetilde{\Bcurv_k}$ and $\left(\widetilde{J\pd{\xi_k}{\zeta}}\right)$ are corresponding 2nd-order case as follows
  \begin{align}
    \label{eq:BFlux_O2}
    &\widetilde{\Bcurv_k}\left((B_j)_l, (B_j)_r,\left(J\pd{\xi_k}{x_j}\right)_{l},\left(J\pd{\xi_k}{x_j}\right)_{r}\right)=\sum_{j=1}^3
    \dfrac14\Bigg(\left(J\pd{\xi_k}{x_j}\right)_{l} + \left(J\pd{\xi_k}{x_j}\right)_{r}\Bigg)\left((B_j)_{l} + (B_j)_{r}\right), \\
    \label{eq:MetricsFlux_O2}
    &\left(\widetilde{J\pd{\xi_k}{\zeta}}\right)
    \left(\left(J\pd{\xi_k}{\zeta}\right)_{l},
    \left(J\pd{\xi_k}{\zeta}\right)_{r}\right)=
    \dfrac12\Bigg(\left(J\pd{\xi_k}{\zeta}\right)_{l} + \left(J\pd{\xi_k}{\zeta}\right)_{r}\Bigg),
  \end{align}
  and the coefficients in the linear combinations satisfy the constraints \cite{Lefloch2002Fully}
  \begin{equation}\label{eq:HOCoeff}
    \sum_{n=1}^p n\alpha_{p,n} = 1,~
    \sum_{n=1}^p n^{2s-1}\alpha_{p,n} = 0,~s=2,\cdots,p,
  \end{equation}
 then the  approximation of the flux derivative $\pd{\Fcurv_k}{\xi_k}$ is $2p$th-order accurate, i.e.
  \begin{equation}
    \dfrac{1}{\Delta\xi_k}\left((\widetilde{\Fcurv_k})_{\bm{i},k,+\frac12}^{\twop}
    -(\widetilde{\Fcurv_k})_{\bm{i},k,-\frac12}^{\twop}\right)
    =\pd{\Fcurv_k}{\xi_k}\Bigg|_{\bm{i}}+\mathcal{O}\left(\Delta\xi_k^{2p}\right),~k=1,2,3.
  \end{equation}
  Similarly, the approximations of the source terms and the spatial derivatives in the VCL are also $2p$th-order accurate.
\end{proposition}

To prove such proposition, let us first consider the following Lemma.

\begin{lem}\rm\label{prop:Der}
  If the smooth two-parameter scalar function $\widetilde{f}(u(\zeta_l), u(\zeta_r))$ and  vector-value function $\widetilde{\bF}(\bU(\zeta_l), \bU(\zeta_r))$
satisfy
  \begin{align*}
\mbox{Consistency}&\qquad    \widetilde{f}(u, u) = f(u),~\widetilde{\bF}(\bU, \bU) = \bF(\bU),
\\
\mbox{Symmetry}&\qquad
    \widetilde{f}(u(\zeta_l), u(\zeta_r)) = \widetilde{f}(u(\zeta_r), u(\zeta_l)),~
    \widetilde{\bF}(\bU(\zeta_l), \bU(\zeta_r)) = \widetilde{\bF}(\bU(\zeta_r), \bU(\zeta_l)),
  \end{align*}
  then the following identities hold
  \begin{subequations}
    \begin{align*}
      2\dfrac{\partial}{\partial\zeta_r}\widetilde{f}(u(\zeta_l), u(\zeta_r)) \Big|_{\zeta_r=\zeta_l}
      &=\dfrac{\partial}{\partial\zeta}f(u(\zeta)) \Big|_{\zeta=\zeta_l}, \\
      2\dfrac{\partial}{\partial\zeta_r}\widetilde{\bF}(\bU(\zeta_l), \bU(\zeta_r)) \Big|_{\zeta_r=\zeta_l}
      &=\dfrac{\partial}{\partial\zeta}\bF(\bU(\zeta)) \Big|_{\zeta=\zeta_l}, \\
      2\dfrac{\partial}{\partial\zeta_r}\left[\widetilde{f}(u(\zeta_l), u(\zeta_r))\widetilde{\bF}(\bU(\zeta_l), \bU(\zeta_r))\right] \Big|_{\zeta_r=\zeta_l}
      &=\dfrac{\partial}{\partial\zeta}\left[f(u(\zeta))\bF(\bU(\zeta))\right] \Big|_{\zeta=\zeta_l}.
    \end{align*}
  \end{subequations}
\end{lem}

\begin{proof}
  The first identity is a special case of the second, which comes from \cite{Chan2018On}.
  Utilizing the symmetry and the consistency of $\widetilde{\bF}$ gives
  \begin{align*}
    2\pd{\widetilde{\bF}(\bU_l,\bU_r)}{\bU_r} \Big|_{\bU_r=\bU_l}
    &=\left(\pd{\widetilde{\bF}(\bU_l,\bU_r)}{\bU_l} + \pd{\widetilde{\bF}(\bU_l,\bU_r)}{\bU_r} \right) \Big|_{\bU_r=\bU_l} \\
    &=\pd{\widetilde{\bF}(\bU,\bU)}{\bU} \Big|_{\bU=\bU_l} = \pd{\bF(\bU)}{\bU} \Big|_{\bU=\bU_l}.
  \end{align*}
  Letting $\bU_l=\bU(\zeta_l),\bU_r=\bU(\zeta_r)$ and using the chain rule gives
  \begin{equation*}
    2\dfrac{\partial}{\partial\zeta_r}\widetilde{\bF}(\bU(\zeta_l), \bU(\zeta_r)) \Big|_{\zeta_r=\zeta_l}= \dfrac{\partial}{\partial\zeta}\bF(\bU(\zeta)) \Big|_{\zeta=\zeta_l}.
  \end{equation*}
  The third identity can be obtained as follows
  \begin{align*}
    &2\dfrac{\partial}{\partial\zeta_r}\left[\widetilde{f}(u(\zeta_l), u(\zeta_r))\widetilde{\bF}(\bU(\zeta_l), \bU(\zeta_r))\right] \Big|_{\zeta_r=\zeta_l} \\
    =&\ 2\dfrac{\partial}{\partial\zeta_r}\widetilde{f}(u(\zeta_l), u(\zeta_r)) \Big|_{\zeta_r=\zeta_l} \bF(\bU(\zeta_l))
    +2f(u(\zeta_l))\dfrac{\partial}{\partial\zeta_r}\widetilde{\bF}(\bU(\zeta_l), \bU(\zeta_r)) \Big|_{\zeta_r=\zeta_l} \\
    =&\ \dfrac{\partial}{\partial\zeta}f(u(\zeta)) \Big|_{\zeta=\zeta_l} \bF(\bU(\zeta_l))
    +f(u(\zeta_l))\dfrac{\partial}{\partial\zeta}\bF(\bU(\zeta)) \Big|_{\zeta=\zeta_l} \\
    =&\ \dfrac{\partial}{\partial\zeta}\left[f(u(\zeta))\bF(\bU(\zeta))\right] \Big|_{\zeta=\zeta_l},
  \end{align*}
  where the first equality uses the product rule.
\end{proof}

Based on the above Lemma, it is ready to prove Proposition \ref{prop:ECScheme_HOacc}.

\begin{proof}
  It suffices to consider the $i_k$-direction and to assume
  the other two independent variables to be fixed and omitted in the following expressions by using ``$\cdots$".
  If taking $\zeta_l=\hat{\xi}_{k},\zeta_r=\tilde{\xi}_{k}$, and
  $$\widetilde{f}=\dfrac12\left(\left(J\pd{\xi_k}{x_j}\right)(\cdots,\hat{\xi}_k,\cdots)+ \left(J\pd{\xi_k}{x_j}\right)(\cdots,\tilde{\xi}_k,\cdots)\right),$$
  $$\widetilde{\bF}=\widetilde{\bF}_j(\bU(\cdots,\hat{\xi}_k,\cdots),\bU(\cdots,\tilde{\xi}_k,\cdots)),$$
  in Proposition \ref{prop:Der}, then one has
  \begin{align}\label{eq:FcurvFluxDer}
    &\dfrac{\partial}{\partial\tilde{\xi}_k}\left[\dfrac12\left(\left(J\pd{\xi_k}{x_j}\right)(\cdots,\hat{\xi}_k,\cdots)+ \left(J\pd{\xi_k}{x_j}\right)(\cdots,\tilde{\xi}_k,\cdots)\right)
    \widetilde{\bF}_j(\bU(\cdots,\hat{\xi}_k,\cdots),\bU(\cdots,\tilde{\xi}_k,\cdots))\right] \Big|_{\tilde{\xi}_k=\hat{\xi}_k} \nonumber\\
    =\ & \dfrac12\dfrac{\partial}{\partial\xi_k}\left[\left(J\pd{\xi_k}{x_j}\right)(\cdots,\xi_{k},\cdots)\bF_j(\bU(\dots,\xi_{k},\dots))\right] \Big|_{\xi_k=\hat{\xi}_k}.
  \end{align}
 If  utilizing \eqref{eq:FcurvFluxDer}, then one can expand $\FCurvFlux_j(\bm{i}; \bm{i},k,\pm n)$ defined in \eqref{eq:UFcircle} at $\xi_{k,i_k}$ by using Taylor series as follows
  \begin{align*}
    \FCurvFlux_j(\bm{i}; \bm{i},k,\pm n)	=&\left[\left(J\pd{\xi_k}{x_j}\right)\bF_j\right]_{\bm{i}} \pm \dfrac{n\Delta \xi_k}{2}\dfrac{\partial}{\partial\xi_k}\left[\left(J\pd{\xi_k}{x_j}\right)\bF_j\right]_{\bm{i}} \\
    &+ \sum_{s=2}^{2p}\dfrac{(\pm n\Delta \xi_k)^s}{s!}{\partial_{\xi_k}^s}\FCurvFlux_j(\bm{i}; \bm{i}) + \mathcal{O}\left(\Delta \xi_k^{2p+1}\right),
  \end{align*}
  so that their difference becomes
  \begin{align*}
    \FCurvFlux_j(\bm{i}; \bm{i},k,+n)-\FCurvFlux_j(\bm{i}; \bm{i},k,-n)
    =&\  n\Delta\xi_k\dfrac{\partial}{\partial\xi_k}\left[\left(J\pd{\xi_k}{x_j}\right)\bF_j\right]_{\bm{i}}\\
    &+ 2\sum_{s=2}^{p} \dfrac{(n\Delta\xi_k)^{2s-1}}{(2s-1)!}{\partial_{\xi_k}^{2s-1}}\FCurvFlux_j(\bm{i}; \bm{i})
    + \mathcal{O}\left(\Delta \xi_k^{2p+1}\right).
  \end{align*}
  Similarly, it can be verified that
  \begin{align*}
    \UCurvFlux(\bm{i}; \bm{i},k,+n)-\UCurvFlux(\bm{i}; \bm{i},k,-n)
    =&\  n\Delta\xi_k\dfrac{\partial}{\partial\xi_k}\left[\left(J\pd{\xi_k}{t}\right)\bU\right]_{\bm{i}}\\
    &+ 2\sum_{s=2}^{p} \dfrac{(n\Delta\xi_k)^{2s-1}}{(2s-1)!}{\partial_{\xi_k}^{2s-1}}\UCurvFlux(\bm{i}; \bm{i})
    + \mathcal{O}\left(\Delta \xi_k^{2p+1}\right).
  \end{align*}
Based on those,  one gets
  \begin{align*}
    &\dfrac{1}{\Delta\xi_k}\left((\widetilde{\Fcurv_k})_{\bm{i},k,+\frac12}^{\twop}
    -(\widetilde{\Fcurv_k})_{\bm{i},k,-\frac12}^{\twop}\right)\\
    =&\sum_{n=1}^{p}\alpha_{p,n}\left(\UCurvFlux(\bm{i};\bm{i},k,+n) + \sum_{j=1}^3\FCurvFlux_j(\bm{i};\bm{i},k,+n)
    - \UCurvFlux(\bm{i};\bm{i},k,-n) - \sum_{j=1}^3\FCurvFlux_j(\bm{i};\bm{i},k,-n)\right) \\
    =&\left(\sum_{n=1}^{p}n\alpha_{p,n}\right)\pd{\Fcurv_k}{\xi_k}\Big|_{\bm{i}}
    + \dfrac{2\Delta\xi_k^{2s-2}}{(2s-1)!}\sum_{s=2}^{p}\left(\sum_{n=1}^{p}n^{2s-1}\alpha_{p,n}\right) \dfrac{\partial^{2s-1}{\left( \UCurvFlux({\bm{i}};{\bm{i}})+\sum_{j=1}^3\FCurvFlux_j({\bm{i}};{\bm{i}}) \right)}}{\partial\xi_k^{2s-1}}
    + \mathcal{O}\left(\Delta \xi_k^{2p}\right)\\
    =&\pd{\Fcurv_k}{\xi_k}\Bigg|_{\bm{i}}+\mathcal{O}\left(\Delta\xi_k^{2p}\right),~k=1,2,3,
  \end{align*}
  where the last equality uses the constraints \eqref{eq:HOCoeff}.
  Similarly it can be proved that the approximations of the source terms and the spatial derivatives in the VCL are also $2p$th-order accurate.
\end{proof}

\subsection{Discrete GCLs}\label{subsec:GCLs}
This section introduces some appropriate discretizations of the spatial metrics $\left(J\pd{\xi_k}{x_j}\right)_{\bm{i}}$
and the temporal metrics $\left(J\pd{\xi_k}{t}\right)_{\bm{i}}$
in order to get the discrete SCLs
\begin{equation}\label{eq:DiscSCL_O2p}
  \sum_{k=1}^3\dfrac{1}{\Delta \xi_k}\left(\left(\widetilde{J\pd{\xi_k}{x_j}}\right)_{\bm{i},k,+\frac12}^{\twop}
  -\left(\widetilde{J\pd{\xi_k}{x_j}}\right)_{\bm{i},k,-\frac12}^{\twop}\right)=0,~j=1,2,3.
\end{equation}
and  the discrete VCL.

For the smooth transformation \eqref{eq:transf},
the following identities hold
\begin{equation*}
  \begin{aligned}
    J\pd{\xi_1}{x_1}=\pd{x_2}{\xi_2}\pd{x_3}{\xi_3}-\pd{x_2}{\xi_3}\pd{x_3}{\xi_2},~
    J\pd{\xi_1}{x_2}=\pd{x_3}{\xi_2}\pd{x_1}{\xi_3}-\pd{x_3}{\xi_3}\pd{x_1}{\xi_2},~
    J\pd{\xi_1}{x_3}=\pd{x_1}{\xi_2}\pd{x_2}{\xi_3}-\pd{x_1}{\xi_3}\pd{x_2}{\xi_2},\\
    J\pd{\xi_2}{x_1}=\pd{x_2}{\xi_3}\pd{x_3}{\xi_1}-\pd{x_2}{\xi_1}\pd{x_3}{\xi_3},~
    J\pd{\xi_2}{x_2}=\pd{x_3}{\xi_3}\pd{x_1}{\xi_1}-\pd{x_3}{\xi_1}\pd{x_1}{\xi_3},~
    J\pd{\xi_2}{x_3}=\pd{x_1}{\xi_3}\pd{x_2}{\xi_1}-\pd{x_1}{\xi_1}\pd{x_2}{\xi_3},\\
    J\pd{\xi_3}{x_1}=\pd{x_2}{\xi_1}\pd{x_3}{\xi_2}-\pd{x_2}{\xi_2}\pd{x_3}{\xi_1},~
    J\pd{\xi_3}{x_2}=\pd{x_3}{\xi_1}\pd{x_1}{\xi_2}-\pd{x_3}{\xi_2}\pd{x_1}{\xi_1},~
    J\pd{\xi_3}{x_3}=\pd{x_1}{\xi_1}\pd{x_2}{\xi_2}-\pd{x_1}{\xi_2}
    \pd{x_2}{\xi_1},
  \end{aligned}
\end{equation*}
which can be reformulated into the divergence form
\begin{equation}\label{eq:CMM_SCL}
  \begin{aligned}
    &J\pd{\xi_1}{x_1}=\dfrac{\partial}{\partial\xi_3}\left(\pd{x_2}{\xi_2}x_3\right)-\dfrac{\partial}{\partial\xi_2}\left(\pd{x_2}{\xi_3}x_3\right),~
    J\pd{\xi_1}{x_2}=\dfrac{\partial}{\partial\xi_3}\left(\pd{x_3}{\xi_2}x_1\right)-\dfrac{\partial}{\partial\xi_2}\left(\pd{x_3}{\xi_3}x_1\right),\\
    &J\pd{\xi_1}{x_3}=\dfrac{\partial}{\partial\xi_3}\left(\pd{x_1}{\xi_2}x_2\right)-\dfrac{\partial}{\partial\xi_2}\left(\pd{x_1}{\xi_3}x_2\right),\\
    &J\pd{\xi_2}{x_1}=\dfrac{\partial}{\partial\xi_1}\left(\pd{x_2}{\xi_3}x_3\right)-\dfrac{\partial}{\partial\xi_3}\left(\pd{x_2}{\xi_1}x_3\right),~
    J\pd{\xi_2}{x_2}=\dfrac{\partial}{\partial\xi_1}\left(\pd{x_3}{\xi_3}x_1\right)-\dfrac{\partial}{\partial\xi_3}\left(\pd{x_3}{\xi_1}x_1\right),\\
    &J\pd{\xi_2}{x_3}=\dfrac{\partial}{\partial\xi_1}\left(\pd{x_1}{\xi_3}x_2\right)-\dfrac{\partial}{\partial\xi_3}\left(\pd{x_1}{\xi_1}x_2\right),\\
    &J\pd{\xi_3}{x_1}=\dfrac{\partial}{\partial\xi_2}\left(\pd{x_2}{\xi_1}x_3\right)-\dfrac{\partial}{\partial\xi_1}\left(\pd{x_2}{\xi_2}x_3\right),~
    J\pd{\xi_3}{x_2}=\dfrac{\partial}{\partial\xi_2}\left(\pd{x_3}{\xi_1}x_1\right)-\dfrac{\partial}{\partial\xi_1}\left(\pd{x_3}{\xi_2}x_1\right),\\
    &J\pd{\xi_3}{x_3}=\dfrac{\partial}{\partial\xi_2}\left(\pd{x_1}{\xi_1}x_2\right)
    -\dfrac{\partial}{\partial\xi_1}\left(\pd{x_1}{\xi_2}x_2\right).
  \end{aligned}
\end{equation}
Those are useful to compute the discrete metrics and to obtain the discrete SCLs by the CMM \cite{Thomas1979}.
Using the same discretizations for the first-order
spatial derivatives in \eqref{eq:CMM_SCL} gives
\begin{equation}\label{eq:SCLCoeff}
  \begin{aligned}
    &\left(J\pd{\xi_1}{x_1}\right)_{\bm{i}}=\dfrac{1}{\Delta\xi_2\Delta\xi_3}
    \left(\delta_3\left[\delta_2\left[x_2\right]x_3\right]-\delta_2\left[\delta_3\left[x_2\right]x_3\right]\right),~
    \left(J\pd{\xi_1}{x_2}\right)_{\bm{i}}=\dfrac{1}{\Delta\xi_2\Delta\xi_3}
    \left(\delta_3\left[\delta_2\left[x_3\right]{x_1}\right]-\delta_2\left[\delta_3\left[x_3\right]{x_1}\right]\right),\\
    &\left(J\pd{\xi_1}{x_3}\right)_{\bm{i}}=\dfrac{1}{\Delta\xi_2\Delta\xi_3}
    \left(\delta_3\left[\delta_2\left[x_1\right]{x_2}\right]-\delta_2\left[\delta_3\left[x_1\right]{x_2}\right]\right),\\
    &\left(J\pd{\xi_2}{x_1}\right)_{\bm{i}}=\dfrac{1}{\Delta\xi_3\Delta\xi_1}
    \left(\delta_1\left[\delta_3\left[x_2\right]{x_3}\right]-\delta_3\left[\delta_1\left[x_2\right]{x_3}\right]\right),~
    \left(J\pd{\xi_2}{x_2}\right)_{\bm{i}}=\dfrac{1}{\Delta\xi_3\Delta\xi_1}
    \left(\delta_1\left[\delta_3\left[x_3\right]{x_1}\right]-\delta_3\left[\delta_1\left[x_3\right]{x_1}\right]\right),\\
    &\left(J\pd{\xi_2}{x_3}\right)_{\bm{i}}=\dfrac{1}{\Delta\xi_3\Delta\xi_1}
    \left(\delta_1\left[\delta_3\left[x_1\right]{x_2}\right]-\delta_3\left[\delta_1\left[x_1\right]{x_2}\right]\right),\\
    &\left(J\pd{\xi_3}{x_1}\right)_{\bm{i}}=\dfrac{1}{\Delta\xi_1\Delta\xi_2}
    \left(\delta_2\left[\delta_1\left[x_2\right]{x_3}\right]-\delta_1\left[\delta_2\left[x_2\right]{x_3}\right]\right),~
    \left(J\pd{\xi_3}{x_2}\right)_{\bm{i}}=\dfrac{1}{\Delta\xi_1\Delta\xi_2}
    \left(\delta_2\left[\delta_1\left[x_3\right]{x_1}\right]-\delta_1\left[\delta_2\left[x_3\right]{x_1}\right]\right),\\
    &\left(J\pd{\xi_3}{x_3}\right)_{\bm{i}}=\dfrac{1}{\Delta\xi_1\Delta\xi_2}
    \left(\delta_2\left[\delta_1\left[x_1\right]{x_2}\right]-\delta_1\left[\delta_2\left[x_1\right]{x_2}\right]\right),
  \end{aligned}
\end{equation}
where
\begin{align*}
  \delta_k[a_{\bm{i}}]=\dfrac12\sum_{n=1}^p\alpha_{p,n}\left(a_{\bm{i},k,+n} - a_{\bm{i},k,-n}\right)
\end{align*}
is the $2p$th-order central difference operator in the $i_k$-direction.
Combing the above discretizations with the fluxes \eqref{eq:MetricsFlux_O2p}, one can verify that the discrete SCLs \eqref{eq:DiscSCL_O2p} are satisfied. For example, for $j=1$, one has
\begin{align*}
  &\sum_{k=1}^3\dfrac{1}{\Delta \xi_k}\left(\left(\widetilde{J\pd{\xi_k}{x_1}}\right)_{\bm{i},k,+\frac12}^{\twop}-\left(\widetilde{J\pd{\xi_k}{x_1}}\right)_{\bm{i},k,-\frac12}^{\twop}\right)
  =\sum_{k=1}^3\dfrac{1}{\Delta \xi_k}\delta_k
  \left[\left(J\pd{\xi_k}{x_1}\right)\right]\\
  =&\ \dfrac{1}{\Delta\xi_1}\delta_1\left[\left(J\pd{\xi_1}{x_1}\right)\right]
  +\dfrac{1}{\Delta\xi_2}\delta_2\left[\left(J\pd{\xi_2}{x_1}\right)\right]
  +\dfrac{1}{\Delta\xi_3}\delta_3\left[\left(J\pd{\xi_3}{x_1}\right)\right]\\
  =&\ \dfrac{1}{\Delta\xi_1\Delta\xi_2\Delta\xi_3}
  \Big(\delta_1\delta_3\left[\delta_2\left[x_2\right]{x_3}\right]-\delta_1\delta_2\left[\delta_3\left[x_2\right]{x_3}\right]
    +\delta_2\delta_1\left[\delta_3\left[x_2\right]{x_3}\right]\\
    &-\delta_2\delta_3\left[\delta_1\left[x_2\right]{x_3}\right]
  +\delta_3\delta_2\left[\delta_1\left[x_2\right]{x_3}\right]-\delta_3\delta_1\left[\delta_2\left[x_2\right]{x_3}\right]\Big)=0,
\end{align*}
since $\delta_j$ and $\delta_k$ are commutative, i.e. $\delta_j\delta_k=\delta_k\delta_j$.

The temporal metrics $\left(J\partial_t{\xi_k}\right)$ satisfy
\begin{equation*}
  J\pd{\xi_k}{t}=-\sum_{j=1}^3\pd{x_j}{t}\left(J\pd{\xi_k}{x_j}\right),~k=1,2,3,
\end{equation*}
so that one has the following approximation
\begin{equation}\label{eq:VCLCoeff}
  \left(J\pd{\xi_k}{t}\right)_{\bm{i}}=-\sum_{j=1}^3(\dot{x}_j)_{\bm{i}}\left(J\pd{\xi_k}{x_j}\right)_{\bm{i}},
\end{equation}
where $(\dot{x}_j)_{\bm{i}},~j=1,2,3$ are the grid velocities at $\bm{i}$,
which will be provided by some given expressions or solving the mesh equations in Section \ref{section:MM}.
Since the quantities $\left(J\pd{\xi_k}{x_j}\right)_{\bm{i}}$ have been obtained in \eqref{eq:SCLCoeff},
the implementation of \eqref{eq:VCLCoeff} is simple and cheap.
Combining \eqref{eq:VCLCoeff} with \eqref{eq:RMHDSemiJ_O2p} and \eqref{eq:MetricsFlux_O2p} gives the semi-discrete VCL.
Moreover, it can be verified the following  free-stream preserving property.

\begin{proposition}\label{prop:GCL}\rm
  If the semi-discrete schemes \eqref{eq:RMHDSemiU_O2p}-\eqref{eq:RMHDSemiJ_O2p} are
  integrated in time with the
  explicit SSP RK scheme from $t=t^n$ to
  $t^{n+1}=t^n+\Delta t^n$, with the time step size $\Delta t^n$, then
  the resulting fully-discrete schemes preserve the free-stream states.
\end{proposition}

\begin{proof}
  The forward Euler time discretization is considered here, since the explicit SSP RK schemes
  are a convex combination of the forward Euler time discretizations.
  Assuming that $\bU_{\bm{i}}^n=\bU_0$ is a physical constant state,
  rewrite the update of the metric Jacobian $J_{\bm{i}}$ and the solution $\bU_{\bm{i}}$ as follows
  \begin{align*}
    J_{\bm{i}}^{n+1}=&~J_{\bm{i}}^{n}-\sum_{k=1}^3\dfrac{\Delta t^n}{\Delta \xi_k}\left(\left(\widetilde{J\pd{\xi_k}{t}}\right)_{\bm{i},k,+\frac12}^{\twop}
    -\left(\widetilde{J\pd{\xi_k}{t}}\right)_{\bm{i},k,-\frac12}^{\twop}\right)
    =J_{\bm{i}}^{n}-\sum_{k=1}^3\dfrac{\Delta t^n}{\Delta \xi_k}\delta_k\left[\left({J\pd{\xi_k}{t}}\right)\right],\\
    (J\bU)_{\bm{i}}^{n+1}=&~(J\bU)_{\bm{i}}^{n}-\sum_{k=1}^3\dfrac{\Delta t^n}{\Delta \xi_k}\left((\widetilde{\Fcurv_k})_{\bm{i},k,+\frac12}^{\twop} - (\widetilde{\Fcurv_k})_{\bm{i},k,-\frac12}^{\twop}\right)\\
    =&~J_{\bm{i}}^n\bU_0-\sum_{k=1}^3\dfrac{\Delta t^n}{\Delta \xi_k}\sum_{n=1}^p\alpha_{p,n}\Bigg[\\
      &+\dfrac12\left(\left(J\pd{\xi_k}{t}\right)_{\bm{i}} + \left(J\pd{\xi_k}{t}\right)_{\bm{i},k,+n}\right)\bU_0
      +\sum\limits_{j=1}^3 \dfrac12\left(\left(J\pd{\xi_k}{x_j}\right)_{\bm{i}} + \left(J\pd{\xi_k}{x_j}\right)_{\bm{i},k,+n}\right)\bF_j(\bU_0) \\
      &-\dfrac12\left(\left(J\pd{\xi_k}{t}\right)_{\bm{i},k,-n} + \left(J\pd{\xi_k}{t}\right)_{\bm{i}}\right)\bU_0
    -\sum\limits_{j=1}^3 \dfrac12\left(\left(J\pd{\xi_k}{x_j}\right)_{\bm{i},k,-n} + \left(J\pd{\xi_k}{x_j}\right)_{\bm{i}}\right)\bF_j(\bU_0)\Bigg]\\
    =&\left(J_{\bm{i}}^{n}-\sum_{k=1}^3\dfrac{\Delta t^n}{\Delta \xi_k}\delta_k\left[\left(J\pd{\xi_k}{t}\right)\right]\right)\bU_0
    -\sum_{j=1}^3\left(\sum_{k=1}^3\dfrac{\Delta t^n}{\Delta \xi_k}\delta_k\left[\left(J\pd{\xi_k}{x_j}\right)\right]\right)\bF_j(\bU_0)\\
    =&~J_{\bm{i}}^{n+1}\bU_0,
  \end{align*}
  where the discrete GCLs have been used in the last equality. Thus $\bU_{\bm{i}}^{n+1}=(J\bU)_{\bm{i}}^{n+1}/J_{\bm{i}}^{n+1}=\bU_0$.
  The proof is completed.
\end{proof}

\subsection{Proof of high-order accuracy and EC property}
This section is devoted to present the high-order accurate EC schemes based on the previous results.
\begin{thm}\label{thm:ECScheme_HOEC}\rm
  The semi-discrete schemes \eqref{eq:RMHDSemiU_O2p}-\eqref{eq:RMHDSemiJ_O2p}
  with the fluxes \eqref{eq:ECFlux_O2p}-\eqref{eq:MetricsFlux_O2p}
  are $2p$th-order accurate and EC in the sense that
  \begin{equation}\label{eq:NumEntropyID_O2p}
    \dfrac{\dd}{\dd t}J_{\bm{i}}\eta(\bU_{\bm{i}}(t))
    +\sum_{k=1}^3\dfrac{1}{\Delta \xi_k}\left((\widetilde{\qcurv_k})_{\bm{i},k,+\frac12}^{\twop}
    -(\widetilde{\qcurv_k})_{\bm{i},k,-\frac12}^{\twop}\right)=0,
  \end{equation}
  with the consistent numerical entropy fluxes
  \begin{equation}\label{eq:NumEntropyFlux_O2p}
    (\widetilde{\qcurv_k})_{\bm{i},k,+\frac12}^{\twop}=
    \sum_{n=1}^p\alpha_{p,n}\sum_{s=0}^{n-1}
    \widetilde{\qcurv_k}\left(\bU_{\bm{i},k,-s}, \bU_{\bm{i},k,-s+n}, \left(J\pd{\xi_k}{\zeta}\right)_{\bm{i},k,-s}, \left(J\pd{\xi_k}{\zeta}\right)_{\bm{i},k,-s+n} \right),
  \end{equation}
  where
  \begin{align}\label{eq:NumEntropyFluxCurv}
    &\widetilde{\qcurv_k}\left(\bU_l, \bU_r, \left(J\pd{\xi_k}{\zeta}\right)_l, \left(J\pd{\xi_k}{\zeta}\right)_r \right)\nonumber\\
    =&\ \dfrac12\left(\bV(\bU_l)+\bV(\bU_r)\right)^\mathrm{T}\widetilde{\Fcurv_k}\left(\bU_l, \bU_r, \left(J\pd{\xi_k}{\zeta}\right)_l, \left(J\pd{\xi_k}{\zeta}\right)_r \right)\nonumber\\
    &-\dfrac14\left(\left(J\pd{\xi_k}{t}\right)_l+\left(J\pd{\xi_k}{t}\right)_r\right)\left(\phi(\bU_l)+\phi(\bU_r)\right)\nonumber\\
    &-\sum_{j=1}^3\dfrac14\left(\left(J\pd{\xi_k}{x_j}\right)_l+\left(J\pd{\xi_k}{x_j}\right)_r\right)\left(\psi_j(\bU_l)+\psi_j(\bU_r)\right)\nonumber\\
    &+\sum_{j=1}^3\dfrac18\Bigg(\left(J\pd{\xi_k}{x_j}\right)_l+\left(J\pd{\xi_k}{x_j}\right)_r\Bigg)\Bigg((B_j)_l+(B_j)_r\Bigg)\left(\Phi(\bU_l)+\Phi(\bU_r)\right).
  \end{align}
\end{thm}

\begin{proof}
  From Proposition \ref{prop:ECScheme_HOacc} and the discretizations of the metrics $\left(J\pd{\xi_k}{\zeta}\right)_{\bm{i}}$ in \eqref{eq:SCLCoeff} and \eqref{eq:VCLCoeff}, ${\zeta}=t,x_1,x_2,x_3$, it is obvious that the semi-discrete schemes \eqref{eq:RMHDSemiU_O2p}-\eqref{eq:RMHDSemiJ_O2p} are $2p$th-order accurate in space.

  Taking the dot product of \eqref{eq:RMHDSemiU_O2p} with $\bV_{\bm{i}}$ and using the chain rule and the semi-discrete VCL \eqref{eq:RMHDSemiJ_O2p} gives
  \begin{align*}
    \dfrac{\dd}{\dd t}(J_{\bm{i}}\eta_{\bm{i}})=&
    -\sum_{k=1}^3\dfrac{1}{\Delta \xi_k} \Bigg\{\bV_{\bm{i}}^\mathrm{T}\left((\widetilde{\Fcurv_k})_{\bm{i},k,+\frac12}^{\twop}
      -(\widetilde{\Fcurv_k})_{\bm{i},k,-\frac12}^{\twop}\right)
      -\phi_{\bm{i}}\Bigg(\left(\widetilde{J\pd{\xi_k}{t}}\right)_{\bm{i},k,+\frac12}^{\twop}
      -\left(\widetilde{J\pd{\xi_k}{t}}\right)_{\bm{i},k,-\frac12}^{\twop}\Bigg) \\
      &+\Phi_{\bm{i}}\left((\widetilde{\Bcurv_k})_{\bm{i},k,+\frac12}^{\twop}
    -(\widetilde{\Bcurv_k})_{\bm{i},k,-\frac12}^{\twop}\right) \Bigg\}.
  \end{align*}
  Further utilizing the discrete SCLs \eqref{eq:DiscSCL_O2p}  can get
  \begin{align}\label{eq:DiscEntropyID_O2p_step1}
    {\dfrac{\dd}{\dd t}(J_{\bm{i}}\eta_{\bm{i}})}=&
    -\sum_{k=1}^3\dfrac{1}{\Delta \xi_k} \Bigg\{\bV_{\bm{i}}^\mathrm{T}\left((\widetilde{\Fcurv_k})_{\bm{i},k,+\frac12}^{\twop}
      -(\widetilde{\Fcurv_k})_{\bm{i},k,-\frac12}^{\twop}\right)
      -\phi_{\bm{i}}\left(\left(\widetilde{J\pd{\xi_k}{t}}\right)_{\bm{i},k,+\frac12}^{\twop}
      -\left(\widetilde{J\pd{\xi_k}{t}}\right)_{\bm{i},k,-\frac12}^{\twop}\right) \nonumber\\
      &-\sum_{j=1}^3(\psi_j)_{\bm{i}}\Bigg(\left(\widetilde{J\pd{\xi_k}{x_j}}\right)_{\bm{i},k,+\frac12}^{\twop}
      -\left(\widetilde{J\pd{\xi_k}{x_j}}\right)_{\bm{i},k,-\frac12}^{\twop}\Bigg)
      +\Phi_{\bm{i}}\left((\widetilde{\Bcurv_k})_{\bm{i},k,+\frac12}^{\twop}
    -(\widetilde{\Bcurv_k})_{\bm{i},k,-\frac12}^{\twop}\right) \Bigg\} \nonumber\\
    &=-\sum_{k=1}^3\sum_{n=1}^p\dfrac{\alpha_{p,n}}{\Delta \xi_k}\left(I_1-I_2-I_3+I_4\right),
  \end{align}
  where
  \begin{align*}
    I_1=&\ \bV_{\bm{i}}^\mathrm{T}
    \left[\widetilde{\Fcurv_k}\left(\bU_{\bm{i}}, \bU_{\bm{i},k,+n},
      \left(J\pd{\xi_k}{\zeta}\right)_{\bm{i}}, \left(J\pd{\xi_k}{\zeta}\right)_{\bm{i},k,+n}\right)
      -\widetilde{\Fcurv_k}\left(\bU_{\bm{i}}, \bU_{\bm{i},k,-n},
    \left(J\pd{\xi_k}{\zeta}\right)_{\bm{i}}, \left(J\pd{\xi_k}{\zeta}\right)_{\bm{i},k,-n}\right)\right],\\
    I_2=&\ \phi_{\bm{i}}
    \left[\dfrac12\left(\left(J\pd{\xi_k}{t}\right)_{\bm{i}}+\left(J\pd{\xi_k}{t}\right)_{\bm{i},k,+n}\right)
    -\dfrac12\left(\left(J\pd{\xi_k}{t}\right)_{\bm{i}}+\left(J\pd{\xi_k}{t}\right)_{\bm{i},k,-n}\right)\right],\\
    I_3=&\ \sum_{j=1}^3(\psi_j)_{\bm{i}}
    \left[\dfrac12\left(\left(J\pd{\xi_k}{x_j}\right)_{\bm{i}}+\left(J\pd{\xi_k}{x_j}\right)_{\bm{i},k,+n}\right)
    -\dfrac12\left(\left(J\pd{\xi_k}{x_j}\right)_{\bm{i}}+\left(J\pd{\xi_k}{x_j}\right)_{\bm{i},k,-n}\right)\right],\\
    I_4=&\ \sum_{j=1}^3\Phi_{\bm{i}}\Bigg[\dfrac14\Bigg(\left(J\pd{\xi_k}{x_j}\right)_{\bm{i}}+\left(J\pd{\xi_k}{x_j}\right)_{\bm{i},k,+n}\Bigg)\Bigg((B_j)_{\bm{i}}+(B_j)_{\bm{i},k,+n}\Bigg) \\
    &-\dfrac14\Bigg(\left(J\pd{\xi_k}{x_j}\right)_{\bm{i}}+\left(J\pd{\xi_k}{x_j}\right)_{\bm{i},k,-n}\Bigg)\Bigg((B_j)_{\bm{i}}+(B_j)_{\bm{i},k,-n}\Bigg)\Bigg].
  \end{align*}
 If splitting $\bV_{\bm{i}}$ as $\dfrac12\left(\bV_{\bm{i}}+\bV_{\bm{i},k,+n}\right)-\dfrac12\left(\bV_{\bm{i},k,+n}-\bV_{\bm{i}}\right)$
  or $\dfrac12\left(\bV_{\bm{i},k,-n}+\bV_{\bm{i}}\right)+\dfrac12\left(\bV_{\bm{i}}-\bV_{\bm{i},k,-n}\right)$, then $I_1$ goes to
  \begin{align}\label{eq:I1_split_HO}
    I_1=&+\dfrac12\left(\bV_{\bm{i}}+\bV_{\bm{i},k,+n}\right)^\mathrm{T}
    \widetilde{\Fcurv_k}\left(\bU_{\bm{i}}, \bU_{\bm{i},k,+n},
    \left(J\pd{\xi_k}{\zeta}\right)_{\bm{i}}, \left(J\pd{\xi_k}{\zeta}\right)_{\bm{i},k,+n}\right)\nonumber\\
    &-\dfrac12\left(\bV_{\bm{i},k,+n}-\bV_{\bm{i}}\right)^\mathrm{T}
    \widetilde{\Fcurv_k}\left(\bU_{\bm{i}}, \bU_{\bm{i},k,+n},
    \left(J\pd{\xi_k}{\zeta}\right)_{\bm{i}}, \left(J\pd{\xi_k}{\zeta}\right)_{\bm{i},k,+n}\right)\nonumber\\
    &-\dfrac12\left(\bV_{\bm{i}}+\bV_{\bm{i},k,-n}\right)^\mathrm{T}\widetilde{\Fcurv_k}\left(\bU_{\bm{i}}, \bU_{\bm{i},k,-n},
    \left(J\pd{\xi_k}{\zeta}\right)_{\bm{i}}, \left(J\pd{\xi_k}{\zeta}\right)_{\bm{i},k,-n}\right)\nonumber\\
    &-\dfrac12\left(\bV_{\bm{i}}-\bV_{\bm{i},k,-n}\right)^\mathrm{T}\widetilde{\Fcurv_k}\left(\bU_{\bm{i}}, \bU_{\bm{i},k,-n},
    \left(J\pd{\xi_k}{\zeta}\right)_{\bm{i}}, \left(J\pd{\xi_k}{\zeta}\right)_{\bm{i},k,-n}\right).
  \end{align}
  Similarly, treating $\phi_{\bm{i}},(\psi_j)_{\bm{i}}$ and $\Phi_{\bm{i}}$ gives
  \begin{align}
    \label{eq:I2_split_HO}
    I_2=&+\dfrac14\left(\phi_{\bm{i}}+\phi_{\bm{i},k,+n}\right)
    \left(\left(J\pd{\xi_k}{t}\right)_{\bm{i}}+\left(J\pd{\xi_k}{t}\right)_{\bm{i},k,+n}\right)
    -\dfrac14\left(\phi_{\bm{i},k,+n}-\phi_{\bm{i}}\right)
    \left(\left(J\pd{\xi_k}{t}\right)_{\bm{i}}+\left(J\pd{\xi_k}{t}\right)_{\bm{i},k,+n}\right)\nonumber\\
    &-\dfrac14\left(\phi_{\bm{i}}+\phi_{\bm{i},k,-n}\right)
    \left(\left(J\pd{\xi_k}{t}\right)_{\bm{i}}+\left(J\pd{\xi_k}{t}\right)_{\bm{i},k,-n}\right)
    -\dfrac14\left(\phi_{\bm{i}}-\phi_{\bm{i},k,-n}\right)
    \left(\left(J\pd{\xi_k}{t}\right)_{\bm{i}}+\left(J\pd{\xi_k}{t}\right)_{\bm{i},k,-n}\right),\\
    \label{eq:I3_split_HO}
    I_3=&+\sum_{j=1}^3\Bigg[\dfrac14\left((\psi_{j})_{\bm{i}}+(\psi_{j})_{\bm{i},k,+n}\right)
      \left(\left(J\pd{\xi_k}{x_j}\right)_{\bm{i}}+\left(J\pd{\xi_k}{x_j}\right)_{\bm{i},k,+n}\right)\nonumber\\
      &-\dfrac14\left((\psi_{j})_{\bm{i},k,+n}-(\psi_{j})_{\bm{i}}\right)
      \left(\left(J\pd{\xi_k}{x_j}\right)_{\bm{i}}+\left(J\pd{\xi_k}{x_j}\right)_{\bm{i},k,+n}\right)\nonumber\\
      &-\dfrac14\left((\psi_{j})_{\bm{i}}+(\psi_{j})_{\bm{i},k,-n}\right)\left(\left(J\pd{\xi_k}{x_j}\right)_{\bm{i}}+\left(J\pd{\xi_k}{x_j}\right)_{\bm{i},k,-n}\right)\nonumber\\
    &-\dfrac14\left((\psi_{j})_{\bm{i}}-(\psi_{j})_{\bm{i},k,-n}\right)\left(\left(J\pd{\xi_k}{x_j}\right)_{\bm{i}}+\left(J\pd{\xi_k}{x_j}\right)_{\bm{i},k,-n}\right)\Bigg],\\
    \label{eq:I4_split_HO}
    I_4=&+\sum_{j=1}^3\Bigg[\dfrac18\left(\Phi_{\bm{i}}+\Phi_{\bm{i},k,+n}\right)
      \left(\left(J\pd{\xi_k}{x_j}\right)_{\bm{i}}+\left(J\pd{\xi_k}{x_j}\right)_{\bm{i},k,+n}\right)
      \Bigg((B_j)_{\bm{i}}+(B_j)_{\bm{i},k,+n}\Bigg)\nonumber\\
      &-\dfrac18\left(\Phi_{\bm{i},k,+n}-\Phi_{\bm{i}}\right)
      \left(\left(J\pd{\xi_k}{x_j}\right)_{\bm{i}}+\left(J\pd{\xi_k}{x_j}\right)_{\bm{i},k,+n}\right)
      \Bigg((B_j)_{\bm{i}}+(B_j)_{\bm{i},k,+n}\Bigg)\nonumber\\
      &-\dfrac18\left(\Phi_{\bm{i}}+\Phi_{\bm{i},k,-n}\right)\left(\left(J\pd{\xi_k}{x_j}\right)_{\bm{i}}+\left(J\pd{\xi_k}{x_j}\right)_{\bm{i},k,-n}\right)
      \Bigg((B_j)_{\bm{i}}+(B_j)_{\bm{i},k,-n}\Bigg)\nonumber\\
      &-\dfrac18\left(\Phi_{\bm{i}}-\Phi_{\bm{i},k,-n}\right)\left(\left(J\pd{\xi_k}{x_j}\right)_{\bm{i}}+\left(J\pd{\xi_k}{x_j}\right)_{\bm{i},k,-n}\right)
    \Bigg((B_j)_{\bm{i}}+(B_j)_{\bm{i},k,-n}\Bigg)\Bigg].
  \end{align}
  Substituting the sufficient condition \eqref{eq:ECConditionCurv} into \eqref{eq:I1_split_HO} yields
  \begin{align}\label{eq:I1_EC_HO}
    I_1=&+\dfrac12\left(\bV_{\bm{i}}+\bV_{\bm{i},k,+n}\right)^\mathrm{T}
    \widetilde{\Fcurv_k}\left(\bU_{\bm{i}}, \bU_{\bm{i},k,+n},
    \left(J\pd{\xi_k}{\zeta}\right)_{\bm{i}}, \left(J\pd{\xi_k}{\zeta}\right)_{\bm{i},k,+n}\right)\nonumber\\
    &-\dfrac12\left(\bV_{\bm{i}}+\bV_{\bm{i},k,-n}\right)^\mathrm{T}\widetilde{\Fcurv_k}\left(\bU_{\bm{i}}, \bU_{\bm{i},k,-n},
    \left(J\pd{\xi_k}{\zeta}\right)_{\bm{i}}, \left(J\pd{\xi_k}{\zeta}\right)_{\bm{i},k,-n}\right)\nonumber\\
    &-\dfrac14\Bigg[\left(\phi_{\bm{i},k,+n}-\phi_{\bm{i}}\right)
      \Bigg(\left(J\pd{\xi_k}{t}\right)_{\bm{i}}+\left(J\pd{\xi_k}{t}\right)_{\bm{i},k,+n}\Bigg)\nonumber\\
      &+\sum_{j=1}^3\left((\psi_j)_{\bm{i},k,+n}-(\psi_j)_{\bm{i}}\right)
      \Bigg(\left(J\pd{\xi_k}{x_j}\right)_{\bm{i}}+\left(J\pd{\xi_k}{x_j}\right)_{\bm{i},k,+n}\Bigg)
    \Bigg]\nonumber\\
    &+\dfrac18\left(\Phi_{\bm{i},k,+n}-\Phi_{\bm{i}}\right)
    \Bigg(\left(J\pd{\xi_k}{x_j}\right)_{\bm{i}}+\left(J\pd{\xi_k}{x_j}
    \right)_{\bm{i},k,+n}\Bigg)\Bigg((B_j)_{\bm{i}}+(B_j)_{\bm{i},k,+n}\Bigg)\nonumber\\
    &-\dfrac14\Bigg[\left(\phi_{\bm{i}}-\phi_{\bm{i},k,-n}\right)
      \Bigg(\left(J\pd{\xi_k}{t}\right)_{\bm{i}}+\left(J\pd{\xi_k}{t}
      \right)_{\bm{i},k,-n}\Bigg)\nonumber\\
      &+\sum_{j=1}^3\left((\psi_j)_{\bm{i}}-(\psi_j)_{\bm{i},k,-n}\right)
      \Bigg(\left(J\pd{\xi_k}{x_j}\right)_{\bm{i}}
      +\left(J\pd{\xi_k}{x_j}\right)_{\bm{i},k,-n}\Bigg)
    \Bigg]\nonumber\\
    &+\dfrac18\left(\Phi_{\bm{i}}-\Phi_{\bm{i},k,-n}\right)
    \Bigg(\left(J\pd{\xi_k}{x_j}\right)_{\bm{i}}+\left(J\pd{\xi_k}{x_j}\right)
    _{\bm{i},k,-n}\Bigg)\Bigg((B_j)_{\bm{i}}+(B_j)_{\bm{i},k,-n}\Bigg).
  \end{align}
  Combining \eqref{eq:I2_split_HO}-\eqref{eq:I1_EC_HO} with \eqref{eq:NumEntropyFluxCurv} gives
  \begin{align*}
    I_1-I_2-I_3+I_4=&+\dfrac12\left(\bV_{\bm{i}}+\bV_{\bm{i},k,+n}\right)^\mathrm{T}
    \widetilde{\Fcurv_k}\left(\bU_{\bm{i}}, \bU_{\bm{i},k,+n},
    \left(J\pd{\xi_k}{\zeta}\right)_{\bm{i}}, \left(J\pd{\xi_k}{\zeta}\right)_{\bm{i},k,+n}\right)\nonumber\\
    &-\dfrac12\left(\bV_{\bm{i},k,-n}+\bV_{\bm{i}}\right)^\mathrm{T}\widetilde{\Fcurv_k}\left(\bU_{\bm{i}}, \bU_{\bm{i},k,-n},
    \left(J\pd{\xi_k}{\zeta}\right)_{\bm{i}}, \left(J\pd{\xi_k}{\zeta}\right)_{\bm{i},k,-n}\right)\nonumber\\
    &-\dfrac14\left(\phi_{\bm{i}}+\phi_{\bm{i},k,+n}\right)
    \Bigg(\left(J\pd{\xi_k}{t}\right)_{\bm{i}}+\left(J\pd{\xi_k}{t}\right)_{\bm{i},k,+n}\Bigg)\nonumber\\
    &+\dfrac14\left(\phi_{\bm{i}}+\phi_{\bm{i},k,-n}\right)
    \Bigg(\left(J\pd{\xi_k}{t}\right)_{\bm{i}}+\left(J\pd{\xi_k}{t}\right)_{\bm{i},k,-n}\Bigg)\nonumber\\
    &-\sum_{j=1}^3\Bigg[\dfrac14\left((\psi_{j})_{\bm{i}}+(\psi_{j})_{\bm{i},k,+n}\right)
      \Bigg(\left(J\pd{\xi_k}{x_j}\right)_{\bm{i}}+\left(J\pd{\xi_k}{x_j}\right)_{\bm{i},k,+n}\Bigg)\nonumber\\
    &-\dfrac14\left((\psi_{j})_{\bm{i}}+(\psi_{j})_{\bm{i},k,-n}\right)\Bigg(\left(J\pd{\xi_k}{x_j}\right)_{\bm{i}}+\left(J\pd{\xi_k}{x_j}\right)_{\bm{i},k,-n}\Bigg)\Bigg]\nonumber\\
    &+\sum_{j=1}^3\Bigg[\dfrac18\left(\Phi_{\bm{i}}+\Phi_{\bm{i},k,+n}\right)
      \Bigg(\left(J\pd{\xi_k}{x_j}\right)_{\bm{i}}+\left(J\pd{\xi_k}{x_j}\right)_{\bm{i},k,+n}\Bigg)
      \Bigg((B_j)_{\bm{i}}+(B_j)_{\bm{i},k,+n}\Bigg)\nonumber\\
      &-\dfrac18\left(\Phi_{\bm{i}}+\Phi_{\bm{i},k,-n}\right)
      \Bigg(\left(J\pd{\xi_k}{x_j}\right)_{\bm{i}}+\left(J\pd{\xi_k}{x_j}\right)_{\bm{i},k,-n}\Bigg)
    \Bigg((B_j)_{\bm{i}}+(B_j)_{\bm{i},k,-n}\Bigg)\Bigg]\nonumber\\
    =&~\widetilde{\qcurv_k}\left(\bU_{\bm{i}}, \bU_{\bm{i},k,+n}, \left(J\pd{\xi_k}{\zeta}\right)_{\bm{i}}, \left(J\pd{\xi_k}{\zeta}\right)_{\bm{i},k,+n} \right)
    -\widetilde{\qcurv_k}\left(\bU_{\bm{i}}, \bU_{\bm{i},k,-n}, \left(J\pd{\xi_k}{\zeta}\right)_{\bm{i}}, \left(J\pd{\xi_k}{\zeta}\right)_{\bm{i},k,-n} \right),
  \end{align*}
  thus \eqref{eq:DiscEntropyID_O2p_step1} becomes the numerical entropy identity \eqref{eq:NumEntropyID_O2p}.
  Moreover, it is easy to check the consistency of the numerical entropy flux $(\widetilde{\qcurv_k})_{\bm{i},k,\pm\frac12}^{\twop}$ with
  $\qcurv_k$. The proof is completed.
\end{proof}

%% file: ESScheme.tex
\section{High-order accurate ES schemes}\label{section:ESScheme}
It is known that for the quasi-linear hyperbolic conservation laws,
the entropy identity is available only if the solution is smooth.
For the discontinuous solutions, one should consider the entropy inequality.
Meanwhile, the EC schemes may produce serious nonphysical oscillations near the discontinuities.
Those motivate us to construct the high-order accurate ES schemes (satisfying the entropy inequality for the given entropy pair).
It can be achieved by adding suitable high-order dissipation to the EC flux \eqref{eq:ECFlux_O2p} to obtain the  $w$th-order ($w=2p-1\geq 3$) accurate ES flux
\begin{align}\label{eq:ESFlux_HO}
  &(\widehat{\Fcurv_k})_{\bm{i},k,+\frac12}^{\wth}=(\widetilde{\Fcurv_k})_{\bm{i},k,+\frac12}^{\twop}-
  \dfrac12 \bm{D}_{\bm{i},k,+\frac12}\bm{Y}_{\bm{i},k,+\frac12}
  \jumpangle{\widetilde{\bV}}_{\bm{i},k,+\frac12}^{\WENO},
\end{align}
where the matrix $\bm{D}_{\bm{i},k,+\frac12}$ is obtained by evaluating $\bm{D}:=\widehat{\lambda}\bT^{-1}\bm{R}(\bT\bU)$ at $\bm{i},k,+\frac12$,
 and $\bT$ is the ``rotational'' matrix,
 which is defined by $\bT=\diag\{1, \bT_0, 1\}$
and $\bT=\diag\{1, \bT_0, 1, \bT_0\}$ in the RHD and RMHD case, respectively, with
\begin{align*}
	&\bT_0 =
	\begin{bmatrix}
		\cos\varphi\cos\theta  & \cos\varphi\sin\theta  & \sin\varphi \\
	    -\sin\theta            & \cos\theta             & 0           \\
		-\sin\varphi\cos\theta & -\sin\varphi\sin\theta & \cos\varphi \\
	\end{bmatrix},\\
  &\theta = \arctan\left(\left(J\pd{\xi_k}{x_2}\right)\Big/\left(J\pd{\xi_k}{x_1}\right)\right),\\
  &\varphi = \arctan\left(\left(J\pd{\xi_k}{x_3}\right)\Bigg/
  \sqrt{\left(J\pd{\xi_k}{x_1}\right)^2+\left(J\pd{\xi_k}{x_2}\right)^2}
  \right).
\end{align*}
Here $\widehat{\lambda}$ is taken as the spectral radius
\begin{equation*}
  \widehat{\lambda}:=\max_m\left\{ \left| J\pd{\xi_k}{t}+L_k\lambda_m(\bT\bU) \right|\right\},
\end{equation*}
with $L_k=\sqrt{\sum\limits_{j=1}^3\left(J\pd{\xi_k}{x_j}\right)^2}$,
and $\bm{R}$ is a set of scaled eigenvectors such that
\begin{equation*}
  \pd{\bU}{\bV}=\bm{R}\bm{R}^\mathrm{T},\
  \pd{\bF_1}{\bU}=\bm{R}\bm{\Lambda}\bm{R}^{-1},\
  \bm{\Lambda}=\mbox{diag}\{\lambda_1,\ldots,\lambda_m\},
\end{equation*}
where $\lambda_1,\cdots,\lambda_m$ are the eigenvalues and $m$ is the equation number (e.g. $m=5$ and 8 for the RHD and RMHD cases respectively when $d=3$). The detailed computation of the eigenvalues and eigenvectors has been given in \cite{Duan2021RHDMM, Duan2020RMHD}.
To obtain high-order accuracy, the high-order WENO reconstruction is performed in the scaled entropy variables.
More specifically,  the $w$th-order ($w=2p-1$) WENO reconstruction \cite{Jiang1996Efficient} is performed on $\{\widetilde{\bV}=\bm{R}_{\bm{i},k,+\frac12}^\mathrm{T}(\bT\bU)\bT_{\bm{i},k,+\frac12}\bV\}$ in the $i_k$-direction to obtain
the left and right limit values denoted by
$\widetilde{\bV}_{\bm{i},k,+\frac12}^{\scriptsize\text{WENO},-}$ and $\widetilde{\bV}_{\bm{i},k,+\frac12}^{\scriptsize\text{WENO},+}$, and then define
\begin{equation*}
  \jumpangle{\widetilde{\bV}}_{\bm{i},k,+\frac12}^{\WENO}=\widetilde{\bV}_{\bm{i},k,+\frac12}^{\scriptsize\text{WENO},+}
  -\widetilde{\bV}_{\bm{i},k,+\frac12}^{\scriptsize\text{WENO},-}.
\end{equation*}
In \eqref{eq:ESFlux_HO}, the diagonal matrix  $\bm{Y}_{\bm{i},k,+\frac12}$  is used to enforce the ``sign'' property,  see \cite{Biswas2018Low},
with the diagonal component given by
\begin{equation*}
	(\bm{Y}_{\bm{i},k,+\frac12})_{l,l}=
	\begin{cases}
		1, & \text{sign}(\jumpangle{\widetilde{\bV}_l}_{\bm{i},k,+\frac12})=\text{sign}(\jump{\widetilde{\bV}_l}_{\bm{i},k,+\frac12}), \\
		0, &\text{otherwise},
	\end{cases}
\end{equation*}
where $\jump{a}_{\bm{i},k,+\frac12}=a_{\bm{i},k,+1}-a_{\bm{i}}$.

\begin{thm}\rm
  By replacing the $2p$th-order EC flux with $w$th-order ES flux \eqref{eq:ESFlux_HO}, the following schemes
  \begin{align}
    \label{eq:RMHDSemiU_O2p_ES}
    &\dfrac{\dd}{\dd t}\bm{\mathcal{U}}_{\bm{i}}=
    -\sum_{k=1}^3\dfrac{1}{\Delta \xi_k}\left((\widehat{\Fcurv_k})_{\bm{i},k,+\frac12}^{\wth}-(\widehat{\Fcurv_k})_{\bm{i},k,-\frac12}^{\wth}\right)
    -\Phi'(\bV_{\bm{i}})^\mathrm{T}\sum_{k=1}^3\dfrac{1}{\Delta \xi_k}\left((\widetilde{\Bcurv_k})_{\bm{i},k,+\frac12}^{\twop}-(\widetilde{\Bcurv_k})_{\bm{i},k,-\frac12}^{\twop}\right),
    \\
    \label{eq:RMHDSemiJ_O2p_ES}
    &\dfrac{\dd}{\dd t}J_{\bm{i}}=
    -\sum_{k=1}^3\dfrac{1}{\Delta \xi_k}\Bigg(\left(\widetilde{J\pd{\xi_k}{t}}\right)_{\bm{i},k,+\frac12}^{\twop}-\left(\widetilde{J\pd{\xi_k}{t}}\right)_{\bm{i},k,-\frac12}^{\twop}\Bigg),
  \end{align}
  are ES. Specially, they satisfy the entropy inequality
  \begin{equation*}
    \dfrac{\dd}{\dd t}J_{\bm{i}}\eta(\bU_{\bm{i}}(t))
    +\sum_{k=1}^3\dfrac{1}{\Delta \xi_k}\left((\widehat{\qcurv_k})_{\bm{i},k,+\frac12}^{\wth}
    -(\widehat{\qcurv_k})_{\bm{i},k,-\frac12}^{\wth}\right)\leqslant0,
  \end{equation*}
  with the consistent numerical entropy fluxes
  \begin{equation}\label{eq:NumEntropyFlux_O2p_ES}
    (\widehat{\qcurv_k})_{\bm{i},k,+\frac12}^{\wth}=
    (\widetilde{\qcurv_k})_{\bm{i},k,+\frac12}^{\twop}
    - \dfrac12\widehat{\lambda}_{\bm{i},k,+\frac12}\mean{\widetilde{\bV}}^\mathrm{T}_{\bm{i},k,+\frac12}
    \bm{Y}_{\bm{i},k,+\frac12}\jumpangle{\widetilde{\bV}}_{\bm{i},k,+\frac12}^{\WENO},
  \end{equation}
  where $\mean{a}_{\bm{i},k,+\frac12}=\frac12(a_{\bm{i},k,+1}+a_{\bm{i}})$.
\end{thm}

\begin{proof}
  Taking the dot product of $\bV_{\bm{i}}$ and \eqref{eq:RMHDSemiU_O2p_ES} gives
  \begin{align*}
    {\dfrac{\dd}{\dd t}(J_{\bm{i}}\eta_{\bm{i}})}=&
    -\sum_{k=1}^3\dfrac{1}{\Delta \xi_k}\left((\widetilde{\qcurv_k})_{\bm{i},k,+\frac12}^{\twop}
    -(\widetilde{\qcurv_k})_{\bm{i},k,-\frac12}^{\twop}\right)\\
    &+\sum_{k=1}^3\dfrac{1}{2\Delta \xi_k}\Big(\widehat{\lambda}_{\bm{i},k,+\frac12}\bV_i^\mathrm{T}
      \bT_{\bm{i},k,+\frac12}^{-1}\bm{R}_{\bm{i},k,+\frac12}(\bT\bU)\bm{Y}_{\bm{i},k,+\frac12}\jumpangle{\widetilde{\bV}}_{\bm{i},k,+\frac12}^{\WENO} \\
      &- \widehat{\lambda}_{\bm{i},k,-\frac12}\bV_i^\mathrm{T}
    \bT_{\bm{i},k,-\frac12}^{-1}\bm{R}_{\bm{i},k,-\frac12}(\bT\bU)\bm{Y}_{\bm{i},k,-\frac12}\jumpangle{\widetilde{\bV}}_{\bm{i},k,-\frac12}^{\WENO} \Big)\\
    =&-\sum_{k=1}^3\dfrac{1}{\Delta \xi_k}\left((\widehat{\qcurv_k})_{\bm{i},k,+\frac12}^{\wth}
    -(\widehat{\qcurv_k})_{\bm{i},k,-\frac12}^{\wth}\right) \\
    &-\sum_{k=1}^3\dfrac{1}{4\Delta \xi_k}\Bigg(\widehat{\lambda}_{\bm{i},k,+\frac12}\jump{\bV}_{\bm{i},k,+\frac12}^\mathrm{T}
      \bT_{\bm{i},k,+\frac12}^{-1}\bm{R}_{\bm{i},k,+\frac12}(\bT\bU)
      \bm{Y}_{\bm{i},k,+\frac12}\jumpangle{\widetilde{\bV}}_{\bm{i},k,+\frac12}^{\WENO} \\
      &+ \widehat{\lambda}_{\bm{i},k,-\frac12}\jump{\bV}_{\bm{i},k,-\frac12}^\mathrm{T}
      \bT_{\bm{i},k,-\frac12}^{-1}\bm{R}_{\bm{i},k,-\frac12}(\bT\bU)
    \bm{Y}_{\bm{i},k,-\frac12}\jumpangle{\widetilde{\bV}}_{\bm{i},k,-\frac12}^{\WENO} \Bigg)\\
    =&-\sum_{k=1}^3\dfrac{1}{\Delta \xi_k}\left((\widehat{\qcurv_k})_{\bm{i},k,+\frac12}^{\wth}
    -(\widehat{\qcurv_k})_{\bm{i},k,-\frac12}^{\wth}\right) \\
    &-\sum_{k=1}^3\dfrac{1}{4\Delta \xi_k}\Big(\widehat{\lambda}_{\bm{i},k,+\frac12}\jump{\widetilde{\bV}}^\mathrm{T}_{\bm{i},k,+\frac12}
      \bm{Y}_{\bm{i},k,+\frac12}\jumpangle{\widetilde{\bV}}_{\bm{i},k,+\frac12}^{\WENO}
      + \widehat{\lambda}_{\bm{i},k,-\frac12}\jump{\widetilde{\bV}}_{\bm{i},k,-\frac12}^\mathrm{T}
    \bm{Y}_{\bm{i},k,-\frac12}\jumpangle{\widetilde{\bV}}_{\bm{i},k,-\frac12}^{\WENO} \Big),
  \end{align*}
  where the 1st equality uses the entropy identity satisfied by the $2p$th-order EC scheme \eqref{eq:NumEntropyID_O2p},
  the 2nd equality uses \eqref{eq:NumEntropyFlux_O2p_ES}.
  From the definition of $\bm{Y}_{\bm{i},k,\pm\frac12}$, one can get
  \begin{equation*}
    \jump{\widetilde{\bV}}^\mathrm{T}_{\bm{i},k,\pm\frac12}
    \bm{Y}_{\bm{i},k,\pm\frac12}\jumpangle{\widetilde{\bV}}_{\bm{i},k,\pm\frac12}^{\WENO}\geqslant 0,
  \end{equation*}
  therefore, it holds
  \begin{equation*}
     \dfrac{\dd}{\dd t}J_{\bm{i}}\eta(\bU_{\bm{i}}(t))
    +\sum_{k=1}^3\dfrac{1}{\Delta \xi_k}\left((\widehat{\qcurv_k})_{\bm{i},k,+\frac12}^{\wth}
    -(\widehat{\qcurv_k})_{\bm{i},k,-\frac12}^{\wth}\right)\leqslant0. \qedhere
  \end{equation*}
\end{proof}

\begin{rmk}\rm
  When the solution is a constant state, the dissipation terms vanish, so that the ES schemes preserve the free-stream state.
\end{rmk}

%% file: MovingMesh.tex
\section{Adaptive moving mesh strategy}\label{section:MM}

This section  presents our adaptive moving mesh strategy at time $t=t^n$ for the completeness of the paper,
but focuses on the mesh iteration redistribution with the solution obtained by
the finite difference scheme. It is similar to that used in \cite{Duan2021RHDMM}, where the mesh iteration redistribution depends on the solution obtained by
the second-order accurate finite volume scheme.
Unless otherwise stated, the dependence of the variables on $t$ will be omitted.

Consider the   mesh adaption functional 
\begin{equation}\label{eq:mesh_func}
  \widetilde{E}(\bx)=\dfrac12\sum_{k=1}^3\int_{\Omega_l}\left(\nabla_{\bm{\xi}}x_k\right)^\mathrm{T}\bm{G}_k\left(\nabla_{\bm{\xi}}x_k\right)\dd\bm{\xi},
\end{equation}
where $\bm{G}_k$ is the given symmetric positive definite matrix, depending on the solution $\bU$.
Solving the Euler-Lagrange equations of \eqref{eq:mesh_func}
\begin{equation}\label{eq:mesh_EL}
  \nabla_{\bm{\xi}}\cdot\left(\bm{G}_k\nabla_{\bm{\xi}}x_k\right)=0,
  ~\bm{\xi}\in\Omega_c,~k=1,2,3,
\end{equation}
will give directly a coordinate transformation
$\bx=\bx(\bm{\xi})$ from the computational domain $\Omega_c$ to the physical domain $\Omega_p$.
%
The concentration of the mesh points is controlled by  $\bm{G}_k$,
which in general depends on the solutions or their derivatives of the
underlying governing equations and is one of the most important elements in the adaptive moving mesh method. Different problems may be equipped with different $\bm{G}_k$.
For example, the Winslow variable diffusion method \cite{Winslow1967Numerical} is considering the simplest choice of $\bm{G}_k$ defined by
\begin{equation*}
  \bm{G}_k=\omega\bm{I}_3,
\end{equation*}
where $\omega$ is a positive weight function, called the monitor function, and  may be taken as
\begin{align}\label{eq:monitor}
  \omega=\sqrt{1+\alpha{\abs{\nabla_{\bm{\xi}}\sigma}}/{\max\abs{\nabla_{\bm{\xi}}\sigma}}},
\end{align}
here $\sigma$ is some physical variable and $\alpha$ is a positive parameter.
There are several other
choices of the monitor functions, see
\cite{Cao1999A,Han2007An,He2012RHD,Tang2006A,Tang2003An}.

\begin{rmk}\rm
  The monitor function is computed from the solutions of the underlying physical equations \eqref{eq:RMHDCurv},
  thus is not smooth in general. To get a smoother (adaptive) mesh, the following low pass filter
  \begin{align*}
    \omega_{i_1,i_2,i_3}\leftarrow&\sum_{j_1,j_2,j_3=0,\pm 1}\left(\dfrac{1}{2}\right)^{\abs{j_1}+\abs{j_2}+\abs{j_3}+3}
    \omega_{i_1+j_1,i_2+j_2,i_3+j_3},
  \end{align*}
  is applied $3\sim 5$ times in this work.
\end{rmk}

The mesh equations \eqref{eq:mesh_EL} are approximated by the central difference scheme
on the computational mesh and then solved by using the Jacobi iteration method
\begin{equation*}
  \begin{aligned}
    \sum_{k=1}^3\left[\left(\omega_{\bm{i}}+\omega_{\bm{i},k,+1}\right)\left(\bx_{\bm{i},k,+1}^{[\nu]}-\bx_{\bm{i}}^{[\nu+1]}\right)
    -\left(\omega_{\bm{i}}+\omega_{\bm{i},k,-1}\right)\left(\bx_{\bm{i}}^{[\nu+1]}-\bx_{\bm{i},k,-1}^{[\nu]}\right)\right]=0, \ \nu=0,1,\cdots,\mu,
  \end{aligned}
\end{equation*}
in parallel, where ${\bx}^{[0]}_{\bm{i}}
:=\bx^n_{\bm{i}}$, and $\omega$ is computed by using the solution $\bU$ at $t^n$.
In our numerical tests, the total iteration number $\mu$ is taken as $10$, unless otherwise stated.

Once the  mesh $\{{\bx}^{[\mu]}_{\bm{i}}\}$ is obtained,
the final adaptive mesh  is given by
\begin{equation*}
  \bx^{n+1}_{\bm{i}}
  :=\bx^n_{\bm{i}}
  +
  {\Delta_\tau}   (\delta_\tau{\bx})^{n}_{\bm{i}},~
  (\delta_\tau {\bx})^{n}_{\bm{i}}:={\bx}^{[\mu]}_{\bm{i}}  -\bx^n_{\bm{i}},
\end{equation*}
where the parameter ${\Delta_\tau}$ is used  to limit the mesh point movement
\begin{equation*}
  {\Delta_\tau}\leqslant
  \begin{cases}
    -\frac{1}{2(\delta_\tau{x_k})_{\bm{i}}}\left[(x_k)^n_{\bm{i}}-(x_1)^n_{\bm{i},k,-1}\right], ~ (\delta_\tau{x_k})_{\bm{i}}<0, \\
    +\frac{1}{2(\delta_\tau{x_k})_{\bm{i}}}\left[(x_k)^n_{\bm{i},k,+1}-(x_1)^n_{\bm{i}}\right], ~ (\delta_\tau{x_k})_{\bm{i}}>0. \\
  \end{cases}
\end{equation*}
Finally, the mesh velocity in \eqref{eq:VCLCoeff} is defined by
$
\dot{\bx}^{n}_{\bm{i}}:=
{\Delta_\tau}   (\delta_\tau{\bx})^{n}_{\bm{i}}/\Delta t^n
$,
where the time step size $\Delta t^n$ is  determined by \eqref{eq:cfl}.

%% file: NumTests.tex
\section{Numerical results}\label{section:NumTests}
This section conducts several 2D and 3D numerical tests in the RHDs and RMHDs to validate the convergence orders of our sixth-order accurate EC schemes on moving meshes (denoted by {\tt MM-O6}),
and  the convergence orders and the shock-capturing ability of our fifth-order accurate ES schemes on moving meshes (denoted by {\tt MM-O5}).
The numerical results are also compared to those obtained by the fifth-order accurate ES schemes on the static uniform mesh (denoted by {\tt UM-O5}) \cite{Duan2020RHD},
and the second-order accurate ES adaptive moving mesh schemes (denoted by {\tt MM-O2}) \cite{Duan2021RHDMM}.
Our schemes are implemented in parallel based on the data structure of the PLUTO code \cite{Mignone2007PLUTO},
and all simulations are performed with the CPU nodes of the High-performance Computing Platform of Peking University
(Linux Red Hat environment, two Intel Xeon E5-2697A V4 (16 cores $\times2$) per node, and core frequency of 2.6GHz).
Unless otherwise stated, the adiabatic index $\Gamma$ is taken as $5/3$ and
the time step size $\Delta t^n$ is determined by the following CFL condition
\begin{equation}\label{eq:cfl}
  {\Delta t}^n= 
   \dfrac{\text{CFL}}{\sum\limits_{k=1}^d\max\limits_{\bm{i}}
  {\varrho_{k,\bm{i}}^n}/{\Delta \xi_k}},
\end{equation}
where $\varrho_{k,\bm{i}}^n$ is the spectral radius of $\partial{\Fcurv_k}/\partial{\Ucurv}+\Phi'(\bV)\partial{\Bcurv_k}/\partial{\Ucurv}$ evaluated at $\bm{i}$ and $t^n$,
and the CFL number is taken as 0.4 and 0.3
for the 2D and  3D tests, respectively.

\subsection{2D tests}
\begin{example}[2D RMHD isentropic vortex problem]\label{ex:RMHD_2DVortex}\rm
  It describes a 2D vortex moving with a constant speed $(-0.5,-0.5)$ and is solved to test the convergence orders and the change of the
  total entropy.
  Specifically, the physical domain $\Omega_p$ is taken as $[-R,R]\times [-R,R]$ with $R=5$ and periodic boundary conditions.
  The explicit analytical solutions at time $t$ and the spatial point $(x_1,x_2)$   given first in \cite{Duan2021Analytical}  are
  \begin{equation*}\label{eq:Vortex1}
    \begin{aligned}
      \rho&=(1-\sigma\exp(1-r^2))^{\frac{1}{\Gamma-1}},~ p=\rho^\Gamma,\\
      \bm{v}&=\frac{1}{4-2(\widetilde{v}_1+\widetilde{v}_2)}((2+\sqrt{2})\widetilde{v}_1+(2-\sqrt{2})\widetilde{v}_2-2,
      ~(2+\sqrt{2})\widetilde{v}_2+(2-\sqrt{2})\widetilde{v}_1-2,~0),\\
      \bm{B}&=\frac12\left((\sqrt{2}+1)\widetilde{B}_1 - (\sqrt{2}-1)\widetilde{B}_2,
      ~(\sqrt{2}+1)\widetilde{B}_2 - (\sqrt{2}-1)\widetilde{B}_1,~0\right),
    \end{aligned}
  \end{equation*}
  where
  \begin{equation*}\label{eq:Vortex2}
    \begin{aligned}
      &\Gamma=5/3,~\sigma=0.2,~B_0=0.05,~r=\sqrt{\widetilde{x}_1^2+\widetilde{x}_2^2},\\
      &\widetilde{x}_k=\widehat{x}_k+(\sqrt{2}-1)(\widehat{x}_1 + \widehat{x}_2)/2,~k=1,2,\\
      &(\widehat{x}_1,\widehat{x}_2)=(2k_1R + x_1 + t/2 - 1,~2k_2R + x_2 + t/2 - 1), ~(\widehat{x}_1,\widehat{x}_2)\in [-R,R]\times[-R,R], ~k_1,k_2\in\mathbb{Z},\\
      & (\widetilde{v}_1,\widetilde{v}_2)=(-\widetilde{x}_2,\widetilde{x}_1)f,~ f=\sqrt{\dfrac{\kappa \exp(1-r^2)}{\kappa r^2\exp(1-r^2) + (\Gamma-1)\rho + \Gamma p}},~
      \kappa = 2\Gamma \sigma\rho + (\Gamma-1)B_0^2(2-r^2),	\\
      & (\widetilde{B}_1,\widetilde{B}_2)=B_0\exp(1-r^2)(-\widetilde{x}_2,\widetilde{x}_1).
    \end{aligned}
  \end{equation*}
  The problem is solved with a series of $N\times N$ meshes until $t=4$.

First, we test the sixth-order EC scheme  on moving meshes ({\tt MM-O6})
 with the following  moving mesh strategy
\begin{equation}\label{eq:RMHD_2DVortex_MovingMesh}
  \begin{split}
    &(x_1)_{i_1,i_2}=\mathring{x}_1 + 0.2\cos(\pi t/4)\sin(3\pi \mathring{x}_2/R),~
    (x_2)_{i_1,i_2}=\mathring{x}_2 + 0.2\cos(\pi t/4)\sin(3\pi \mathring{x}_1/R),\\
    &\mathring{x}_1=2i_1R/(N-1),~\mathring{x}_2=2i_2R/(N-1),
    ~i_1,i_2=0,1,\cdots,N-1.
  \end{split}
\end{equation}
The time step size is chosen as $\Delta t^n=\text{CFL}\Delta\xi_1^2$ to make the spatial error dominant.
Figure \ref{fig:RMHD_2DVortex_EC6_mesh} gives the $10$ equally spaced contours of the rest-mass density and the moving meshes with $N=40$ at different times. One can see that
the shape of the vortex is preserved well.

Next, the  problem is resolved by using the fifth-order ES scheme with the adaptive moving mesh ({\tt MM-O5}) and the following monitor function
\begin{equation}\label{eq:RMHD_2DVortex_Monitor}
  \omega=\sqrt{1 + 20\abs{\nabla_{\bm{\xi}}\rho}/\max{\abs{\nabla_{\bm{\xi}}\rho}}
  +10\abs{\Delta_{\bm{\xi}}\rho}/\max{\abs{\Delta_{\bm{\xi}}\rho}}}.
\end{equation}
The time step size is chosen as $\Delta t^n=\text{CFL}\Delta\xi_1^{5/3}$ to make the spatial error dominant.
Figure \ref{fig:RMHD_2DVortex_ES5_mesh} plots the adaptive meshes of $N=40$ at different times, which show that
the concentration of the mesh points  follows the propagation of the vortex well.

Figure \ref{fig:RMHD_2DVortex_err} plots corresponding errors in the rest-mass density $\rho$ and  convergence orders of {\tt MM-O6} and {\tt MM-O5}. One can see that
{\tt MM-O6} and {\tt MM-O5} can achieve sixth- and fifth-order accuracies respectively.

Finally, we examine the EC and ES property of our schemes.
Figure \ref{fig:RMHD_2DVortex_TotalEntropy} presents the evolution of the discrete total entropy
$\sum_{i_1,i_2} J_{i_1,i_2} \eta(\bU_{i_1,i_2})/N^2$
with respect to time obtained by {\tt MM-O6} and {\tt MM-O5} with $N=160$.
We can see that the total entropy of the EC scheme
almost keeps unchanged, while the total entropy of the ES scheme decays as expected.
\end{example}

It should be noted that
{\tt MM-O5} with the moving mesh \eqref{eq:RMHD_2DVortex_MovingMesh}
and {\tt MM-O6} with the adaptive moving mesh and the monitor \eqref{eq:RMHD_2DVortex_Monitor}  can also respectively get fifth-order and sixth-order.
Their results are omitted here due to limited space.


\begin{figure}[!ht]
  \centering
  \begin{subfigure}[b]{0.3\textwidth}
    \centering
    \includegraphics[width=1.0\textwidth]{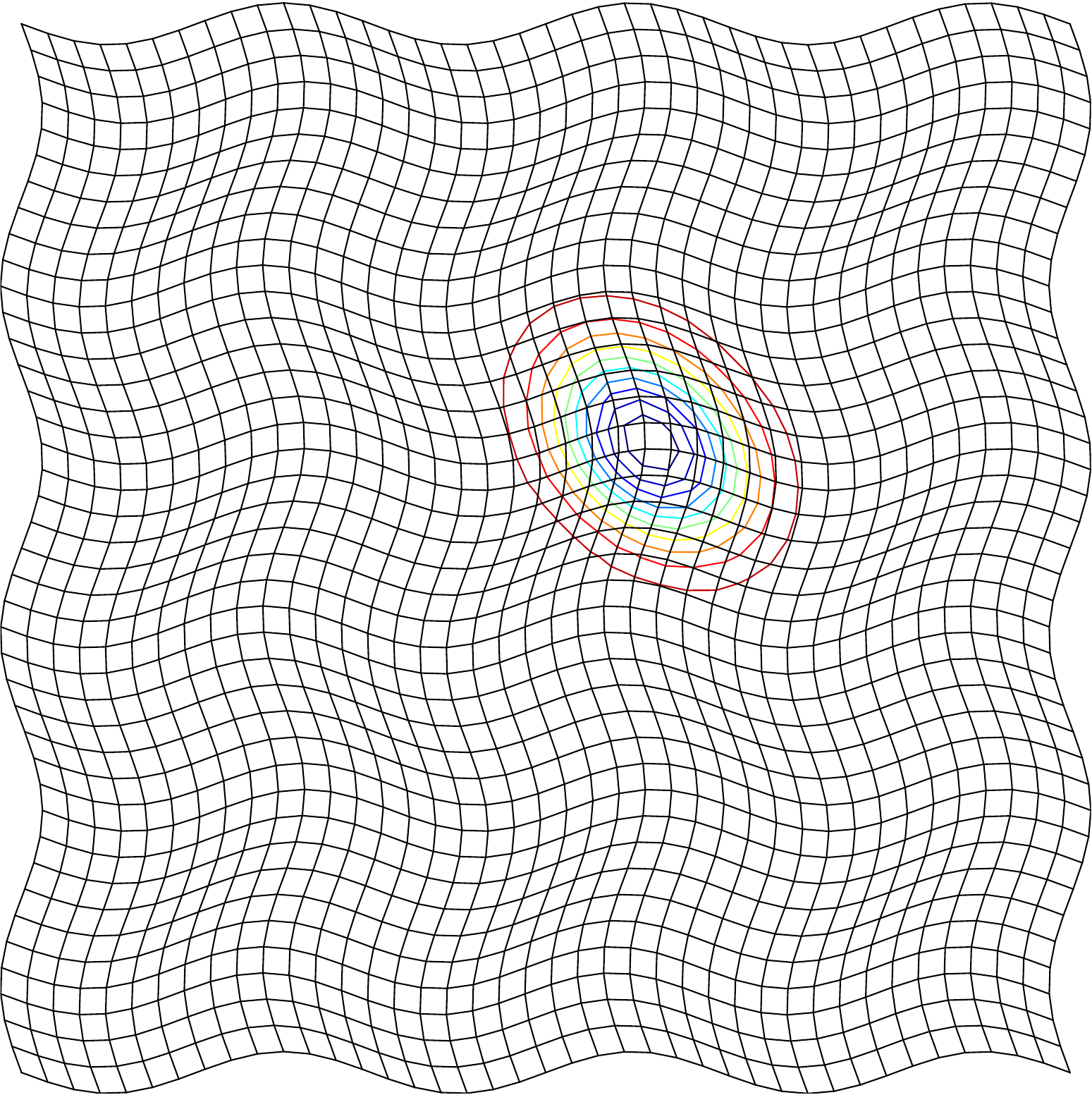}
    \caption{$t=0$}
  \end{subfigure}
  \begin{subfigure}[b]{0.3\textwidth}
    \centering
    \includegraphics[width=1.0\textwidth]{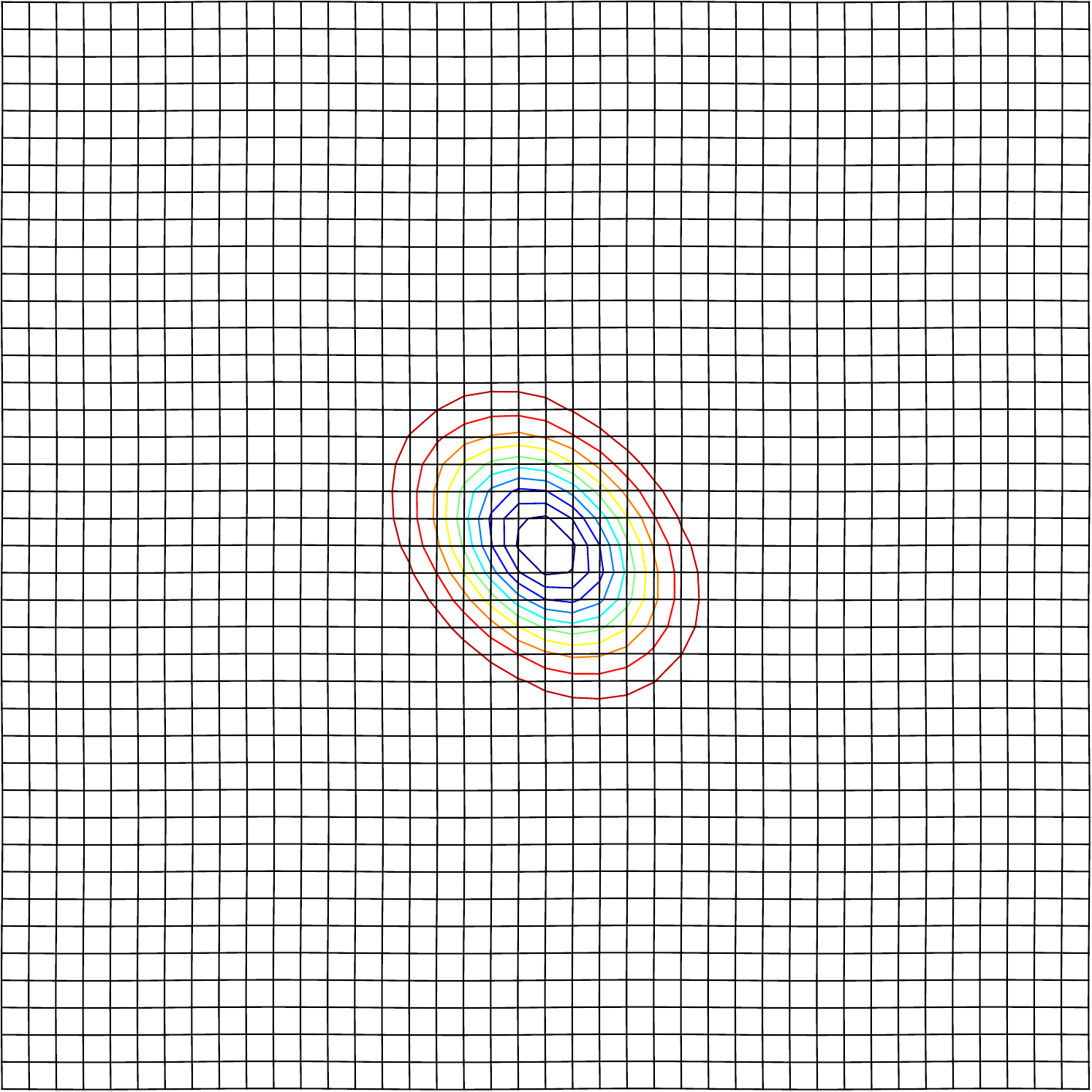}
    \caption{$t=2$}
  \end{subfigure}
  \begin{subfigure}[b]{0.3\textwidth}
    \centering
    \includegraphics[width=1.0\textwidth]{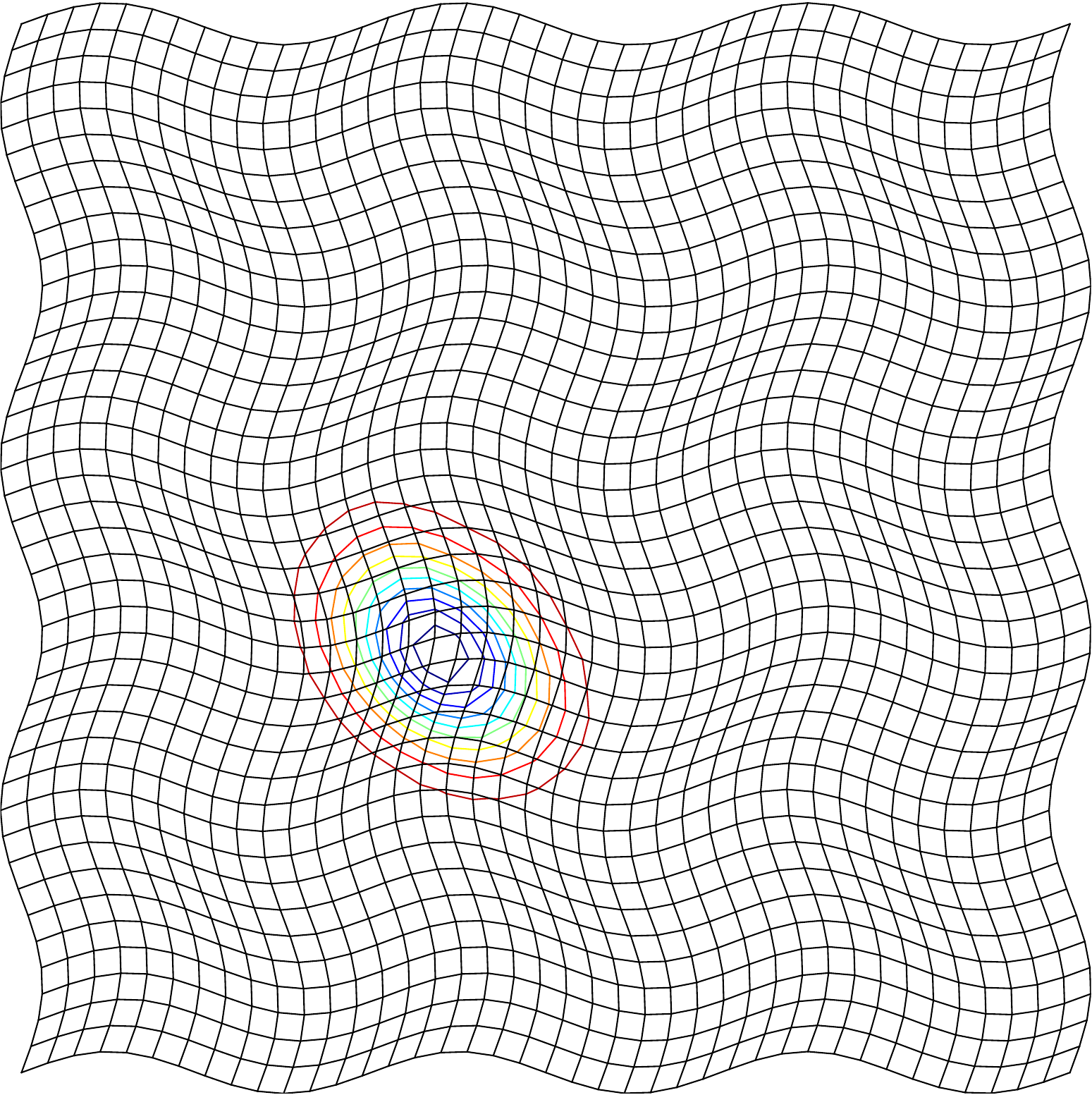}
    \caption{$t=4$}
  \end{subfigure}
  \caption{Example \ref{ex:RMHD_2DVortex}: Adaptive meshes and  rest-mass density contours at different times obtained by {\tt MM-O6} with the moving mesh \eqref{eq:RMHD_2DVortex_MovingMesh}. $N=40$ and $10$ equally spaced contour lines.}
  \label{fig:RMHD_2DVortex_EC6_mesh}
\end{figure}

\begin{figure}[!ht]
  \centering
  \begin{subfigure}[b]{0.3\textwidth}
    \centering
    \includegraphics[width=1.0\textwidth]{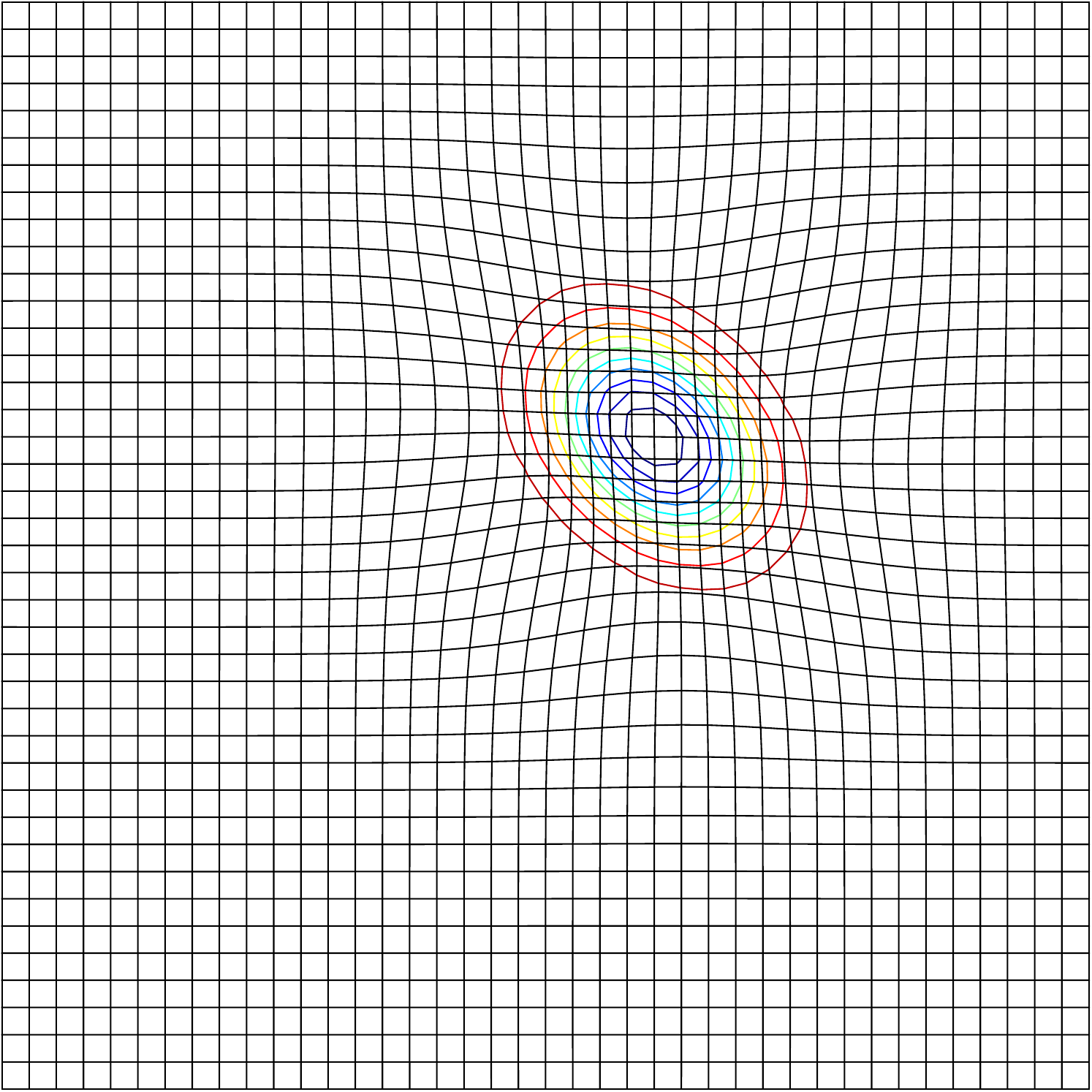}
    \caption{$t=0$}
  \end{subfigure}
  \begin{subfigure}[b]{0.3\textwidth}
    \centering
    \includegraphics[width=1.0\textwidth]{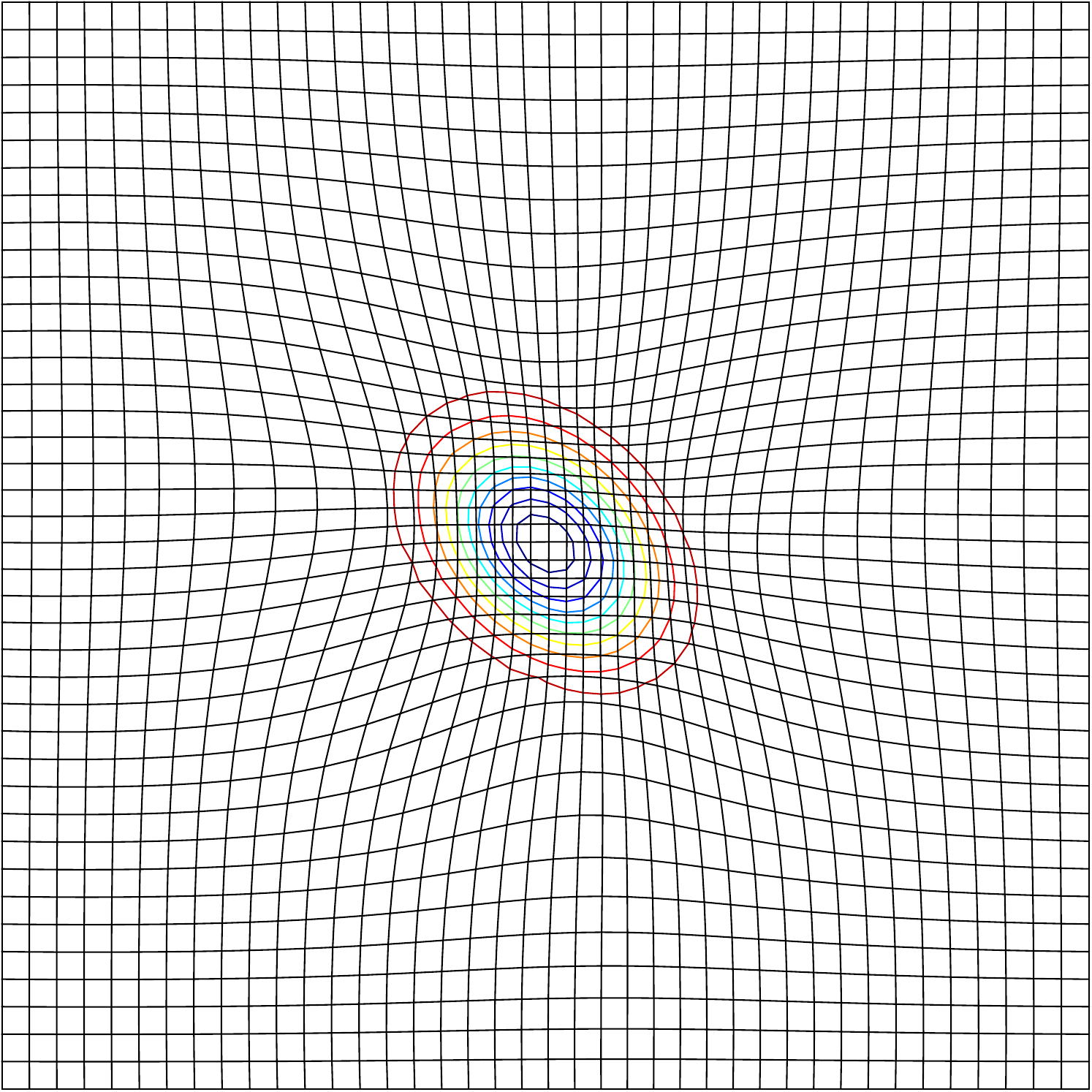}
    \caption{$t=2$}
  \end{subfigure}
  \begin{subfigure}[b]{0.3\textwidth}
    \centering
    \includegraphics[width=1.0\textwidth]{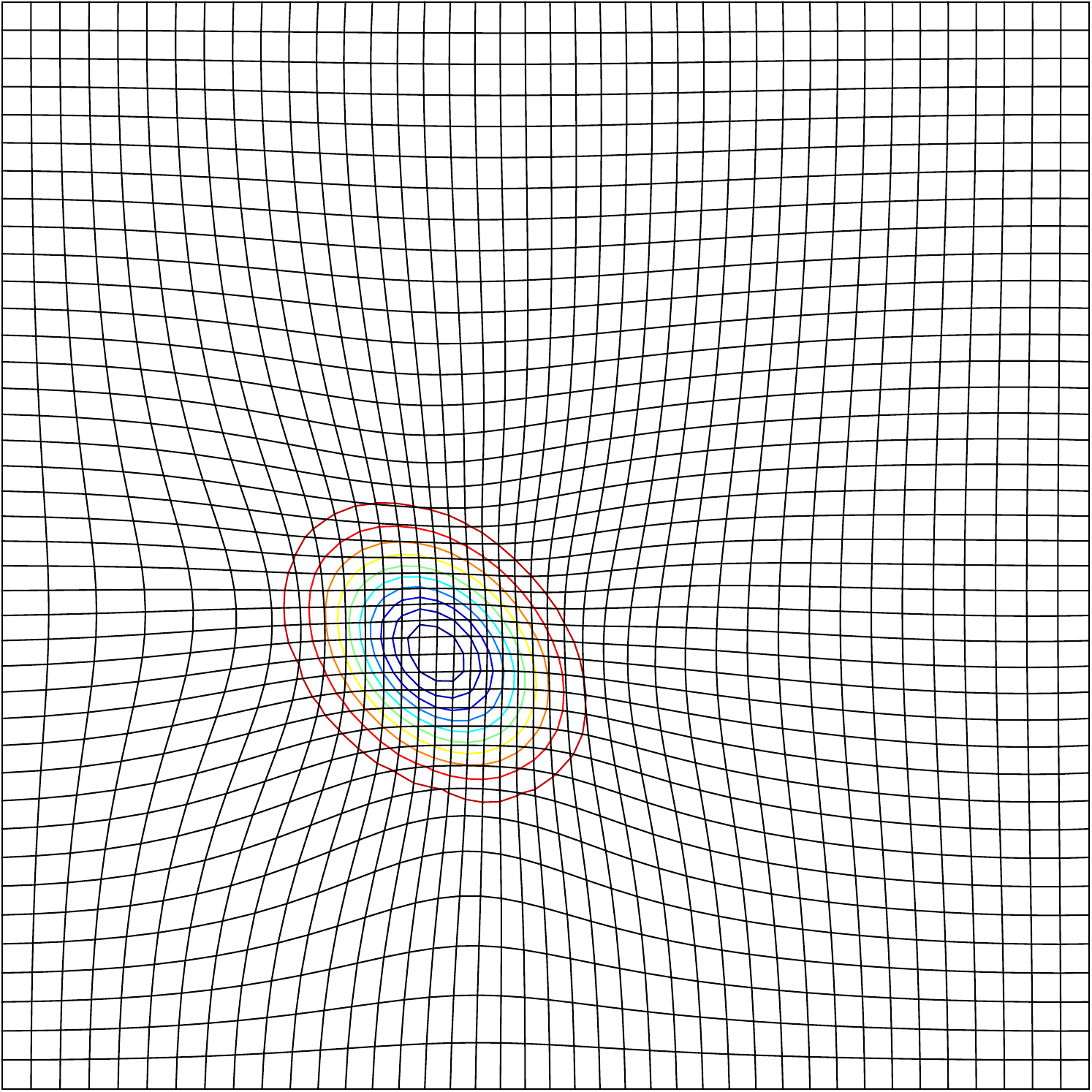}
    \caption{$t=4$}
  \end{subfigure}
  \caption{Example \ref{ex:RMHD_2DVortex}:  Adaptive meshes and  rest-mass density contours at different times obtained by {\tt MM-O5} with adaptive mesh velocity and the monitor \eqref{eq:RMHD_2DVortex_Monitor}. $N=40$ and $10$ equally spaced contour lines.}
  \label{fig:RMHD_2DVortex_ES5_mesh}
\end{figure}

\begin{figure}[!ht]
  \centering
  \begin{subfigure}[b]{0.48\textwidth}
    \centering
    \includegraphics[width=1.0\textwidth]{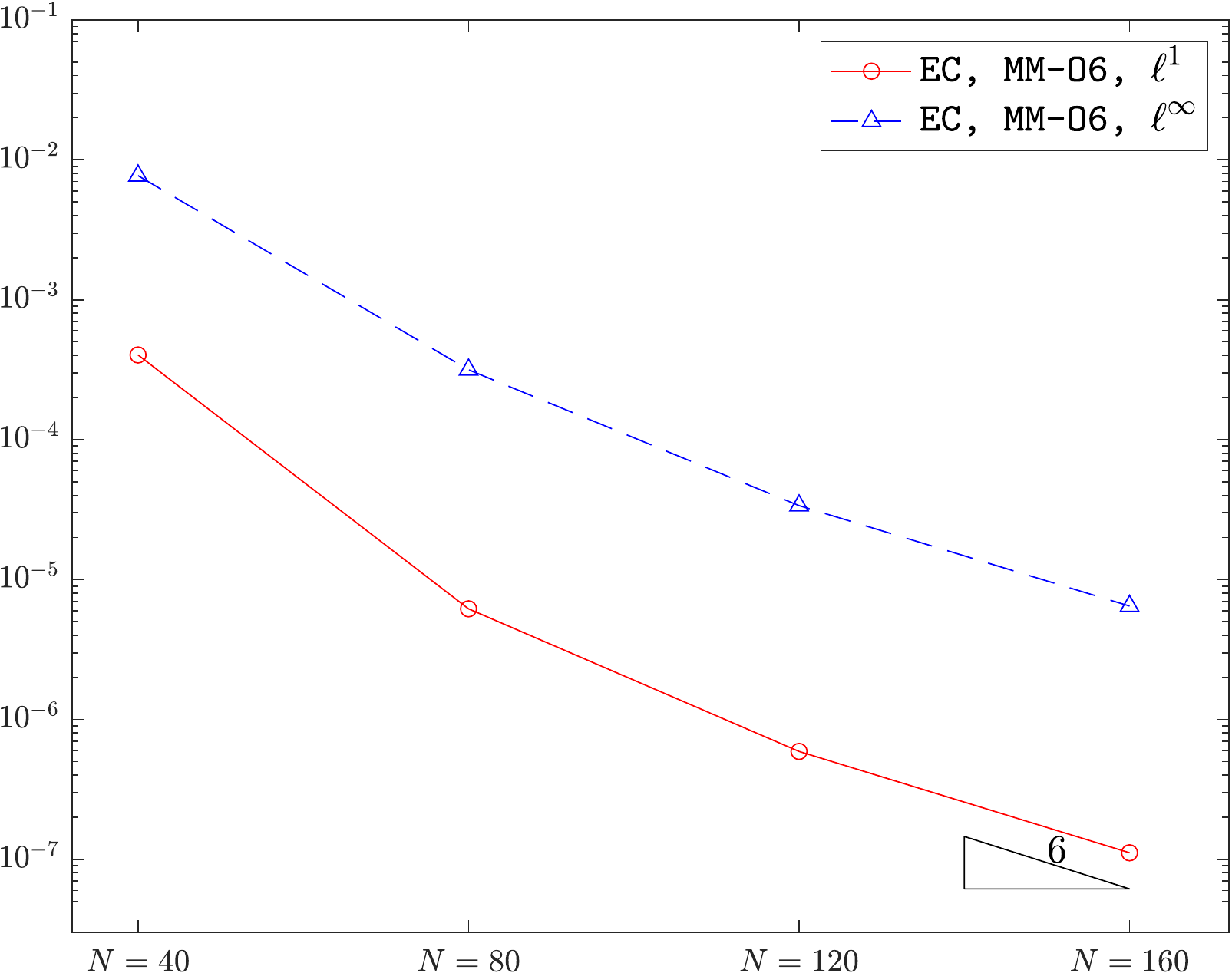}
    \caption{EC, \tt{MM-O6}}
    \label{fig:RMHD_2DVortex_err_EC}
  \end{subfigure}
  \begin{subfigure}[b]{0.48\textwidth}
    \centering
    \includegraphics[width=1.0\textwidth]{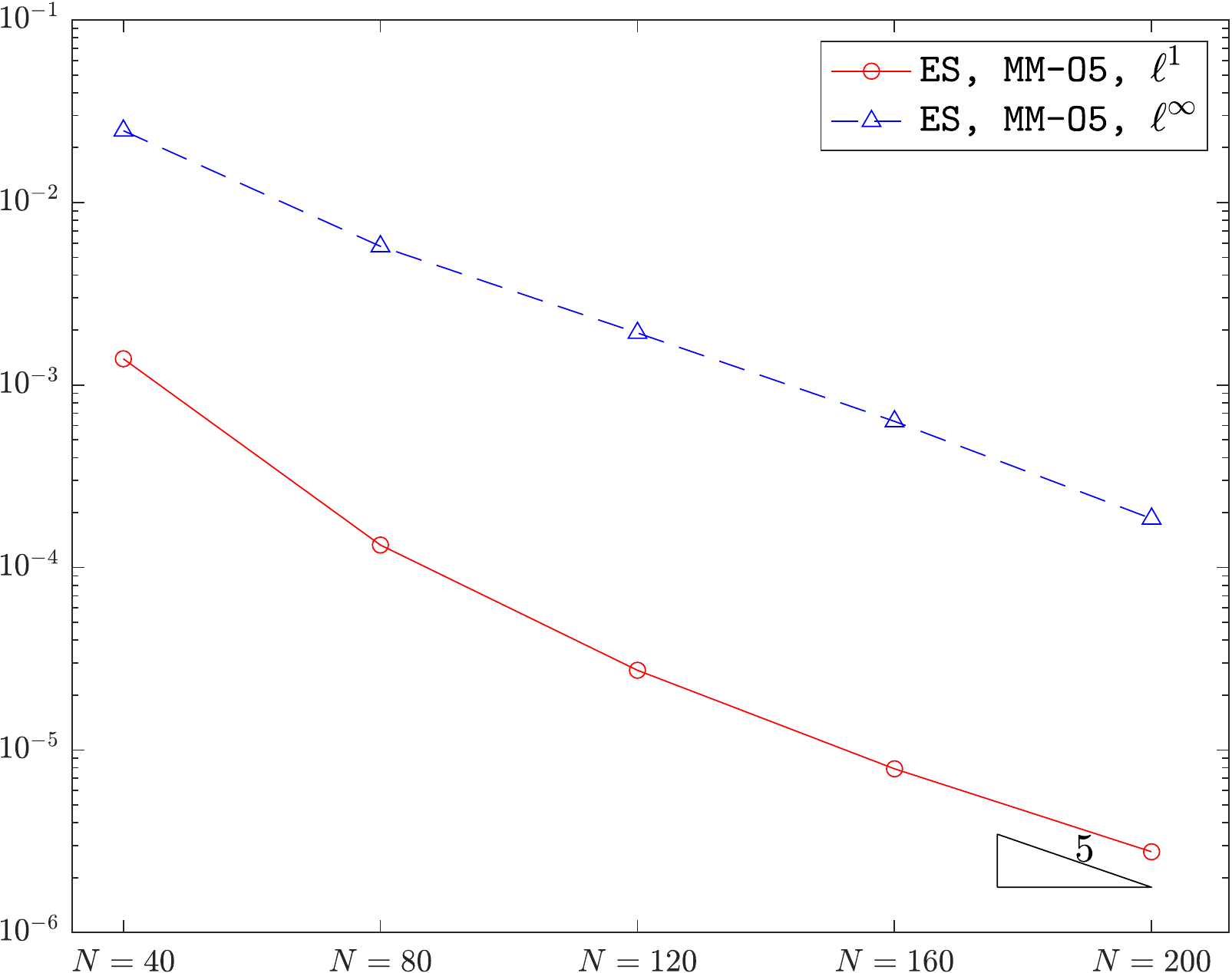}
    \caption{ES, \tt{MM-O5}}
    \label{fig:RMHD_2DVortex_err_ES}
  \end{subfigure}
  \caption{Example \ref{ex:RMHD_2DVortex}: The errors and convergence orders in $\rho$ at $t=4$.}
  \label{fig:RMHD_2DVortex_err}
\end{figure}

\begin{figure}[!ht]
  \centering
  \includegraphics[width=0.5\textwidth]{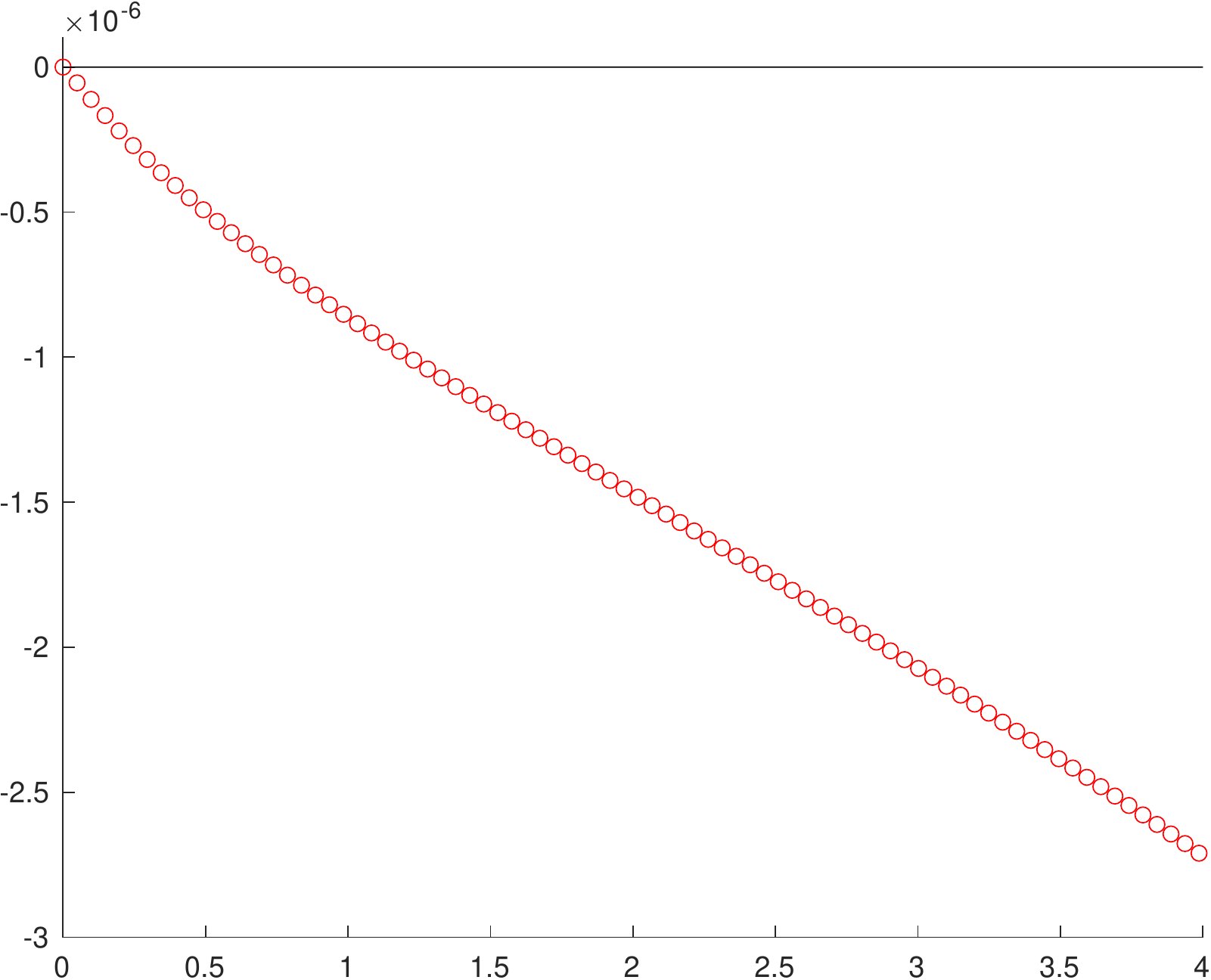}
  \caption{Example \ref{ex:RMHD_2DVortex}: The evolution of the discrete total entropy with $160\times 160$ meshes. The line and symbols aare obtained by using the EC scheme {\tt MM-O6}, and the ES scheme {\tt MM-O5}, respectively.}
  \label{fig:RMHD_2DVortex_TotalEntropy}
\end{figure}

\begin{example}[RHD Riemann problem \uppercase\expandafter{\romannumeral1}]\label{ex:2DRP1}\rm
  This example considers the 2D RHD Riemann problem with the initial data
  \begin{align*}
    (\rho,v_1,v_2,p)=\begin{cases}
      (0.5,~0.5,-0.5,~5), &\quad x_1>0.5,~x_2>0.5,\\
      (1,~0.5,~0.5,~5),   &\quad x_1<0.5,~x_2>0.5,\\
      (3,-0.5,~0.5,~5),   &\quad x_1<0.5,~x_2<0.5,\\
      (1.5,-0.5,-0.5,~5), &\quad x_1>0.5,~x_2<0.5.
    \end{cases}
  \end{align*}
  It describes the interaction of four contact discontinuities (vortex sheets) with the same sign (the negative sign).
\end{example}

The monitor function is chosen as \eqref{eq:monitor} with $\alpha=1200$ and $\sigma=\ln\rho$.
Figure \ref{fig:RHD_RP1} shows the adaptive mesh of {\tt MM-O5}, $40$ equally spaced contour lines of $\ln\rho$, and the cut lines of $\ln\rho$ along $x_2=x_1$ at $t=0.4$ obtained by using  our ES schemes  with $N\times N$ meshes.
As time increases, a spiral with the low rest-mass density around the point (0.5,0.5) emerges,
and the adaptive concentration of the mesh points  follows the spiral formation well, see Figure \ref{fig:RHD_RP1_mesh}, so that some important features are well-captured.
 Figure \ref{fig:RHD_RP1_cut} shows the solution of {\tt MM-O5} with $N=200$ is very close to that of {\tt UM-O5} with $N=500$, and {\tt MM-05} does not cause spurious oscillations near $(0.86,0.86)$, see the small box in the upper right corner in Figure \ref{fig:RHD_RP1_cut}.
The CPU times (see the  parentheses in the captions of Figures \ref{fig:RHD_RP1_MMO5_N200} and \ref{fig:RHD_RP1_UMO5_N500}) clearly highlight the efficiency of the adaptive moving mesh scheme, since it takes only $17.8\%$ CPU time of the latter.
Figures \ref{fig:RHD_RP1_MMO5_N150} and \ref{fig:RHD_RP1_MMO2_N200}
show that 
the fifth-order scheme {\tt MM-O5} gives better results with comparable CPU time than the second-order scheme {\tt MM-O2} \cite{Duan2021RHDMM},
thus {\tt MM-O5}  outperforms {\tt MM-O2}.


\begin{figure}[ht!]
  \centering
  \begin{subfigure}[b]{0.32\textwidth}
    \centering
    \includegraphics[width=\textwidth]{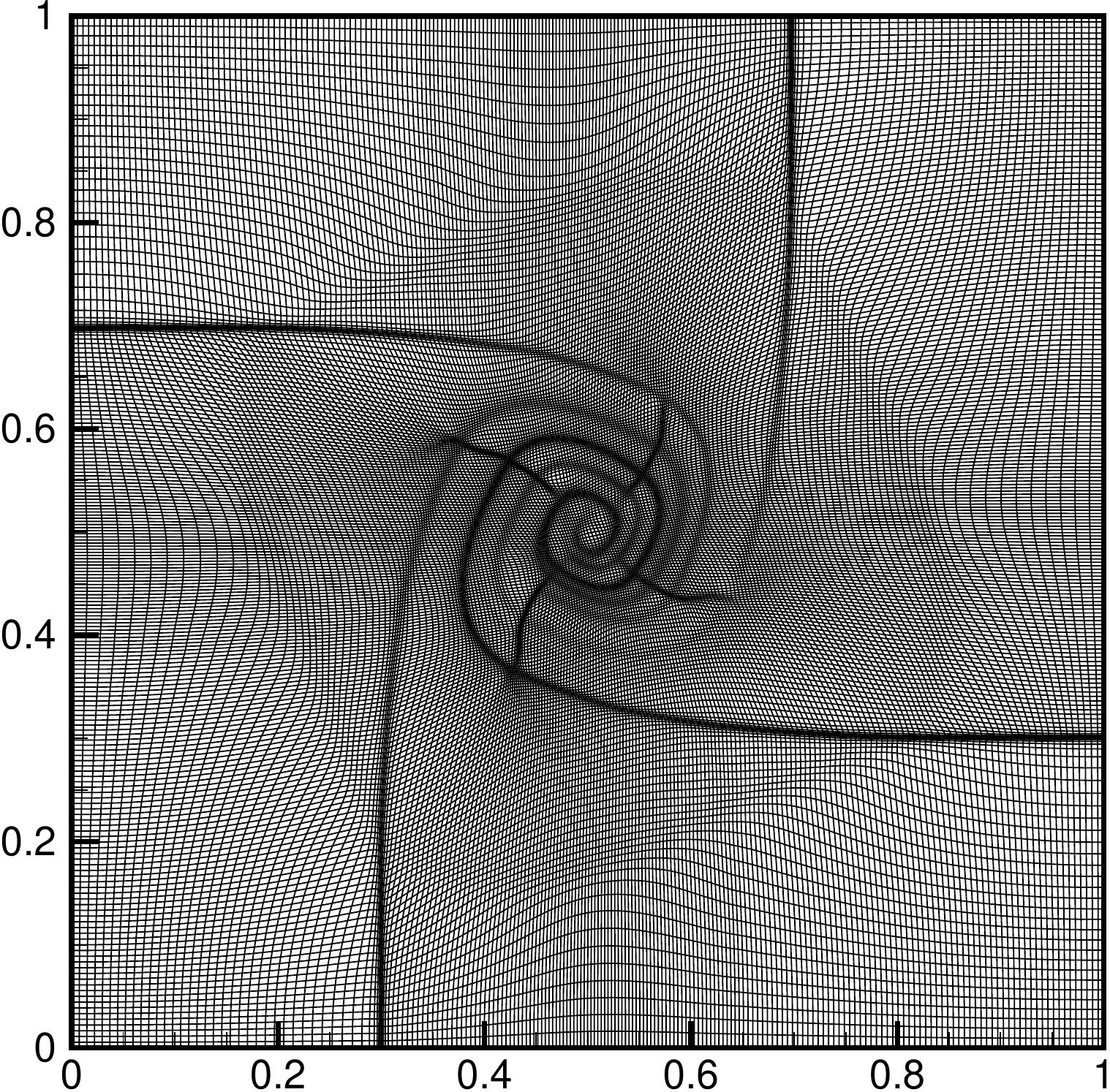}
    \caption{{\tt MM-O5}  with $N=200$}
    \label{fig:RHD_RP1_mesh}
  \end{subfigure}
  \begin{subfigure}[b]{0.32\textwidth}
    \centering
    \includegraphics[width=\textwidth]{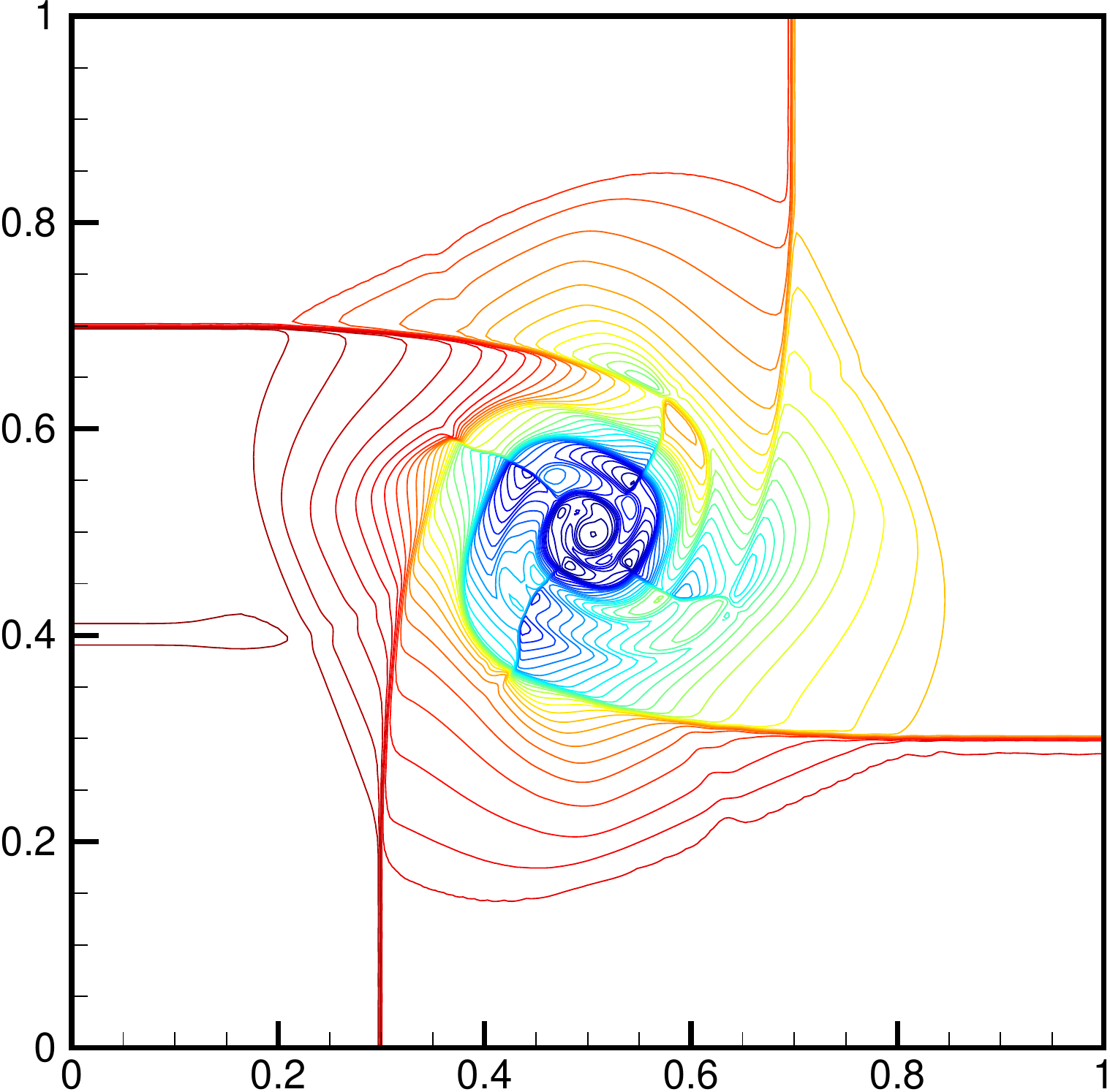}
    \caption{{\tt MM-O5} with $N=200$ (1m02s) }
    \label{fig:RHD_RP1_MMO5_N200}
  \end{subfigure}
  \begin{subfigure}[b]{0.32\textwidth}
    \centering
    \includegraphics[width=\textwidth]{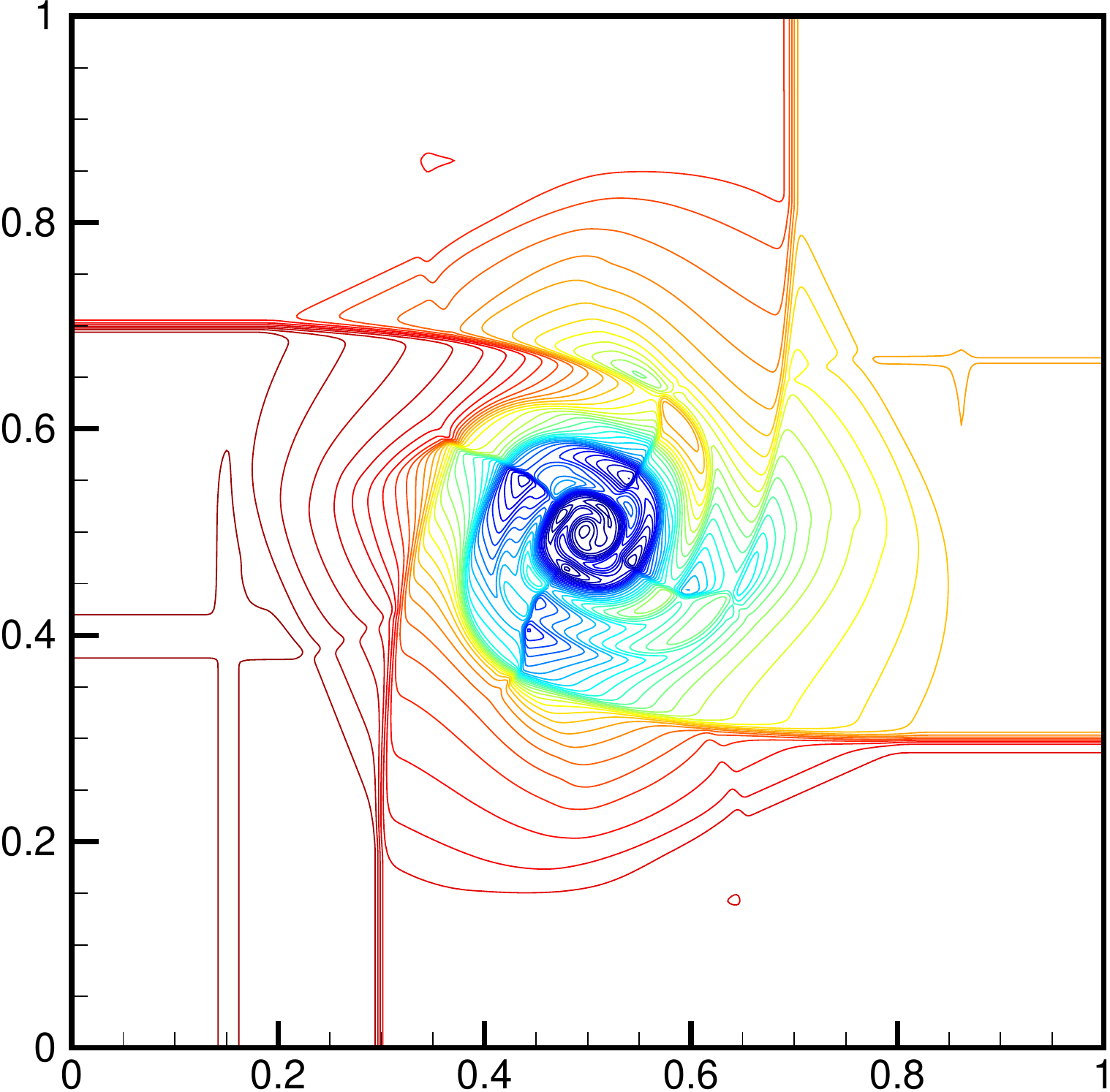}
    \caption{{\tt UM-O5} with $N=500$ (5m49s)}
    \label{fig:RHD_RP1_UMO5_N500}
  \end{subfigure}

  \begin{subfigure}[b]{0.32\textwidth}
    \centering
    \includegraphics[width=\textwidth]{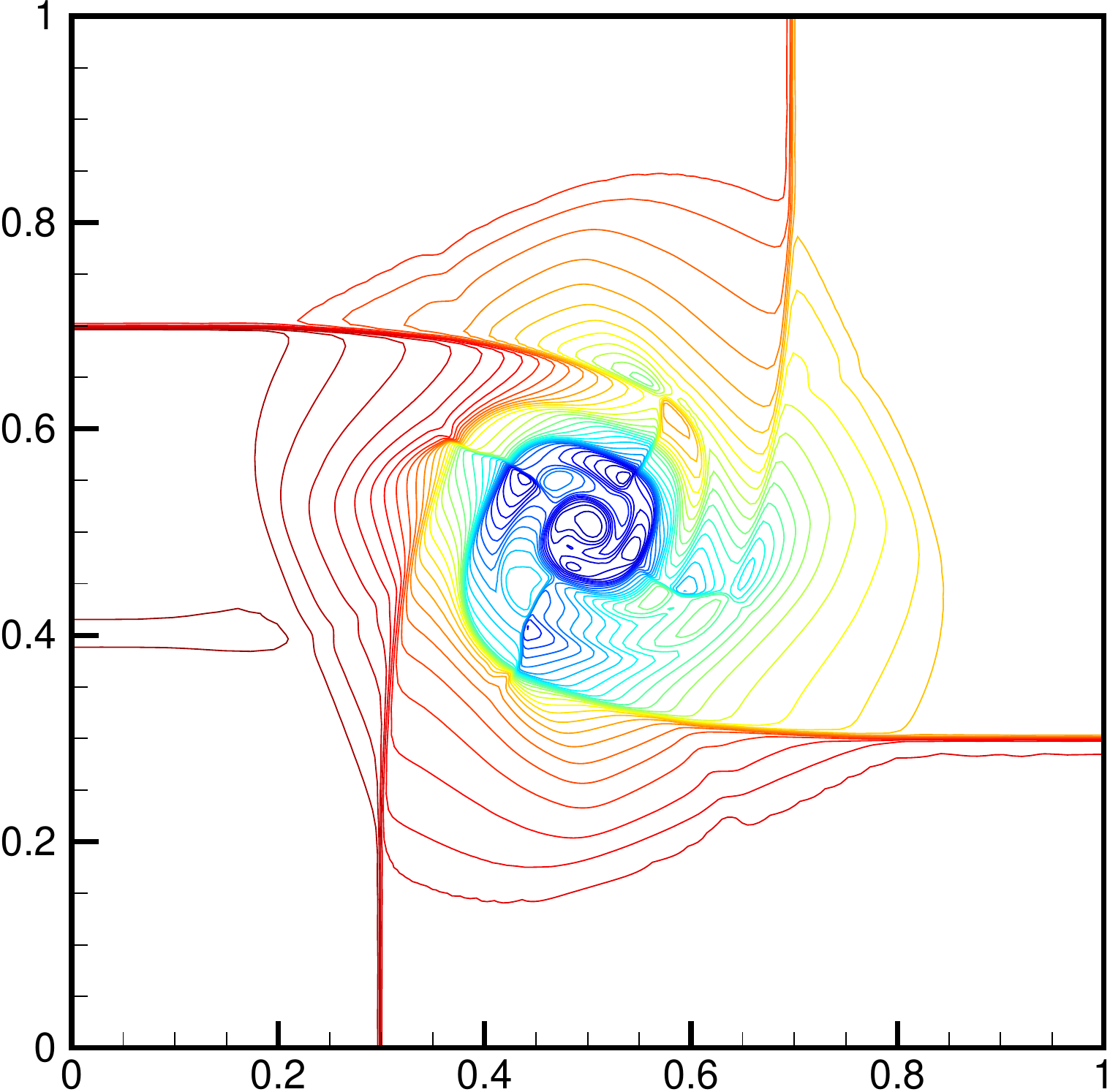}
    \caption{{\tt MM-O5} with $N=150$ (29s) }
    \label{fig:RHD_RP1_MMO5_N150}
  \end{subfigure}
  \begin{subfigure}[b]{0.32\textwidth}
    \centering
    \includegraphics[width=\textwidth]{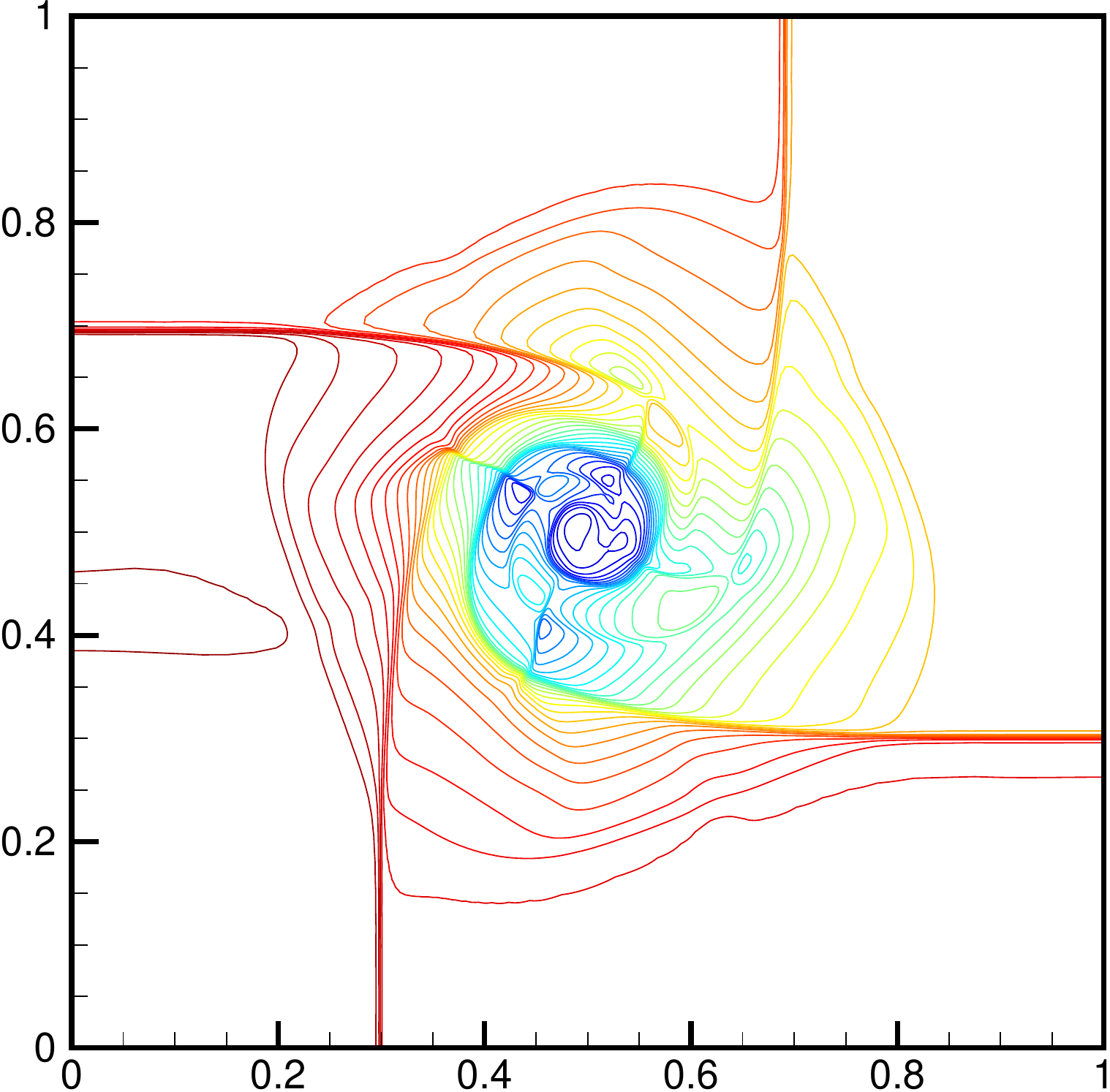}
    \caption{{\tt MM-O2} with $N=200$ (30s)}
    \label{fig:RHD_RP1_MMO2_N200}
  \end{subfigure}
  \begin{subfigure}[b]{0.32\textwidth}
    \centering
    \includegraphics[width=1.06\textwidth]{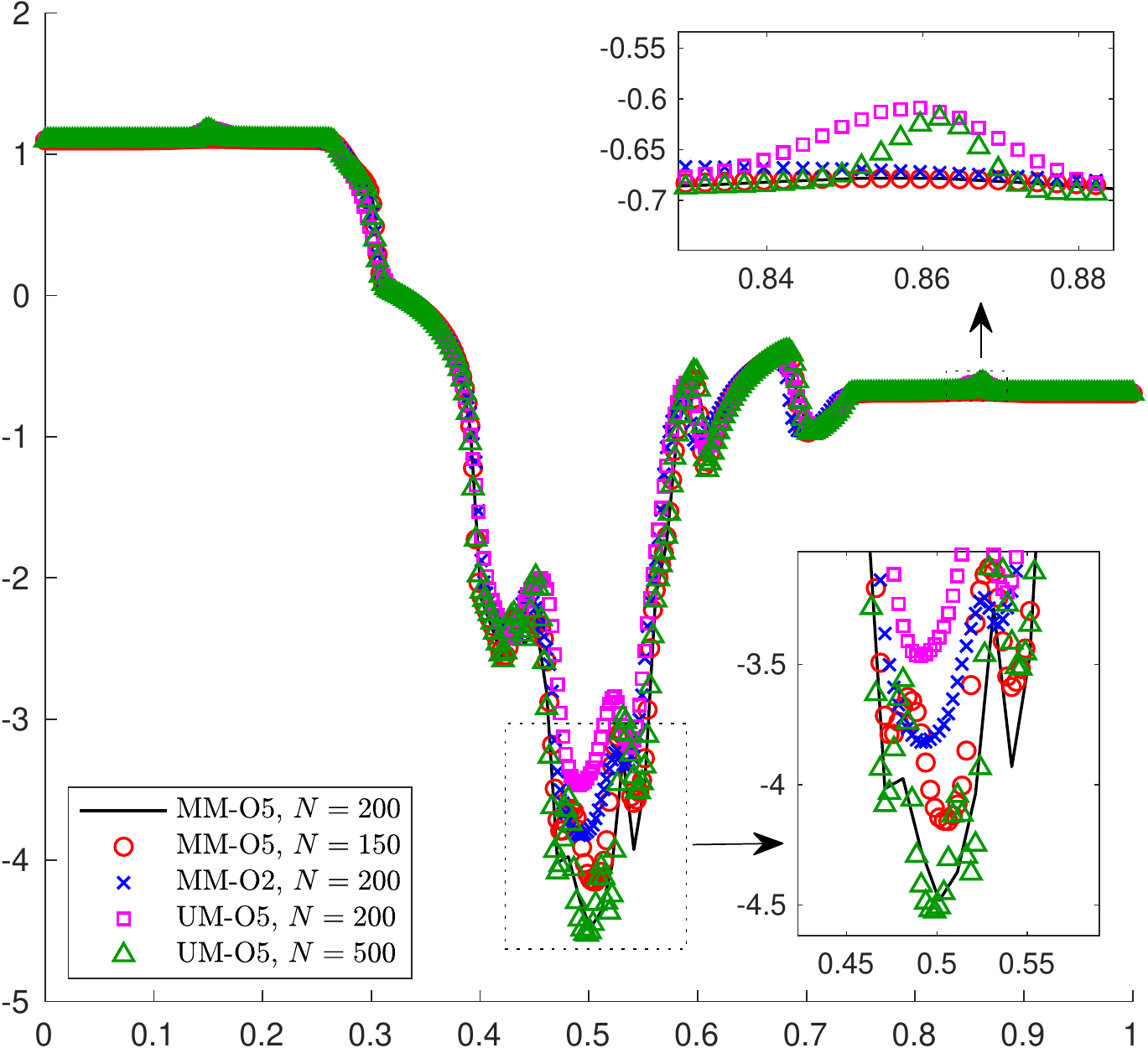}
    \caption{$\ln\rho$ along $x_2=x_1$}
    \label{fig:RHD_RP1_cut}
  \end{subfigure}
  \caption{Example \ref{ex:2DRP1}. Adaptive mesh of {\tt MM-O5} with $N=200$,    $40$ equally spaced contour lines of $\ln \rho$, and
  cut lines of $\ln\rho$ along $x_2=x_1$ obtained by using ES schemes.
  CPU times are listed in  parentheses.}
  \label{fig:RHD_RP1}
\end{figure}

\begin{example}[RHD Riemann problem \uppercase\expandafter{\romannumeral2}]\label{ex:2DRP2}\rm
  The initial data of this 2D RHD Riemann problem are
  \begin{align*}
    (\rho,v_1,v_2,p)=\begin{cases}
      (1,~0,~0,~1),             & \quad x_1>0.5,~x_2>0.5, \\
      (0.5771,-0.3529,~0,~0.4), & \quad x_1<0.5,~x_2>0.5, \\
      (1,-0.3529,-0.3529,~1),   & \quad x_1<0.5,~x_2<0.5, \\
      (0.5771,~0,-0.3529,~0.4), & \quad x_1>0.5,~x_2<0.5,
    \end{cases}
  \end{align*}
  which is about the interaction of four rarefaction waves.
\end{example}

The monitor function is the same as that in the last example.
Figure \ref{fig:RHD_RP2} presents the adaptive mesh of {\tt MM-O5}, the contours of the density logarithms $\ln\rho$
with $40$ equally spaced lines, and $\ln\rho$  along $x_2=x_1$ at $t=0.4$.
The results show that those four initial discontinuities first evolve as four rarefaction waves
and then interact each other and form two (almost parallel) curved shock waves perpendicular to the line $x_2=x_1$ as time increases.
It is seen that the adaptive moving mesh schemes capture the rarefaction waves and the shock waves well.
Figure \ref{fig:RHD_RP2_cut}  compares the results of  {\tt MM-O5} with $N=200$  to {\tt UM-O5} with $N=500$, which are very close to each other, but the former takes about $30.6\%$ CPU time.
One can also find from Figure \ref{fig:RHD_RP2_cut} that {\tt MM-O5} with $N=150$ gives  better results than {\tt MM-O2} with $N=200$ when using comparable CPU time.

\begin{figure}[ht!]
  \centering
  \begin{subfigure}[b]{0.32\textwidth}
    \centering
    \includegraphics[width=\textwidth]{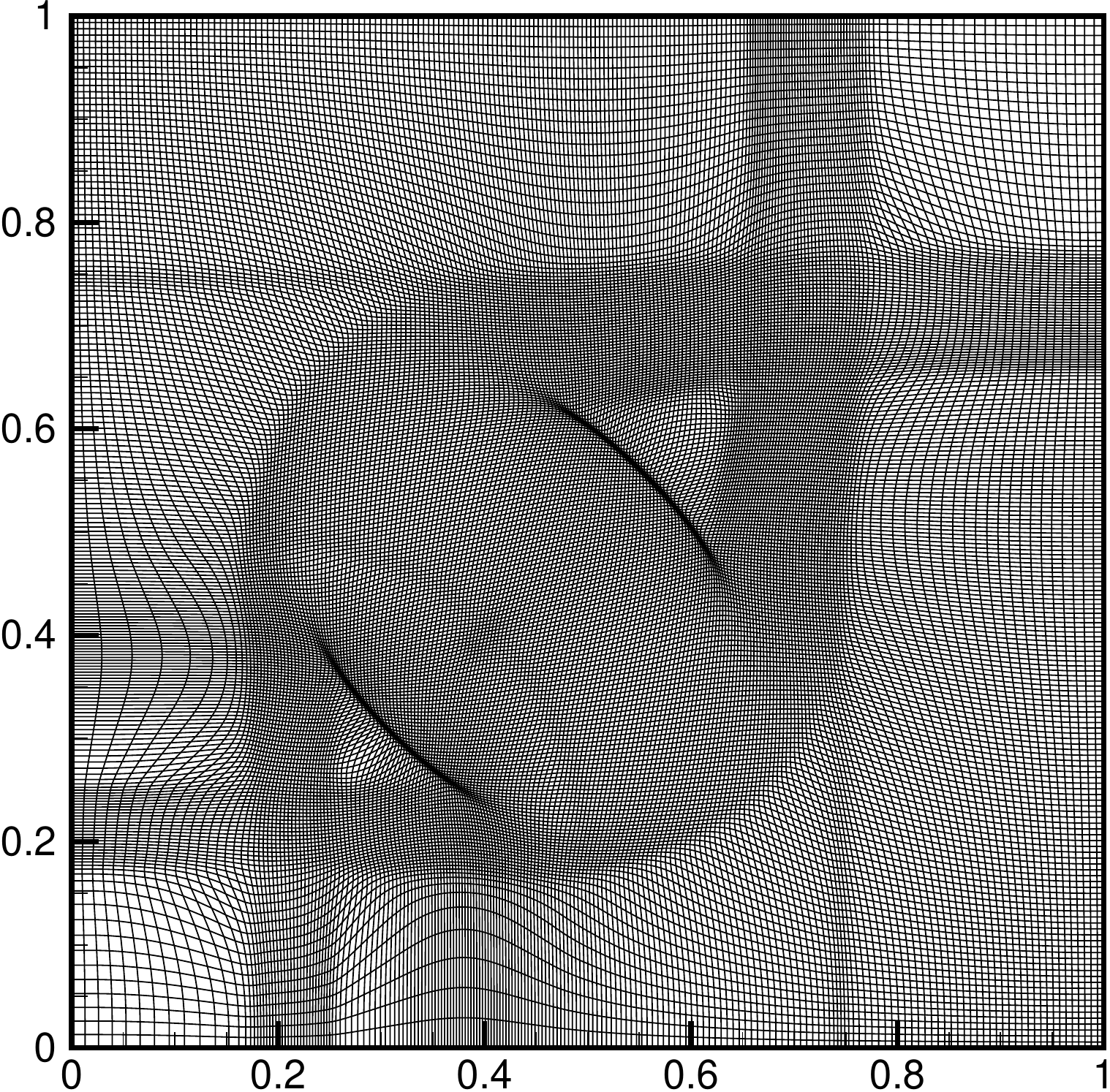}
    \caption{{\tt MM-O5}  with $N=200$}
    \label{fig:RHD_RP2_mesh}
  \end{subfigure}
  \begin{subfigure}[b]{0.32\textwidth}
    \centering
    \includegraphics[width=\textwidth]{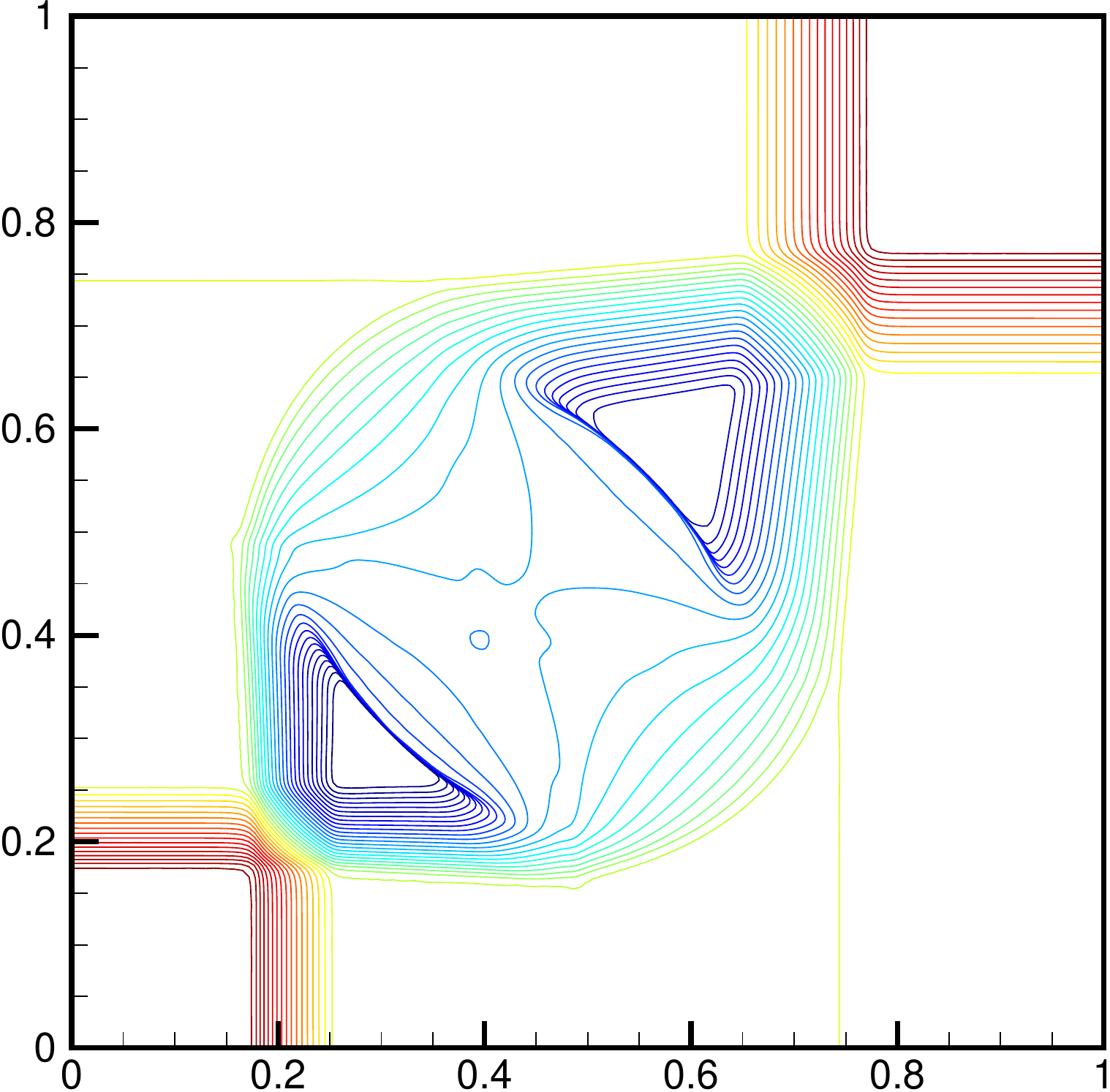}
    \caption{{\tt MM-O5} with $N=200$ (1m25s) }
    \label{fig:RHD_RP2_MMO5_N200}
  \end{subfigure}
  \begin{subfigure}[b]{0.32\textwidth}
    \centering
    \includegraphics[width=\textwidth]{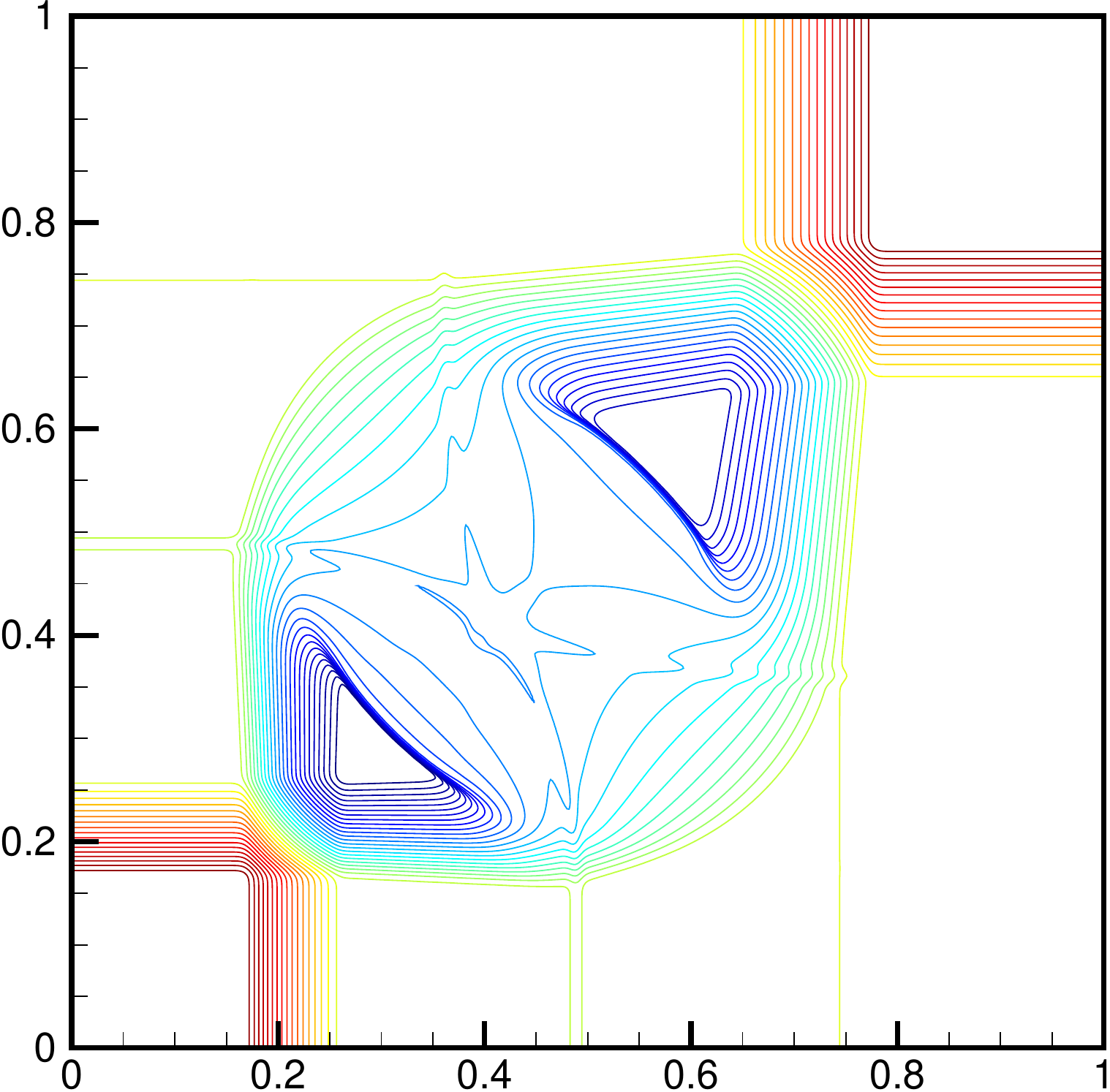}
    \caption{{\tt UM-O5} with $N=500$ (4m38s)}
    \label{fig:RHD_RP2_UMO5_N500}
  \end{subfigure}

  \begin{subfigure}[b]{0.32\textwidth}
    \centering
    \includegraphics[width=\textwidth]{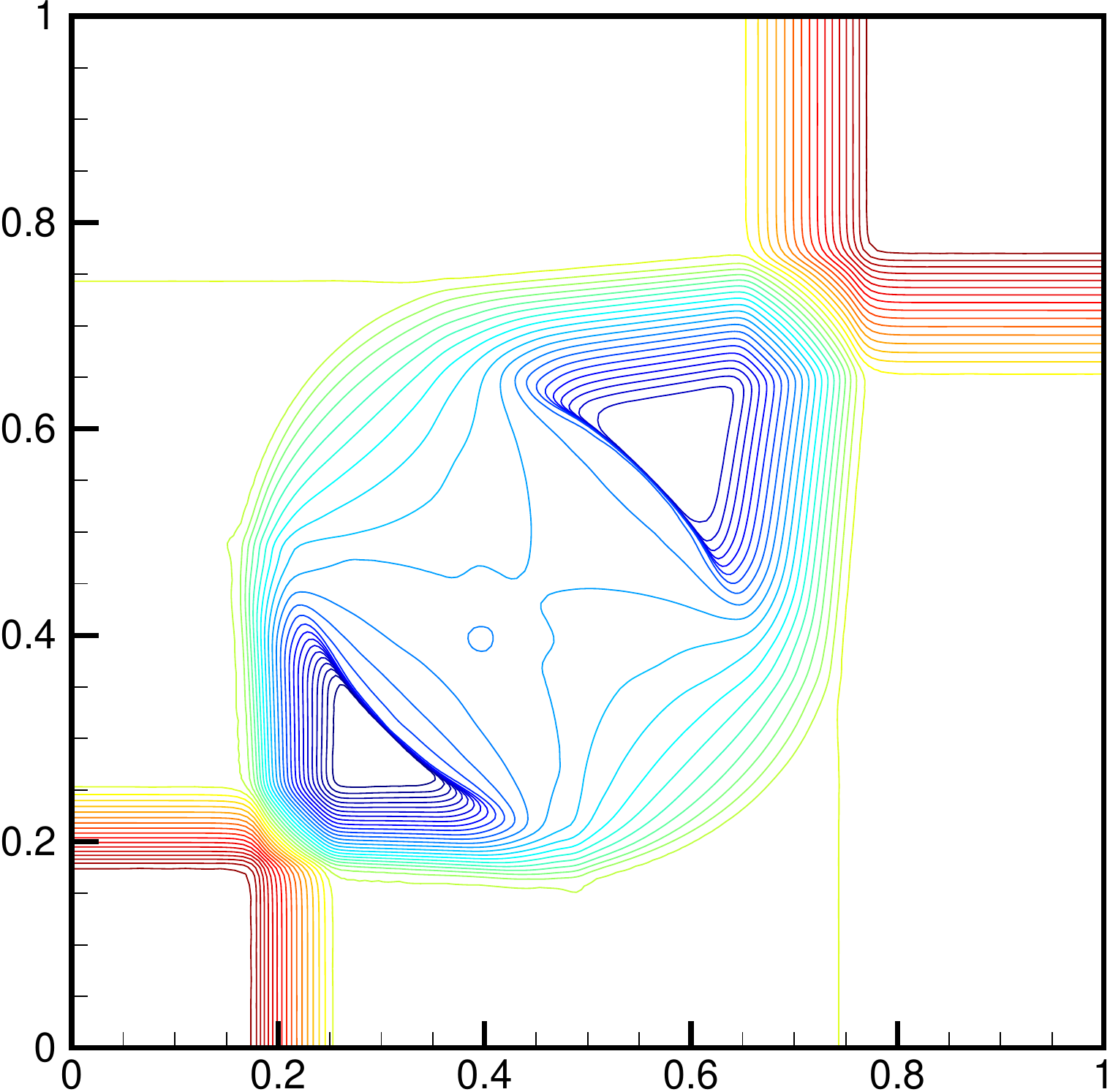}
    \caption{{\tt MM-O5} with $N=150$ (35s) }
    \label{fig:RHD_RP2_MMO5_N150}
  \end{subfigure}
  \begin{subfigure}[b]{0.32\textwidth}
    \centering
    \includegraphics[width=\textwidth]{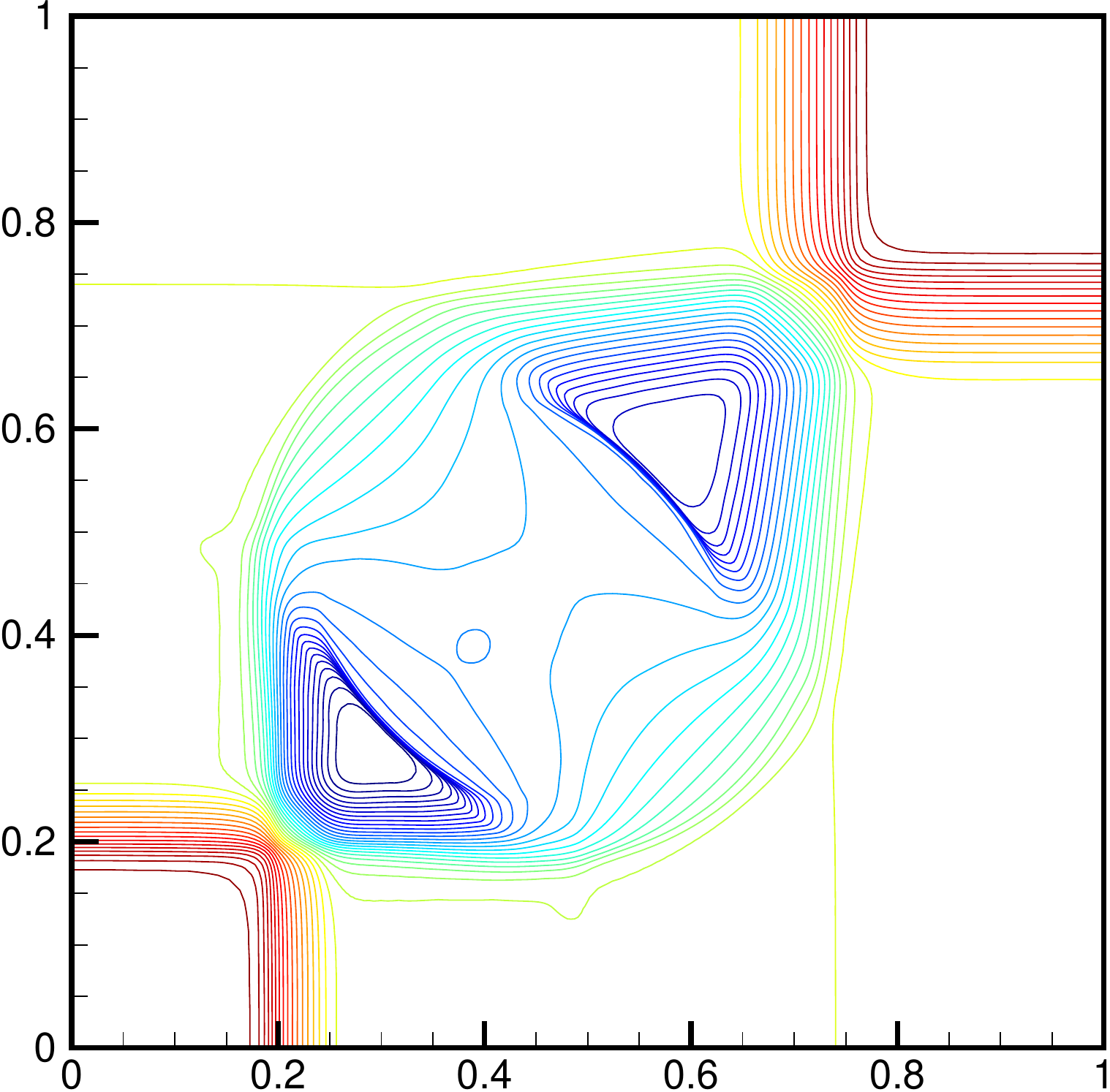}
    \caption{{\tt MM-O2} with $N=200$ (38s)}
    \label{fig:RHD_RP2_MMO2_N200}
  \end{subfigure}
  \begin{subfigure}[b]{0.32\textwidth}
    \centering
    \includegraphics[width=1.06\textwidth]{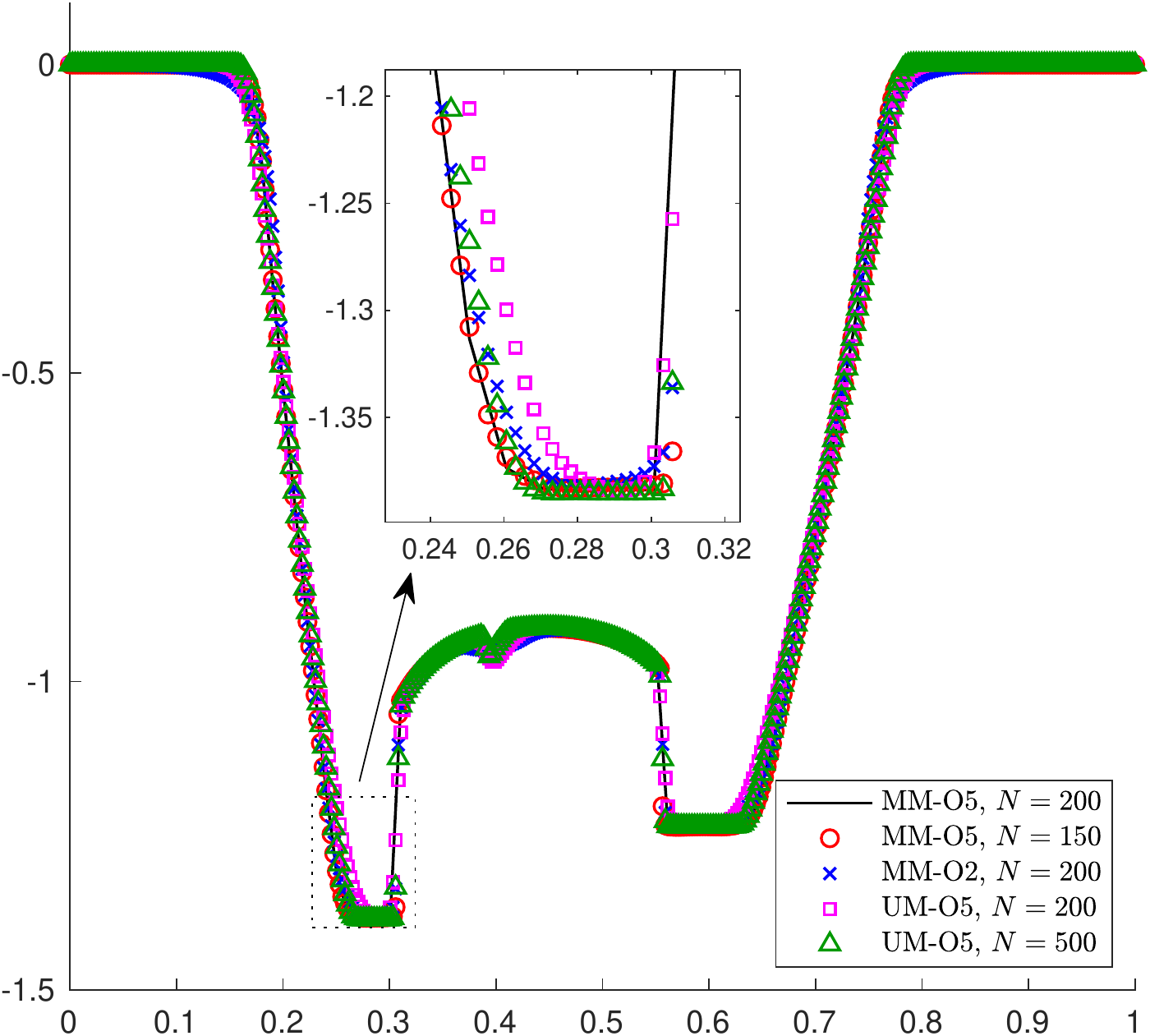}
    \caption{$\ln\rho$ along $x_2=x_1$}
    \label{fig:RHD_RP2_cut}
  \end{subfigure}
  \caption{Example \ref{ex:2DRP2}. Adaptive mesh of {\tt MM-O5} with $N=200$,
    $40$ equally spaced contour lines of $\ln \rho$, and the cut lines of $\ln\rho$ along $x_2=x_1$.
  CPU times are listed in  parentheses.}
  \label{fig:RHD_RP2}
\end{figure}

\begin{example}[RHD Riemann problem \uppercase\expandafter{\romannumeral3}]\label{ex:2DRP3}\rm
  The initial data of the third 2D RHD Riemann problem are
  \begin{align*}
    (\rho,v_1,v_2,p)=\begin{cases}
      (0.035145216124503,~0,~0,~0.162931056509027), & \quad x_1>0.5,~x_2>0.5, \\
      (0.1,~0.7,~0,~1),                             & \quad x_1<0.5,~x_2>0.5, \\
      (0.5,~0,~0,~1),                               & \quad x_1<0.5,~x_2<0.5, \\
      (0.1,~0,~0.7,~1),                             & \quad x_1>0.5,~x_2<0.5,
    \end{cases}
  \end{align*}
  where the left and bottom discontinuities are two contact discontinuities and
  the top and right are two shock waves.
\end{example}

The monitor function is the same as above.
The adaptive mesh of {\tt MM-O5} with $N=200$, the contours of the density logarithms $\ln\rho$
with $40$ equally spaced lines, and $\ln\rho$ cut along $x_2=x_1$ at $t=0.4$
are shown in Figure \ref{fig:RHD_RP3}.
Similar to the last two examples, from Figure \ref{fig:RHD_RP3_MMO5_N150} and \ref{fig:RHD_RP3_MMO2_N200},
one can see that {\tt MM-O5} gives better results than {\tt MM-O2} when using comparable CPU time, especially around the central ``mushroom cloud'', which forms after the interaction of the initial discontinuities.
The solution obtained by {\tt MM-O5} with $N=200$ is much better than {\tt UM-O5} with $N=200$, see Figure \ref{fig:RHD_RP3_cut}, and agrees well with that
of {\tt UM-O5} with $N=600$, while the adaptive moving mesh scheme only takes $13.7\%$ CPU time, verifying the high efficiency of our high-order accurate ES adaptive moving mesh schemes.

\begin{figure}[ht!]
  \centering
  \begin{subfigure}[b]{0.32\textwidth}
    \centering
    \includegraphics[width=\textwidth]{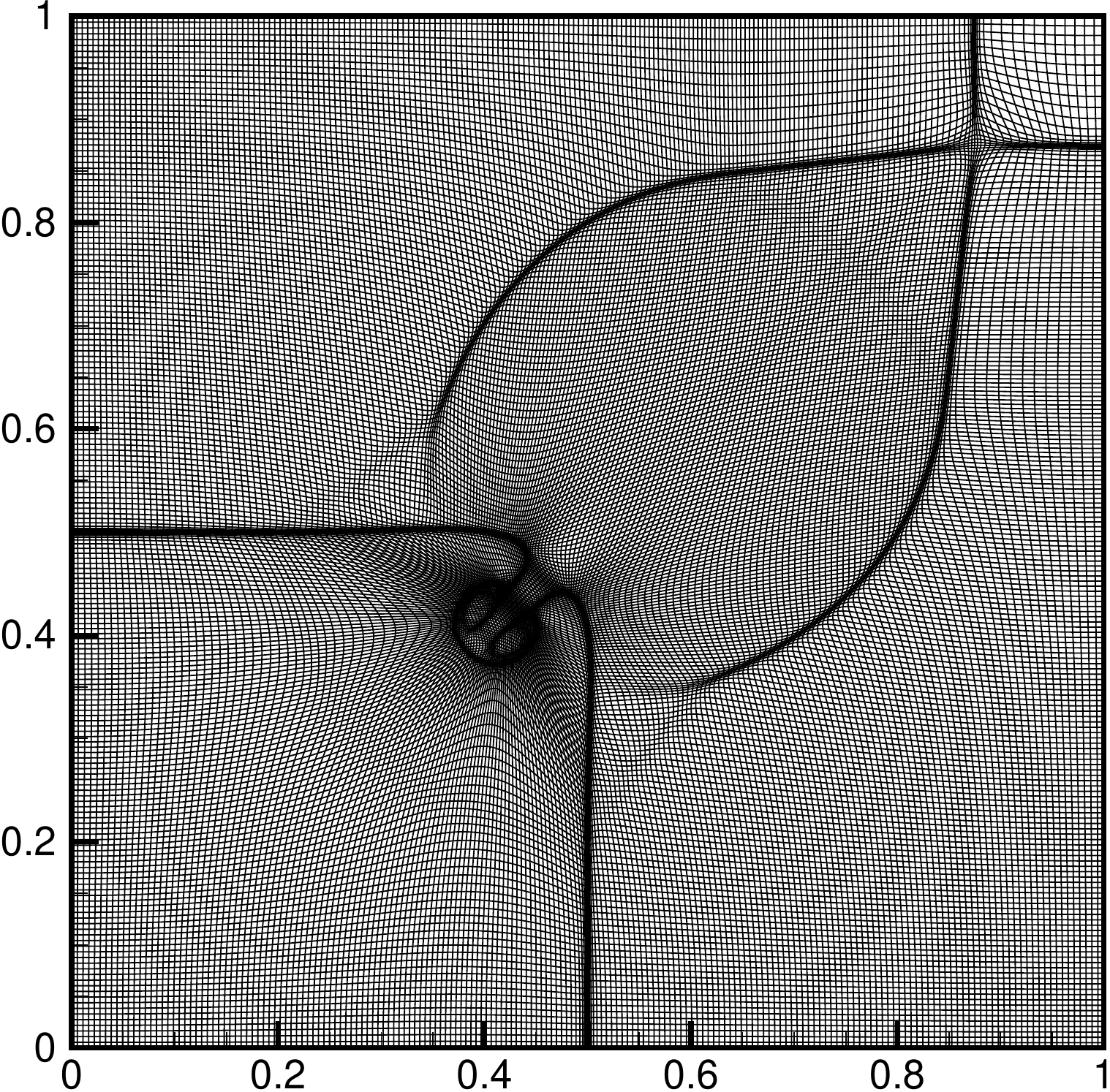}
    \caption{{\tt MM-O5}  with $N=200$}
    \label{fig:RHD_RP3_mesh}
  \end{subfigure}
  \begin{subfigure}[b]{0.32\textwidth}
    \centering
    \includegraphics[width=\textwidth]{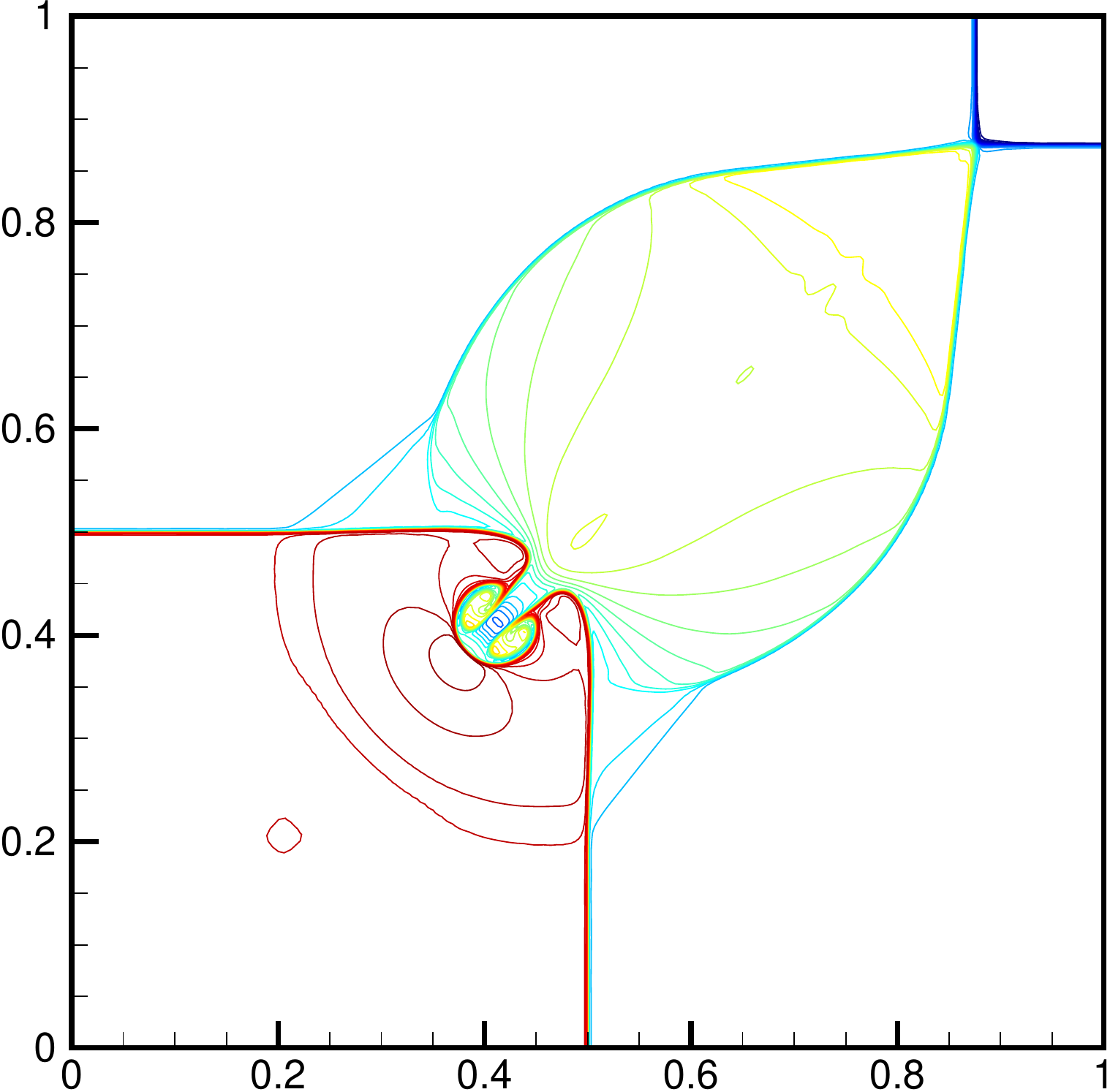}
    \caption{{\tt MM-O5} with $N=200$ (1m16s) }
    \label{fig:RHD_RP3_MMO5_N200}
  \end{subfigure}
  \begin{subfigure}[b]{0.32\textwidth}
    \centering
    \includegraphics[width=\textwidth]{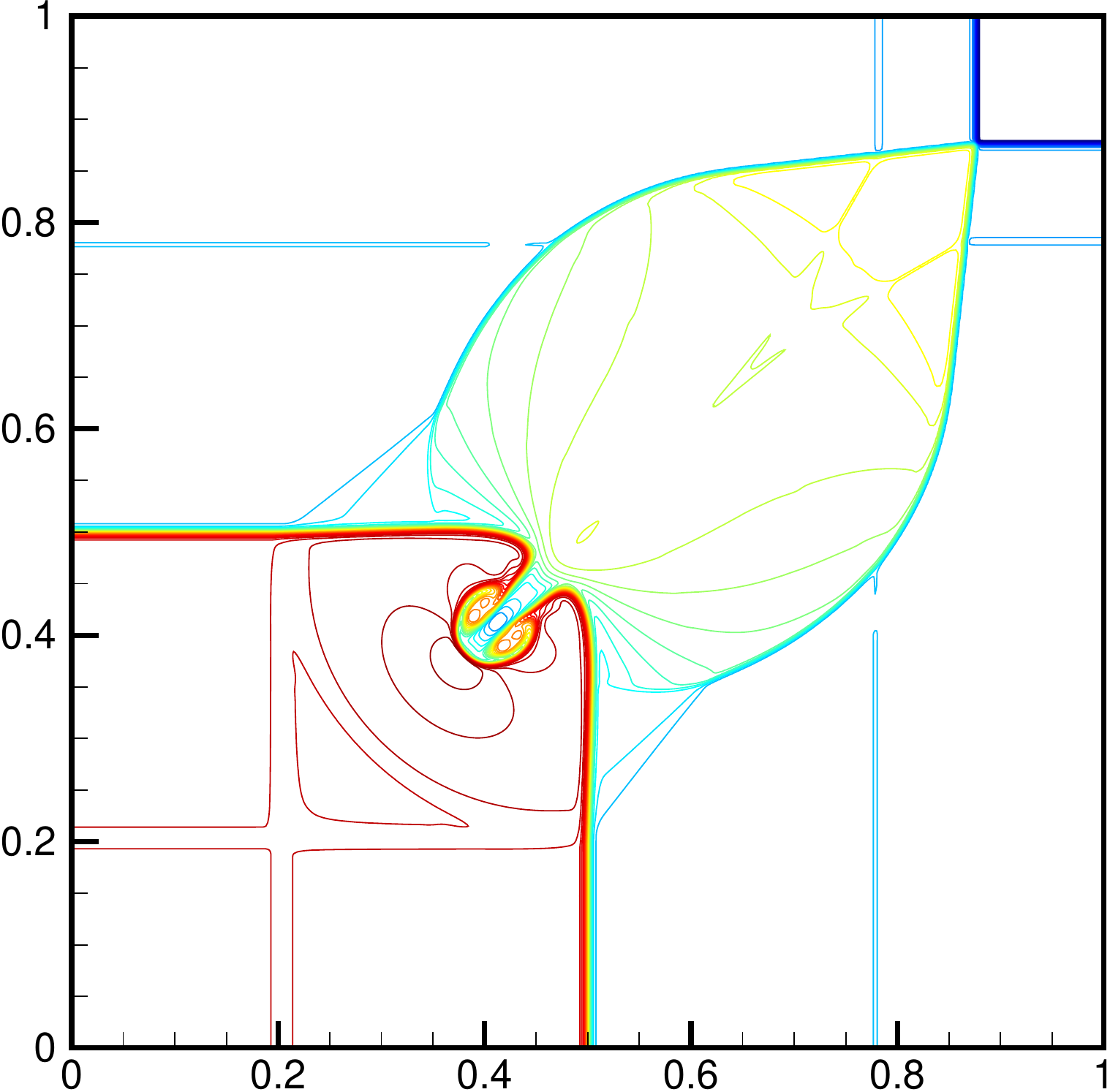}
    \caption{{\tt UM-O5} with $N=600$ (9m16s)}
    \label{fig:RHD_RP3_UMO5_N600}
  \end{subfigure}

  \begin{subfigure}[b]{0.32\textwidth}
    \centering
    \includegraphics[width=\textwidth]{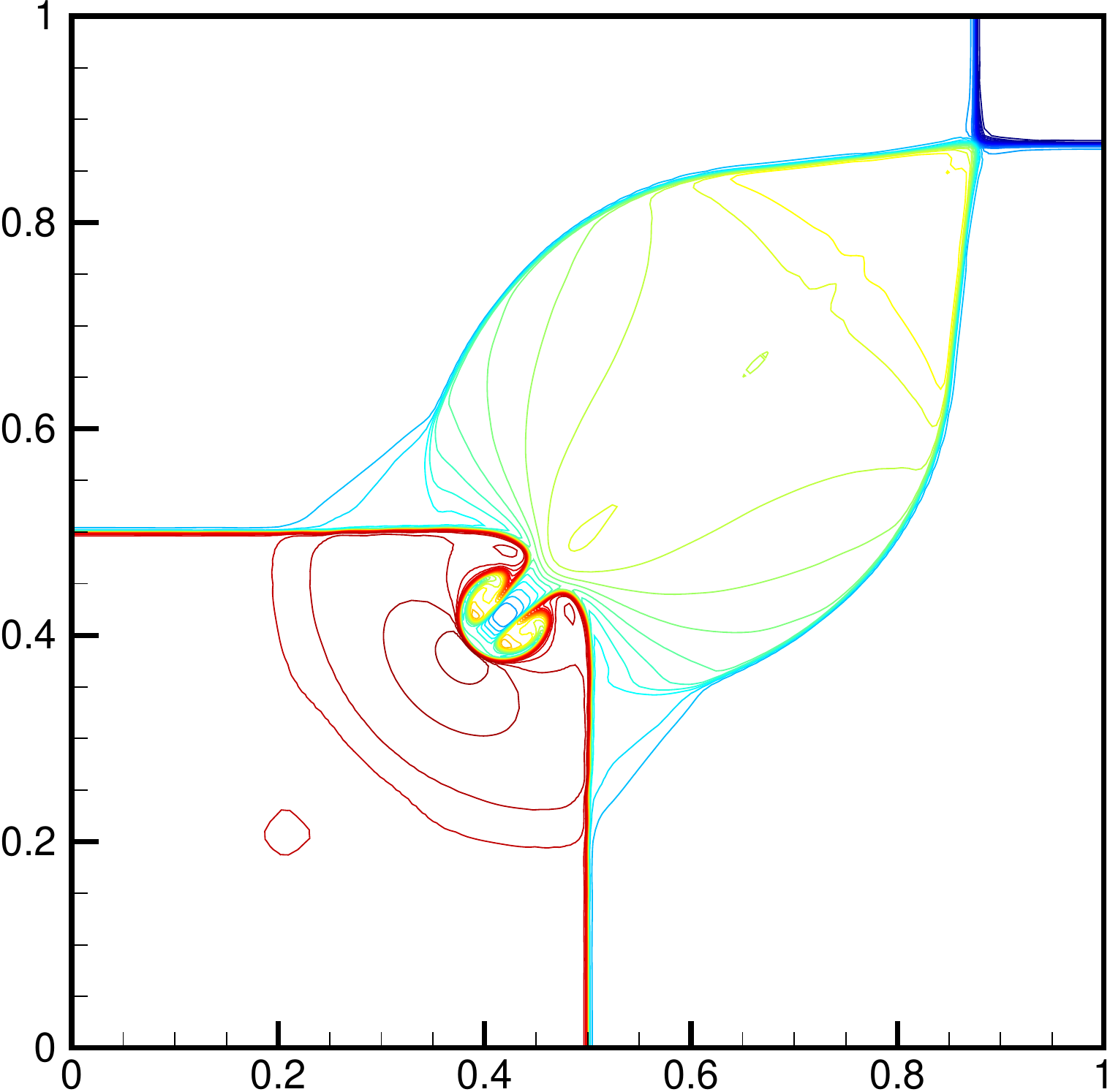}
    \caption{{\tt MM-O5} with $N=150$ (32s) }
    \label{fig:RHD_RP3_MMO5_N150}
  \end{subfigure}
  \begin{subfigure}[b]{0.32\textwidth}
    \centering
    \includegraphics[width=\textwidth]{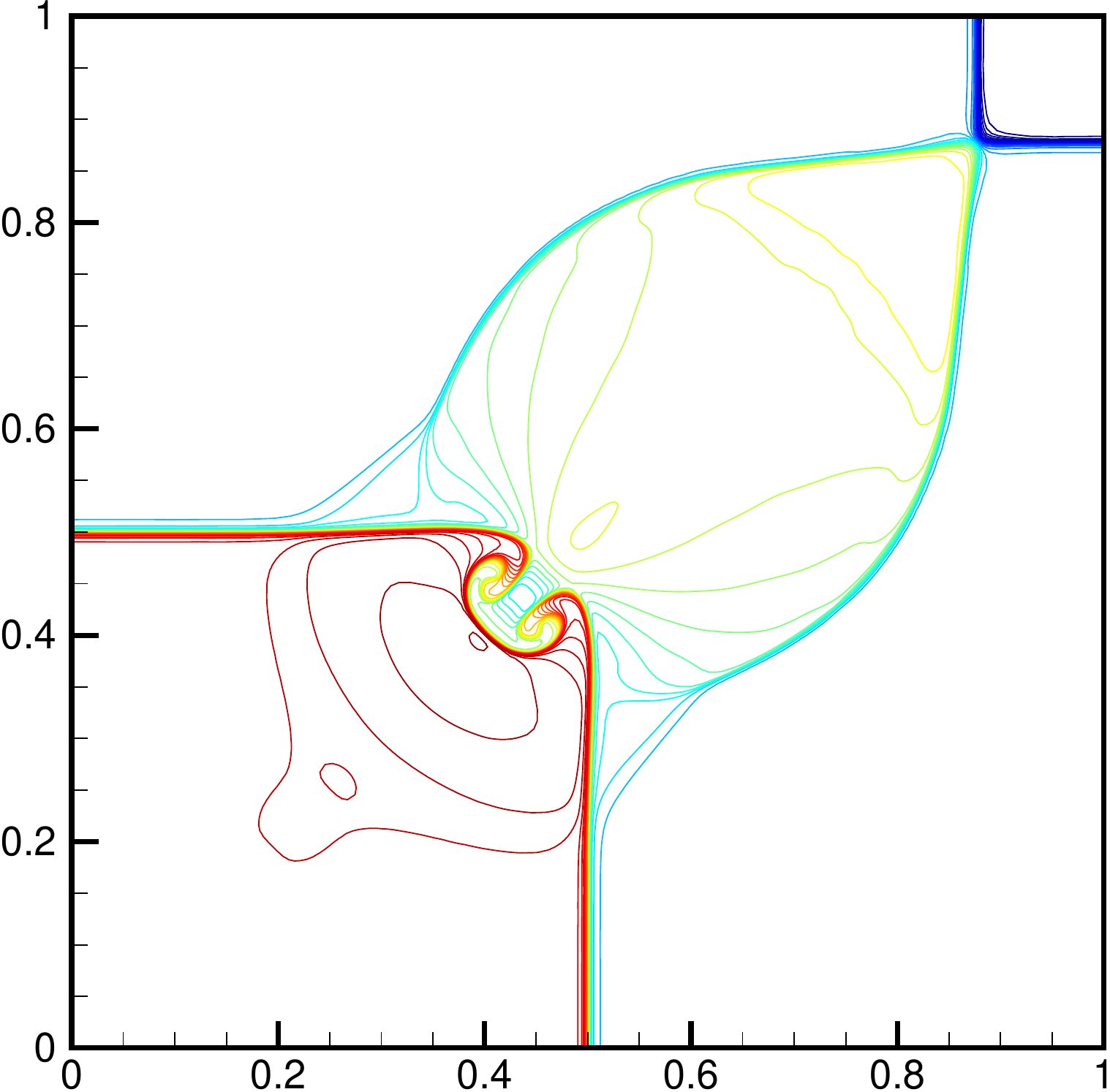}
    \caption{{\tt MM-O2} with $N=200$ (39s)}
    \label{fig:RHD_RP3_MMO2_N200}
  \end{subfigure}
  \begin{subfigure}[b]{0.32\textwidth}
    \centering
    \includegraphics[width=1.0\textwidth]{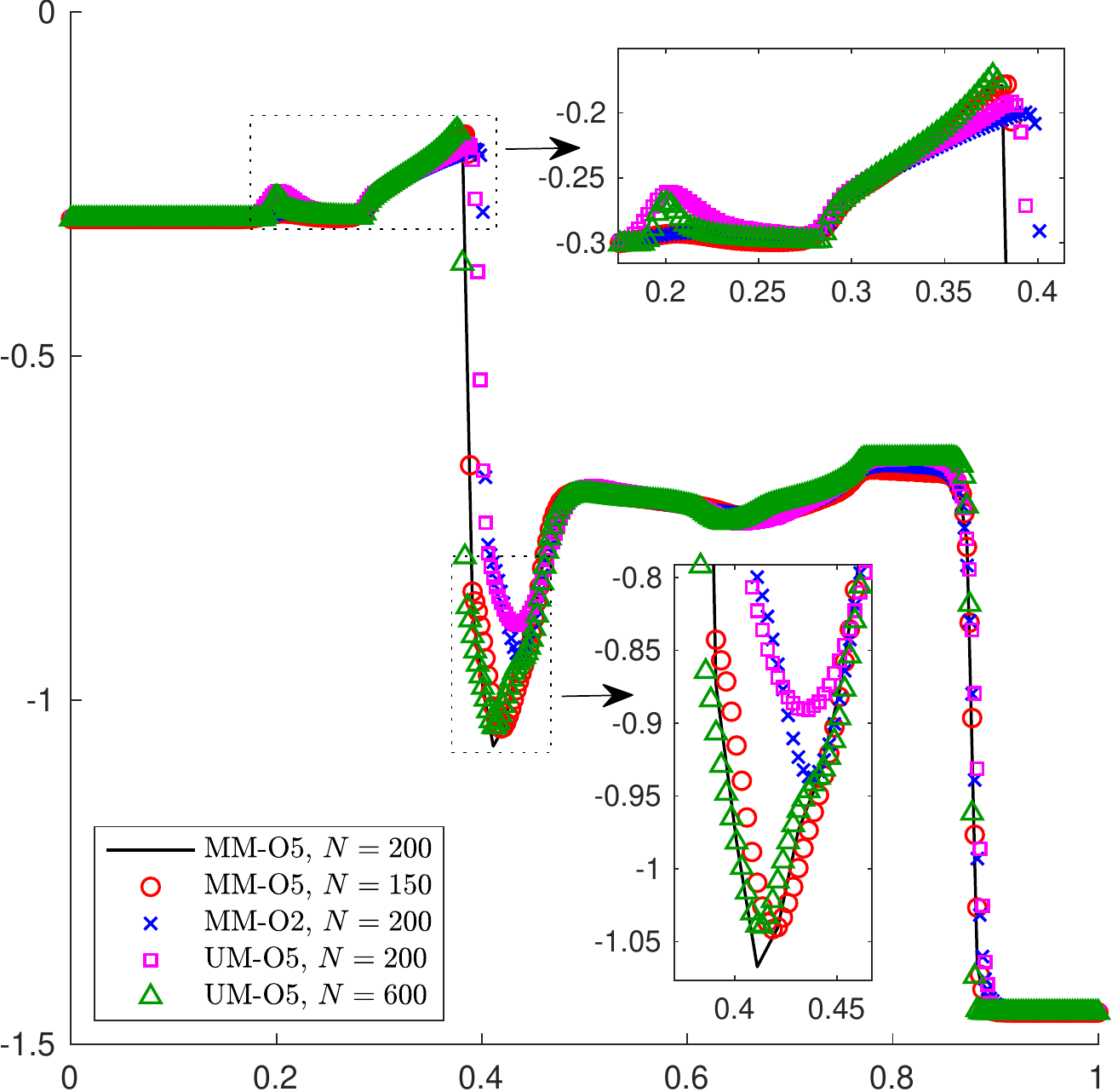}
    \caption{$\ln\rho$ along $x_2=x_1$}
    \label{fig:RHD_RP3_cut}
  \end{subfigure}
  \caption{Example \ref{ex:2DRP3}. Adaptive mesh of {\tt MM-O5} with $N=200$,
    $40$ equally spaced contour lines of $\ln \rho$, and  cut lines of $\ln\rho$ along $x_2=x_1$.
  CPU times are listed in   parentheses.}
  \label{fig:RHD_RP3}
\end{figure}

\begin{example}[2D RMHD blast problem]\label{ex:RMHD_2DBlast}\rm
  It is a benchmark test problem for the RMHD, and the initial setup in \cite{Balsara2016A,DelZanna2003An,Mignone2006An} is adopted.
  The physical domain is $[-6,6]^2$ with outflow boundary conditions, and divided into three parts at initial time.
  The inner part is the explosion zone with a radius of $0.8$, and $\rho=0.01,~p=1$;  and the outer part is the ambient medium with the radius larger than $1$, and $\rho=10^{-4},~p=5\times 10^{-4}$; while
  the intermediate part is a linear taper applied to the density and the pressure from the radius $0.8$ to $1$.
  The magnetic field is only initialized in the $x_1$-direction as
  $\Bx=0.1$ and the adiabatic index $\Gamma=4/3$.
  This problem is solved by using the fifth-order ES adaptive moving mesh scheme with $N\times N$ meshes until $t=4$.
\end{example}
The monitor is the same as that in the last example except for $\alpha=800$.
Figure \ref{fig:RMHD_2DBlast} shows the adaptive mesh and $40$ equally spaced contour lines obtained by using {\tt MM-O5} with $150\times 150$ mesh at $t=4$.
One can see that the mesh points adaptively concentrate near the large gradient of $\ln\rho$ due to the choice of the monitor function, and increase the resolution of the shock waves.
To compare the results of the fifth-order ES schemes on the adaptive moving mesh and the static uniform mesh, the cut lines of $p$ and $W$ are plotted in Figure \ref{fig:RMHD_2DBlast_cut}.
It is seen that the results obtained by using {\tt MM-O5} with $N=150$ are much better than those of {\tt UM-O5} with the same grid number, and comparable to those of {\tt UM-O5} with $N=600$.
From Table \ref{tab:RMHD_2D_CPU}, one can see that {\tt MM-O5} is more efficient than {\tt UM-O5}, since the former takes only $7.26\%$ CPU time of the latter, highlighting the high efficiency of our high-order accurate ES adaptive moving mesh schemes.

\begin{figure}[ht!]
  \centering
  \begin{subfigure}[b]{0.48\textwidth}
    \centering
    \includegraphics[width=\textwidth]{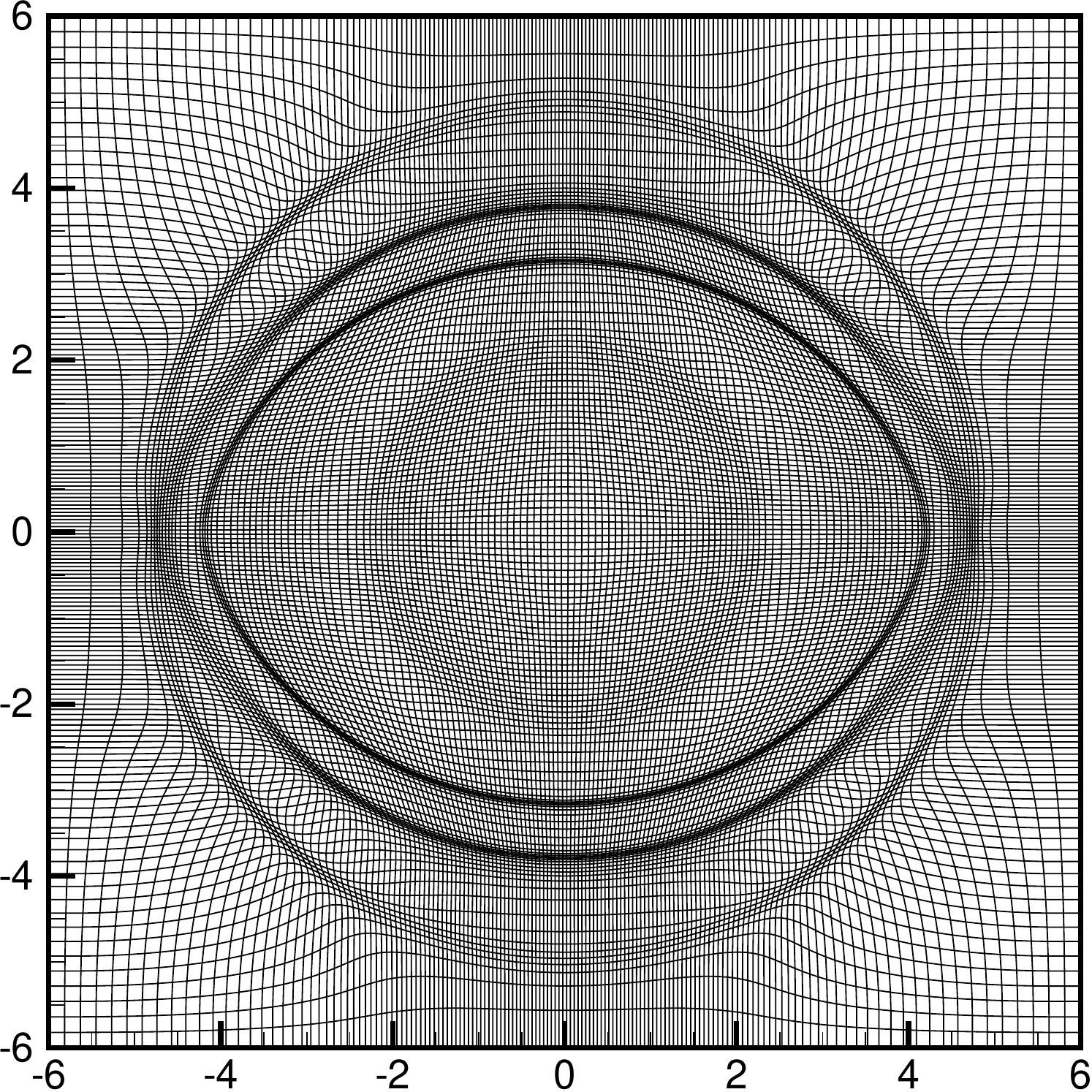}
    \caption{Adaptive mesh}
  \end{subfigure}
  \begin{subfigure}[b]{0.48\textwidth}
    \centering
    \includegraphics[width=\textwidth]{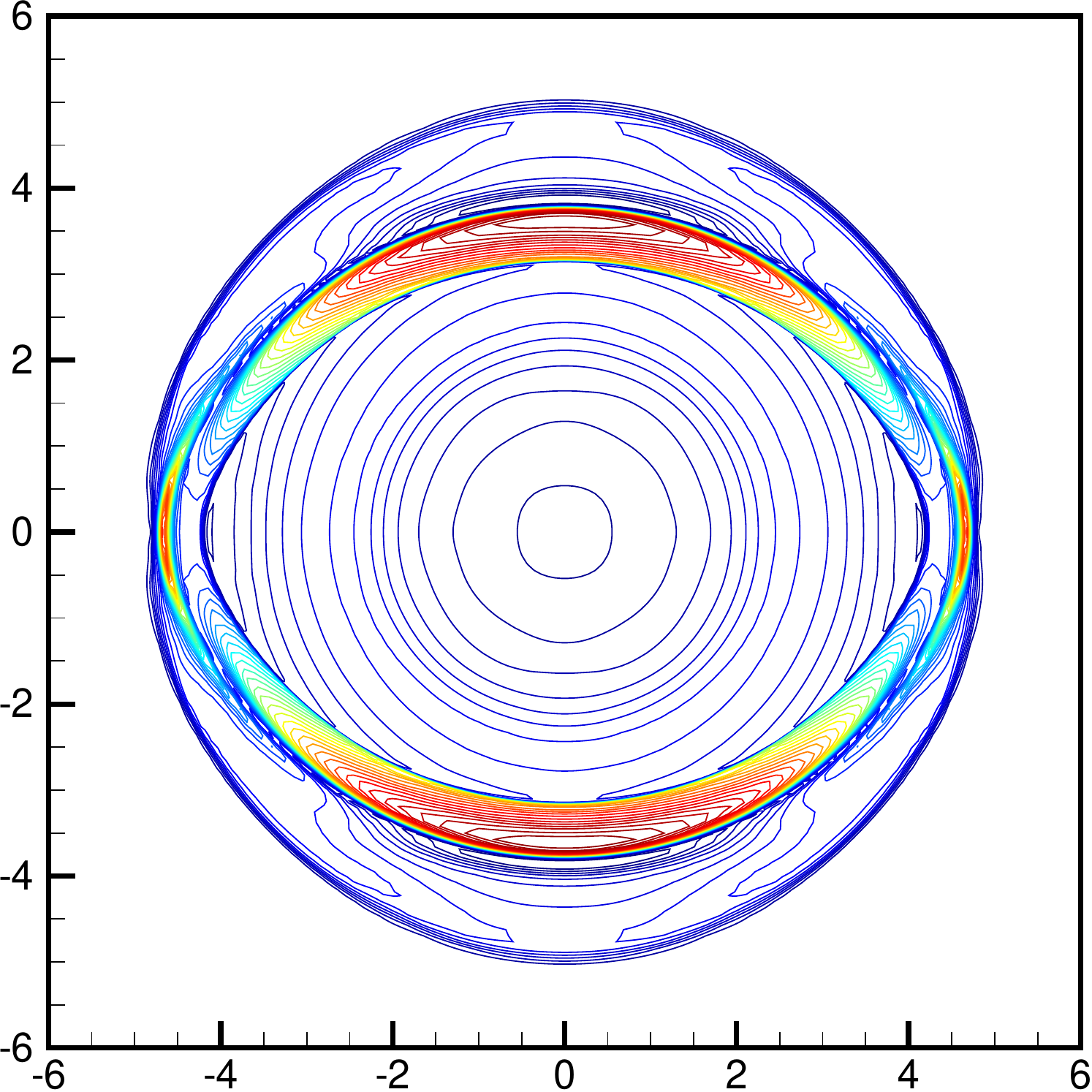}
    \caption{$\rho$}
  \end{subfigure}

  \begin{subfigure}[b]{0.48\textwidth}
    \centering
    \includegraphics[width=\textwidth]{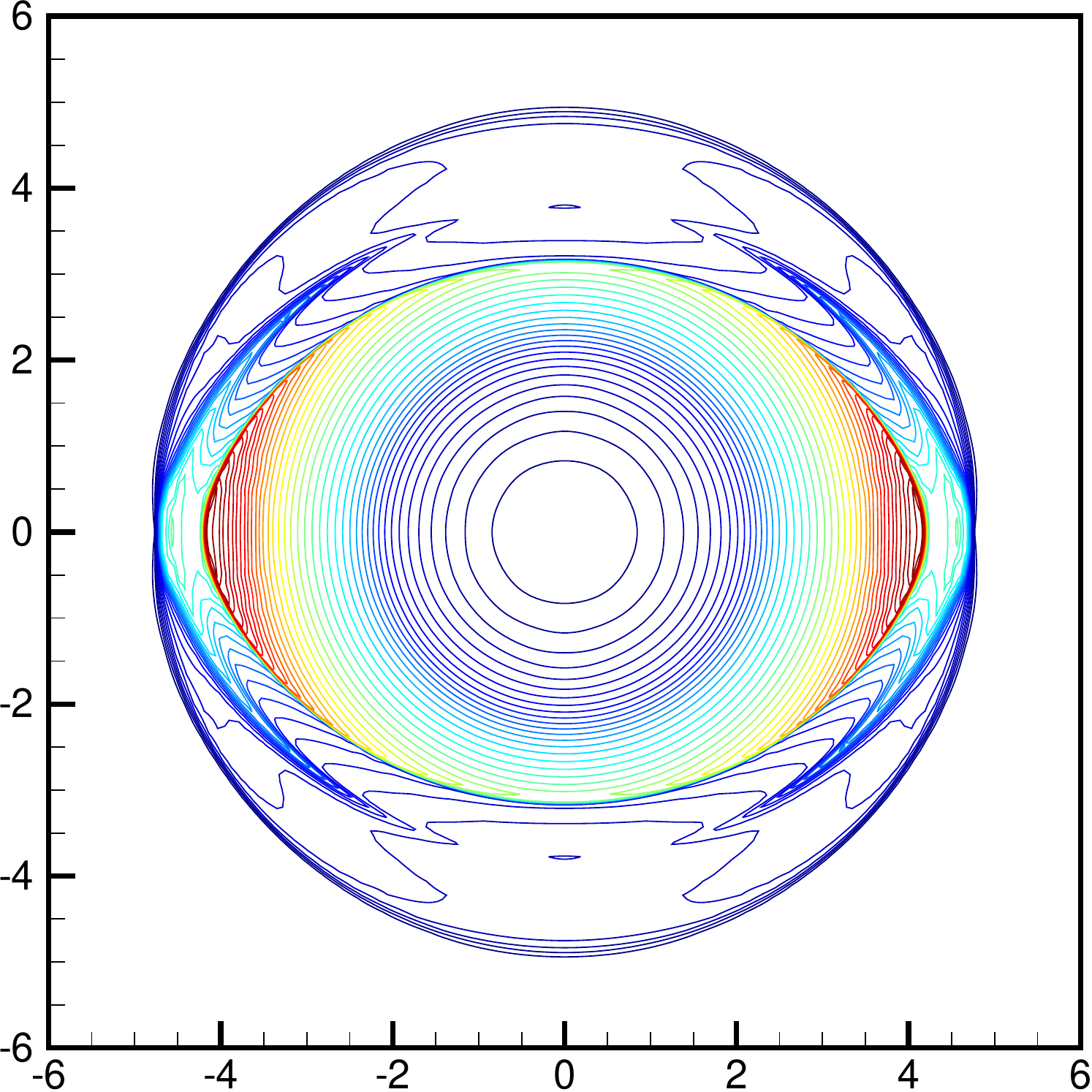}
    \caption{$W$}
  \end{subfigure}
  \begin{subfigure}[b]{0.48\textwidth}
    \centering
    \includegraphics[width=1.02\textwidth]{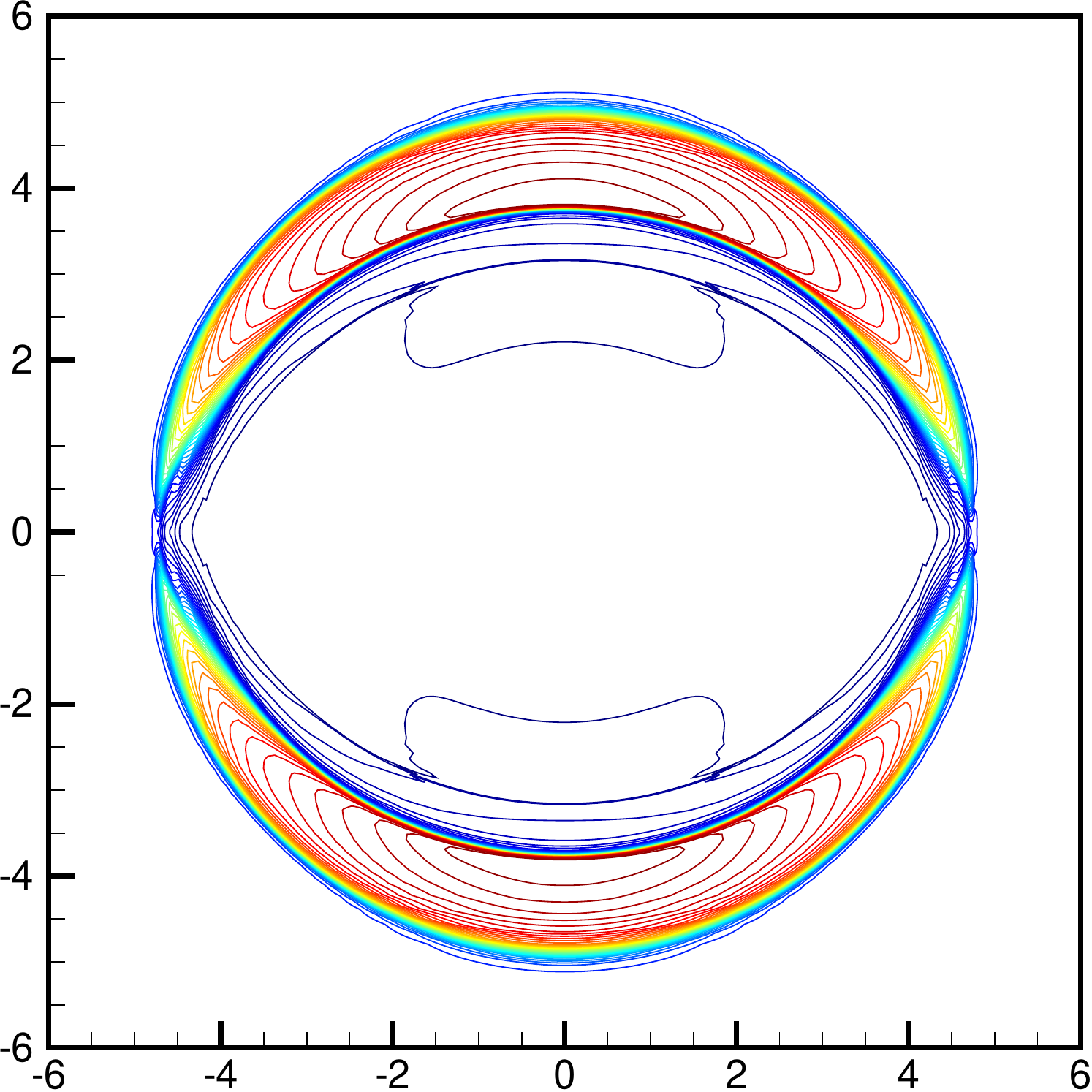}
    \caption{$\abs{\bm{B}}$}
  \end{subfigure}
  \caption{Example \ref{ex:RMHD_2DBlast}. Adaptive mesh and $40$ equally spaced contour lines obtained by {\tt MM-O5} with $150\times 150$ mesh.}
  \label{fig:RMHD_2DBlast}
\end{figure}

\begin{figure}[ht!]
  \centering
  \begin{subfigure}[b]{0.48\textwidth}
    \centering
    \includegraphics[width=\textwidth]{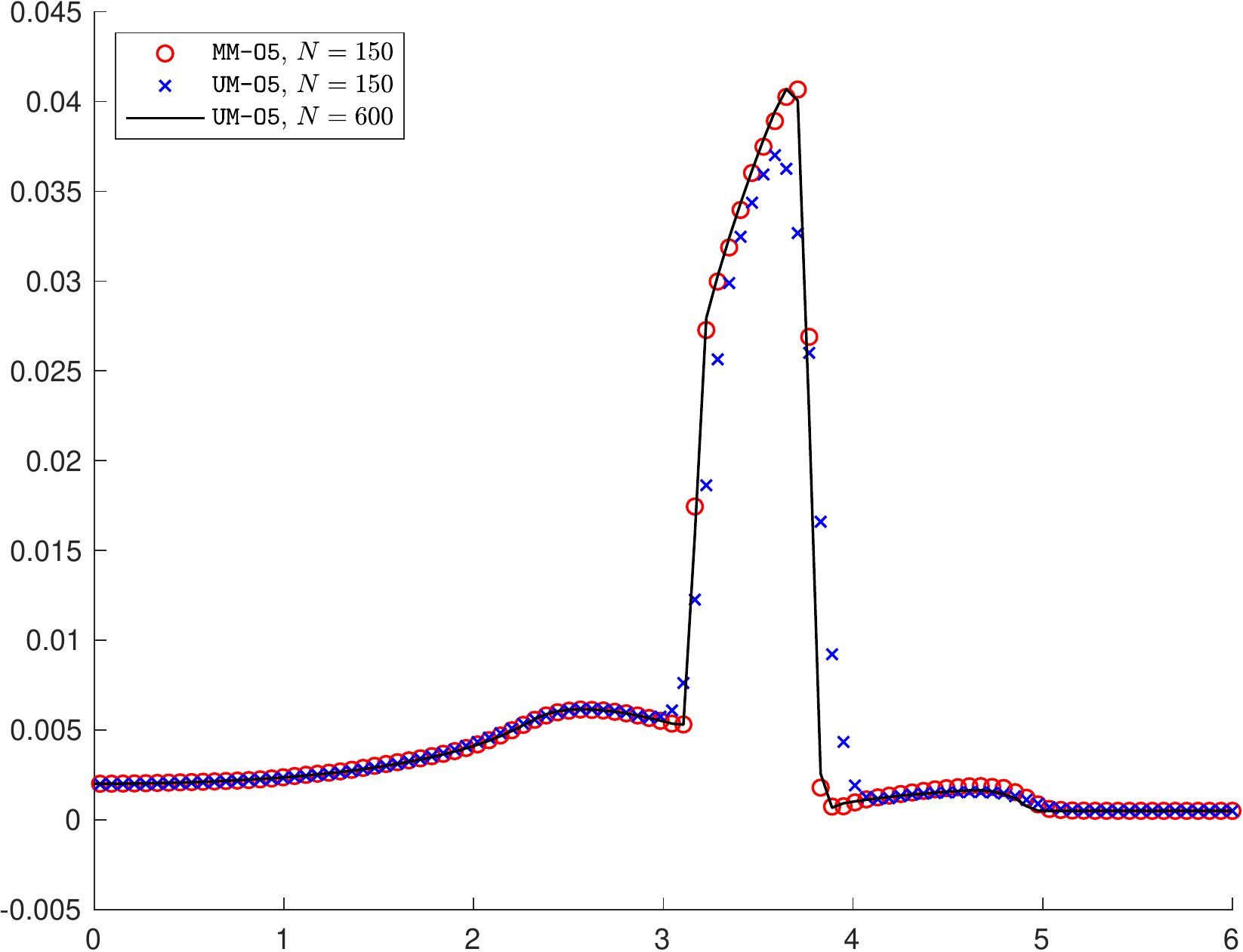}
    \caption{$p$}
  \end{subfigure}
  \begin{subfigure}[b]{0.48\textwidth}
    \centering
    \includegraphics[width=\textwidth]{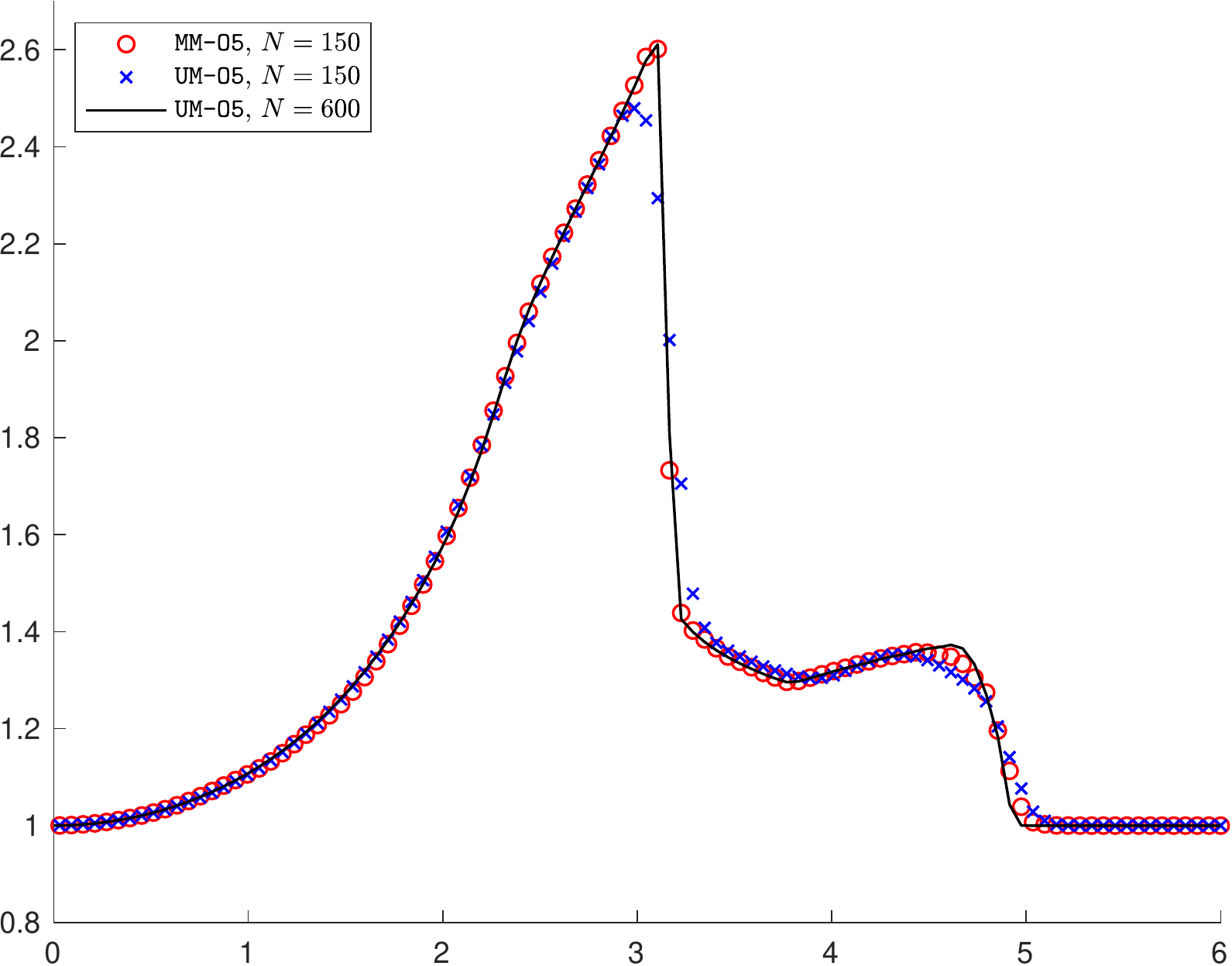}
    \caption{$W$}
  \end{subfigure}
  \caption{Example \ref{ex:RMHD_2DBlast}. Cut lines of  $p$ and $W$ along $x_1=0$ ($x_2\in[0,6]$).}
  \label{fig:RMHD_2DBlast_cut}
\end{figure}

\begin{example}[2D RMHD shock-cloud interaction]\label{ex:RMHD_2DSC}\rm
  It is about a strong shock wave interacts with a high density cloud \cite{He2012RMHD}.
  The physical domain is $[-0.2,1.2]\times [0,1]$	with the inflow boundary condition specified on the left boundary, and the outflow boundary conditions on the other  boundaries.
  A planar shock wave moves from $x_1=0.05$ to the right with the left and right states
  \begin{align*}
    (\rho, \bv, p, \bm{B})=\begin{cases}
      (3.86859, ~0.68, ~0, ~0, ~1.25115, ~0, ~0.84981,  -0.84981), ~& ~x_1<0.05, \\
      (1, ~0, ~0, ~0, ~0, ~0.16106, ~0.16106, ~0.05), ~&\text{otherwise}. \\
    \end{cases}
  \end{align*}
 The circular cloud of radius $0.15$ with a high density $\rho = 30$ is centered at $(0.25, 0.5)$.
  This problem is solved by using the fifth-order ES adaptive moving mesh scheme until $t=1.2$.
\end{example}
The monitor is the same as that in the last example.
Figure \ref{fig:RMHD_2DSC_mesh} shows the $210\times 150$ adaptive mesh obtain by {\tt MM-O5}, where the mesh points adaptively concentrate near the cloud.
To give comparable results presented in \cite{He2012RMHD},
the numerical schlieren images generated by using $\phi_1=\exp(-50\abs{\nabla\ln\rho}/\abs{\nabla\ln\rho}_{\text{max}})$ and
$\phi_2=\exp(-50\abs{\nabla\abs{\bm{B}}}/\abs{\nabla\abs{\bm{B}}}_{\text{max}})$
are presented in Figures \ref{fig:RMHD_2DSC_phi1}-\ref{fig:RMHD_2DSC_phi2}.
The results obtained by {\tt MM-O5} with $210\times 150$ mesh are shown in the upper half parts, while {\tt UM-O5} with $210\times 150$ and $560\times 400$ meshes are  respectively shown in the  lower half parts of the left and right plots, so that one can compare the results more clearly.
Similar to the last example, {\tt MM-O5} gives the comparable results to {\tt UM-O5} with a finer mesh, while takes only $10.5\%$ CPU time, see Table \ref{tab:RMHD_2D_CPU}.

\begin{figure}
  \centering
  \includegraphics[width=0.48\textwidth]{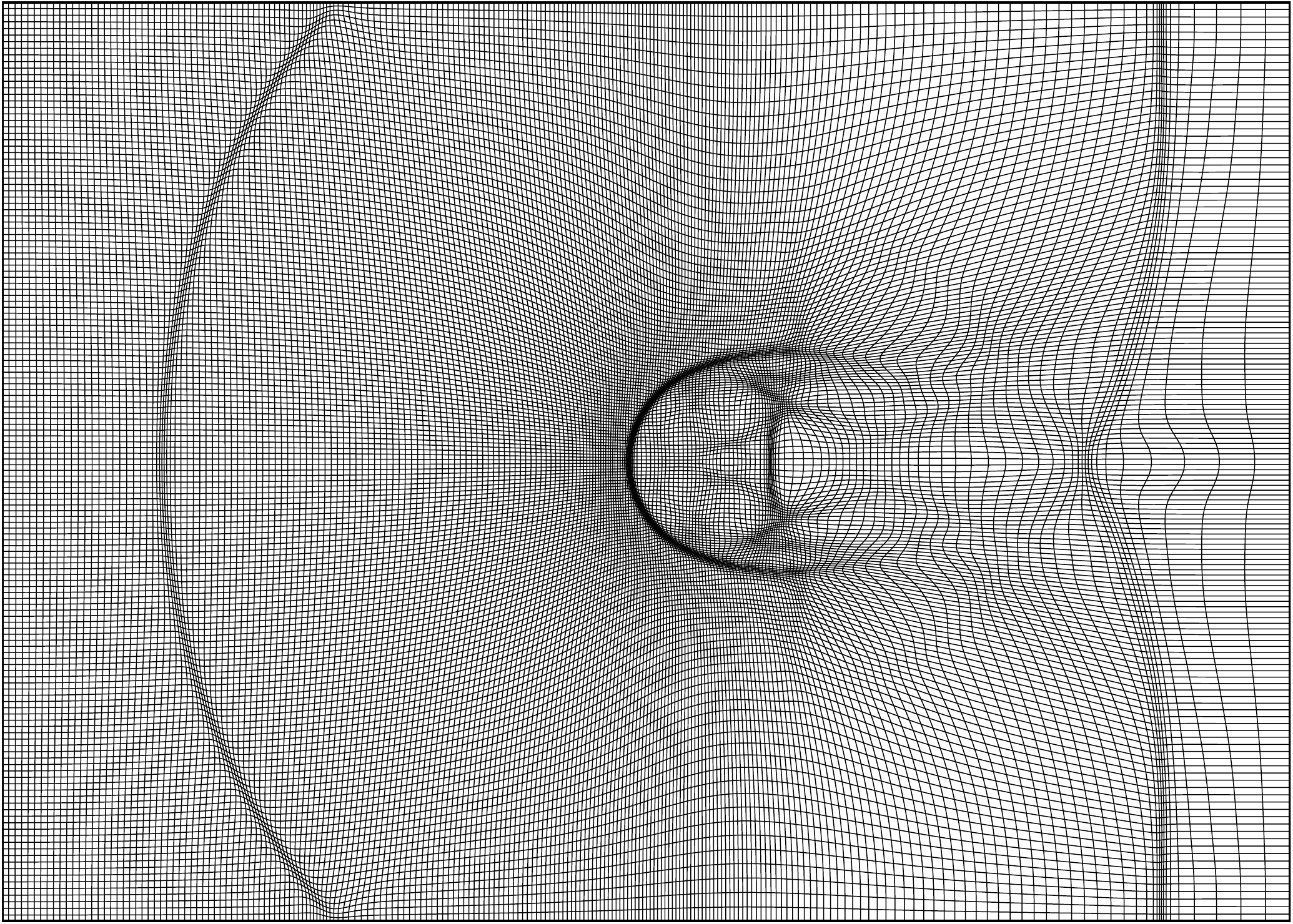}
  \caption{Example \ref{ex:RMHD_2DSC}. $210\times 150$ adaptive mesh obtained by {\tt  MM-O5} at $t=1.2$.}
  \label{fig:RMHD_2DSC_mesh}
\end{figure}

\begin{figure}
  \centering
  \includegraphics[width=0.48\textwidth, trim=2 2 2 2, clip]{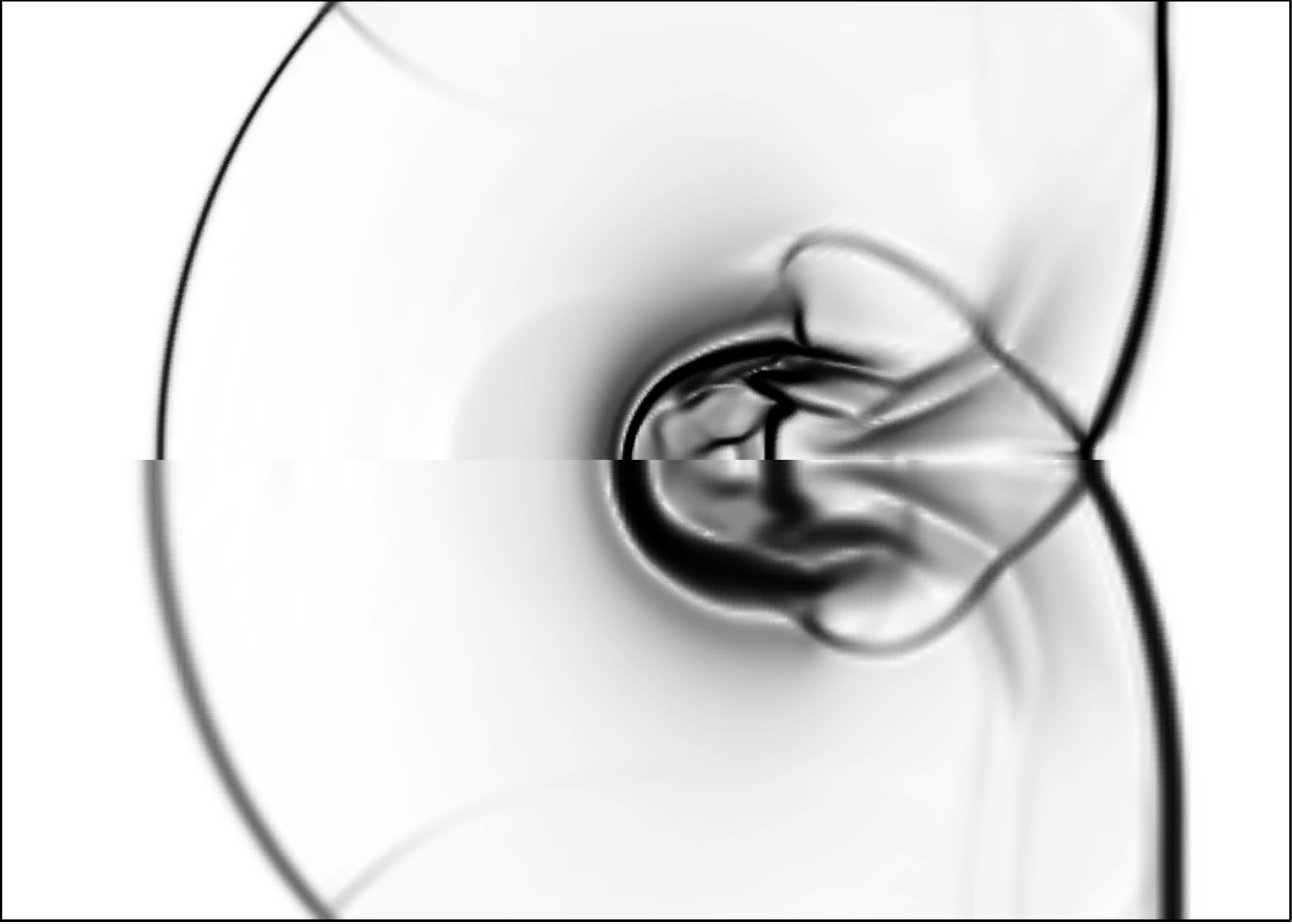}
  \includegraphics[width=0.48\textwidth, trim=2 2 2 2, clip]{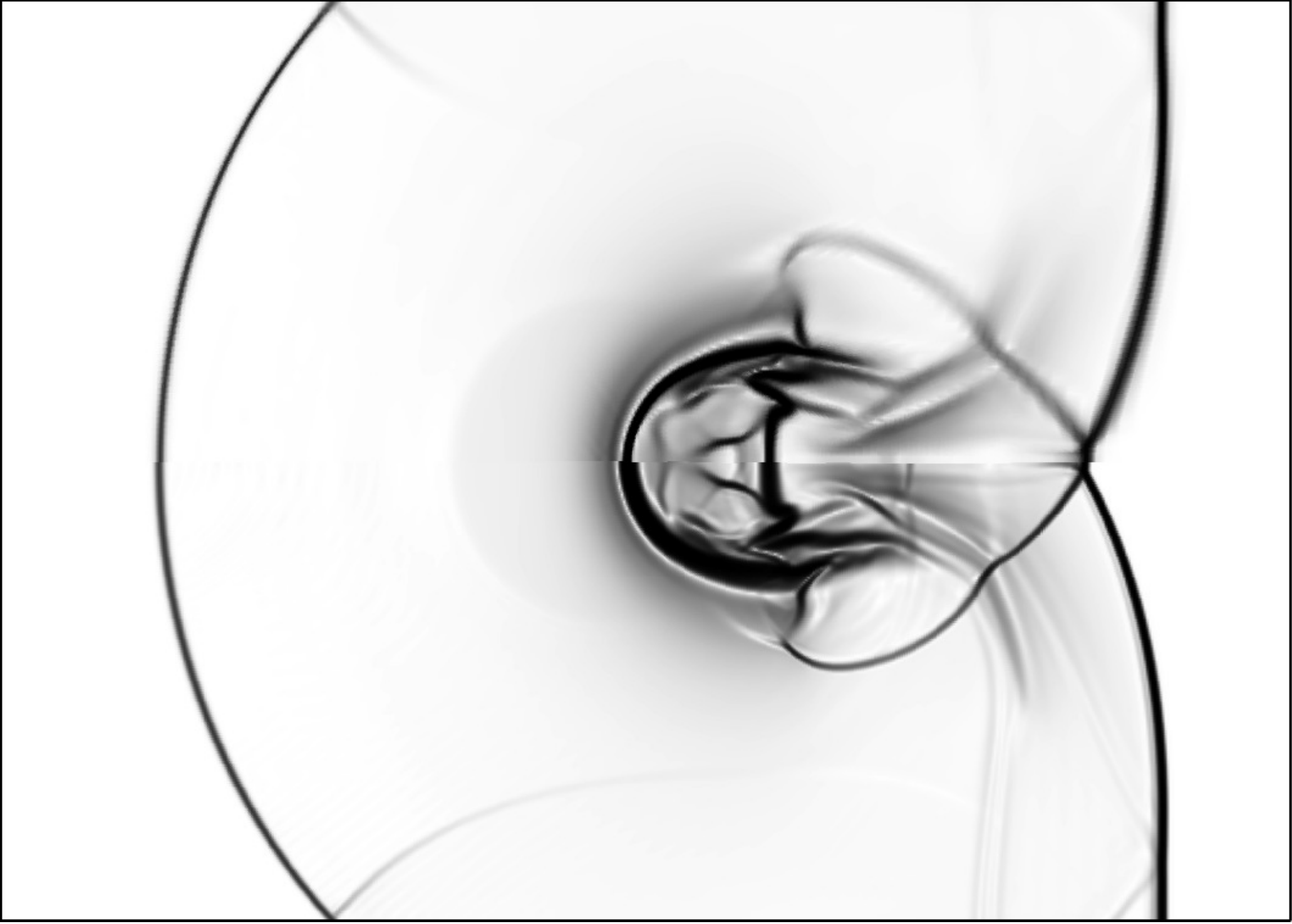}
  \caption{Example \ref{ex:RMHD_2DSC}. Numerical schlieren images of $\phi_1$ at $t=1.2$.
    Left: {\tt  MM-O5} with $210\times 150$ mesh (upper half) and {\tt UM-O5} with $210\times 150$ mesh (lower half).
    Right: {\tt MM-O5} with $210\times 150$ mesh (upper half) and {\tt UM-O5} with $560\times 400$ mesh (lower half).
  }
  \label{fig:RMHD_2DSC_phi1}
\end{figure}

\begin{figure}
  \centering
  \includegraphics[width=0.48\textwidth, trim=2 2 2 2, clip]{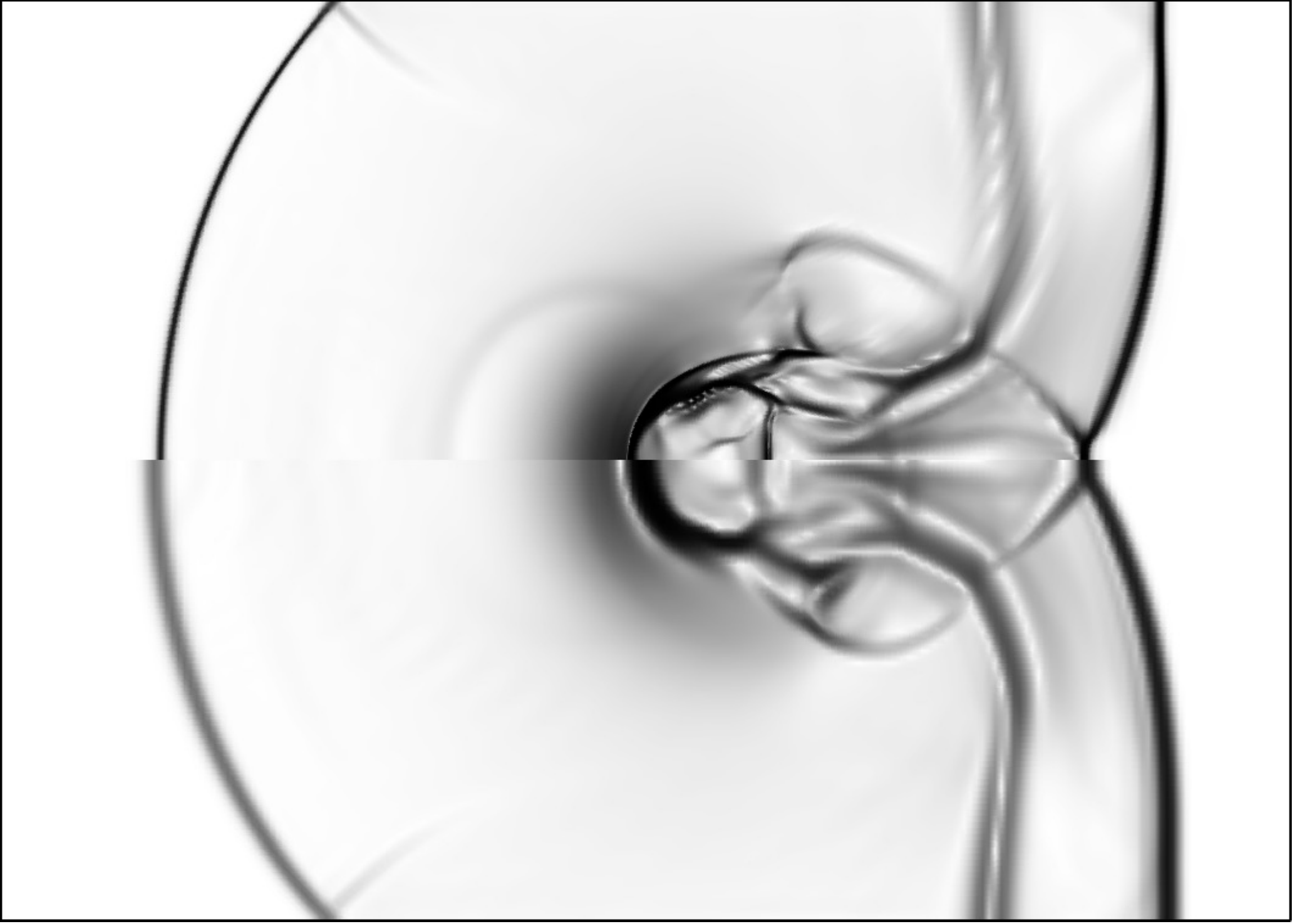}
  \includegraphics[width=0.48\textwidth, trim=2 2 2 2, clip]{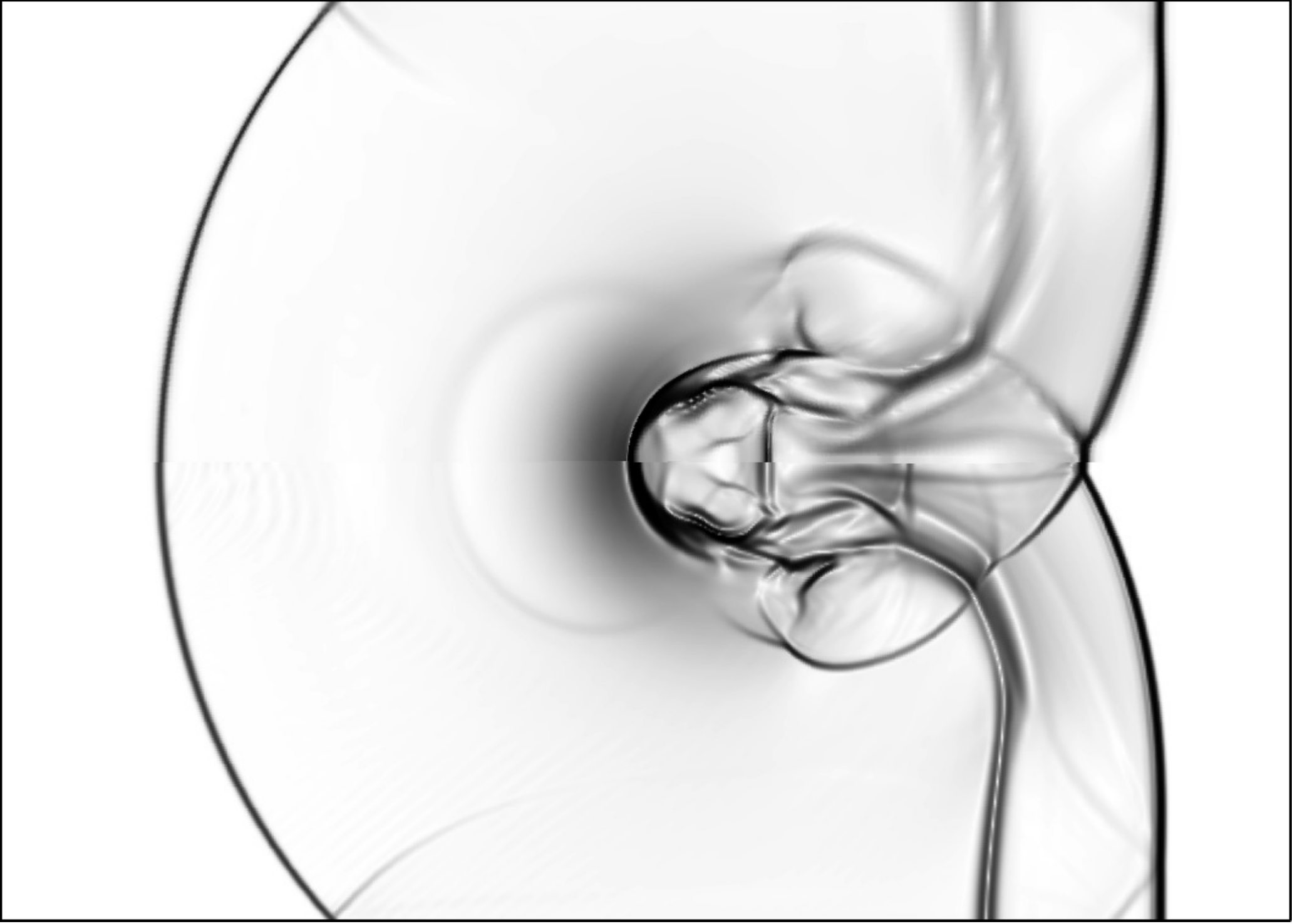}
  \caption{Same as Figure \ref{fig:RMHD_2DSC_phi1} except for $\phi_2$.}
  \label{fig:RMHD_2DSC_phi2}
\end{figure}

\begin{table}[!ht]
  \centering
  \begin{tabular}{c|c|c}
    \hline
    Scheme & Example \ref{ex:RMHD_2DBlast} & Example \ref{ex:RMHD_2DSC} \\ \hline
    {\tt MM-O5} & 2m03s ($150\times 150$) & 4m16s  ($210\times 150$) \\
    {\tt UM-O5} & 30s ($150\times 150$)  &  2m05s ($210\times 150$) \\
    {\tt UM-O5} & 28m14s   ($600\times 600$) & 40m28s  ($560\times 400$) \\
    \hline
  \end{tabular}
  \caption{CPU times of Examples \ref{ex:RMHD_2DBlast}-\ref{ex:RMHD_2DSC} (4 cores are used).}
  \label{tab:RMHD_2D_CPU}
\end{table}

\subsection{3D tests}
\begin{example}[3D RMHD  isentropic vortex problem]\label{ex:RMHD_3DVortex}\rm
  It is  given in \cite{Duan2021Analytical} and used  here to verify the accuracy of the 3D EC and ES moving mesh schemes.
  The analytical solutions at time $t$ and the spatial point $(x_1,x_2,x_3)$ in the physical domain $[-R,R]\times[-R,R]\times[-5R,5R]$ with $R=5$ and the periodic boundary conditions can be given by
  \begin{equation*}
    \begin{aligned}
      \rho&=(1-\sigma\exp(1-r^2))^{\frac{1}{\Gamma-1}},~ p=\rho^\Gamma,\\
      \bm{v}&=\frac{1}{6-3(\widetilde{v}_1+\widetilde{v}_2)}(4\widetilde{v}_1+\widetilde{v}_2-3,
      ~4\widetilde{v}_2+\widetilde{v}_1-3, ~\widetilde{v}_1+\widetilde{v}_2-3),\\
      \bm{B}&=\frac13\left(5\widetilde{B}_1 - \widetilde{B}_2,
        ~5\widetilde{B}_2 - \widetilde{B}_1,
      ~-\widetilde{B}_1 - \widetilde{B}_2\right),
    \end{aligned}
  \end{equation*}
  where
  \begin{equation*}
    \begin{aligned}
      &\Gamma=5/3,~\sigma=0.2,~B_0=0.05,~r=\sqrt{\widetilde{x}_1^2+\widetilde{x}_2^2},\\
      &(\widetilde{x}_1,\widetilde{x}_2)=(40/3k_1 + 10/3k_2 + \widehat{x}_1,~10/3k_1 + 40/3k_2 + \widehat{x}_2), ~(\widetilde{x}_1,\widetilde{x}_2)\in \Omega_0, ~k_1,k_2\in\mathbb{Z},\\
      &\widehat{x}_k=x_k+({x}_1 + {x}_2 + {x}_3)/3+t,~k=1,2,3,\\
      &(\widetilde{v}_1,\widetilde{v}_2)=(-\widetilde{x}_2,\widetilde{x}_1)f,~ f=\sqrt{\dfrac{\kappa \exp(1-r^2)}{\kappa r^2\exp(1-r^2) + (\Gamma-1)\rho + \Gamma p}},~
      \kappa = 2\Gamma \sigma\rho + (\Gamma-1)B_0^2(2-r^2),	\\
      & (\widetilde{B}_1,\widetilde{B}_2)=B_0\exp(1-r^2)(-\widetilde{x}_2,\widetilde{x}_1).
    \end{aligned}
  \end{equation*}
  The problem is solved until $t=0.1$ with a series of $N\times N\times 5N$ meshes.

Similar to the 2D isentropic vortex problem, two  mesh movements are used.
The first  is generated by using the adaptive moving mesh strategy in Section \ref{section:MM} based on the monitor being similar to the 2D case \eqref{eq:RMHD_2DVortex_Monitor}, while the second is given by the following expressions
\begin{equation}\label{eq:RMHD_3DVortex_MovingMesh}
  \begin{split}
    &(x_1)_{\bm{i}}=\mathring{x}_1 + 0.2\cos(\pi t/4)\sin(3\pi \mathring{x}_2/R)\sin(3\pi \mathring{x}_3/5R),\\
    &(x_2)_{\bm{i}}=\mathring{x}_2 + 0.2\cos(\pi t/4)\sin(3\pi \mathring{x}_3/5R)\sin(3\pi \mathring{x}_1/R),\\
    &(x_3)_{\bm{i}}=\mathring{x}_3 + 0.2\cos(\pi t/4)\sin(3\pi \mathring{x}_1/R)\sin(3\pi \mathring{x}_2/R),\\
    &\mathring{x}_1=2i_1R/(N-1),~\mathring{x}_2=2i_2R/(N-1),
    ~i_1,i_2=0,1,\cdots,N-1,\\
    &\mathring{x}_3=10i_3R/(5N-1),~i_3=0,1,\cdots,5N-1.
  \end{split}
\end{equation}
Figure \ref{fig:RMHD_3DVortex_err} plots the errors and convergence orders in $\rho$, from which one can see that
{\tt MM-O5} with the adaptive moving mesh gets fifth-order, while {\tt MM-O6} with the moving mesh \eqref{eq:RMHD_3DVortex_MovingMesh} achieves sixth-order accuracy.
Figure \ref{fig:RMHD_3DVortex_TotalEntropy} presents the time evolution of the discrete total entropy
$\sum_{\bm{i}} J_{\bm{i}} \eta(\bU_{\bm{i}})/5/N^3$
obtained by {\tt MM-O6} and {\tt MM-O5} with $N=160$,
verifying the EC and ES property of our schemes.
\end{example}


\begin{figure}[!ht]
  \centering
  \begin{subfigure}[b]{0.48\textwidth}
    \centering
    \includegraphics[width=1.0\textwidth]{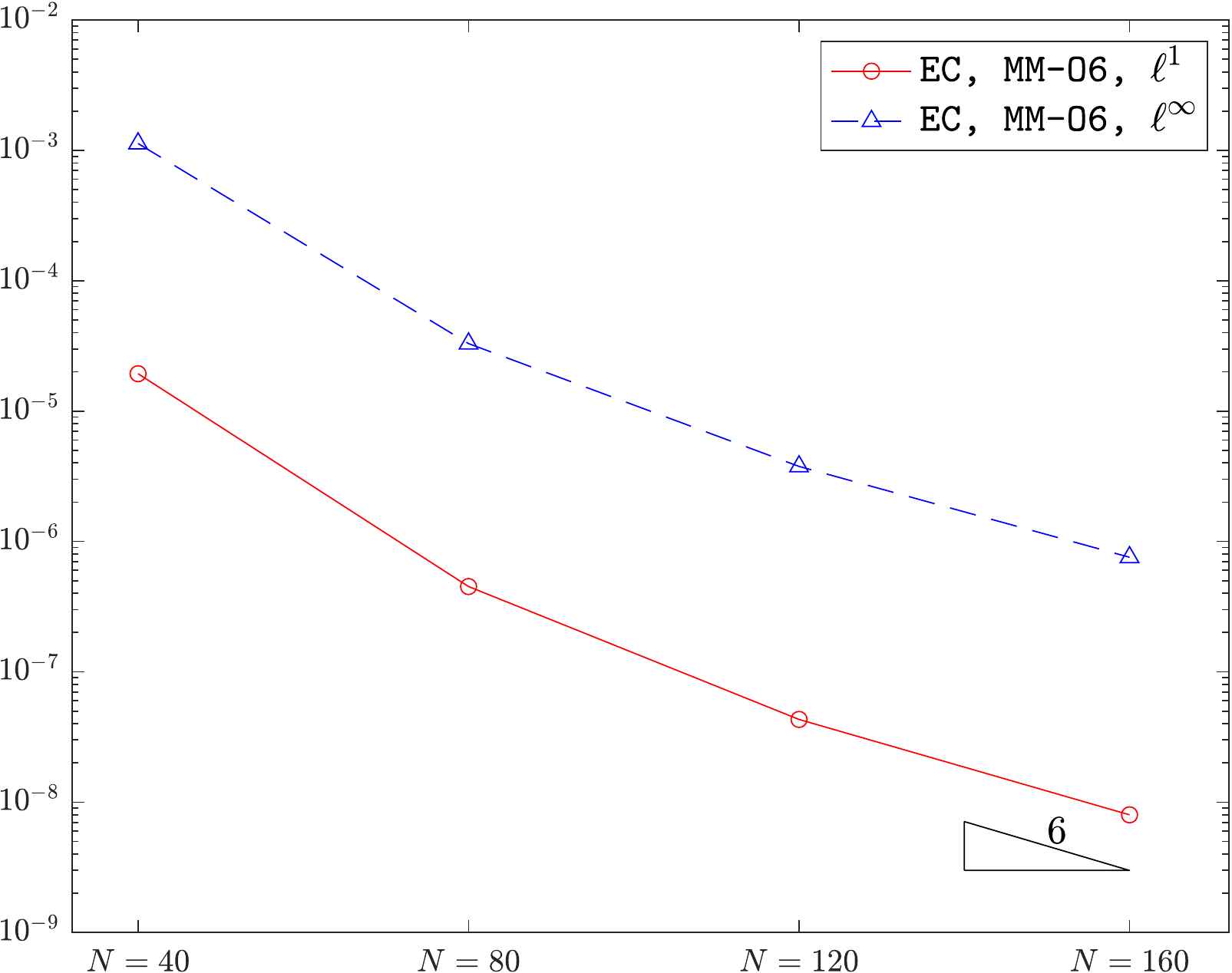}
  \end{subfigure}
  \begin{subfigure}[b]{0.48\textwidth}
    \centering
    \includegraphics[width=1.0\textwidth]{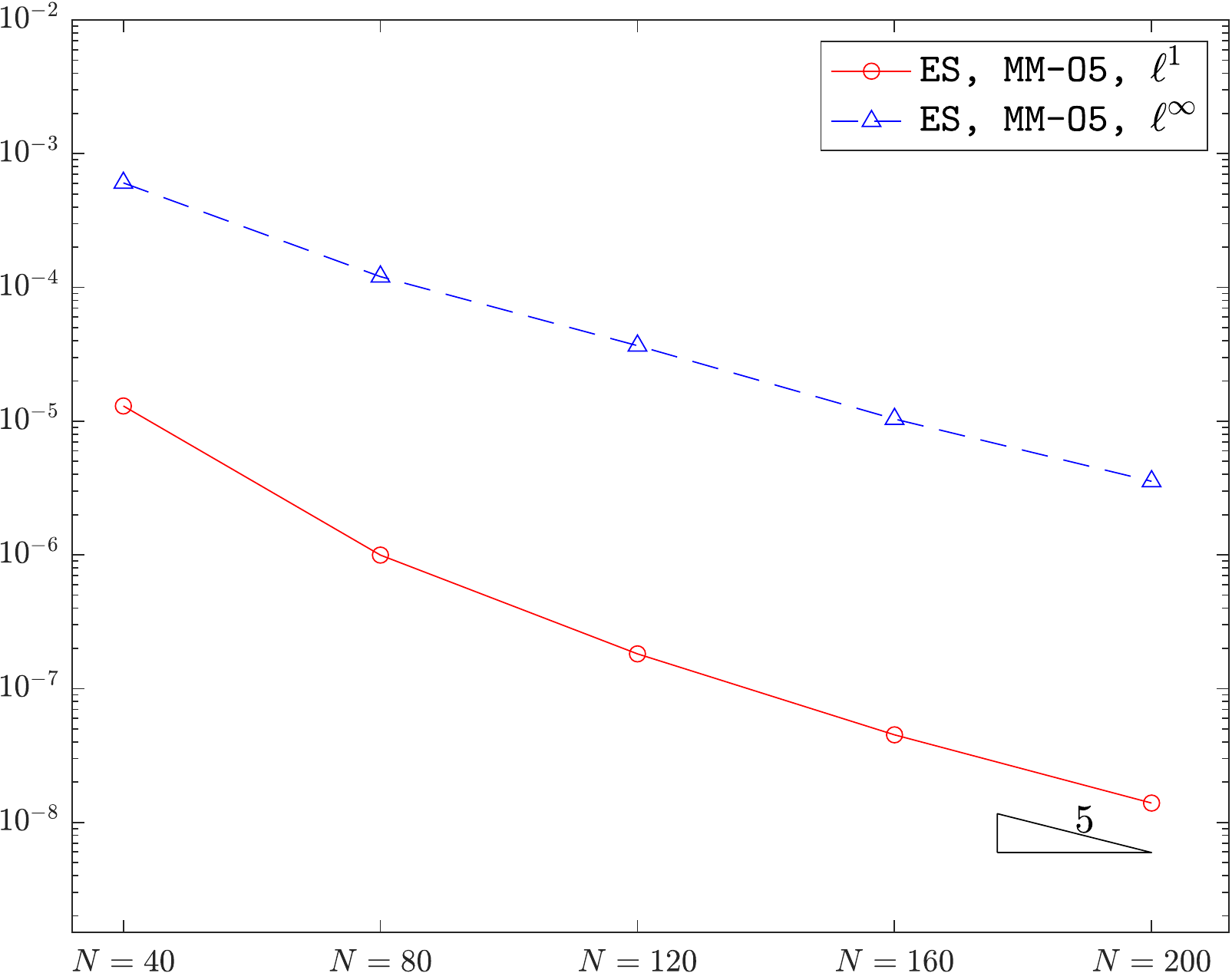}
  \end{subfigure}
  \caption{Example \ref{ex:RMHD_3DVortex}. Errors and convergence orders in $\rho$ at $t=0.1$.}
  \label{fig:RMHD_3DVortex_err}
\end{figure}

\begin{figure}[!ht]
  \centering
  \includegraphics[width=0.5\textwidth]{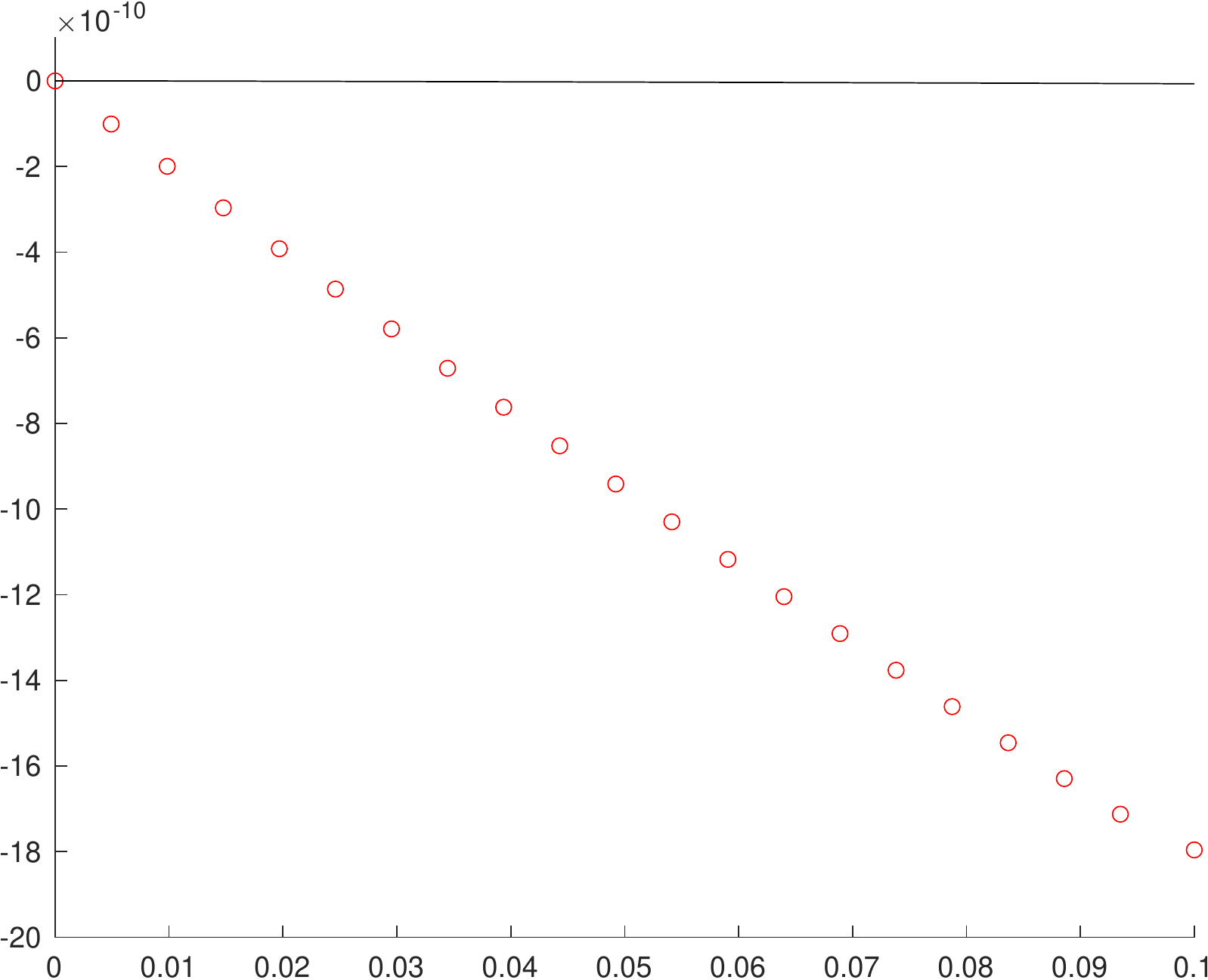}
  \caption{Example \ref{ex:RMHD_3DVortex}. Discrete total entropy obtained by EC and ES schemes with $N=160$.}
  \label{fig:RMHD_3DVortex_TotalEntropy}
\end{figure}

\begin{example}[3D RHD spherical symmetric Riemann problem]\label{ex:3DSymmRP}\rm
  This problem has a reference solution so that it is suitable to serve as the first example to verify our 3D high-order accurate ES adaptive moving mesh schemes.
  The reference solution is obtained by using a second-order TVD scheme to solve the RHD equations in the 1D spherical coordinates.
  The initial data are
  \begin{equation*}
    (\rho,\bv,p)=\begin{cases}
      (10,~0,~0,~0,40/3),     & ~ r = {\sqrt{x_1^2+x_2^2+x_3^2}} < 0.5, \\
      (1, ~0,~0,~0, 10^{-2}), & ~ \text{otherwise},
    \end{cases}
  \end{equation*}
  and $N\times N\times N$ meshes are used.
\end{example}
The monitor function is chosen as \eqref{eq:monitor} with $\alpha=800$ and $\sigma=\ln\rho$.
Figure \ref{fig:RHD_3DSymmRP} gives the $100\times 100\times 100$ adaptive mesh obtained by {\tt MM-O5},
and the comparison of $\rho$ along the volume diagonal connecting $(0,0,0)$ and $(1,1,1)$ at $t=0.4$.
Table \ref{tab:3D_CPU} lists the CPU times of different cases.
It is obvious that all the schemes give correct solutions, and the mesh points adaptively concentrate near where the large gradient in $\ln\rho$ occurs, increasing the discontinuity resolution.
{\tt MM-O5} gives better results than {\tt MM-O2} near the head and tail of the rarefaction wave, indicating that the present high-order accurate scheme outperforms the second-order scheme.
The results of {\tt MM-O5} with $N=100$ and {\tt UM-O5} with $N=200$ are comparable, while the former costs $13.8\%$ CPU time, verifying the efficiency of our high-order accurate ES adaptive moving mesh scheme.

\begin{figure}[!ht]
  \centering
  \begin{subfigure}[b]{0.48\textwidth}
    \centering
    \includegraphics[width=1.0\textwidth, trim=1 1 1 1, clip]{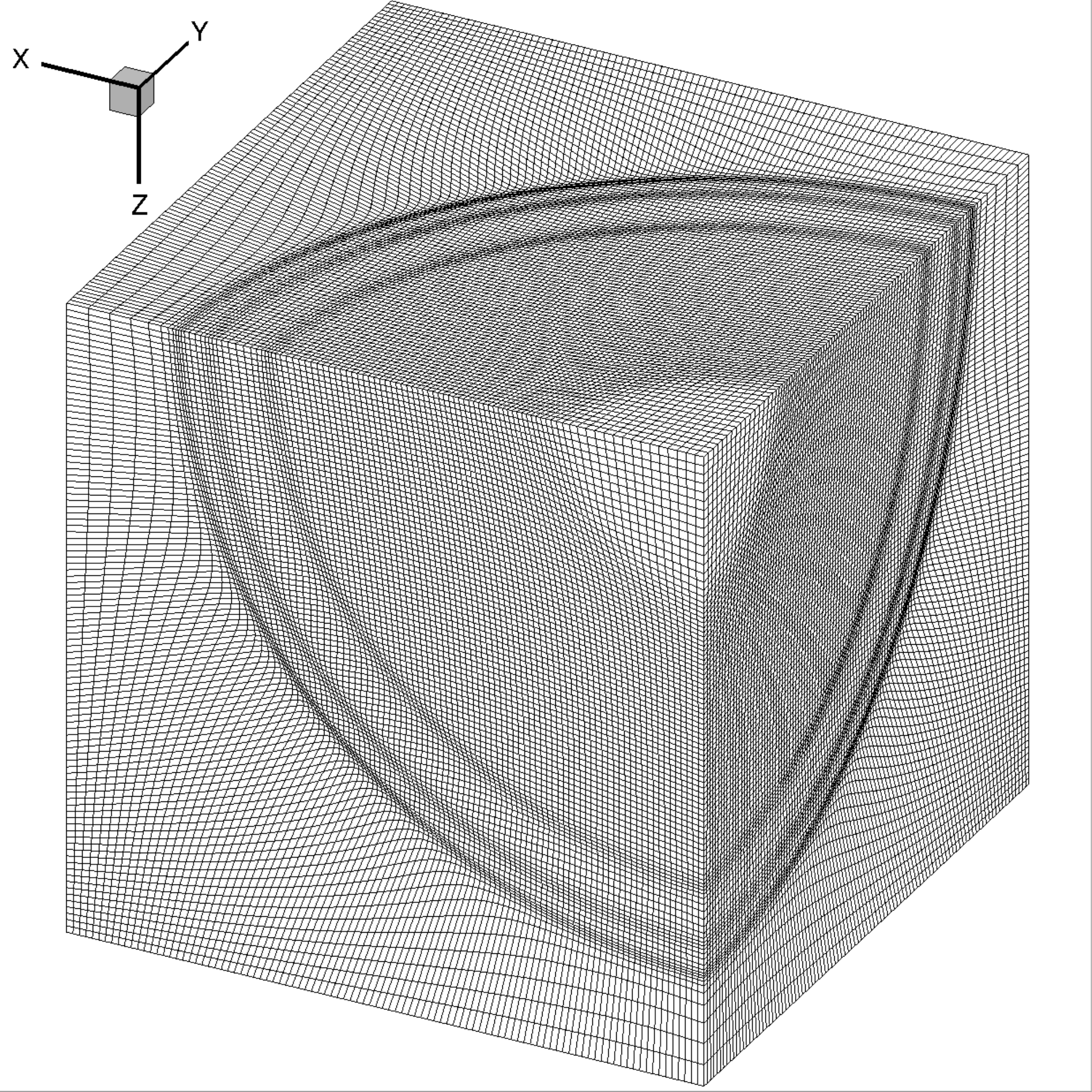}
    \caption{Adaptive mesh of {\tt  {\tt MM-O5}} with $N=100$}
  \end{subfigure}
  \begin{subfigure}[b]{0.48\textwidth}
    \centering
    \includegraphics[width=1.0\textwidth]{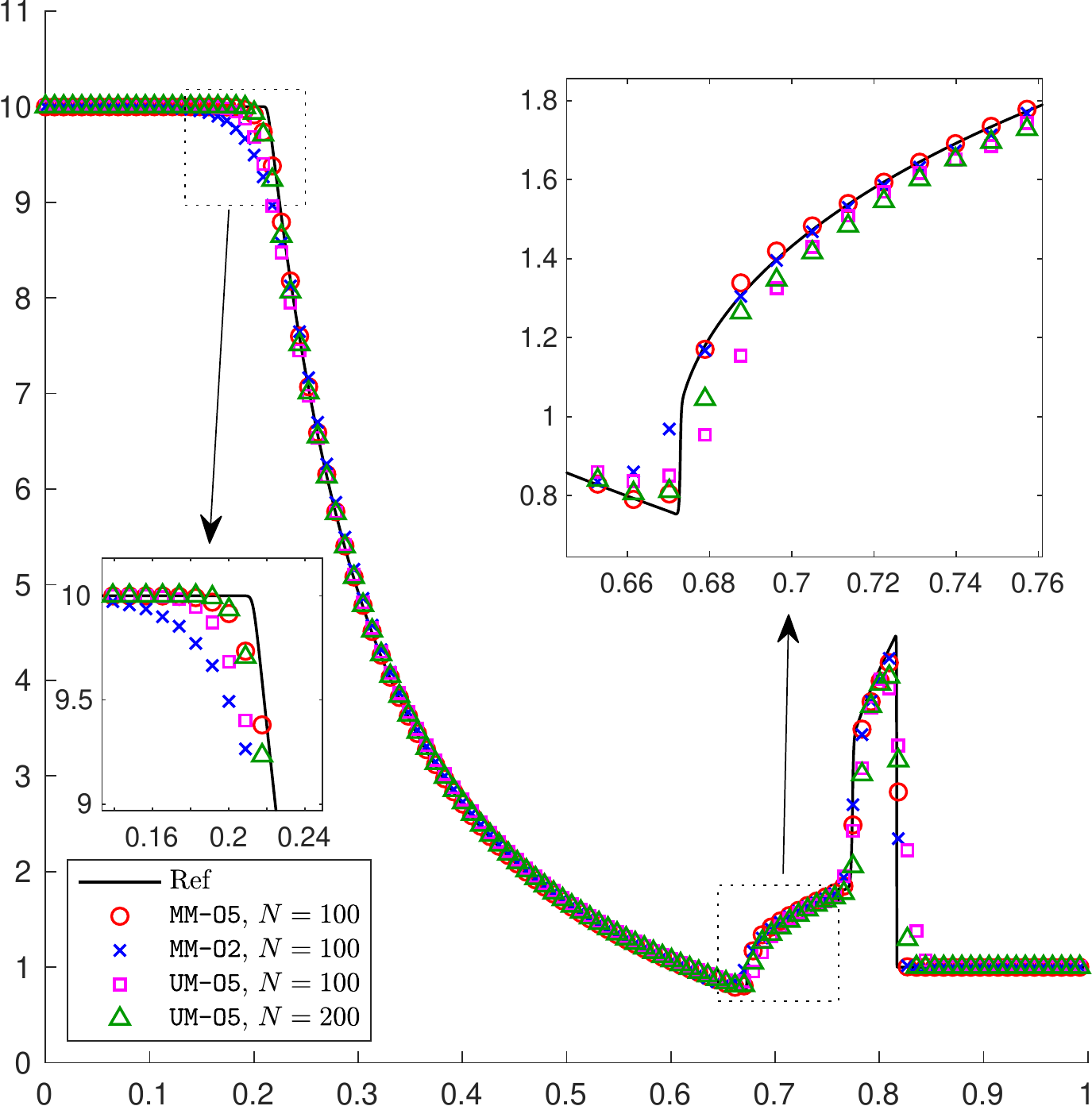}
    \caption{$\rho$ along the line connecting $(0, 0, 0)$ and $(1, 1, 1)$}
  \end{subfigure}
  \caption{Example \ref{ex:3DSymmRP}.  Adaptive mesh and cut lines of $\rho$ at $t=0.4$. }
  \label{fig:RHD_3DSymmRP}
\end{figure}

\begin{table}[!ht]
  \centering
  \begin{tabular}{c|c|c|c}
    \hline
    Scheme & Example \ref{ex:3DSymmRP} & Example \ref{ex:RHD_3DSBL} & Example \ref{ex:RMHD_3DSC} \\ \hline
    {\tt MM-O5} & 5m40s ($100\times 100\times 100$) & 2h14m44s  ($325\times 90\times 90$) & 3h9m57s ($210\times 150\times 150$) \\
    {\tt MM-O2} & 2m51s ($100\times 100\times 100$)  &  1h10m29s ($325\times 90\times 90$) & - \\
    {\tt UM-O5} & 3m08s    ($100\times 100\times 100$)  & 51m18s   ($325\times 90\times 90$) & 2h8m44s ($210\times 150\times 150$) \\
    {\tt UM-O5} & 41m08s   ($200\times 200\times 200$) & 12h34m43s  ($650\times 180\times 180$) & 34h46m49s ($420\times 300\times 300$) \\
    \hline
  \end{tabular}
  \caption{CPU times of Examples \ref{ex:3DSymmRP}-\ref{ex:RMHD_3DSC} (32 cores are used).}
  \label{tab:3D_CPU}
\end{table}

\begin{example}[3D RHD shock-bubble interaction]\label{ex:RHD_3DSBL}\rm
  This example considers a moving planar shock wave interacts with a light bubble within the physical domain
  $[0,325]\times[-45,45]\times[-45,45]$, which is extended from the 2D case \cite{He2012RHD}, and also used in \cite{Duan2021RHDMM}.
  The initial pre- and post-shock states are
  \begin{equation*}
    (\rho,\bv,p)=\begin{cases}
      (1,~0,~0,~0,~0.05),                                 & x_1<265, \\
      (1.865225080631180,-0.196781107378299,~0,~0,~0.15), & x_1>265,
    \end{cases}
  \end{equation*}
  and the state in the bubble is $$(\rho,\bv,p)=(0.1358,~0,~0,~0,~0.05),\quad
  \sqrt{(x_1-215)^2+x_2^2+x_3^2}\leqslant 25.$$
  The output times are $t=90,180,270,360,450$.
\end{example}
The monitor is the same as that in the last example.
Figure \ref{fig:RHD_3DSBL_mesh} presents the iso-surfaces of $\rho=0.7$,
the close-up of the adaptive mesh and two surface meshes near the bubble at $t=450$.
One can see that the mesh points concentrate near the shock wave  and the bubble according to the choice of the monitor function, which helps to obtain the sharp interfaces.
Figure \ref{fig:RHD_3DSBL_rho} gives the adaptive meshes and numerical schlieren images generated by $\phi=\exp(-10\abs{\nabla\rho}/\abs{\nabla\rho}_{\text{max}})$ on the slice $x_2=0$ at $t=90,180,270,360,450$ (from top to bottom).
The results obtained by {\tt MM-O5} with $325\times90\times90$ meshes are shown in the upper half parts in each row, while the adaptive meshes and numerical schlieren images obtained by {\tt MM-O2} with $325\times90\times90$ meshes are shown in the left and middle lower half parts in each row, respectively, and those obtained by {\tt UM-O5} with $650\times180\times180$ meshes are shown in the right lower half parts.
Those plots clearly show the dynamics of the interaction between the shock wave and the bubble, and our high-order accurate ES adaptive moving mesh schemes well capture the sharp interfaces of the bubble at different output times.
One can see that as time increases, the fifth-order scheme gives sharper interfaces than the second-order scheme, since the high-order accurate scheme has lower dissipation.
From the CPU times listed in Table \ref{tab:3D_CPU}, {\tt MM-O5} is more efficient than {\tt UM-O5}, because it takes only $17.8\%$ CPU time to give comparable results.

\begin{figure}[!ht]
  \centering
  \begin{subfigure}[b]{0.48\textwidth}
    \centering
    \includegraphics[width=1.0\textwidth, trim=1 1 1 1, clip]{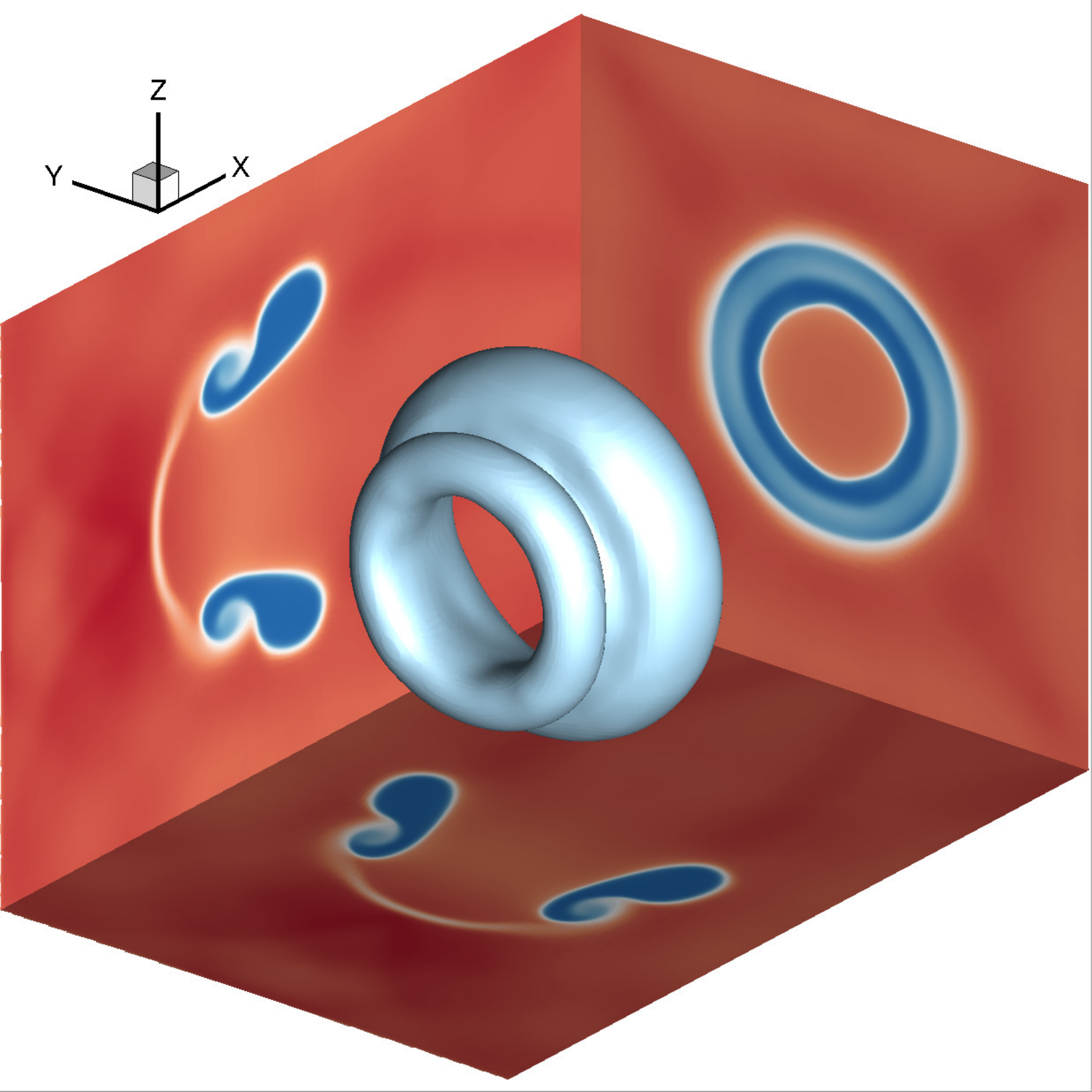}
    \caption{Iso-surface of $\rho=0.7$ and three offset 2D slices taken at $x_1=125,x_2=0,x_3=0$}
  \end{subfigure}
  \begin{subfigure}[b]{0.48\textwidth}
    \centering
    \includegraphics[width=1.0\textwidth, trim=1 1 1 1, clip]{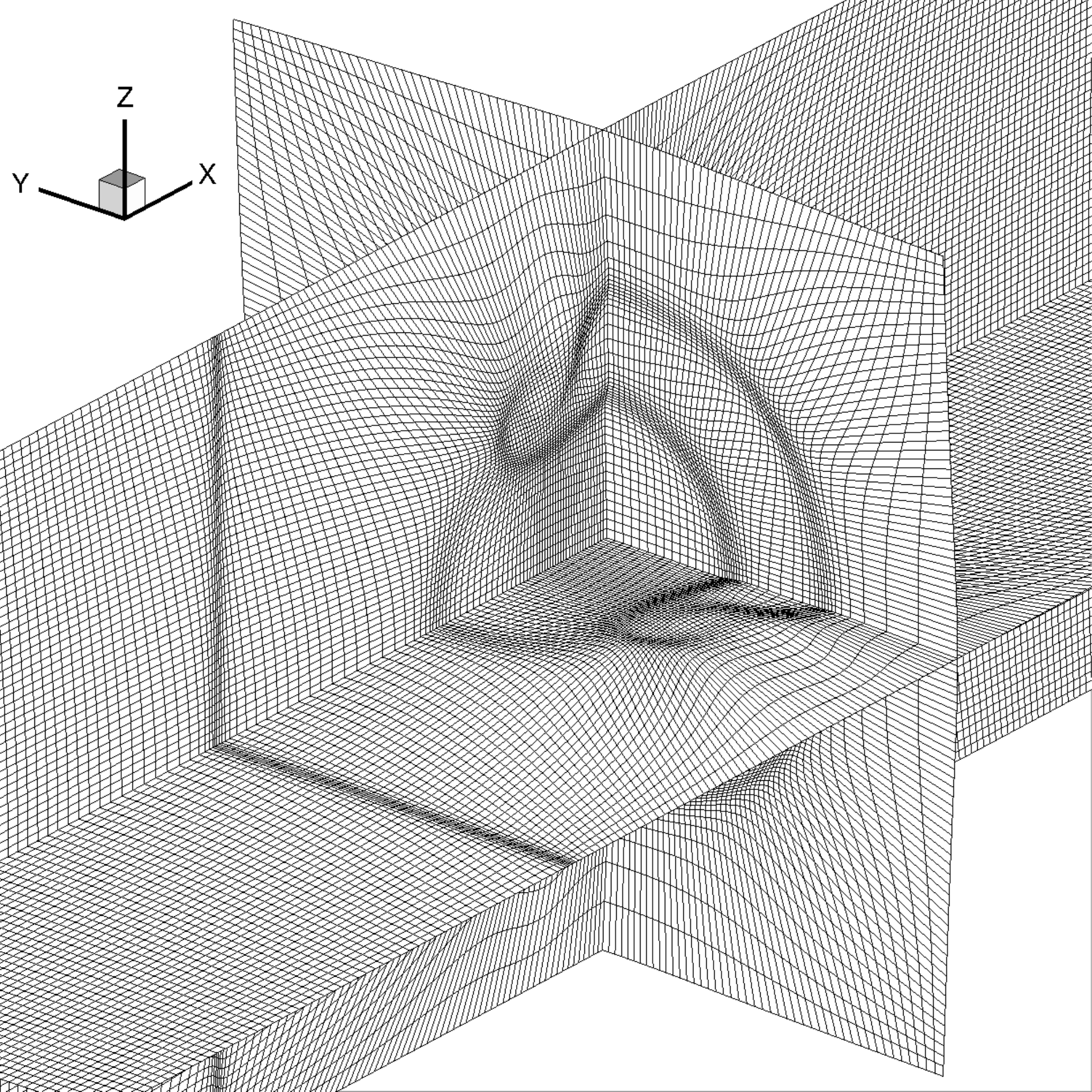}
    \caption{Adaptive meshes on three surfaces of $i_1=125,i_2=45,i_3=45$}
  \end{subfigure}

  \begin{subfigure}[b]{0.48\textwidth}
    \centering
    \includegraphics[width=1.0\textwidth]{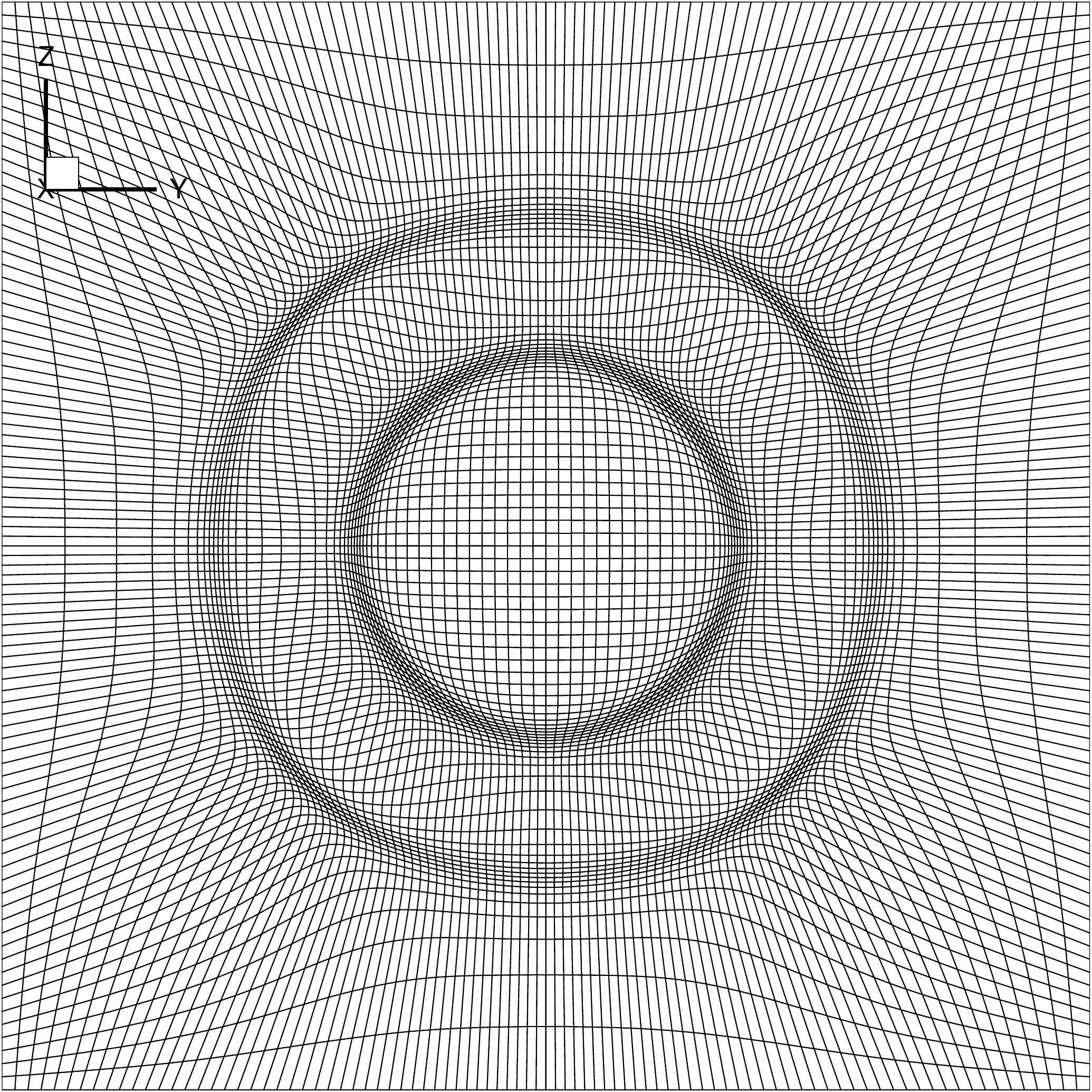}
    \caption{Close-up of adaptive mesh on  surface of $i_1=125$}
  \end{subfigure}
  \begin{subfigure}[b]{0.48\textwidth}
    \centering
    \includegraphics[width=1.0\textwidth]{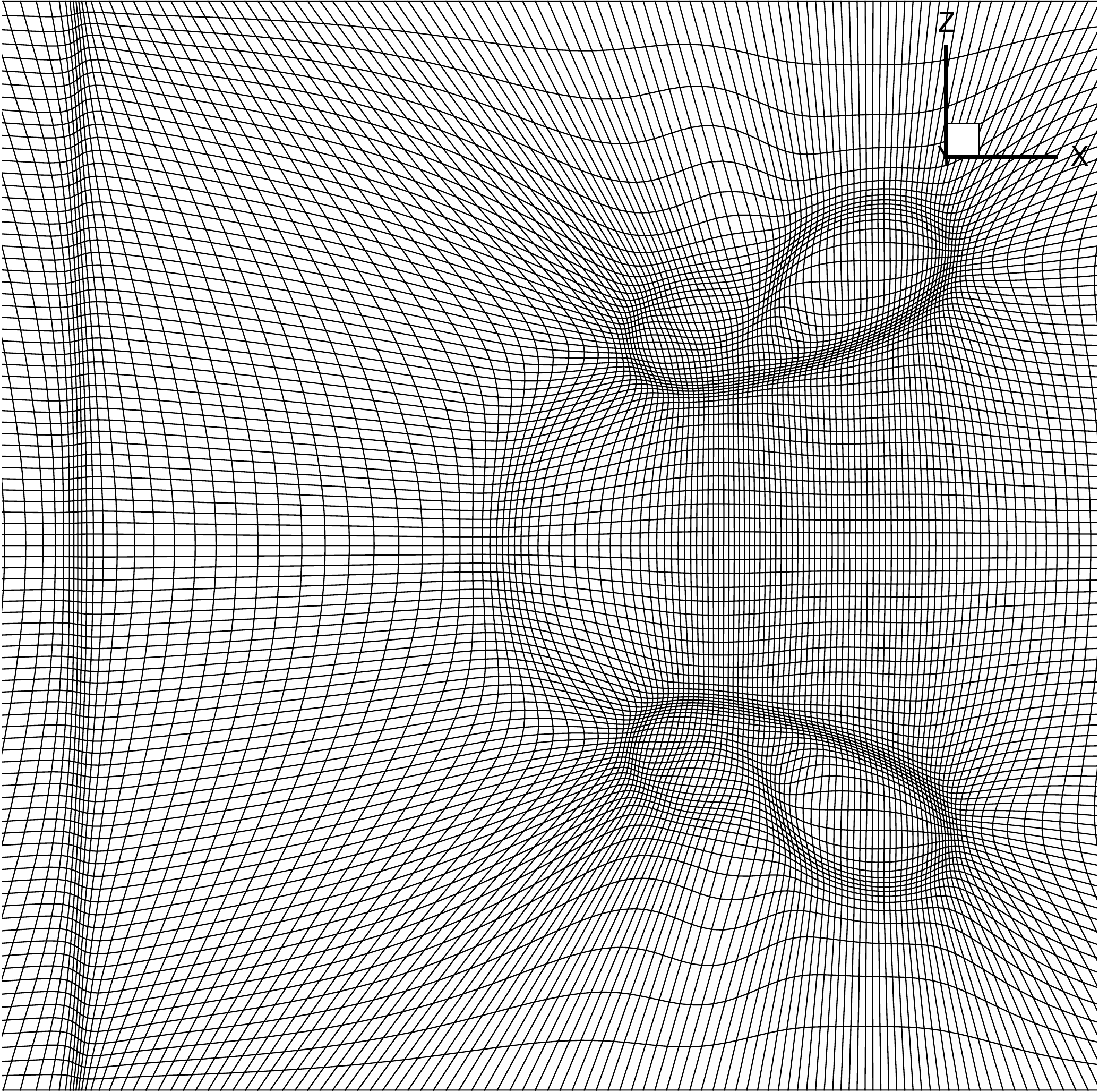}
    \caption{Close-up of  adaptive mesh on  surface of $i_2=45$}
  \end{subfigure}
  \caption{Example \ref{ex:RHD_3DSBL}. Adaptive meshes and $\rho$ at  $t=450$.}
  \label{fig:RHD_3DSBL_mesh}
\end{figure}

\begin{figure}
  \centering
  \includegraphics[width=0.25\textwidth, trim=1 1 1 1, clip]{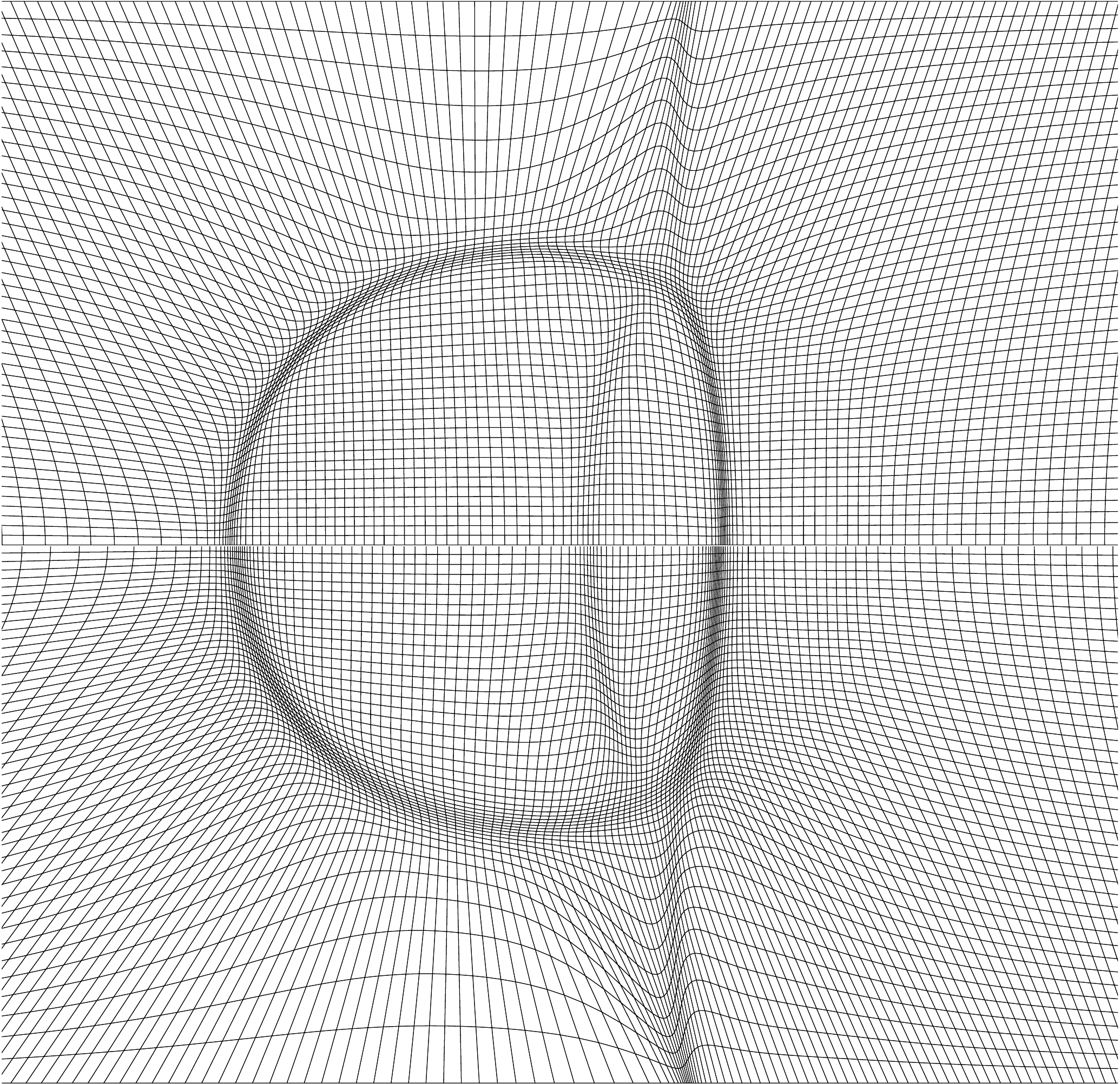}
  \includegraphics[width=0.25\textwidth, trim=1 1 1 1, clip]{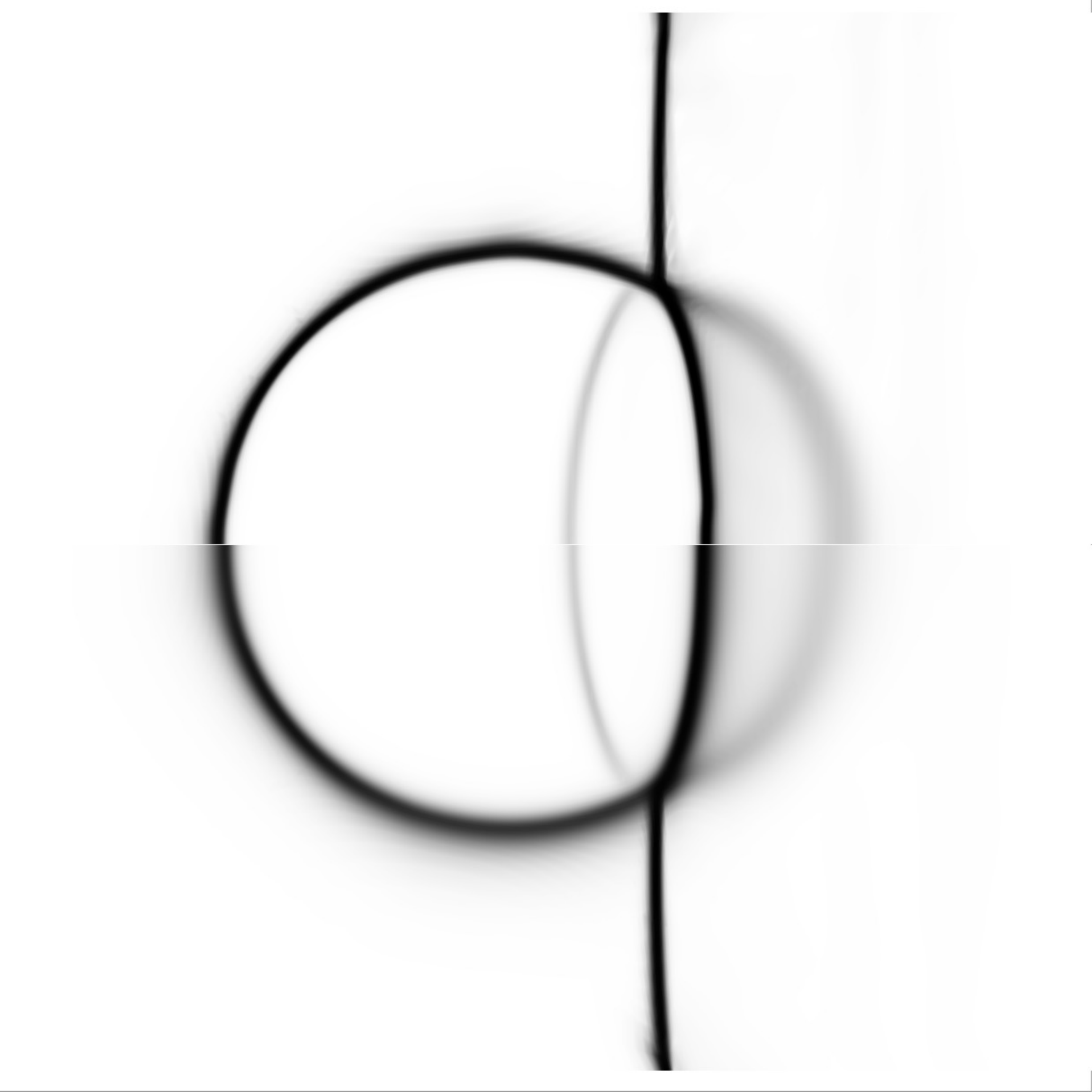}
  \includegraphics[width=0.25\textwidth, trim=1 1 1 1, clip]{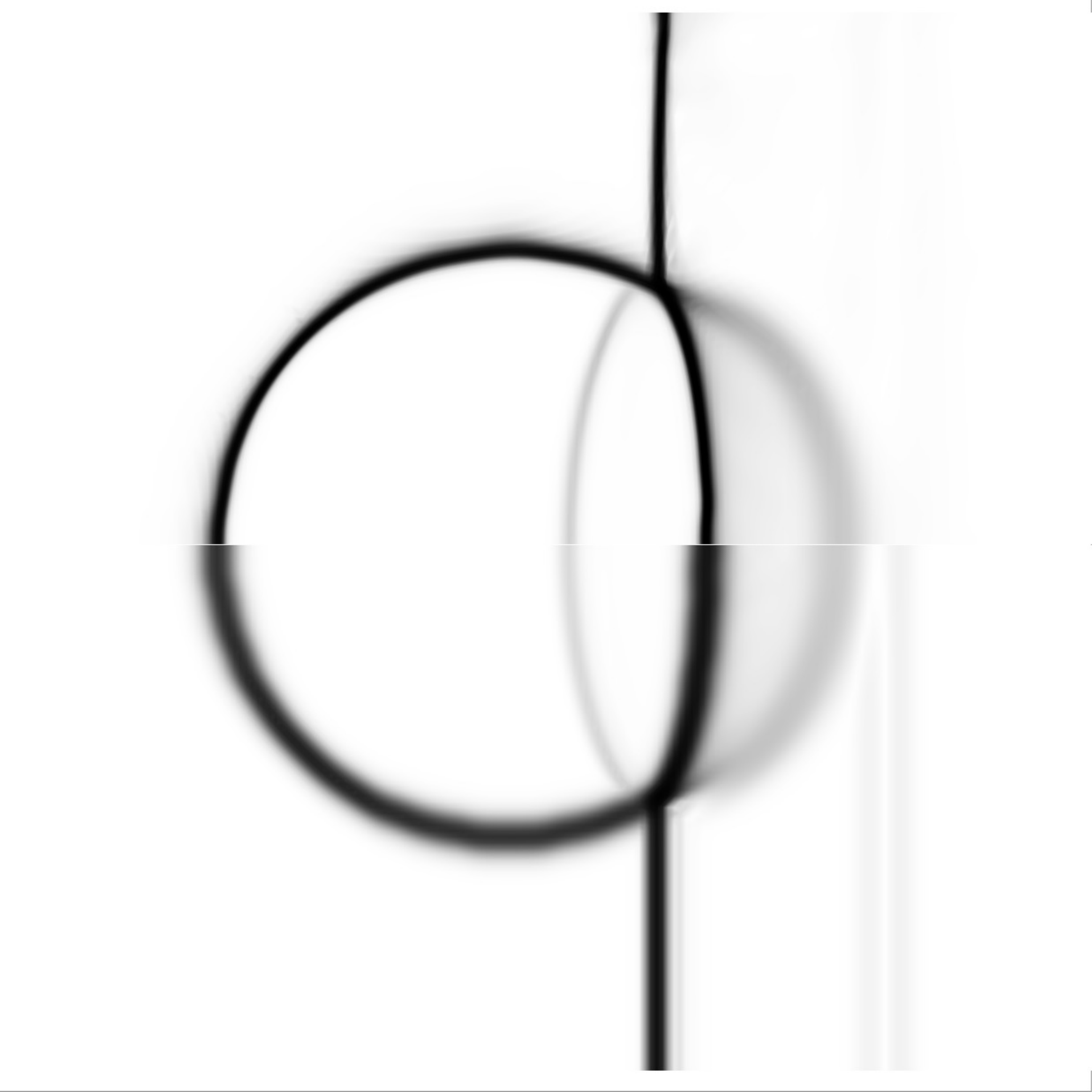}

  \includegraphics[width=0.25\textwidth, trim=1 1 1 1, clip]{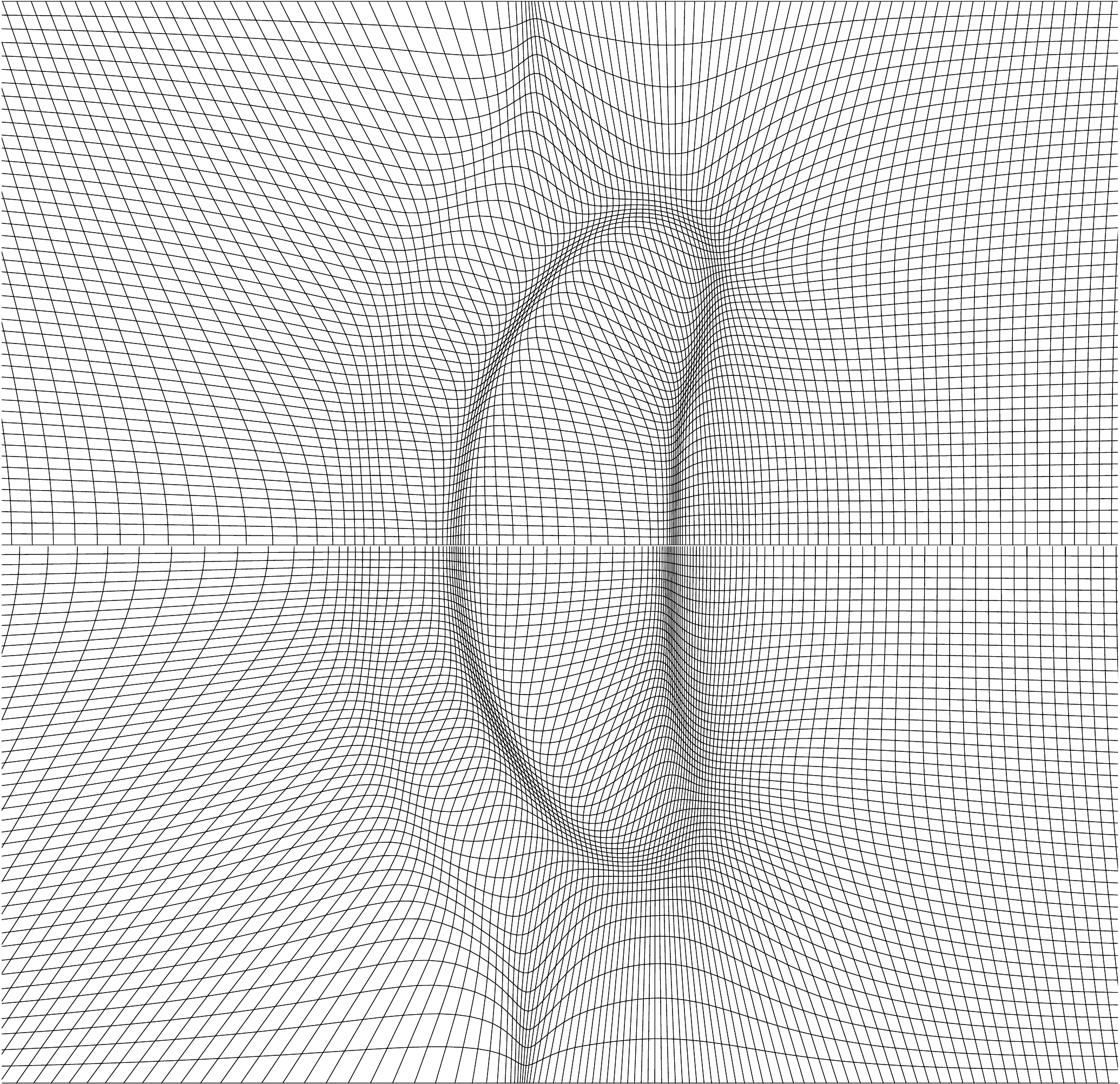}
  \includegraphics[width=0.25\textwidth, trim=1 1 1 1, clip]{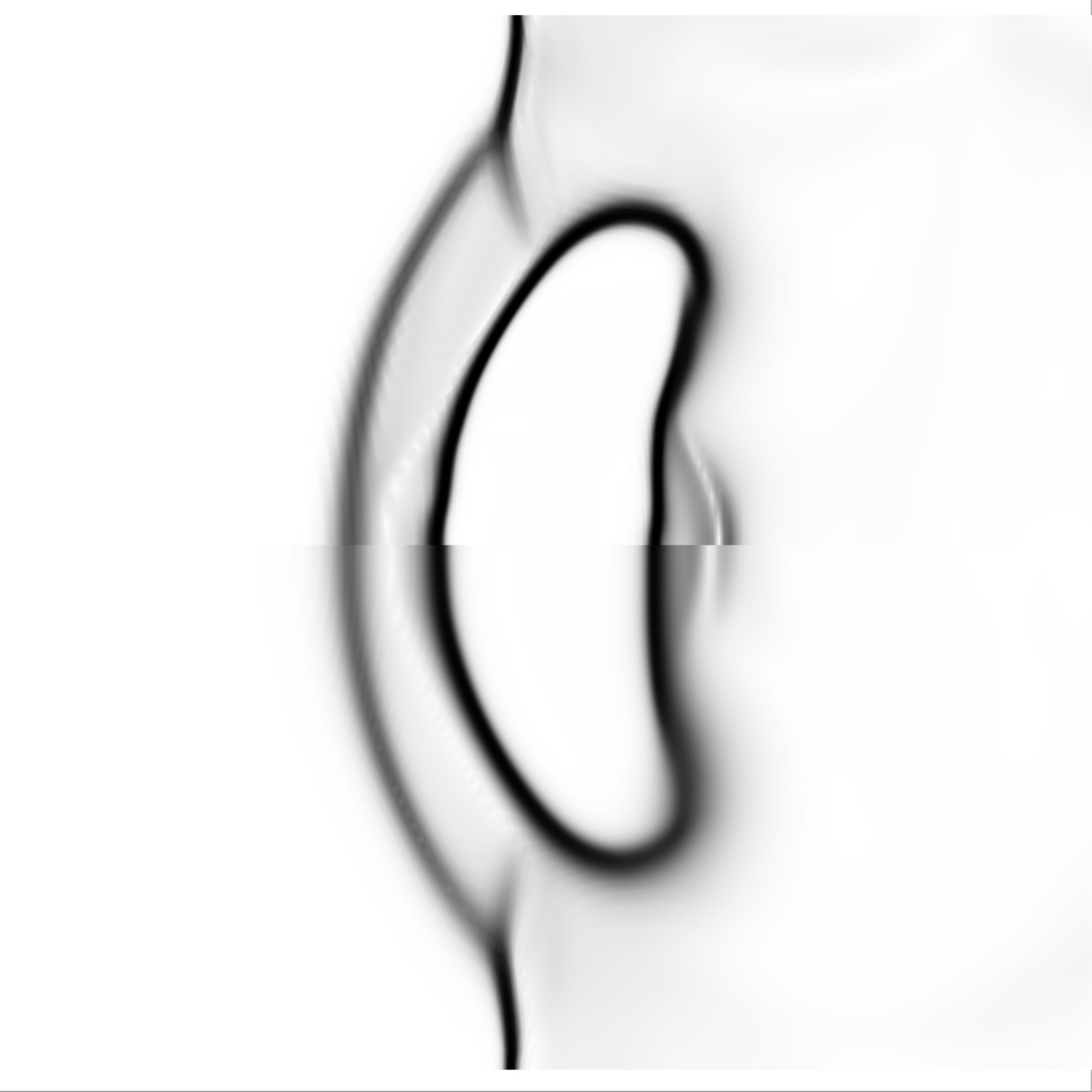}
  \includegraphics[width=0.25\textwidth, trim=1 1 1 1, clip]{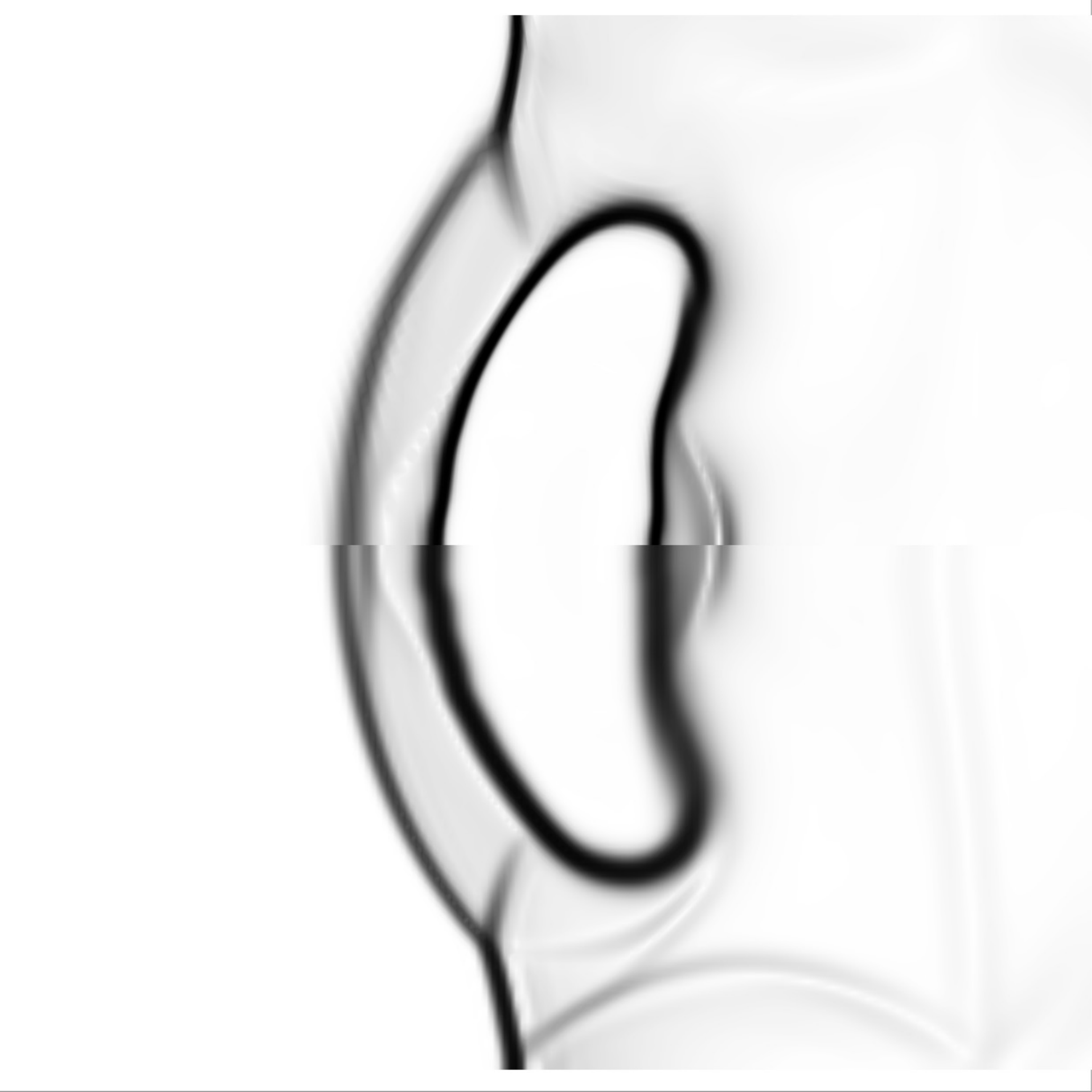}

  \includegraphics[width=0.25\textwidth, trim=1 1 1 1, clip]{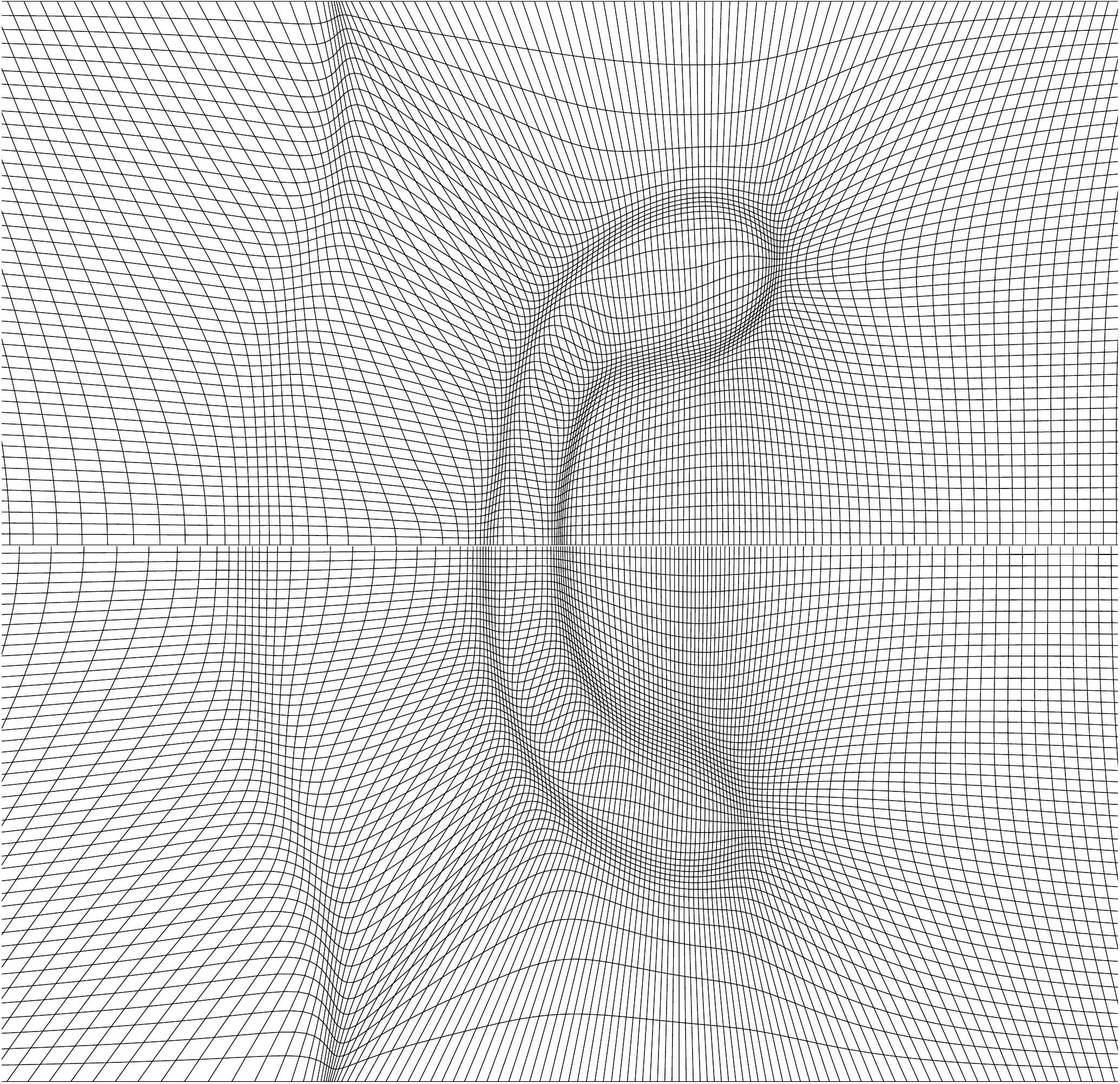}
  \includegraphics[width=0.25\textwidth, trim=1 1 1 1, clip]{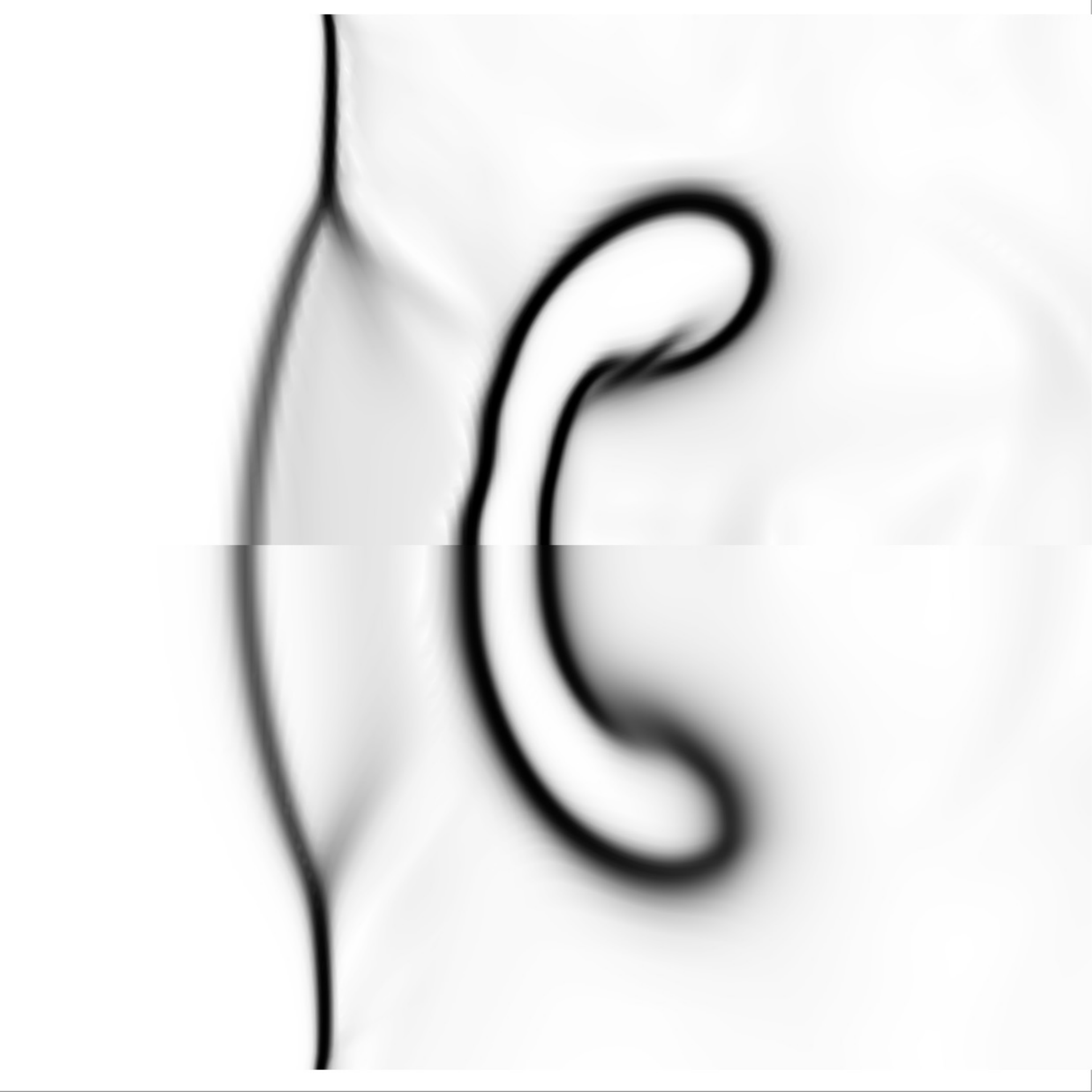}
  \includegraphics[width=0.25\textwidth, trim=1 1 1 1, clip]{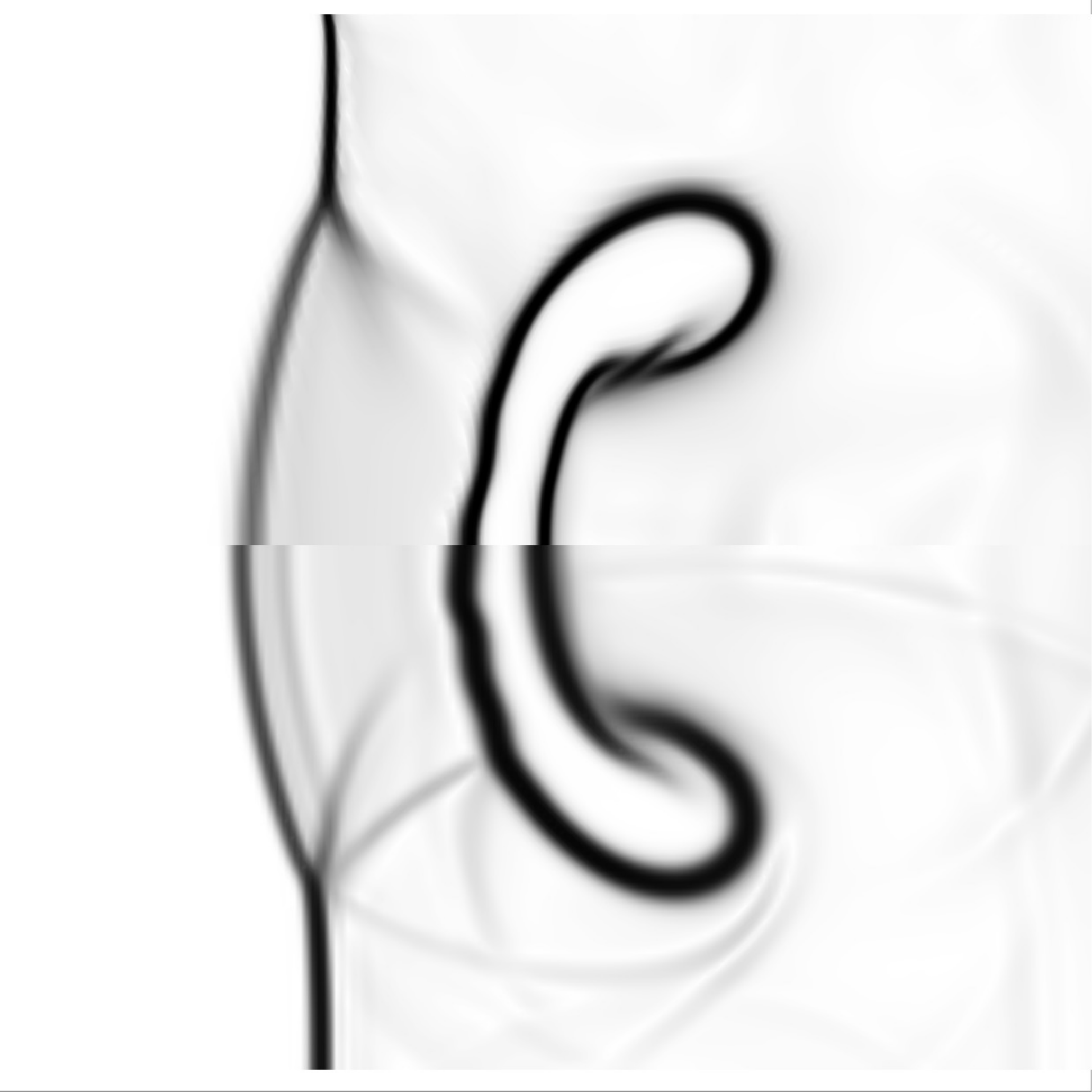}

  \includegraphics[width=0.25\textwidth, trim=1 1 1 1, clip]{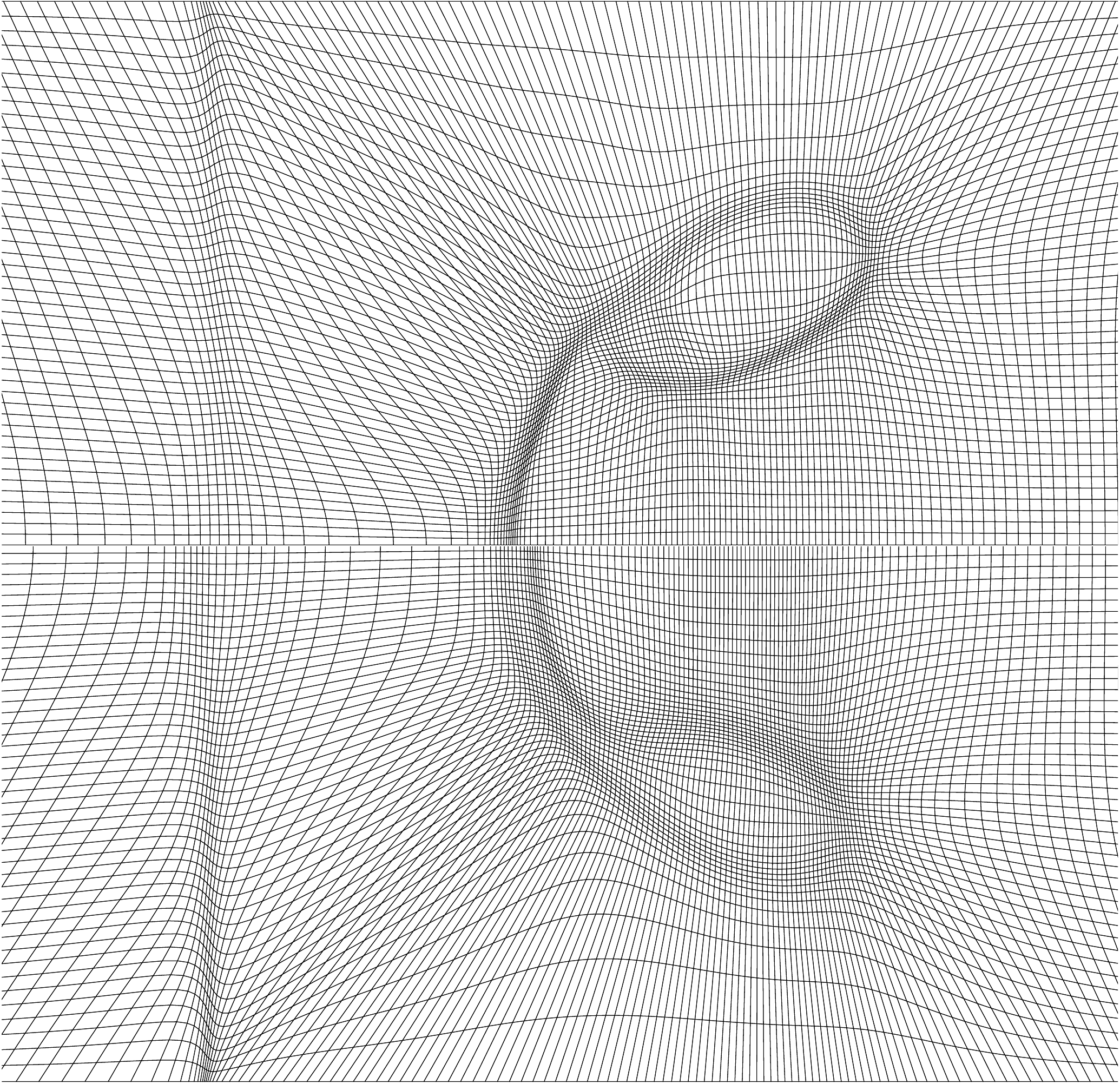}
  \includegraphics[width=0.25\textwidth, trim=1 1 1 1, clip]{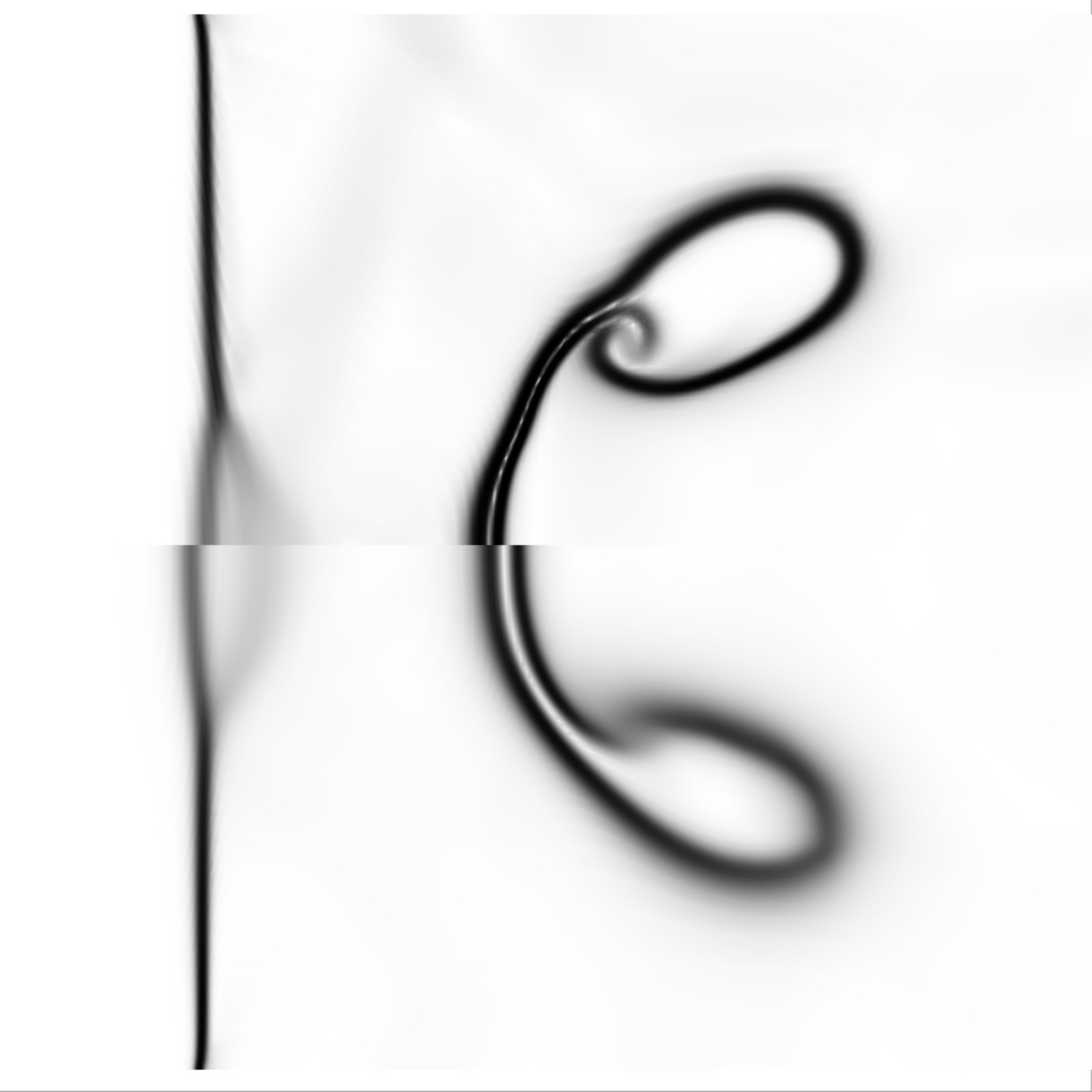}
  \includegraphics[width=0.25\textwidth, trim=1 1 1 1, clip]{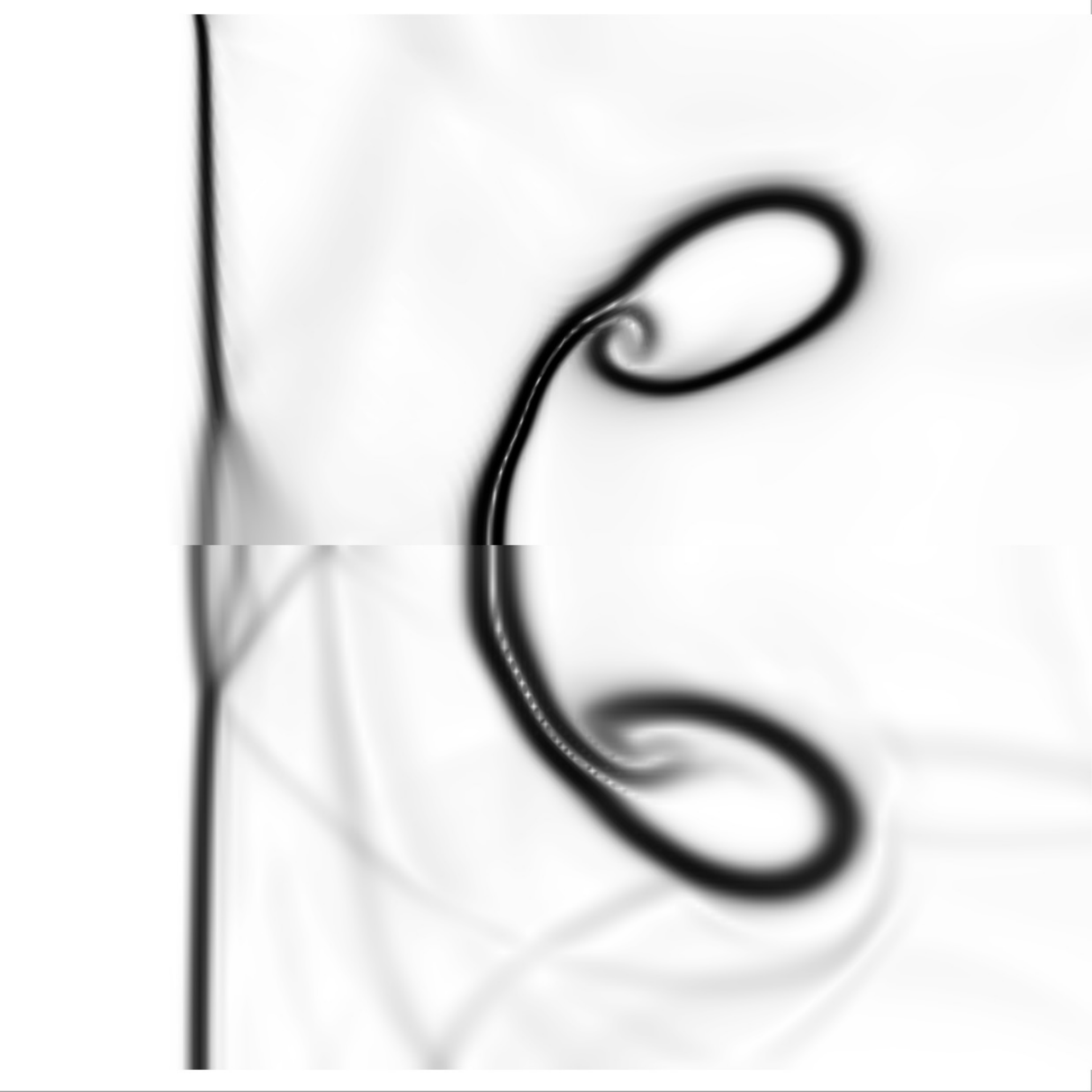}

  \includegraphics[width=0.25\textwidth, trim=1 1 1 1, clip]{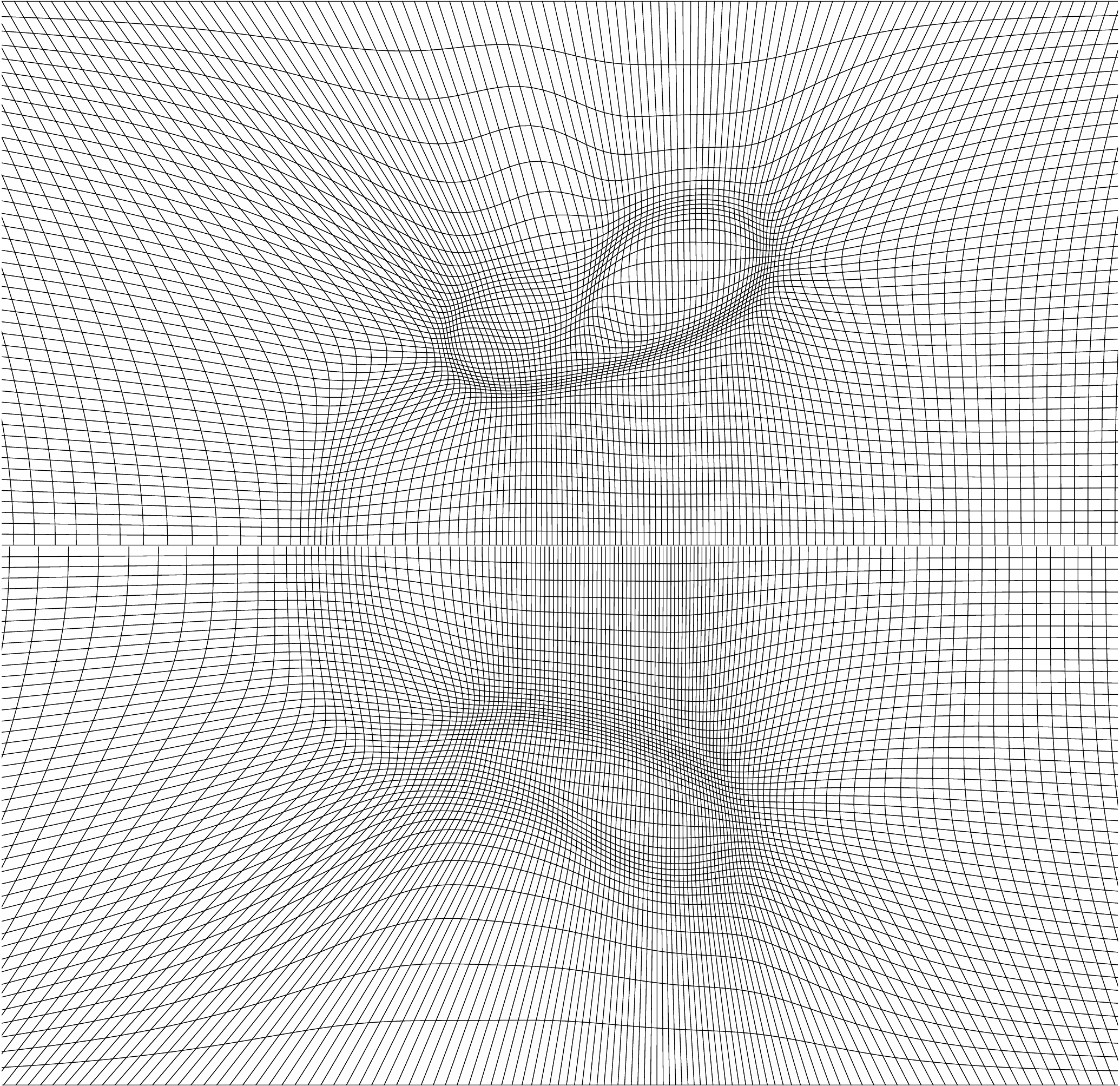}
  \includegraphics[width=0.25\textwidth, trim=1 1 1 1, clip]{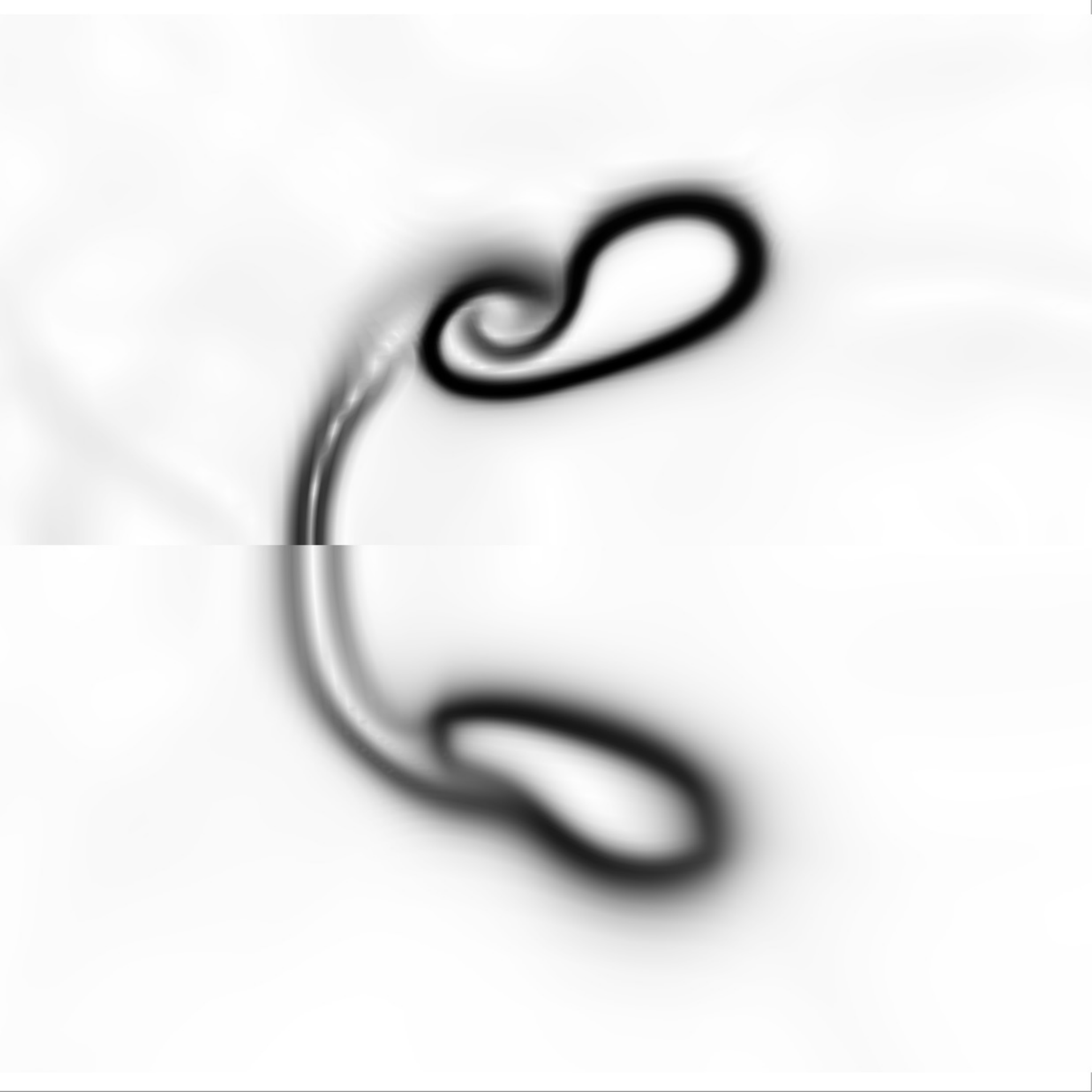}
  \includegraphics[width=0.25\textwidth, trim=1 1 1 1, clip]{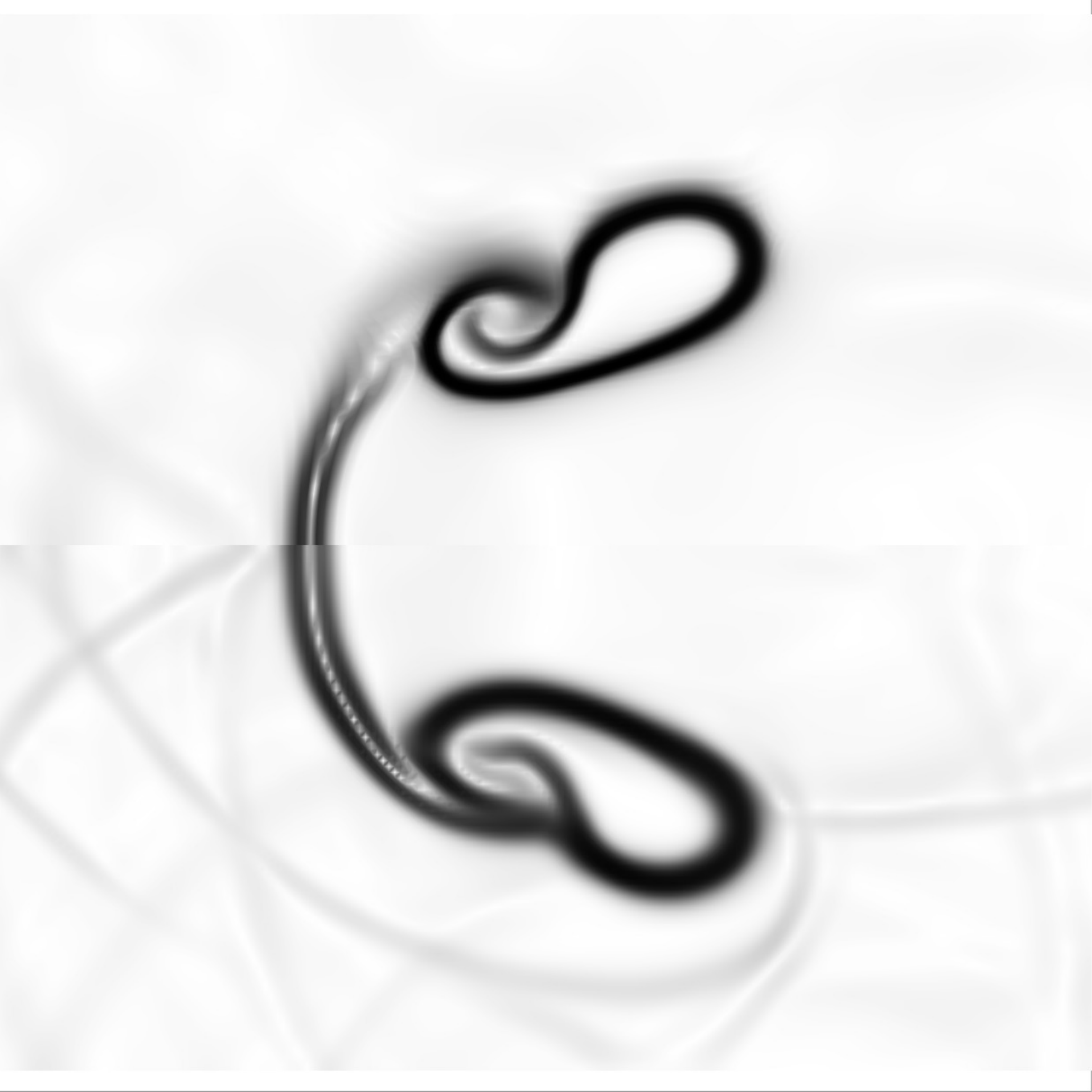}

  \caption{Example \ref{ex:RHD_3DSBL}.  From top to bottom: $t=90,180,270,360,450$.
    Left: adaptive meshes on  surface of $i_2=45$ obtained with {\tt MM-O5} (upper half) and {\tt MM-O2} (lower half) with $325\times 90\times 90$ mesh.
    Middle: numerical schlieren images of $\phi$ on  slice $x_2=0$ obtained with {\tt MM-O5} (upper half) and {\tt MM-O2} (lower half) with $325\times 90\times 90$ mesh.
    Right: numerical schlieren images of $\phi$ on the slice $x_2=0$, obtained with {\tt MM-O5} (upper half) with $325\times 90\times 90$ mesh and {\tt UM-O5} (lower half) with $650\times 180\times 180$ mesh.
  }
  \label{fig:RHD_3DSBL_rho}
\end{figure}

\begin{example}[3D RMHD shock-cloud interaction]\label{ex:RMHD_3DSC}\rm
  It is a 3D extension of Example \ref{ex:RMHD_2DSC}.
  The physical domain is $[-0.2,1.2]\times [0,1]\times [0,1]$,
  and the circular cloud is modified as a spherical cloud of radius $0.15$ centered at $(0.25, 0.5, 0.5)$ with invariant density.
  The initial data of the pre- and post-shock remain unchanged.
  This problem is solved by using the fifth-order ES adaptive moving mesh scheme until $t=1.2$.
\end{example}
The monitor is the same as the last example.
The iso-surfaces of $\rho=1.52$,
the close-up of the adaptive mesh and two surface meshes near the bubble at $t=1.2$ are given in Figure \ref{fig:RMHD_3DSC_mesh}.
The mesh points adaptively concentrate near the complicated structures formed after the interaction of the shock wave and the cloud, improving the nearby resolution.
Figures \ref{fig:RMHD_3DSC_phi1}-\ref{fig:RMHD_3DSC_phi2} show the numerical schlieren images of $\phi_1$ and $\phi_2$ defined in Example \ref{ex:RMHD_2DSC} on the slice $x_2=0$.
The results obtained by {\tt MM-O5} with $210\times150\times150$ meshes are plotted in the upper half parts, while those obtained by {\tt UM-O5} with $210\times150\times150$ and $420\times300\times300$ meshes are shown in the left and right lower half parts, respectively.
One can see that {\tt MM-O5} gives better results than {\tt UM-O5} with the same grid number,
and the former takes only $9.06\%$ CPU time to give comparable results when the latter uses finer mesh,
which again shows the high efficiency of our high-order accurate ES adaptive moving mesh schemes.

\begin{figure}[!ht]
  \centering
  \begin{subfigure}[b]{0.48\textwidth}
    \centering
    \includegraphics[width=1.0\textwidth, trim=1 1 1 1, clip]{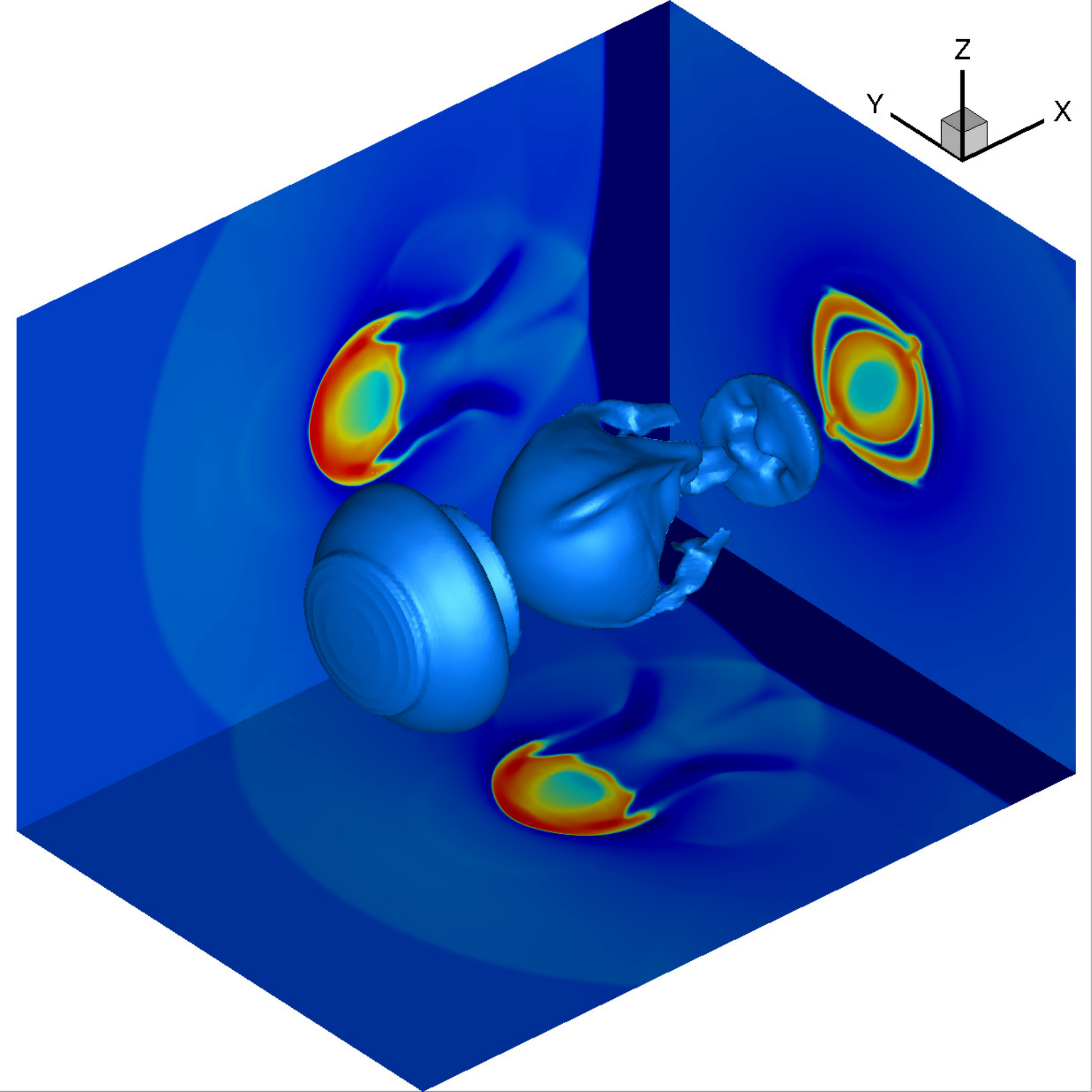}
    \caption{Iso-surface of $\ln\rho=1.52$ and three offset 2D slices taken at $x_1=0.58,x_2=0.5,x_3=0.5$}
  \end{subfigure}
  \begin{subfigure}[b]{0.48\textwidth}
    \centering
    \includegraphics[width=1.0\textwidth, trim=1 1 1 1, clip]{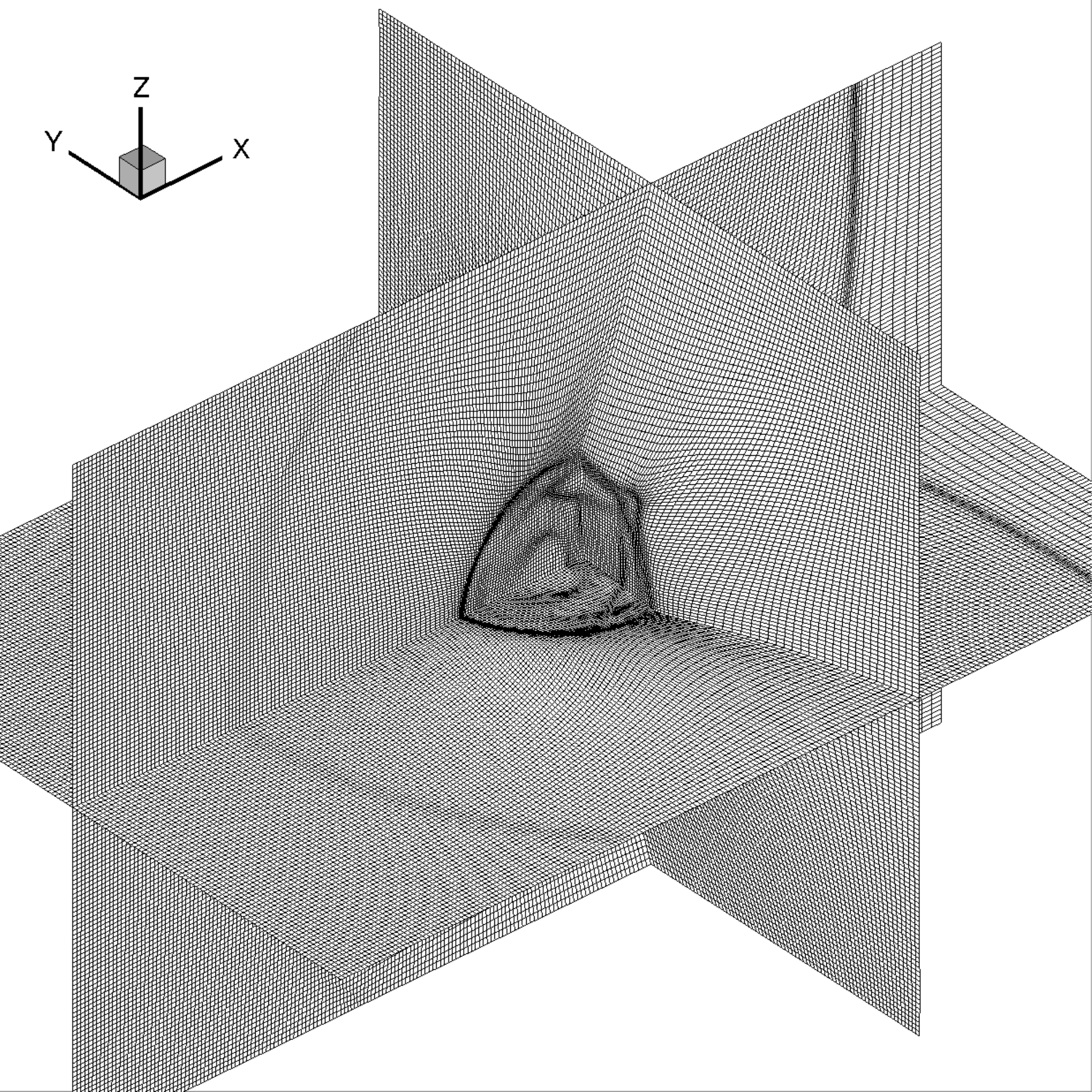}
    \caption{Adaptive meshes on three surfaces of $i_1=150,i_2=75,i_3=75$}
  \end{subfigure}

  \begin{subfigure}[b]{0.48\textwidth}
    \centering
    \includegraphics[width=1.0\textwidth]{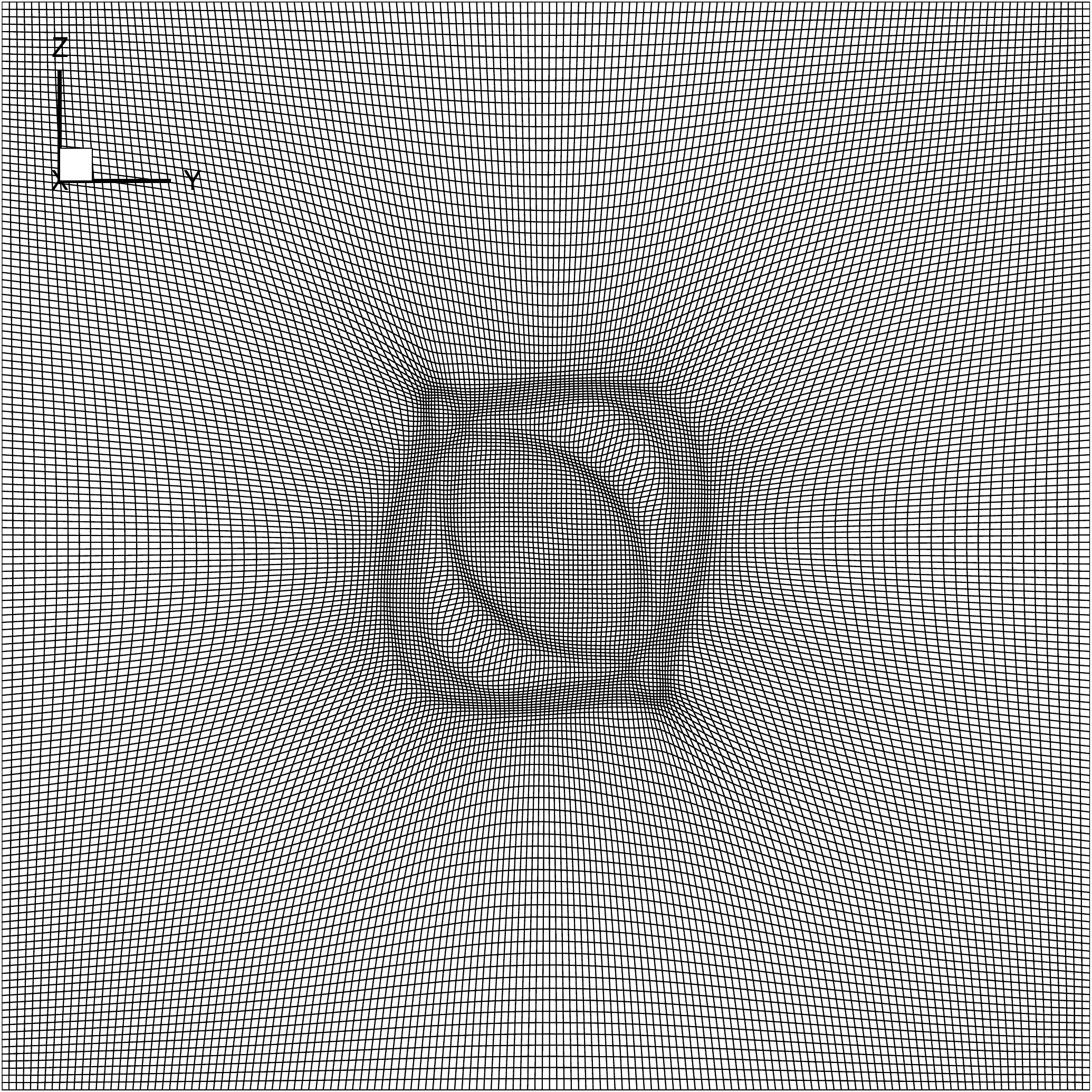}
    \caption{Close-up of  adaptive mesh on   surface of $i_1=150$}
  \end{subfigure}
  \begin{subfigure}[b]{0.48\textwidth}
    \centering
    \includegraphics[width=1.0\textwidth]{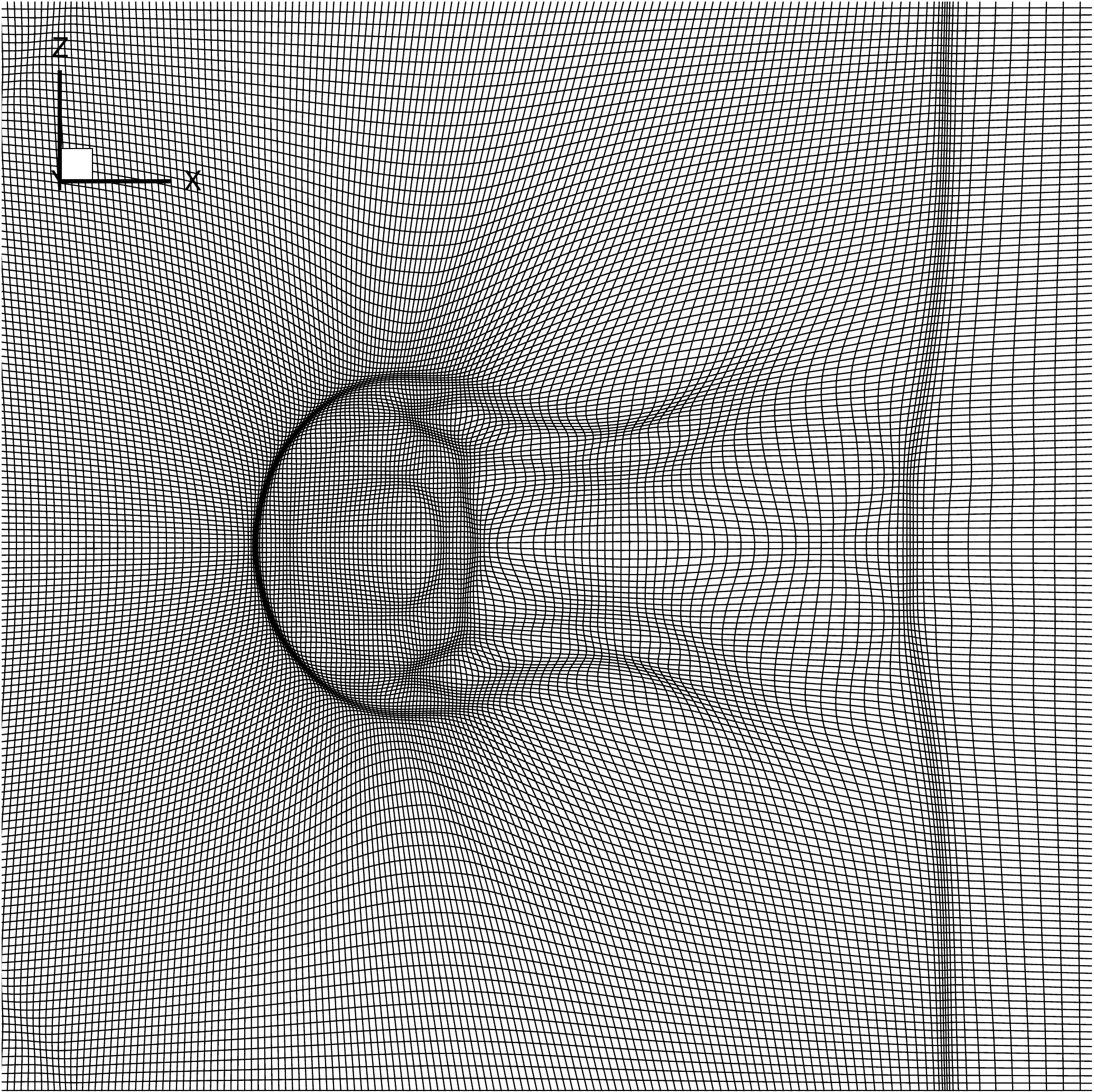}
    \caption{Close-up of  adaptive mesh on   surface of $i_2=75$}
  \end{subfigure}
  \caption{Example \ref{ex:RMHD_3DSC}. Adaptive meshes and $\ln\rho$ at $t=1.2$.}
  \label{fig:RMHD_3DSC_mesh}
\end{figure}

\begin{figure}
  \centering
  \includegraphics[width=0.48\textwidth, trim=2 2 2 2, clip]{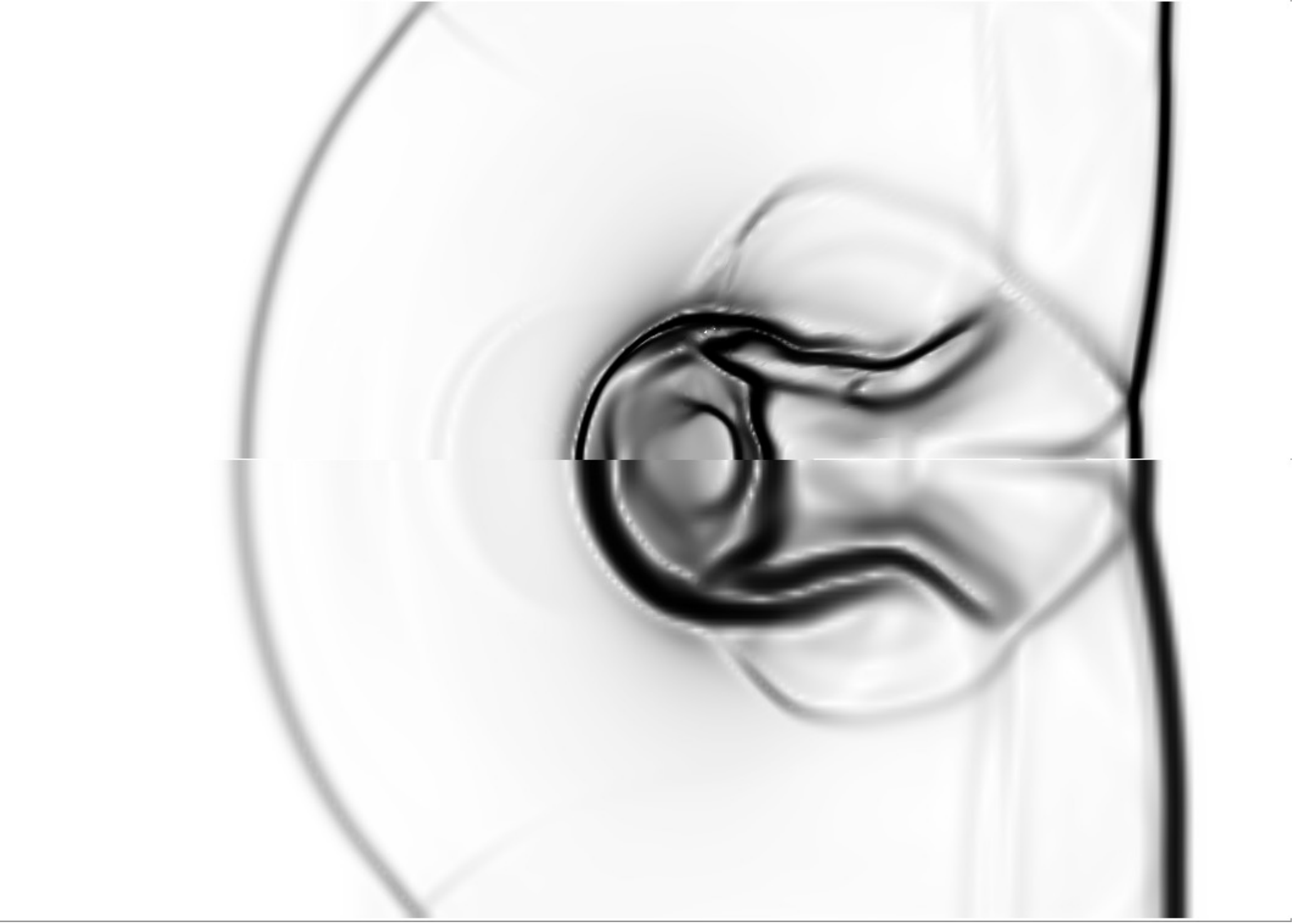}
  \includegraphics[width=0.48\textwidth, trim=2 2 2 2, clip]{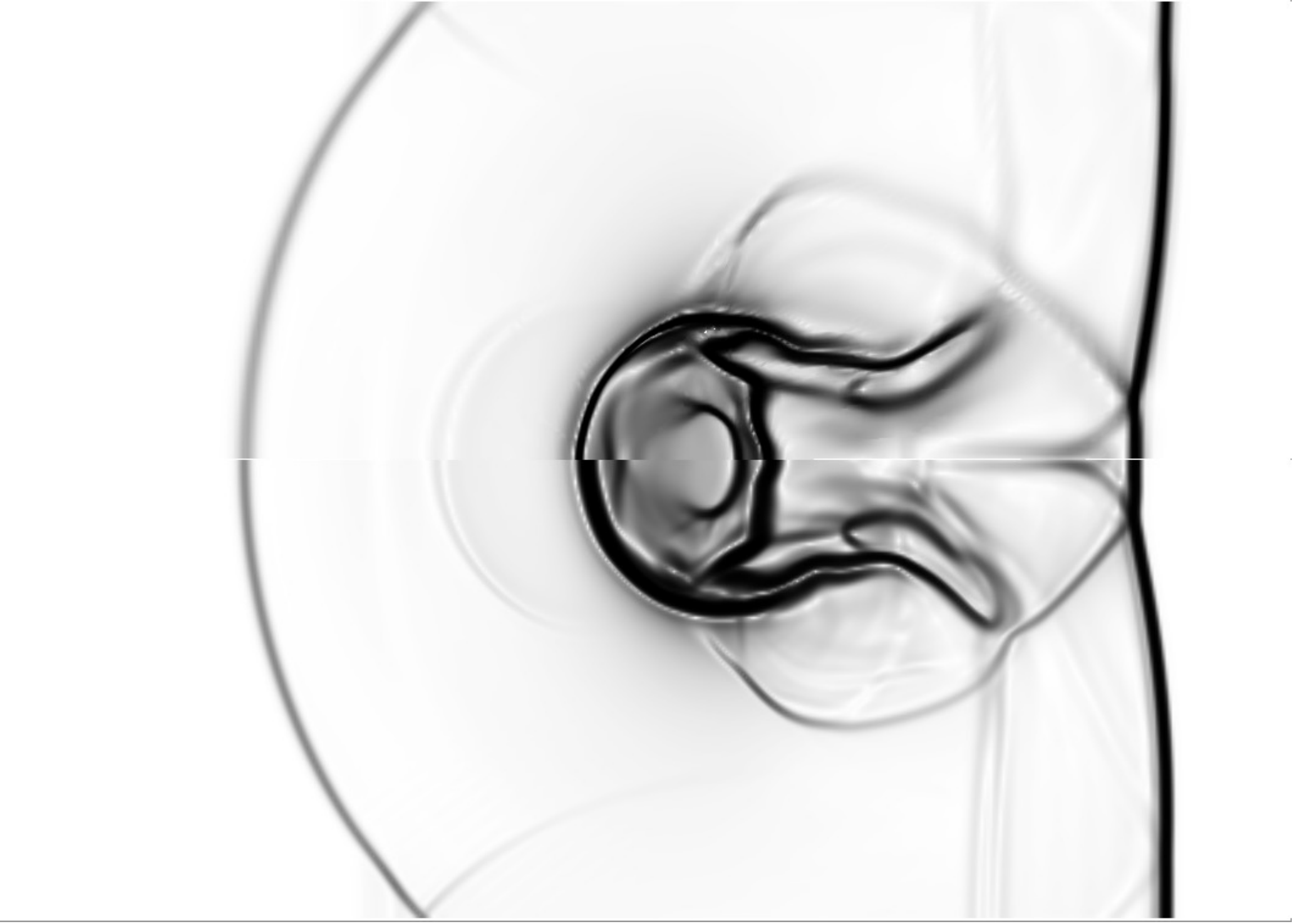}
  \caption{Example \ref{ex:RMHD_3DSC}. Numerical schlieren images of $\phi_1$ at $t=1.2$.
    Left: {\tt MM-O5} with $210\times 150\times 150$ mesh (upper half) and {\tt UM-O5} with $210\times 150\times 150$ mesh (lower half).
    Right: {\tt MM-O5} with $210\times 150\times 150$ mesh (upper half) and {\tt UM-O5} with $420\times 300\times 300$ mesh (lower half).
  }
  \label{fig:RMHD_3DSC_phi1}
\end{figure}

\begin{figure}
  \centering
  \includegraphics[width=0.48\textwidth, trim=2 2 2 2, clip]{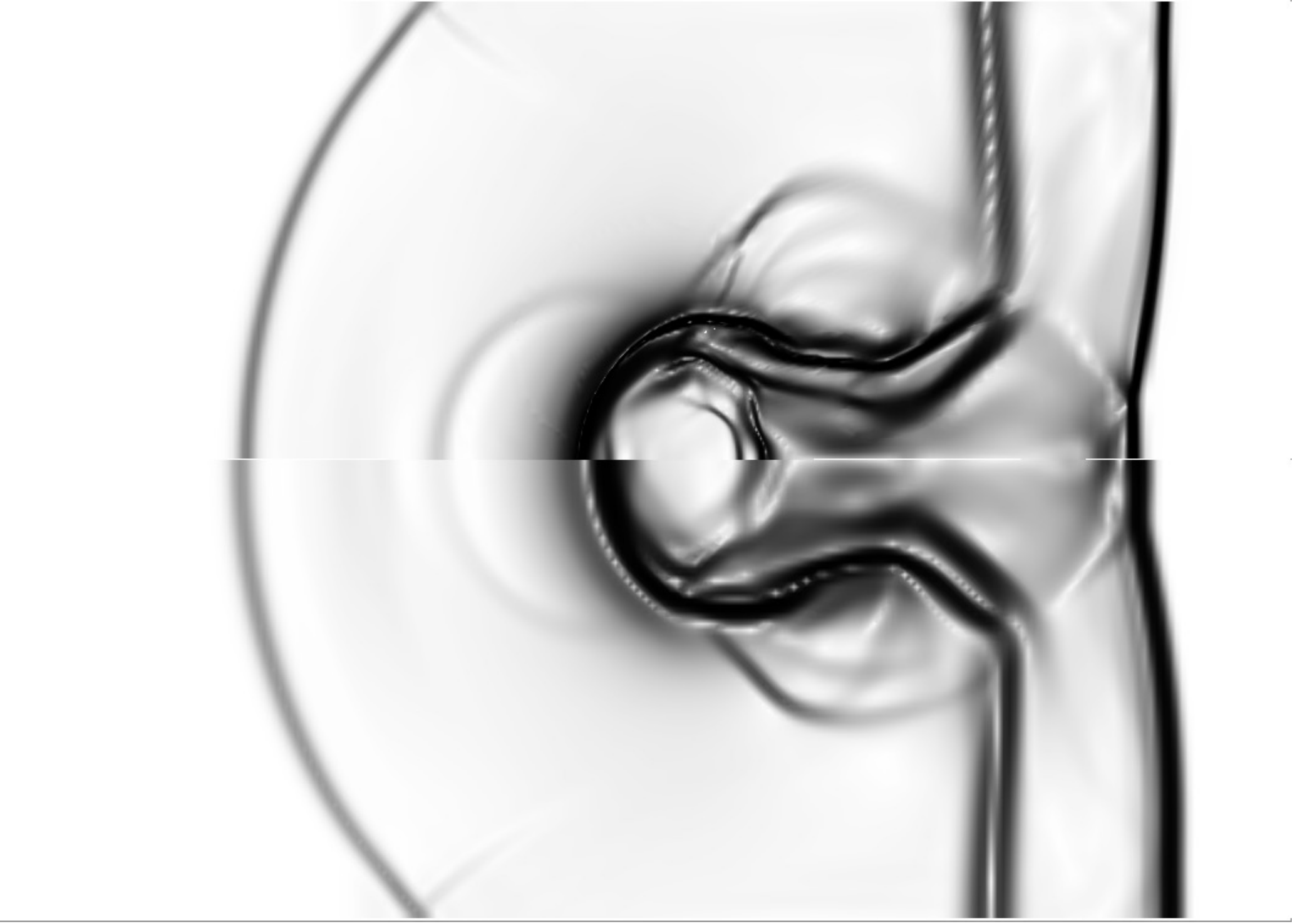}
  \includegraphics[width=0.48\textwidth, trim=2 2 2 2, clip]{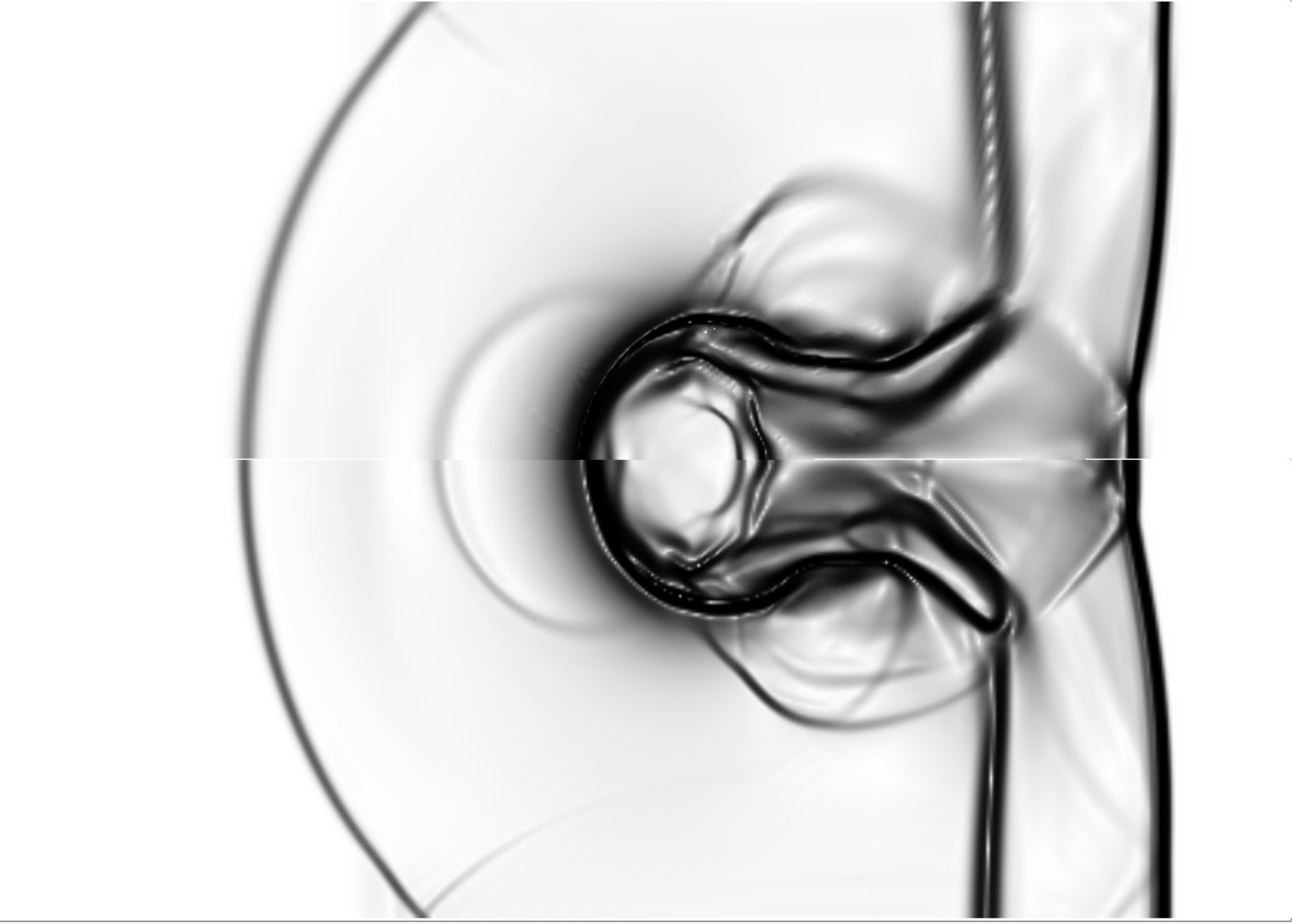}
  \caption{Same as Figure \ref{fig:RMHD_3DSC_phi1} except for $\phi_2$.}
  \label{fig:RMHD_3DSC_phi2}
\end{figure}

%% file: Conclusion.tex
\section{Conclusions}\label{section:Conclusion}
This paper presented the high-order accurate ES adaptive moving mesh schemes for the 2D and 3D special RHD and RMHD equations.
Our schemes were built on the ES finite difference approximation in the curvilinear coordinates, the discrete GCLs, and the adaptive mesh redistribution  built on the minimization of the mesh adaption functional, and  consisted of the following main parts.
\begin{enumerate}
 \item The two-point EC flux $\widetilde{\Fcurv_k}$ for the modified RMHD equations (involving the RHD equations) in the curvilinear coordinates for the given entropy pair was first derived,  see \eqref{eq:ECFluxCurv}, and then
     the high-order EC flux $(\widetilde{\Fcurv_k})_{\bm{i},k,\pm\frac12}^{\twop}$ was proposed
 by using some linear combinations of the  two-point EC flux $\widetilde{\Fcurv_k}$,
 so that the approximation of the flux derivatives in space was $2p$th-order accurate, which was an extension of the high-order accurate EC schemes in the Cartesian coordinates \cite{Lefloch2002Fully} to the curvilinear coordinates.
	
\item  The $2p$th-order accurate approximations of the spatial derivatives in the source terms and the VCL were given by
designing $(\widetilde{\Bcurv_k})_{\bm{i},k,\pm\frac12}^{\twop}$ and	$\left(\widetilde{J\pd{\xi_k}{t}}\right)_{\bm{i},k,\pm\frac12}^{\twop}$ as the linear combination of corresponding 2nd-order case with the same coefficients as above. The discretization of the latter degenerated to the $2p$th-order accurate central difference.

	\item The spatial metrics $\left(\widetilde{J\pd{\xi_k}{x_j}}\right)_{\bm{i}}$ used in the above two parts were discretized by using the $2p$th-order central difference based on the conservative metrics method (CMM) \cite{Thomas1979}, such that the SCLs held in the discrete level.

\item The semi-discrete schemes
built on the above three parts, see \eqref{eq:RMHDSemiU_O2p}-\eqref{eq:RMHDSemiJ_O2p}, were proved to be $2p$th-order accurate in space and EC by   mimicking the derivation of the continuous entropy identity in the curvilinear coordinates.

\item Some suitable high-order dissipation term utilizing WENO reconstruction in the scaled entropy variables was added to the EC flux to get the high-order accurate ES schemes satisfying the semi-discrete entropy inequality, in order to avoid the numerical oscillation produced by the EC scheme around the discontinuities.

  \item The fully-discrete ES schemes were obtained by integrating
the above semi-discrete ES schemes in time by
using the third-order accurate explicit strong-stability preserving Runge-Kutta schemes,
and proved to be free-stream preserving.

\item The mesh points were adaptively redistributed by solving the Euler-Lagrange equation of the mesh adaption functional
on the computational mesh at each time step with the suitably chosen monitor functions.
\end{enumerate}

Several 2D and 3D numerical results showed that the high-order accurate ES adaptive moving mesh schemes effectively captured the localized structures, such as the sharp transitions or discontinuities,
and outperformed both their counterparts on the uniform mesh and the 2nd-order ES adaptive moving mesh schemes.

%% file: App.tex
\appendix

\section{1D EC schemes}\label{app:1DEC}
This Appendix presents the semi-discrete 1D EC schemes.
	Consider the case of $d=1$ and omit the subscripts ``$1$"  denoting the $\xi_1$-direction.
	The system \eqref{eq:RMHDCurv} and the GCLs \eqref{eq:GCL} reduce to
	\begin{equation*}
\pd{\Ucurv}{\tau}+\pd{\Fcurv}{\xi}=-\Phi'(\bV)^\mathrm{T}\pd{B_1}{\xi},
	\end{equation*}
	and
	\begin{subequations}
		\begin{align*}
			&\text{VCL:}\quad \pd{J}{\tau}+\dfrac{\partial}{\partial\xi}{\left(\pd{x}{t}\right)}=0,\\
			&\text{SCL:}\quad \dfrac{\partial}{\partial\xi}{\left(J\pd{\xi}{x}\right)}\equiv0,
		\end{align*}
	\end{subequations}
	where
	\begin{equation*}
		J=\pd{x}{\xi},~
		\Ucurv=J\bU,~
		\Fcurv=\left(J\pd{\xi}{t}\bU\right)+\bF=\left(\pd{x}{t}\bU\right)+\bF.
	\end{equation*}
It is easy to see that the SCL holds automatically in this case.
	If replacing $\bm{i}$ with $i$, then the $2p$th-order EC schemes become
	\begin{align*}
		&\dfrac{\dd}{\dd t}\bm{\mathcal{U}}_{i}=
		-\dfrac{1}{\Delta \xi}\left(\widetilde{\Fcurv}_{i+\frac12}^{\twop}-\widetilde{\Fcurv}_{i-\frac12}^{\twop}\right)
		-\Phi'(\bV_{i})^\mathrm{T}\dfrac{1}{\Delta \xi}\left((\widetilde{B_1})_{i+\frac12}^{\twop}-(\widetilde{B_1})_{i-\frac12}^{\twop}\right),
		\\
		&\dfrac{\dd}{\dd t}J_{i}=
		-\dfrac{1}{\Delta \xi}\Bigg(\left(\widetilde{\pd{x}{t}}\right)_{i+\frac12}^{\twop}-\left(\widetilde{\pd{x}{t}}\right)_{i-\frac12}^{\twop}\Bigg),
	\end{align*}
	where
	\begin{align*}
		&\widetilde{\Fcurv}_{i+\frac12}^{\twop}=
		\sum_{n=1}^p\alpha_{p,n}\sum_{s=0}^{n-1}
		\left[\dfrac12\left(\left(\pd{x}{t}\right)_{i-s} + \left(\pd{x}{t}\right)_{i-s+n}\right)\widetilde{\bU}\left(\bU_{i-s}, \bU_{i-s+n}\right)
		+\widetilde{\bF}\left(\bU_{i-s}, \bU_{i-s+n}\right)\right],\\
		&(\widetilde{B_1})_{i+\frac12}^{\twop}=
		\sum_{n=1}^p\alpha_{p,n}\sum_{s=0}^{n-1}\dfrac12\left((B_1)_{i-s}+(B_1)_{i-s+n}\right), \\
		&\left(\widetilde{\pd{x}{t}}\right)_{i+\frac12}^{\twop}=
		\sum_{n=1}^p\alpha_{p,n}\sum_{s=0}^{n-1}
		\dfrac12\left(\left(\pd{x}{t}\right)_{i-s}+\left(\pd{x}{t}\right)_{i-s+n}\right),
	\end{align*}
	and $\left(\pd{x}{t}\right)\Big|_i$ is the mesh velocity at $\xi_i$.
	
\section{2D EC schemes}\label{app:2DEC}
This Appendix presents the semi-discrete 2D EC schemes.
Consider the  case of $d=2$, and
replace $(\xi_1,\xi_2)$ and $(x_1,x_2)$ with $(\xi,\eta)$ and $(x,y)$, respectively.
The system \eqref{eq:RMHDCurv} and the GCLs \eqref{eq:GCL} reduce to
\begin{equation*}
\pd{\Ucurv}{\tau}+\pd{\Fcurv_1}{\xi}+\pd{\Fcurv_2}{\eta}=-\Phi'(\bV)^\mathrm{T}\left(\pd{\Bcurv_1}{\xi} + \pd{\Bcurv_2}{\eta}\right),
\end{equation*}
and
\begin{subequations}
	\begin{align*}
		\text{VCL:}\quad &\pd{J}{\tau}+\dfrac{\partial}{\partial\xi}{\left(J\pd{\xi}{t}\right)}+\dfrac{\partial}{\partial\eta}{\left(J\pd{\eta}{t}\right)}=0,\\
		\text{SCLs:}\quad &\dfrac{\partial}{\partial\xi}{\left(J\pd{\xi}{x}\right)}
		+\dfrac{\partial}{\partial\eta}{\left(J\pd{\eta}{x}\right)}=0,\\
		&\dfrac{\partial}{\partial\xi}{\left(J\pd{\xi}{y}\right)}
		+\dfrac{\partial}{\partial\eta}{\left(J\pd{\eta}{y}\right)}=0,
	\end{align*}
\end{subequations}
where
\begin{align*}
	&J=\pd{x}{\xi}\pd{y}{\eta} - \pd{x}{\eta}\pd{y}{\xi},~
	\Ucurv=J\bU, \\
	&\Fcurv_1=\left(J\pd{\xi}{t}\bU\right)+\left(J\pd{\xi}{x}\right)\bF_1+\left(J\pd{\xi}{y}\right)\bF_2,
	~\Bcurv_1=\left(J\pd{\xi}{x}\right)B_1+\left(J\pd{\xi}{y}\right)B_2,\\
	&\Fcurv_2=\left(J\pd{\eta}{t}\bU\right)+\left(J\pd{\eta}{x}\right)\bF_1+\left(J\pd{\eta}{y}\right)\bF_2,
	~\Bcurv_2=\left(J\pd{\eta}{x}\right)B_1+\left(J\pd{\eta}{y}\right)B_2.
\end{align*}
If replacing $\bm{i}$ with $\{i,j\}$, then the $2p$th-order EC schemes become
\begin{align*}
	\dfrac{\dd}{\dd t}\bm{\mathcal{U}}_{i,j}=
	&-\dfrac{1}{\Delta \xi}\left((\widetilde{\Fcurv_1})_{i+\frac12,j}^{\twop}-(\widetilde{\Fcurv_1})_{i-\frac12,j}^{\twop}\right)
	-\dfrac{1}{\Delta \eta}\left((\widetilde{\Fcurv_2})_{i,j+\frac12}^{\twop}-(\widetilde{\Fcurv_2})_{i,j-\frac12}^{\twop}\right)\\
	&-\Phi'(\bV_{i,j})^\mathrm{T}\dfrac{1}{\Delta \xi}\left((\widetilde{\Bcurv_1})_{i+\frac12,j}^{\twop}-(\widetilde{\Bcurv_1})_{i-\frac12,j}^{\twop}\right)
	-\Phi'(\bV_{i,j})^\mathrm{T}\dfrac{1}{\Delta \eta}\left((\widetilde{\Bcurv_2})_{i,j+\frac12}^{\twop}-(\widetilde{\Bcurv_2})_{i,j-\frac12}^{\twop}\right),
	\\
	\dfrac{\dd}{\dd t}J_{i,j}=
	&-\dfrac{1}{\Delta \xi}\Bigg(\left(\widetilde{J\pd{\xi}{t}}\right)_{i+\frac12,j}^{\twop}-\left(\widetilde{J\pd{\xi}{t}}\right)_{i-\frac12,j}^{\twop}\Bigg)
	-\dfrac{1}{\Delta \eta}\Bigg(\left(\widetilde{J\pd{\eta}{t}}\right)_{i,j+\frac12}^{\twop}
-\left(\widetilde{J\pd{\eta}{t}}\right)_{i,j-\frac12}^{\twop}\Bigg),
\end{align*}
where
\begin{align*}
	(\widetilde{\Fcurv_1})_{i+\frac12,j}^{\twop}=
	\sum_{n=1}^p\alpha_{p,n}\sum_{s=0}^{n-1}
	\Bigg[&\dfrac12\left(\left(J\pd{\xi}{t}\right)_{i-s,j}
	+ \left(J\pd{\xi}{t}\right)_{i-s+n,j}\right)\widetilde{\bU}
	\left(\bU_{i-s,j}, \bU_{i-s+n,j}\right) \\
	+&\dfrac12\left(\left(J\pd{\xi}{x}\right)_{i-s,j}
	+ \left(J\pd{\xi}{x}\right)_{i-s+n,j}\right)
	\widetilde{\bF_1}\left(\bU_{i-s,j}, \bU_{i-s+n,j}\right)\\
	+&\dfrac12\left(\left(J\pd{\xi}{y}\right)_{i-s,j}
	+ \left(J\pd{\xi}{y}\right)_{i-s+n,j}\right)
	\widetilde{\bF_2}\left(\bU_{i-s,j}, \bU_{i-s+n,j}\right)\Bigg],\\
	(\widetilde{\Fcurv_2})_{i,j+\frac12}^{\twop}=
	\sum_{n=1}^p\alpha_{p,n}\sum_{s=0}^{n-1}
	\Bigg[&\dfrac12\left(\left(J\pd{\eta}{t}\right)_{i,j-s}
	+ \left(J\pd{\eta}{t}\right)_{i,j-s+n}\right)\widetilde{\bU}
	\left(\bU_{i,j-s}, \bU_{i,j-s+n}\right) \\
	+&\dfrac12\left(\left(J\pd{\eta}{x}\right)_{i,j-s}
	+ \left(J\pd{\eta}{x}\right)_{i,j-s+n}\right)
	\widetilde{\bF_1}\left(\bU_{i,j-s}, \bU_{i,j-s+n}\right)\\
	+&\dfrac12\left(\left(J\pd{\eta}{y}\right)_{i,j-s}
	+ \left(J\pd{\eta}{y}\right)_{i,j-s+n}\right)
	\widetilde{\bF_2}\left(\bU_{i,j-s}, \bU_{i,j-s+n}\right)\Bigg],\\
	(\widetilde{\Bcurv_1})_{i+\frac12,j}^{\twop}=
	\sum_{n=1}^p\alpha_{p,n}\sum_{s=0}^{n-1}
	&\Bigg[\dfrac14\left(\left(J\pd{\xi}{x}\right)_{i-s,j}
	+ \left(J\pd{\xi}{x}\right)_{i-s+n,j}\right)
	\left((B_1)_{i-s,j}+(B_1)_{i-s+n,j}\right) \\
	&+\dfrac14\left(\left(J\pd{\xi}{y}\right)_{i-s,j}
	+ \left(J\pd{\xi}{y}\right)_{i-s+n,j}\right)
	\left((B_2)_{i-s,j}+(B_2)_{i-s+n,j}\right)\Bigg], \\
	(\widetilde{\Bcurv_2})_{i,j+\frac12}^{\twop}=
	\sum_{n=1}^p\alpha_{p,n}\sum_{s=0}^{n-1}
	&\Bigg[\dfrac14\left(\left(J\pd{\eta}{x}\right)_{i,j-s}
	+ \left(J\pd{\eta}{x}\right)_{i,j-s+n}\right)
	\left((B_1)_{i,j-s}+(B_1)_{i,j-s+n}\right) \\
	&+\dfrac14\left(\left(J\pd{\eta}{y}\right)_{i,j-s}
	+ \left(J\pd{\eta}{y}\right)_{i,j-s+n}\right)
	\left((B_2)_{i,j-s}+(B_2)_{i,j-s+n}\right)\Bigg], \\
	\left(\widetilde{J\pd{\xi}{t}}\right)_{i+\frac12,j}^{\twop}=
	\sum_{n=1}^p\alpha_{p,n}\sum_{s=0}^{n-1}
	&\dfrac12\left(\left(J\pd{\xi}{t}\right)_{i-s,j}+\left(J\pd{\xi}{t}\right)_{i-s+n,j}\right),\\
	\left(\widetilde{J\pd{\eta}{t}}\right)_{i,j+\frac12}^{\twop}=
	\sum_{n=1}^p\alpha_{p,n}\sum_{s=0}^{n-1}
	&\dfrac12\left(\left(J\pd{\eta}{t}\right)_{i,j-s}+\left(J\pd{\eta}{t}\right)_{i,j-s+n}\right),\\
	\left(J\pd{\xi}{t}\right)_{i,j}=-(\dot{x})_{i,j}\left(J\pd{\xi}{x}\right)_{i,j}&
	-(\dot{y})_{i,j}\left(J\pd{\xi}{y}\right)_{i,j},\\
	\left(J\pd{\eta}{t}\right)_{i,j}=-(\dot{x})_{i,j}\left(J\pd{\eta}{x}\right)_{i,j}&
	-(\dot{y})_{i,j}\left(J\pd{\eta}{y}\right)_{i,j},\\
	\left(J\pd{\xi}{x}\right)_{i,j}=+\sum_{n=1}^p\dfrac{\alpha_{p,n}}{2}
	(y_{i,j+n}&-y_{i,j-n}),\\
	\left(J\pd{\xi}{y}\right)_{i,j}=-\sum_{n=1}^p\dfrac{\alpha_{p,n}}{2}
	(x_{i,j+n}&-x_{i,j-n}),\\
	\left(J\pd{\eta}{x}\right)_{i,j}=-\sum_{n=1}^p\dfrac{\alpha_{p,n}}{2}
	(y_{i+n,j}&-y_{i-n,j}),\\
	\left(J\pd{\eta}{y}\right)_{i,j}=+\sum_{n=1}^p\dfrac{\alpha_{p,n}}{2}
	(x_{i+n,j}&-x_{i-n,j}).
\end{align*}